\documentclass[oneside]{amsbook}

\usepackage{graphics,color,pgf,comment}
\usepackage{epsfig}

 \usepackage[ansinew]{inputenc}
 \usepackage[all]{xy}
 \usepackage{hyperref}
\newdir{ >}{!/8pt/@{}*@{>}}

\usepackage{amssymb, amsmath,amsthm, mathtools}

\makeatletter
\newtheorem*{rep@theorem}{\rep@title}
\newcommand{\newreptheorem}[2]{%
\newenvironment{rep#1}[1]{%
 \def\rep@title{#2 \ref{##1}}%
 \begin{rep@theorem}}%
 {\end{rep@theorem}}}
\makeatother

\theoremstyle{plain}
\newtheorem{Teorema}{Theorem}[section]
\newtheorem{thm}[Teorema]{Theorem}
\newreptheorem{theorem}{Theorem}  
\newtheorem{cor}[Teorema]{Corollary}
\newtheorem{Lemma}[Teorema]{Lemma}
\newtheorem{lemma}[Teorema]{Lemma}
\newtheorem{Corollario}[Teorema]{Corollary}
\newtheorem{Proposizione}[Teorema]{Proposition}
\newtheorem{prop}[Teorema]{Proposition}

\newtheorem{thm_intro}{Theorem}
\newtheorem{cor_intro}[thm_intro]{Corollary}
\newtheorem{quest_intro}[thm_intro]{Question}

\theoremstyle{definition}
\newtheorem{Definizione}[Teorema]{Definition}

\newtheorem{defn}[Teorema]{Definition}
\newtheorem{quest}[Teorema]{Question}

\theoremstyle{remark}
\newtheorem{Nota}[Teorema]{Note}

\newtheorem{Notazione}[Teorema]{Notation}
\newtheorem{Osservazione}[Teorema]{Remark}
\newtheorem{rem}[Teorema]{Remark}

\newtheorem{Esempio}[Teorema]{Example}
\newtheorem{Esempi}[Teorema]{Examples}

\newtheorem*{thm:repeated:isometry-lemma}{Theorem \ref{isometry-lemma-intro}}
\newtheorem*{thm:repeated:Van-Theor-intro}{Theorem \ref{Van-Theor-intro}}
\newtheorem*{thm:repeated:Fin-Theor-intro}{Theorem \ref{Fin-Theor-intro}}
\newtheorem*{thm:repeated:vanishing1_intro}{Theorem \ref{vanishing1_intro}}
\newtheorem*{thm:repeated:homotopy-weak-intro}{Theorem \ref{homotopy-weak-intro}}
\newtheorem*{thm:repeated:mapping1_intro}{Theorem \ref{mapping1_intro}}
\newtheorem*{thm:repeated:mapping-intro}{Theorem \ref{mapping-intro}}
\newtheorem*{thm:repeated:van-cor-intro}{Corollary \ref{van-cor-intro}}
\newtheorem*{thm:repeated:invisible:new-intro}{Theorem \ref{invisible:new-intro}}
\newtheorem*{thm:repeated:tre-intro}{Theorem \ref{tre-intro}}
\newtheorem*{thm:repeated:va-cor-intro}{Corollary \ref{va-cor-intro}}
\newtheorem*{thm:repeated:fi-cor-intro}{Corollary \ref{fi-cor-intro}}

\newcommand{\wdtX}{\widetilde{X}}

\newcommand{\G}{\ensuremath {\Gamma}}
\newcommand{\calC} {\ensuremath {\mathcal{C}}}
\newcommand{\calP} {\ensuremath {\mathcal{P}}}
\newcommand{\calS} {\ensuremath {\mathcal{S}}}
\newcommand{\calM} {\ensuremath {\mathcal{M}}}

\newcommand{\calK} {\ensuremath {\mathcal{K}}}
\newcommand{\calL} {\ensuremath {\mathcal{L}}}
\newcommand{\calA} {\ensuremath {\mathcal{A}}}
\newcommand{\calO} {\ensuremath {\mathcal{O}}}
\newcommand{\calR} {\ensuremath {\mathcal{R}}}
\newcommand{\calW} {\ensuremath {\mathcal{W}}}

\DeclareMathOperator{\Hom}{Hom}
\DeclareMathOperator{\aut}{Aut}
\DeclareMathOperator{\mult}{mult}
\DeclareMathOperator{\supp}{supp}

\DeclareMathOperator{\inte}{int}
\DeclareMathOperator{\id}{Id}
\DeclareMathOperator{\vol}{Vol}

\newcommand{\calU}{\ensuremath {\mathcal{U}}}
\newcommand{\matN}{\ensuremath {\mathbb{N}}}
\newcommand{\R} {\ensuremath {\mathbb{R}}}

\newcommand{\matZ} {\ensuremath {\mathbb{Z}}}

\newcommand{\hl}{\ensuremath{H^{\ell_1}}}
\newcommand{\cl}{\ensuremath{C^{\ell_1}}}

\newcommand{\lf}{\ensuremath{\textrm{lf}}}

\newcommand{\alt}{\ensuremath{\mathrm{alt}}}
\newcommand{\red}{\ensuremath{\mathrm{red}}}
\newcommand{\actson}{\ensuremath{\curvearrowright}}

\newcommand{\simf}{\ensuremath{\mathfrak{S}_{\rm fin}}}
\newcommand{\todo}[1]{\vspace{5mm}\par \noindent
\framebox{\begin{minipage}[c]{0.95 \textwidth} \tt #1\end{minipage}} \vspace{5mm} \par}
\numberwithin{section}{chapter}
\numberwithin{subsection}{section}

\title{Gromov's theory of multicomplexes\\ with applications to bounded cohomology \\ and simplicial volume}
\author[]{Roberto Frigerio}
\address{Dipartimento di Matematica, Universit\`a di Pisa, Largo B. Pontecorvo 5, 56127 Pisa, Italy}
\email{roberto.frigerio@unipi.it}

\author[]{Marco Moraschini}
\address{Fakult\"{a}t f\"{u}r Mathematik, Universit\"{a}t Regensburg, 93040 Regensburg, Germany}
\email{marco.moraschini@ur.de}

\thanks{The second author was supported both by the \emph{PhD program} of the University of Pisa and by the CRC~1085 \emph{Higher Invariants}
  (Universit\"at Regensburg, funded by the DFG)}

\begin{document}

\begin{abstract}
The simplicial volume is a homotopy invariant of manifolds introduced by Gromov in his pioneering paper \emph{Volume and bounded cohomology}. 
In order to study the main properties of simplicial volume, 
Gromov himself initiated the dual theory of bounded cohomology, which then developed into a very active and independent research field. 
Gromov's theory of bounded cohomology of topological spaces was based on the use of multicomplexes, which are simplicial structures that generalize
simplicial complexes without allowing all the degeneracies appearing in simplicial sets. 

The first aim of this paper is to lay the foundation of the theory of multicomplexes. After setting the main definitions, we construct the singular multicomplex
$\calK(X)$ associated
to a topological space $X$, and we prove that the geometric realization of $\calK(X)$ is homotopy equivalent to $X$ for every CW complex $X$. Following Gromov, we introduce the notion
of completeness, which, roughly speaking, translates into the context of multicomplexes the Kan condition for simplicial sets. We then develop the homotopy theory of complete
multicomplexes, and we show that $\calK(X)$ is complete for every CW complex $X$.

In the second part of this work we apply the theory of multicomplexes to the study of the bounded cohomology of topological spaces. 
Our constructions and arguments culminate in the 
complete proofs of Gromov's Mapping Theorem (which implies
in particular that the bounded cohomology of a space only depends on its fundamental group) and of Gromov's Vanishing Theorem, which ensures the vanishing of the simplicial
volume of 
closed manifolds admitting an amenable cover of small multiplicity. 

The third and last part of the paper is devoted to the study of locally finite chains on non-compact spaces, hence to the simplicial volume of open manifolds. We expand some ideas of Gromov to provide detailed proofs of a criterion for the vanishing and a criterion for the finiteness of the simplicial volume of open manifolds. As a by-product of these results,
we prove a criterion for the $\ell^1$-invisibility of closed manifolds in terms of amenable covers. As an application, we give the first detailed proof of the vanishing
of the simplicial volume of the product of three open manifolds. 
\end{abstract}

\maketitle

\tableofcontents

\chapter*{Introduction}

\subsection*{Simplicial volume}
The simplicial volume is an invariant of manifolds introduced by Gromov in his pioneering paper \emph{Volume and bounded cohomology}~\cite{Grom82}. If $X$
is a topological space and $S_n(X)$ denotes the set of singular $n$-simplices with values in $X$, then the space $C_n(X)$ of singular $n$-chains with real coefficients is endowed with the 
$\ell^1$-norm defined by
$$
\left\| \sum_{\sigma\in S_n(X)} a_\sigma \sigma\right\| = \sum_{\sigma\in S_n(X)} |a_\sigma|\ .
$$ 
This norm descends to a seminorm on the homology $H_*(X)$ with real coefficients, and if $M$ is a closed oriented $n$-dimensional manifold, then the simplicial volume $\|M\|$ of $M$
is the $\ell^1$-seminorm of the real fundamental class $[M]\in H_n(M)$.

\subsection*{Simplicial volume and Riemannian geometry}
While being defined only in terms of singular chains, the simplicial volume is deeply related to many invariants of geometric nature. 
For example,
it vanishes for manifolds admitting a Riemannian metric with non-negative Ricci curvature, and it is positive for negatively curved manifolds~\cite{Grom82,Inoue}. 
Indeed, Gromov's paper~\cite{Grom82} describes many phenomena which relate the simplicial volume to other invariants of differential geometric nature. 
If $M$ is a compact Riemannian manifold, the so-called \emph{Gromov's main inequality} establishes an explicit upper bound for the simplicial volume in terms of the Riemannian  volume and the Ricci curvature of $M$. Using this, Gromov provided new estimates from below for the minimal volume in terms of the simplicial volume 
(the minimal volume of a Riemannian manifold $M$ is the infimum of the volumes of complete Riemannian metrics on $M$ subject to the condition that the absolute value of all sectional curvatures is  not bigger than $1$). In particular, Gromov proved in~\cite{Grom82} that the non-vanishing of the simplicial volume of $M$ implies the non-vanishing  of
the minimal volume of $M$, which in turn implies that $M$ cannot collapse while keeping the curvature uniformly bounded (i.e.~there exists no sequence of Riemannian
metrics on $M$ with uniformly bounded sectional curvature whose volumes converge to $0$). This result was exploited e.g.~in the proof of a collapsing theorem which plays a fundamental role in  
the last step of Perelman's proof of Thurston's Geometrization Conjecture~\cite{BBBMP,geom:book}.

Similar estimates hold when one replaces the minimal volume with the minimal entropy, an invariant which, roughly speaking, measures the rate of growth of balls in the universal covering. 
Gromov proved in~\cite{Grom82} that the non-vanishing of the simplicial volume implies the non-vanishing  of
 the minimal entropy. Since then, the simplicial volume has
 been extensively exploited in 
the study of volumes of balls and of systolic inequalities in
 Riemannian manifolds (see e.g.~\cite{Samb1,Guth,Bal}). We refer the reader to~\cite{BaSa} for recent results in this direction.

Even more can be said when working with manifolds with constant curvature or, more in general, with locally symmetric spaces.
For closed hyperbolic manifolds, a fundamental result by Gromov and Thurston shows that $\|M\|=\vol(M)/v_n$, where $v_n$
 is a constant only depending on the dimension $n$ of $M$~\cite{Thurston, Grom82}. This 
result (which provides one of the few exact computations of the simplicial volume) describes the hyperbolic volume explicitly in terms of a topological invariant, and plays a 
fundamental role in a celebrated proof of Mostow Rigidity Theorem due to Gromov and Thurston. 

  More in general, Gromov's Proportionality Principle~\cite{Grom82} ensures that the ratio $\|M\|/\vol(M)$ between the simplicial 
 volume and the classical volume of a closed Riemannian manifold only depends on the isometry type of the universal covering of $M$ (see also
 \cite{Thurston, Loh, Bucher, Frigerio, Franceschini1,Str-Lip}). 
 Lafont and Schmidt proved in~\cite{Lafont-Schmidt}
that the proportionality constant between $\|M\|$ and $\vol(M)$ is positive for every
locally symmetric space of non-compact type (see also~\cite{michelleprimo} for a case not covered by
Lafont and Schmidt's argument). Thus locally symmetric spaces of non-compact type have positive simplicial volume, and this answers a 
 question posed by Gromov~\cite{Grom82}.

Other long-standing questions on the relationship between curvature and  simplicial volume for Riemannian manifolds  are still unsolved.
 For example, Gromov conjectured in~\cite{Grom82} that the simplicial volume of a non-positively curved manifold with negative Ricci curvature should be positive. 
 We refer the reader to~\cite{CW1,CW2} for some recent progress on this topic.

 The simplicial volume has also been studied in relation with the Chern Conjecture, which predicts that the Euler characteristic of a closed manifold
 admitting an affine structure should vanish~\cite{BuGe,BCL}, or, equivalently, that if the tangent bundle $TM$ of a closed manifold $M$ admits a torsion-free flat connection, then
its Euler number of $TM$ should vanish. We refer the reader e.g.~to \cite[Chapter 13]{frigerio-libro} for  the use of simplicial volume in a possible approach to the Chern Conjecture.

\subsection*{Simplicial volume and the topology of manifolds}

Of course, the simplicial volume has found applications also to problems  which are more topological in nature. The elementary remark that the degree
of any map $f\colon M\to N$ between closed manifolds of the same dimension is bounded above by the ratio $\|M\|/\|N\|$ laid the foundation for many results
on domination between manifolds (an orientable manifold $M$ dominates the orientable manifold $N$ if there exists a map $f\colon M\to N$ of non-vanishing degree).
We refer the reader e.g.~to~\cite{Derbez,BRW,Neo1,Neo2,Neo3} for some results on domination between 3-manifolds that were obtained via the study of the simplicial volume. 

Also when dealing with the topological features of simplicial volume, it turns out that some long-standing questions are still unsolved. For example,
 Gromov asked in~\cite[page~232]{gromovasymptotic}  whether the $\textup{L}^2$-Betti numbers  of a closed aspherical manifold $M$ with $\|M\|=0$ should vanish
 (see also \cite[3.1.~(e) on page~769]{gromov-cycles}). If answered in the affirmative, Gromov's question would imply, for example, that if $M$ is an orientable aspherical closed manifold
 admitting a self-map of degree $d\notin \{-1,0,1\}$, then the Euler characteristic of $M$ vanishes. 
 In order to approach his question, Gromov himself introduced a variation of the classical simplicial volume,
 called the \emph{integral foliated simplicial volume}~\cite{Gromovbook}. For some results on this topic we refer e.g.~to~\cite{Sauer, Sthesis, FFM,LP,FLPS,FLF,FLMQ}.

Other (even more elementary) open questions on the simplicial volume are related to products and fiber bundles. 
If $M,F$ are closed orientable manifolds of dimension $m,n$ respectively, then  the simplicial volume of $M\times F$ satisfies the bounds
$
\| M\|\cdot \| F\|\leq \| M\times F\|\leq 
c_{n,m}
\|M\| \cdot \| F\|
$, where $c_{n,m}$ is a constant only depending on $n,m$
(see e.g.~\cite{Grom82}).
Of course, products are just a special case of fiber bundles, so one may wonder whether the simplicial volume of the total space $E$ of
a fiber bundle with base $M$ and fiber $F$ could be estimated in terms of the simplicial volumes of $M$ and $F$. 
In the case when $F$ is a surface it was proved
in~\cite{HK} that the inequality $\| M\|\cdot \| F\|\leq \|E\|$ still holds. This estimate was then improved in~\cite{Michelle:fiber}. It it still an open question whether
the inequality
$
\|E\|\geq \|M\|\cdot \|F\|
$
holds for every fiber bundle $E$ with fiber $F$ and base $M$, without any restriction on the dimensions of $F$ and $M$.

\bigskip

The simplicial volume has been exploited to the study of the topology and of the geometry of manifolds in further directions. Here we just mention the following ones:
the study of the geometry of smooth maps (e.g.~via the analysis of the complexity of their critical sets and of their fibers) ~\cite{gromov-cycles,GuthGromov};
the counting of trajectories of vector fields~\cite{Katz, Alpert1, Alpert2,Campa}.

\subsection*{Bounded cohomology}
Computing the simplicial volume has proved to be a very challenging task. Besides hyperbolic manifolds, the only other closed manifolds for which the exact value of the simplicial volume
is known are $4$-dimensional manifolds locally isometric to the product 
$\mathbb{H}^2\times \mathbb{H}^2$
 of two hyperbolic planes~\cite{Bucher3} (for example,
the product of two hyperbolic surfaces). Starting from these examples, more values for the simplicial volume can be obtained by taking connected sums or amalgamated sums along submanifolds with an amenable fundamental group. 

In order to study the simplicial volume, Gromov himself developed in~\cite{Grom82} the dual theory of \emph{bounded cohomology}. If $\varphi\in C^n(X)$ is a singular cochain with real coefficients, then one can define the $\ell^\infty$-norm of $\varphi$ by setting
$$
\|\varphi\|_\infty =\sup \{ |\varphi(\sigma)|\, ,\ \sigma\in S_n(X)\}\ \in\ [0,+\infty]\ .
$$
A cochain is bounded if its $\ell^\infty$-norm is finite. It is easily seen that bounded cochains define a subcomplex $C^*_b(X)$ of $C^*(X)$, whose cohomology 
is the bounded cohomology $H^*_b(X)$ of $X$. The inclusion $C^*_b(X)\hookrightarrow C^*(X)$ induces the \emph{comparison map}
$$
c^*\colon H^*_b(X)\to H^*(X)\ .
$$
For every $n\in\mathbb{N}$, the normed space $C^n_b(X)$ coincides with the topological dual of $C_*(X)$ (endowed with the $\ell^1$-norm). Using this, it is not difficult to show that the
vanishing of $H^n_b(X)$ implies the vanishing of the $\ell^1$-seminorm on $H_n(X)$ (hence, of the simplicial volume, if $X$ is a closed $n$-dimensional manifold). This and other duality results have been exhaustively exploited in the study of the simplicial volume. For example, Bucher's computation of the simplicial volume of manifolds locally isometric to $\mathbb{H}^2\times \mathbb{H}^2$ takes place in the context of bounded cohomology, and it is based on the study of bounded cocycles rather than of cycles. As a quite surprising consequence, there is no description
of any fundamental cycle for the product of two hyperbolic surfaces whose $\ell^1$-seminorm approximates the value of the simplicial volume.

One of the most peculiar features of the bounded cohomology of a space $X$ is that it only depends on the fundamental group of $X$. In particular, the bounded cohomology of any simply connected space vanishes. 
In order to prove
this fundamental
result (and other related results), Gromov introduced in~\cite{Grom82} the notion of \emph{multicomplex}. Since then, 
multicomplexes have been exploited in several papers on the simplicial volume (see e.g.~\cite{Kuessner, KimKue, strz-unp}).

While being based on very neat geometric ideas, Gromov's
theory of multicomplexes involves a certain amount of technicalities. In view of this, 
a lot of progress in the study of the bounded cohomology of topological spaces has been achieved thanks to an alternative approach
 exploiting tools from homological algebra.
This alternative path to the study of bounded cohomology was initiated by Brooks~\cite{Brooks}
and  developed by
Ivanov in his foundational paper~\cite{Ivanov} (and, years later and in a wider context, by Burger and Monod~\cite{BM1,Monod,BM2}; see also~\cite{Bualg}).
Indeed, the first detailed proof that the bounded cohomology of a space only depends on its fundamental group appeared in~\cite{Ivanov}
(under the assumption that the space be homotopy equivalent to a countable CW complex;  recently, Ivanov modified his proof that applies now to any topological space~\cite{ivanov3}).

In this paper we  come back to Gromov's original approach and exploit multicomplexes to give a detailed proof that 
the bounded cohomology of a space only depends on its fundamental group. In order to apply our argument
(which stays as close as possible to Gromov's) to any topological space, we build
on the invariance of bounded cohomology with respect to weak homotopy equivalences, which was recently proved in~\cite{ivanov3}.
For a self-contained proof based on
the techniques developed here that applies to any 
CW complex (or, more in general, to any 
\emph{good} topological space --
see Definition~\ref{good:def}), we address  the reader to~\cite[Addendum 4.A]{Marco:tesi}. 

We refer the reader e.g.~to~\cite{Monod:inv} for a description of the wide range of applications of bounded cohomology (of spaces and of groups)
to different fields in geometry and algebraic topology (see also the books~\cite{Monod, frigerio-libro}). 


\subsection*{Multicomplexes}
The first aim of this paper is to lay the foundation of the theory of multicomplexes. We believe that renovating the interest towards this topic could be fruitful for at least two
reasons. First of all, approaching the simplicial volume via multicomplexes (rather than from the dual point of view of bounded cohomology) means working with cycles (in a quite concrete way) rather than with 
cocycles (via duality). In our view, this allows a more direct understanding of the topology and the geometry of cycles,
and could be of help in finding new approaches to some long-standing open questions on the simplicial volume. Secondly, while bounded cohomology is very effective 
in dealing with finite singular chains (hence, with the simplicial volume of compact manifolds), the use of multicomplexes is still necessary
in the study of locally finite chains (hence, of the simplicial
volume of open manifolds) -- we refer the reader to~\cite[Remark C.4]{Lothesis} 
for a discussion of this issue. Similarly, we hope that multicomplexes could help to understand
some peculiar features of \emph{relative} bounded cohomology (hence, of the simplicial volume of manifolds with boundary) -- see  the discussion
in Section~\ref{relative:sec}.

Most results proved in this work  on the simplicial volume of \emph{closed} manifolds admit alternative proofs which do not make use of multicomplexes.
On the contrary, 
here we provide the first detailed proofs of several results stated by Gromov in~\cite{Grom82} on the simplicial volume of open manifolds. 
Indeed, in~\cite{Grom82} one may find outlines for the proofs of various results stated here. However, when trying to fill in the details of Gromov's arguments,
we often needed to face substantial difficulties that we were not able to overcome without diverging from Gromov's original path. 
In some cases, Gromov's proofs were so concise that it is not even clear whether our work  follows his ideas or not.
In expanding Gromov's hints, we developed tools and ideas that will hopefully 
be of use also in future works. Unfortunately, we also needed to introduce some (unavoidable, to the best of our efforts) technicalities, and 
we did our best to supply the most technical passages with an informal description of what is going on.

A natural question is whether one could obtain the results proved in this paper by working with simplicial sets rather than with multicomplexes. This
would give the unquestionable advantage of working with well-established and well-understood simplicial structures, for which  an exhaustive literature is available.
However, after trying to translate Gromov's ideas into the less exotic context of simplicial sets, we finally could not
disagree with  Gromov's  deep intuition that multicomplexes should be preferred as effective tools for the study of the bounded cohomology 
of topological spaces. We refer the reader e.g.~to Remark~\ref{rem:embedding:no:simpl:set} for a brief discussion of this issue.

Before stating the main results proved in this paper, let us briefly introduce the notion of multicomplex. 
As stated in~\cite{Grom82}, a multicomplex is 
``a set $K$ divided
into the union of closed affine simplices $\Delta_i$, $i\in I$,
such that the intersection of any
two simplices $\Delta_i\cap \Delta_j$ is a (simplicial) subcomplex in $\Delta_i$ as well as in $\Delta_j$''. More formally, a multicomplex is an unordered $\Delta$-complex in which every simplex
has distinct vertices, or, equivalently, a symmetric simplicial set in which every non-degenerate simplex has distinct vertices.
We refer the reader to Definition~\ref{multicomplex:def} for the precise definition of multicomplex, and to 
Section~\ref{comparison:subsec} for a thorough discussion of the relationship between
multicomplexes and other well-known simplicial structures.

\subsection*{The singular multicomplex} If $X$ is a topological space, then the \emph{singular multicomplex} $\calK(X)$ associated to $X$ is a multicomplex
whose $n$-simplices are given by the singular $n$-simplices with distinct vertices in $X$, up to affine automorphisms of the standard simplex $\Delta^n$.
This definition evokes the notion of \emph{singular complex} $\calS(X)$ of $X$, which is 
 the simplicial set having as simplices the singular simplices with values in $X$. The geometric realizations $|\calS(X)|$ and $|\calK(X)|$ of both $\calS(X)$ and $\calK(X)$ come equipped with natural projections
 $j\colon |\calS(X)|\to X$, $S\colon |\calK(X)|\to X$. A fundamental result in the theory of simplicial sets ensures that $j$ is a weak homotopy equivalence. In Section~\ref{Sec:weak:hom:type:K(X)} we establish the same result for the map $S$:
  
 \begin{thm_intro}
 \label{homotopy-weak-intro}
 Let $X$ be a good space.
  The natural projection
  $$
  S\colon |\calK(X)|\to X
  $$
  is a weak homotopy equivalence.
 \end{thm_intro}
 We refer the reader to Definition~\ref{good:def} for the definition of good topological
 space. Here we just anticipate that every CW complex is good.

Gromov states in~\cite[Ex.~(c), page 42]{Grom82} that Theorem~\ref{homotopy-weak-intro} can be proved by a ``standard argument'', and refers the reader to the paper~\cite{moore}, 
which describes
some applications to homotopy theory  of the classical theory of simplicial sets.
However, 
due to the lack of functoriality of the singular multicomplex (see Section~\ref{Sec:no:functorial}) we were not able to adapt 
the classical ideas described in Moore's paper from the context of simplicial sets to the context of multicomplexes. Indeed, our proof 
of Theorem~\ref{homotopy-weak-intro}
is rather inspired to
Milnor's paper~\cite{milnor-geom}. It is also worth mentioning that 
Theorem~\ref{homotopy-weak-intro} is stated  in~\cite{Grom82} without any assumption on the topology of $X$. 
The question whether
Theorem~\ref{homotopy-weak-intro} could hold for any topological space is discussed in Question~\ref{question:whe} and in 
Remarks~\ref{nospecial:rem} and~\ref{controesempioa:rem}.

\subsection*{The Isometry Lemma} As mentioned above, Ivanov recently proved that weak homotopy equivalences induce isometric isomorphisms on bounded cohomology~\cite{ivanov3}. Therefore, by Theorem~\ref{homotopy-weak-intro} we can compute
the bounded cohomology of a good space $X$ by looking at the singular bounded cohomology of $|\calK(X)|$. Just as for simplicial sets, the simplicial cohomology of a multicomplex $K$
(which will be denoted by the symbol $H^*(K)$)
is isomorphic to the singular cohomology of its geometric realization. There is no hope to extend this result to bounded cohomology for general multicomplexes
(for example, if a multicomplex $K$ has a finite number of simplices, then its simplicial bounded cohomology $H^*_b(K)$ obviously coincides with the ordinary simplicial cohomology $H^*(K)$, which is in turn
isomorphic to the singular cohomology $H^*(|K|)$ of its geometric realization; but $H^*(|K|)$ is in general very far from being isomorphic to the singular bounded cohomology $H^*_b(|K|)$ of $|K|$). However, 
Gromov showed in~\cite{Grom82} that the inclusion of
bounded simplicial cochains into bounded singular cochains induces an isometric isomorphism $H^*_b(K)\to H^*_b(|K|)$ provided that the multicomplex $K$ is \emph{complete} and \emph{large}. We provide a detailed proof of this fact in Section~\ref{isometry:lemma:sec}:

\begin{thm_intro}[Isometry Lemma, {[Gro82, page 43]}]\label{isometry-lemma-intro}
Let $K$ be a complete and large multicomplex. Then for every $n\in\mathbb{N}$ there exists a canonical isometric isomorphism
$$
H^n_b(K)\to H^n_b(|K|)\ .
$$
\end{thm_intro}

\subsection*{Complete, minimal and aspherical multicomplexes}
A multicomplex $K$ is large if every connected component of $K$ contains infinitely many vertices. Completeness is a much subtler notion, which evokes in the context of multicomplexes the Kan condition for simplicial sets (see Definition~\ref{completeness:defn} and Remark~\ref{kan:rem}). Completeness
is particularly useful in the study of homotopy groups of geometric realizations. Indeed, the combinatorial description of homotopy groups of Kan simplicial sets is a classical topic in the theory of simplicial sets, 
and in Chapter~\ref{chap2:hom} we develop a similar theory for complete multicomplexes. 
In particular, following Gromov, we introduce the notion of minimal multicomplex
and of aspherical multicomplex, and we show that to every large and complete multicomplex $K$ there are associated a minimal and complete
multicomplex $L$ homotopy equivalent to $K$, and an aspherical, minimal and complete multicomplex $A$ having the same fundamental group of $K$.

A crucial intuition of Gromov is that, if $K,L$ and $A$ are as above, then, for every $n\in\mathbb{N}$,
 the $n$-skeleton of the aspherical multicomplex $A$ is the quotient of the $n$-skeleton of $L$ with respect to the action of an amenable group. This result ultimately depends
 (in a very indirect way)
 on the fact that $A$ is obtained from $L$ by killing all higher homotopy groups, and higher homotopy groups are abelian, hence amenable. The 
 fact that higher homotopy groups are abelian
is exploited also by Ivanov in his proof that the bounded cohomology of a space only depends on its fundamental group, which is based on 
the construction of a tower of weakly principal bundles with amenable (but infinite-dimensional!) structure groups.  
Here we develop  Gromov's intuition into self-contained and complete proofs.
To this aim,
we need to introduce some modifications not only to Gromov's arguments, but even to some (fundamental) definitions given in~\cite{Grom82}.
For example, our definition of the group of simplicial automorphisms $\G$ such that $L/\G\cong A$ is different from Gromov's. 
We refer the reader to Remark~\ref{rem:ip:agg:gruppi:gromov} for a detailed discussion of this issue (see also Remark~\ref{nohomorem}).

The combinatorial description of the homotopy groups of a complete multicomplex is carried out in Theorems~\ref{complete:special} and~\ref{complete:minimal:special},
and it plays a fundamental role in many of our proofs. For example, it allows us to prove that the singular multicomplex $\calK(X)$ associated to a good topological space is complete (see Theorem~\ref{K(X)-compl}), and 
to establish useful criteria to recognize completeness, minimality and asphericity of a given multicomplex (see Proposition~\ref{aspherical:char}). 
Theorems~\ref{complete:special} and~\ref{complete:minimal:special} are basically taken for granted in~\cite{Grom82},
where they are never explicitly stated (but they are extensively used). 

\subsection*{Bounded cohomology is determined by the fundamental group}
The second part of this paper is devoted to the use of multicomplexes in the study of bounded cohomology and of the simplicial volume of closed manifolds.
As just mentioned, a key result in the theory of multicomplexes is that, for every good topological space $X$, the singular multicomplex $\calK(X)$ is complete. 
Therefore, putting together Theorem~\ref{homotopy-weak-intro} with Theorem~\ref{isometry-lemma-intro} and the invariance of bounded cohomology
with respect to weak homotopy equivalences~\cite{ivanov3}, one obtains that $H^*_b(X)$ is isometrically isomorphic to
the bounded simplicial cohomology $H^*_b(\calK(X))$ of $\calK(X)$ for every good topological space $X$ (Gromov
states the same result for \emph{any} topological space in~\cite{Grom82}, see Remark~\ref{controesempioa:rem} for a discussion of this issue). 
Moreover,
one may define a complete and minimal multicomplex $\calL(X)$ and a complete, minimal and aspherical multicomplex $\calA(X)$ associated to $X$. 
Since the action of an amenable
group is invisible to bounded cohomology, from all these facts we can deduce that the bounded cohomology of $X$ is isometrically isomorphic to the simplicial
bounded cohomology of the aspherical multicomplex $\calA(X)$ (which satisfies $\pi_1(|\calA(X)|)\cong \pi_1(X)$). 
Using this, in Section~\ref{bd:fg:sc} we prove that the bounded cohomology of $X$ only depends on the fundamental group of $X$:

\begin{thm_intro}\label{mapping1_intro}
Let $f\colon X\to Y$ be a continuous map between path connected spaces, and suppose that $f$ induces an isomorphism on fundamental groups.
Then the induced map
$$
H^n_b(f)\colon H^n_b(Y)\to H^n_b(X)
$$
is an isometric isomorphism for every $n\in\mathbb{N}$.
\end{thm_intro}

As mentioned above, a complete proof of Theorem~\ref{mapping1_intro} which exploits tools from homological algebra may be found in~\cite{ivanov3} 
(see also~\cite{Ivanov} for the case when $X$ and $Y$ are countable CW complexes). Here
we first prove Theorem~\ref{mapping1_intro} by using multicomplexes under the assumption that $X,Y$ are good topological spaces. In order to deal with \emph{any} pair of topological
spaces $X,Y$, we then make use 
of the invariance of bounded cohomology with respect to weak homotopy equivalences~\cite{ivanov3}. The same strategy is applied in the proofs of Theorems~\ref{mapping-intro} and~\ref{vanishing1_intro} below.

An immediate corollary of Theorem~\ref{mapping1_intro} is 
the vanishing of bounded cohomology for simply connected spaces:

\begin{cor_intro}\label{simply:cor}
Let $X$ be a simply connected topological space. Then $$H^n_b(X)=0$$ for every $n\geq 1$. 
\end{cor_intro}

\subsection*{The Mapping Theorem}
With more work, in Chapter~\ref{chap:theorems} we prove in fact a stronger statement, which is usually known as \emph{Gromov's Mapping Theorem}:

\begin{thm_intro}[Mapping Theorem]\label{mapping-intro}
Let $X,Y$ be path connected topological spaces, let $f\colon X\to Y$ be a continuous map, and suppose that the map $f_*\colon \pi_1(X)\to \pi_1(Y)$ induced by $f$
on fundamental groups is surjective and has an amenable kernel. Then for every $n\in\mathbb{N}$ the map
$$
H^n_b(f)\colon H^n_b(Y)\to H^n_b(X)
$$
is an isometric isomorphism.
\end{thm_intro}

The proof of Theorem~\ref{mapping-intro} exploits the action on $\calA(X)$ of the group $\Pi(X,X)$, which was first defined by Gromov in~\cite[page 47]{Grom82}. 
Every element of the group $\Pi(X,X)$ consists of a family of homotopy classes (relative to the endpoints) of paths in $X$, subject to some additional conditions. 
Every normal subgroup $N$ of $\pi_1(X)$ gives rise to a normal subgroup $\widehat{N}$ of $\Pi(X,X)$, and we prove in Theorem~\ref{quoziente:che:vogliamo} that
the quotient of $\calA(X)$ by the action of $\widehat{N}$ is a complete, minimal and aspherical multicomplex whose fundamental group is isomorphic to 
$\pi_1(X)/N$ (this fact is stated in~\cite[page 47]{Grom82} without proof). Building on this fact, with some care one can then reduce Gromov's Mapping
Theorem to Theorem~\ref{mapping1_intro} above. 

\subsection*{Amenable subsets and the Vanishing Theorem}
The group $\Pi(X,X)$ plays a fundamental role in many results we prove in this paper, both when studying singular chains and bounded cohomology,
and when dealing with locally finite chains on open manifolds. Before switching our attention to the latter topic, let us state  a vanishing theorem whose proof 
 exploits
the action of $\Pi(X,X)$ just mentioned.

Let $X$ be a topological space and let 
$i\colon U\to X$  be the inclusion of a subset $U$ of $X$. Then $U$ is \emph{amenable} in $X$ if for every path connected
component $U'$ of $U$ the image of $i_* : \pi_1(U') \to \pi_1(X)$ is an amenable subgroup of
$\pi_1(X)$. A cover $\calU=\{U_i\}_{i\in I}$ of $X$ is amenable if every element of $\calU$ is amenable in $X$. 
We denote by 
$\mult(\calU)$ 
the multiplicity of $\calU$
(see Section~\ref{amenable:vanishing:sec}).

\begin{thm_intro}[Vanishing Theorem]\label{vanishing1_intro}
Let $X$ be a topological space and let  $\calU$ be an amenable open cover of $X$. Then for every $n\geq \mult(\calU)$
the comparison map
$$
c^n\colon H^n_b(X)\to H^n(X)
$$
vanishes.
\end{thm_intro}

Theorem~\ref{vanishing1_intro} 
already appears in~\cite[page 40]{Grom82}, and 
is proved in Chapter~\ref{vanishing-thm:chap}.
Via duality, Theorem~\ref{vanishing1_intro} implies the vanishing of the simplicial volume for closed manifolds admitting amenable covers of small multiplicity
(see Section~\ref{amenable:vanishing:sec}):

\begin{cor_intro}
\label{van-cor-intro}
Let $X$ be a topological space admitting an open amenable cover of multiplicity $m$, and let $n\geq m$. Then
$$
\|\alpha\|_1=0
$$
for every $\alpha\in H_n(X)$. 
In particular, if
$M$ is a closed  oriented manifold admitting an open amenable cover $\calU$ such that 
$\mult(\calU)\leq \dim M$,
then
$$
\|M\|=0\ .
$$
\end{cor_intro}

As originally proved by Ivanov~\cite{Ivanov,ivanov3}, in many cases of interest the Vanishing Theorem~\ref{vanishing1_intro} may be strengthened as follows:
if $\calU$ is an amenable open cover of the space $X$ and some mild additional hypotheses hold, then  
the comparison map
$
c^*\colon H^*_b(X)\to H^*(X)
$
factors through the ordinary singular cohomology of the \emph{nerve} of $\calU$ (see Section~\ref{amenable:vanishing:sec}), which is obviously null in degree $n\geq \mult(\calU)$.
We refer the reader to Theorem~\ref{van-thm2} for a precise statement.

\subsection*{The simplicial volume of open manifolds}
The third part of this work is devoted to the study of the simplicial volume of open (i.e.~connected, non-compact and without boundary) manifolds.
If $M$ is an open $n$-dimensional manifold, then $H_n(M)=0$. Therefore, in order to define a fundamental class for $M$, one needs to 
work with the complex $C^\lf_*(M)$ of \emph{locally finite} chains on $M$, rather than with the usual complex $C_*(M)$ of finite singular chains.
As in the  case of finite chains, for every topological space $X$ and for every $n\in\mathbb{N}$, the space of locally finite chains $C_n^\lf(X)$ may be endowed with an $\ell^1$-norm $\|\cdot \|_1$, which 
induces an $\ell^1$-seminorm $\|\cdot \|_1$
on the locally finite homology $H^\lf_n(X)$ of $X$ (see Section~\ref{simplicial:open:sec} for the details).
If $M$ is an $n$-dimensional oriented open manifold, then $H_n^\lf(M)$ is canonically isomorphic to $\R$ and generated by a preferred element $[M]\in H^\lf_n(M)$, called
the \emph{fundamental class} of $M$. The simplicial volume $\|M\|$ of $M$ is then defined by setting
$$
\|M\|=\|[M]\|_1\ \in\ [0,+\infty]\ .
$$

The simplicial volume of open manifolds is still quite mysterious. If $M$ is \emph{tame}, i.e.~it is the internal part of a compact manifold with boundary
$\overline{M}$, then it is known that $\|\overline{M}\|\leq \|M\|$, and that the equality holds provided that the fundamental group of every connected component
of $\partial \overline{M}$ is amenable~\cite{Loeh, KimKue} (we refer the reader to Section~\ref{relative:sec} for the definition of the simplicial volume of a manifold with boundary).
However, in the case when $\partial \overline{M}$ is not amenable, no example is known for which $\|M\|<+\infty$ and $\|M\|>\|\overline{M}\|$. Neither is known 
any example of an open manifold $M$ 
with non-amenable topological ends for which $\|M\|\notin \{0,+\infty\}$ (while we refer the reader e.g.~to~\cite{Loh-Sauer2} for the exact computation
of the simplicial volume of some manifolds
which compactify to manifolds with an amenable boundary). For example, a long-standing open question is the following:

\begin{quest_intro}
Let $\Sigma=(S^1\times S^1)\setminus \{p\}$ be a once-punctured torus. What is the exact value of $\|\Sigma\times \Sigma\|$?
In particular, is $\|\Sigma\times \Sigma\|$ positive or is it null? 
\end{quest_intro}

The simplicial volume of open manifolds lacks several topological and geometric properties enjoyed by the simplicial volume of closed manifolds.
For example, neither a biLipschitz estimate of $\|M_1\times M_2\|$ in terms
of $\|M_1\|\cdot \|M_2\|$ nor a Proportionality Principle holds for the simplicial volume of open manifolds (but see e.g.~\cite{Loh-Sauer2,MichelleKim} for some results
in this direction) and these facts somewhat 
illustrate the difficulties in understanding the topological and the geometric meaning of this invariant.

Gromov himself introduced in~\cite{Grom82} some variations of the simplicial volume for open manifolds. Among them, the most studied
is probably the so called \emph{Lipschitz simplicial volume}, for which one can recover, for example, both a product formula and the Proportionality Principle
(see e.g.~\cite{Loh-Sauer,  Franceschini1, strz-unp}; see also~\cite{Str-Lip}, where the additivity of the Lipschitz simplicial volume under connected sums
is proved by exploiting multicomplexes). However, in this work we will only deal with the classical simplicial volume of open manifolds.

\subsection*{Finiteness and Vanishing Theorems for non-compact spaces}
Even if exact computations  are very difficult, there 
exist criteria that provide conditions for the vanishing or the finiteness of the simplicial volume of open manifolds. Just as in the closed case, the vanishing of the simplicial volume of an open manifold may be deduced from the existence of a suitable open cover
 of small multiplicity. 
Indeed, Chapters~\ref{finvan:chap},
 \ref{diffusion:chap}, \ref{admissible:chap} and~\ref{vanfin:proof:chap} are  devoted to the following results, that were first stated in~\cite[page 58]{Grom82}. Recall that a \emph{triangulable space} is a topological space which is homeomorphic to the geometric realization of a simplicial complex.
We refer the reader to Definition~\ref{amenable:infinity} for the notion of 
 \emph{amenability at infinity} for a 
sequence  $\{U_i\}_{i\in\mathbb{N}}$ of open subsets of $X$.

\begin{thm_intro}[Vanishing Theorem for locally finite homology]\label{Van-Theor-intro}
Let $X$ be a connected non-compact triangulable topological space.
Let $\calU=\{U_j\}_{j\in\mathbb{N}}$ be an amenable open cover of $X$ such that each $U_j$ is relatively compact in $X$. 
Also suppose that 
the sequence $\{U_j\}_{j\in\mathbb{N}}$ is amenable at infinity.
Then for every $k\geq \mult(\calU)$ and every $h \in \, H^{\lf}_{k}(X)$ we have $$\lVert h \rVert_1 = 0\ .$$
\end{thm_intro}

A subset $W$ of $X$ is \emph{large} if $X\setminus W$ is relatively compact in $X$.

\begin{thm_intro}[Finiteness Theorem]\label{Fin-Theor-intro}
Let $X$ be a connected non-compact triangulable topological space.
Let $W$ be a large open subset of $X$, and
let $\calU=\{U_j\}_{j\in\mathbb{N}}$ be an open cover of $W$ such that each $U_j$ is relatively compact in $X$. 
Also suppose that 
the sequence $\{U_j\}_{j\in\mathbb{N}}$ is amenable at infinity (in particular, $\calU$ is locally finite in $X$).
Then for every $k\geq \mult(\calU)$  and every $h \in \, H^{\lf}_{k}(X)$ we have $$\lVert h \rVert_1 <+\infty\  .$$
\end{thm_intro}

The following corollaries provide the main applications of Theorems~\ref{Van-Theor-intro} and~\ref{Fin-Theor-intro} to the simplicial volume of open manifolds.

\begin{cor_intro}\label{va-cor-intro}
Let $M$ be an oriented open triangulable manifold of dimension $m$ and let $\calU=\{U_j\}_{j\in\mathbb{N}}$ be an amenable open cover of $M$ such that each $U_j$ is relatively compact in $M$. Also  suppose that the sequence $\{U_j\}_{j\in\mathbb{N}}$ is amenable at infinity, and that $\mult(\calU)\leq m$. Then
$$
\|M\|=0\ .
$$
 \end{cor_intro}

\begin{cor_intro}\label{fi-cor-intro}
Let $M$ be an oriented open triangulable manifold of dimension $m$. Let $W$ be a large open subset of $M$, and
let $\calU=\{U_j\}_{j\in\mathbb{N}}$ be an open cover of $W$ such that each $U_j$ is relatively compact in $M$. 
Also suppose that 
the sequence $\{U_j\}_{j\in\mathbb{N}}$ is amenable at infinity (in particular, $\calU$ is locally finite in $M$), and that $\mult(\calU)\leq m$.
Then
$$
\|M\|\ <\ +\infty\ .
$$
\end{cor_intro}

In order to prove Theorems~\ref{Van-Theor-intro} and~\ref{Fin-Theor-intro},
after introducing diffusion of chains, one needs to select
a suitable submulticomplex $\mathcal{AD}(X)$ of $\calK(X)$ with the following  property: 
an infinite family of simplices of  $\mathcal{AD}(X)$
 leaves every compact subset of $X$ provided that the set of vertices of the simplices in the family do so. Such a
multicomplex fits the need to relate the local finiteness of chains in $X$ to controlled combinatorial properties of the corresponding simplicial
chains in  $\mathcal{AD}(X)$.

A peculiarity of the proofs of Theorems~\ref{Van-Theor-intro} and~\ref{Fin-Theor-intro} is that they completely avoid any reference to bounded cohomology. 
Theorem~\ref{Van-Theor-intro} extends Corollary~\ref{van-cor-intro} (which deals with the usual singular homology) to the locally finite case.
However, Corollary~\ref{van-cor-intro} was obtained via duality from Theorem~\ref{vanishing1_intro}, which concerns the comparison map defined on bounded cohomology.
On the contrary, the proofs of Theorems~\ref{Van-Theor-intro} and~\ref{Fin-Theor-intro} are based on manipulations of  cycles. Just as in the finite case,
the amenability 
 of certain groups of simplicial automorphisms plays a fundamental role in our arguments. However, here amenability allows us to properly exploit the diffusion of locally finite chains,
 rather than to obtain vanishing results in bounded cohomology. In order 
  to illustrate the 
general ideas involving diffusion of chains
 while avoiding all the technicalities
 needed to deal with locally finite chains, in Section~\ref{toy:sec} we describe a proof of (some special cases of) Corollary~\ref{van-cor-intro} which is based on diffusion of chains
 rather than on bounded cohomology.
 
Even if inspired by Gromov's suggestions and by the ideas developed in the previous chapters, the  arguments described in the third part of this paper
are substantially new. Indeed, Gromov's strategy to prove Theorems~\ref{Van-Theor-intro} and~\ref{Fin-Theor-intro} is just sketched in~\cite{Grom82}, and  
it seems to underestimate several difficulties that emerge e.g.~in the case when $X$ is not assumed to be aspherical
or when proving the amenability of some relevant groups of simplicial automorphisms
(see Remark~\ref{gamma2}).  
For example,
the simplicial automorphisms belonging to these groups are allowed to move the vertices of $\mathcal{AD}(X)$.
As a consequence, when showing that these groups are amenable via a comparison with suitable subgroups of $\Pi(X,X)$, one needs to carefully take into account the role
played by basepoints (see again Remark~\ref{rem:ip:agg:gruppi:gromov}).

 \subsection*{Amenable covers and $\ell^1$-invisibility} 
 Other results on the vanishing and/or the finiteness of the simplicial volume of open manifolds were obtained e.g.~in \cite{Loeh, Loh-Sauer}. In particular,
L{\"o}h provides in \cite{Loeh} a complete criterion for the finiteness of the simplicial volume of tame open manifolds, in terms of the so-called
\emph{$\ell^1$-invisibility} of the boundary components of the manifolds (see Definitions~\ref{invisible:def} and~\ref{tame:defn} for the definition of $\ell^1$-invisible and tame manifold, respectively,  
and Theorem~\ref{invisible:thm} for L\"oh's criterion). 
As a by-product of our results, by comparing
the Finiteness Theorem~\ref{Fin-Theor-intro} with L{\"o}h's result, in Section~\ref{invisible:sec} we obtain a new sufficient condition for a closed manifold to be $\ell^1$-invisible:

\begin{thm_intro}\label{invisible:new-intro}
 Let $M$ be a closed oriented triangulable $n$-dimensional manifold admitting an amenable cover $\calU$ such that $\mult(\calU)\leq n$. Then $M$ is $\ell^1$-invisible.
\end{thm_intro}

Since the simplicial volume of an $\ell^1$-invisible closed manifold vanishes, Theorem~\ref{invisible:new-intro} strengthens
(at least in the case of triangulable manifolds) the last statement of Corollary~\ref{van-cor-intro}. We refer the reader to~\cite{Fri:amenable} for further progress in this direction.

\subsection*{The simplicial volume of the product of three open manifolds}
Let us conclude this introduction with an interesting consequence of Theorem~\ref{Van-Theor-intro} already pointed out by Gromov. 
As mentioned above,
it is not known whether the simplicial volume of the product of two puctured tori vanishes or not. More in general, it is not known whether the simplicial volume of the product
of two open manifolds $M_1,M_2$ should be positive, provided that $\|M_1\|>0$ and $\|M\|_2>0$.
Actually, there exist no examples of open manifolds $M_1,M_2$ for which the simplicial volume $\|M_1\times M_2\|$ is known to be positive and finite.
When considering the product of three open manifolds, we have the following
striking result:

\begin{thm_intro}\label{tre-intro}
Let $M_1, M_2,M_3$ be tame open oriented PL manifolds of positive dimension.
 Then
$$\lVert M_1 \times M_2 \times M_3 \rVert = 0\ .$$
\end{thm_intro}

We refer the reader to Section~\ref{PL:section} for the definition of PL manifold. Theorem~\ref{tre-intro} is stated in~\cite[Ex.~(A), page 59]{Grom82} as an almost direct application of Corollary~\ref{va-cor-intro}. 
After clarifying some statements by Gromov on coamenable subcomplexes of open manifolds, in Section~\ref{tre-sec}
we exploit a  construction inspired by~\cite[Theorem 5.3]{Loh-Sauer} to provide a detailed proof of Theorem~\ref{tre-intro}.



\part{The general theory of multicomplexes}

\chapter{Multicomplexes}\label{multicomplexes:chapter}
The main goal of this chapter is to introduce the notion of \textit{multicomplex}, defined by M. Gromov in \cite{Grom82}. 
Multicomplexes belong to the family of
simplicial structures. Roughly speaking, one can order the most common simplicial structures according to the degeneracies allowed for their simplices. In this spirit,
the least singular objects one can consider are simplicial complexes, while the most general ones are probably simplicial sets.
Somewhere between them we have $\Delta$-complexes, which were introduced in \cite{eil-zil} by Eilenberg and Zilber under the name of \textit{semi-simplicial complexes}.
As we will see, multicomplexes are special $\Delta$-complexes, so they may be considered as a mild generalization of simplicial complexes, the main difference being that simplices
in multicomplexes are not determined by their vertices (but they are still embedded: in particular, the vertices of any simplex of a multicomplex are necessarily pairwise distinct).
We refer the reader to Section~\ref{comparison:subsec} for more details about the relationship between multicomplexes and their relatives.

\section{Basic definitions}\label{Sec:basic:def:multi}

As mentioned above, simplicial multicomplexes (or, for short, multicomplexes) provide a generalization of simplicial complexes. Indeed, multicomplexes may be considered as 
simplicial complexes where distinct simplices may intersect in an arbitrary union of common faces (in particular, they may have the same set of vertices).
Gromov's definition of multicomplex summarizes these properties in a very informative and short sentence: according to~\cite{Grom82}, a multicomplex
is ``a set $K$ divided
into the union of closed affine simplices $\Delta_i$, $i\in I$,
such that the intersection of any
two simplices $\Delta_i\cap \Delta_j$ is a (simplicial) subcomplex in $\Delta_i$ as well as in $\Delta_j$''. As a topological space, a multicomplex 
is endowed with the weak topology associated to its decomposition into simplices. 
This concrete definition of multicomplex would probably suffice to provide a clear insight into the topological  features of multicomplexes. Nevertheless, 
in the sequel we will be interested in settling the foundation
of the (bounded) cohomology of multicomplexes. To this aim, it is convenient to have a more formal (and combinatorial) definition available.

For every set $V$ we denote by $\mathcal{P}_{f}(V)$ the set of finite subsets of $V$. 

\begin{Definizione}\label{multicomplex:def}
A multicomplex $K$ is a triple
$$
K=\left( V,\, I=\bigcup_{A\in \mathcal{P}_f(V)} I_A,\, \Omega\right)\ ,  
$$
where:
\begin{enumerate}
\item $V$ is any set, that will be called the \emph{set of vertices} of $K$;
\item for every $A\in\mathcal{P}_f(V)$, $I_A$ is a (possibly empty) set, that will be called the \emph{set of simplices
with vertex set $A$};
\item if $A=\{v\}$ is a singleton, then $I_{A}$ is also a singleton;
\item $\Omega$ is a set of maps $\{\partial_{A,B}\colon I_A\to I_B,\, A,B\in \mathcal{P}_f(V), A\supseteq B\}$ (that will be called the \emph{boundary maps} of the multicomplex)
such that
$$
\partial_{A,A}={\rm Id}_A
$$
for every $A\in\mathcal{P}_f(V)$, and
$$
\partial_{B,C}\circ \partial_{A,B}=\partial_{A,C}\quad \textrm{whenever}\ A\supseteq B\supseteq C\ .
$$ 
 \end{enumerate}
\end{Definizione}

For every finite set of vertices $A\subseteq V$, each element $\sigma \in I_A$ represents a simplex with vertex set $A$, and under this identification the 
element $\partial_{A,B} \sigma$ should be thought as the face of $\sigma$ with vertex set $B$. Note that, if $I_A\neq \emptyset$ and $B\subseteq A$,
then the existence of the  map $\partial_{A,B}\colon I_A\to I_B$ ensures that $I_B\neq \emptyset$. The condition $\partial_{B,C}\circ \partial_{A,B}=\partial_{A,C}$ just formalizes the fact that the face of a face
of a simplex $\sigma$ is itself a face of $\sigma$. For obvious reasons, we say that the dimension of a simplex $\sigma\in I_A$ is equal to $|A|-1$; simplices of dimension $n$
will be referred to as $n$-simplices. When this does not cause any ambiguity, we will denote simply by $\partial_B$ the map $\partial_{A,B}$, for every $A\supseteq B$.

Let us immediately point out two peculiarities of multicomplexes with respect to simplicial complexes and simplicial sets.
First of all, degenerate simplices are not allowed: each $n$-simplex has exactly $(n+1)$ distinct vertices. Moreover,
there is no ordering on the set of vertices of a simplex. As a (rather disappointing, at first sight) consequence, it makes no sense to speak of the $i$-th face of a simplex.
Therefore, multicomplexes are \emph{not} simplicial sets (but to every multicomplex one can canonically associate a simplicial set, see Section~\ref{comparison:subsec}). 
Indeed, the two properties we have just pointed out make perhaps multicomplexes more similar to simplicial complexes than to simplicial sets.
For more details, we refer the reader to Section~\ref{comparison:subsec}, where we will see that
(the geometric realization of) a multicomplex is what Hatcher calls a \emph{regular unordered $\Delta$-complex}~\cite{hatcher}.
Here we just point out that in this paper we will understand the following definition of simplicial complex,
which is in fact equivalent to the classical definition that may be found e.g.~in~\cite{hatcher, munkres}:

\begin{Definizione}\label{simpl:compl:defn}
 A \emph{simplicial complex} is a multicomplex $(V,I,\Omega)$ in which every simplex is determined by its vertices, i.e.~$I_A$ contains at most one element for every $A\in \calP_f(V)$.
\end{Definizione}

Having introduced the objects of our interest, let us define the notion of \emph{simplicial map} between multicomplexes. Even if multicomplexes do not contain degenerate simplices,
we will allow simplicial maps to shrink the dimension of simplices. 

\begin{Definizione}
Let $K=(V,I,\Omega)$, $K'=(V',I',\Omega')$ be multicomplexes. A (simplicial) map from $K$ to $K'$ is determined by the following data: a map $f\colon V\to V'$, and maps
$f_A\colon I_A\to I'_{f(A)}$, $A\in \mathcal{P}_f(V)$, such that 
$$
\partial_{f(A),f(B)}(f_A(\sigma))=f_B(\partial_{A,B} (\sigma))
$$
for every $A,B\in\mathcal{P}_f(V)$, $B\subseteq A$, $\sigma \in I_A$.

We say that a simplicial map $f$  as above is injective (resp. surjective) if both $f$ and each $f_A$ is injective (resp. surjective). Injective simplicial maps will often be called \emph{simplicial
embeddings}.

When this does not cause any possible misunderstanding, we will simply denote by $f(\sigma)$ the simplex
$f_A(\sigma)$, when $\sigma\in I_A$. 
\end{Definizione}

\begin{rem}
 When defining maps between multicomplexes,
 the absence of degenerate simplices introduces some complications from the point of view of combinatorics.
 For example, in order to be reasonably flexible we need to agree that simplicial maps can shrink simplices to simplices of lower dimensions. In the context of simplicial sets,
 degenerate simplices take care of this phenomenon, while in the context of multicomplexes we need to allow simplicial maps to be non-homogeneous with respect to the dimension of simplices.
\end{rem}

There is an obvious identity on every multicomplex, which of course is 
 a simplicial map; moreover, the composition of simplicial maps (which is defined in the natural way) is a simplicial map. The following definition
 will play an important role in the theory of multicomplexes:
 
 \begin{defn}
 Let $K=(V,I,\Omega)$ be a multicomplex.
 A simplicial map $f\colon K\to K'$ between multicomplexes is \emph{non-degenerate} if it maps every $n$-simplex of $K$ to an $n$-simplex of $K'$, or, in other words,
 if $f|_A$ is injective whenever $I_A\neq \emptyset$.
 \end{defn}

A \emph{submulticomplex} $K'$ of a multicomplex $K=(V,I,\Omega)$ is a triple $K'=(V',I',\Omega')$, such that 
$V'\subseteq V$, $I'_{A'}\subseteq I_{A'}$ for every $A'\in \mathcal{P}_f(V')\subseteq \mathcal{P}_f(V)$, and 
$$
\partial'_{A',B'}=\left(\partial_{A',B'}\right)|_{I'_{A'}}
$$
for every $A',B'\in \mathcal{P}_f(V')$, where we denote by $\partial_{A,B}$ the boundary maps of $K$, and by $\partial'_{A',B'}$ those of $K'$. Of course, a submulticomplex
is a multicomplex itself, and every submulticomplex of a multicomplex $K$ admits an obvious simplicial embedding in $K$. Moreover,
a submulticomplex is uniquely determined by the simplices it contains, so we will often define submulticomplexes just by describing their sets of simplices
(and leaving to the reader to check that such sets are closed with respect to the boundary maps of the ambient multicomplex).

For every $n\in\mathbb{N}$ we can now introduce the \emph{$n$-skeleton} $K^n=(V^n,I^n,\Omega^n)$ of $K=(V,I,\Omega)$ as the unique submulticomplex of $K$ such that
$V^n=V$ and $I^n_A=I_A$ if $|A|\leq n+1$, while $I^n_A=\emptyset$ if $|A|>n+1$. Any simplicial map $f\colon K\to K'$ induces
maps $f^n\colon K^n\to (K')^n$ on the skeleta of $K,K'$.

If $x_0$ is a vertex of $K$, then we say that the pair $(K,x_0)$ is a \emph{pointed} multicomplex.

\section{The geometric realization}
Let us now define the geometric realization of a  multicomplex $K=(V,I,\Omega)$. For every $n\in\mathbb{N}$ we denote by $\Delta^n$ the standard $n$-dimensional simplex with vertices
$e_0,\ldots,e_n$.
Then, we take one copy of $\Delta^n$ for every $n$-simplex of $K$, and we glue these simplices according to the boundary maps prescribed by $\Omega$. More formally, we proceed as follows.
For every $n\in\mathbb{N}$ let us denote by $I^n$ the set of $n$-simplices of $K$, i.e.~the set
$$
I^n=\bigcup \{I_A\, |\, A\in\mathcal{P}_f(V),\, |A|=n+1\}\ ,
$$
and let us set
$$
X^n=I^n\times \Delta^n \, ,\qquad X=\bigsqcup_{n\in\mathbb{N}} X^n\ .
$$
We now define an equivalence relation that encodes the gluings between the simplices of $K$. To this aim, 
for every $ A\in\mathcal{P}_f(V)$ with $|A|=n+1$, we first fix an auxiliary arbitrary
labelling of the vertices of $\Delta^n$ by the elements of $A$, i.e.~a bijection between $A$ and the set of vertices of $\Delta^n$. 
Then, if $(\sigma_A,x)\in I_A\times \Delta^n$ and $(\sigma_B,y)\in I_B\times \Delta^m$, where $n=|A|-1$, $m=|B|-1$, we set
$(\sigma_A,x)\sim (\sigma_B,y)$ if the following condition holds: $A\supseteq B$, and
$$
x=\varphi (y)\ ,
$$
where $\varphi\colon \Delta^m\to \Delta^n$ is the unique affine inclusion that preserves the labelling of the vertices of $\Delta^m$ and of $\Delta^n$ induced by
$\sigma_A,\sigma_B$, respectively.
Finally, we denote by $\approx$ the smallest equivalence relation containing $\sim$, and we set
$$
|K|=X \big/_\approx\ .
$$
The fact that multicomplexes do not contain degenerate simplices (i.e.~that every $n$-simplex has $(n+1)$ distinct vertices) readily implies that,
for every $(\sigma,\Delta^n)$, $\sigma\in I^n$, $n\in\mathbb{N}$, the composition
$$
\Delta^n \to |K|\, ,\qquad x\mapsto [(\sigma,x)]
$$
is injective. Therefore, if 
for every $\sigma\in I^n$ we denote by $|\sigma|\subseteq |K|$ the image of $\{\sigma\}\times \Delta^n$ in $|K|$, 
then we may endow $|\sigma|$ with the Euclidean topology inherited from $\{\sigma\}\times \Delta^n$.
We then endow $|K|$ with the topology for which a subset $Y\subseteq |K|$ is open if and only if its intersection
with every $|\sigma|$ is open. It is immediate to check that
$|K|$ is a CW complex, whose closed $n$-cells are given by the subsets $|\sigma|$ defined above, as  $\sigma$ varies in $I^n$.
Closed cells are embedded in $|K|$, i.e.~$|K|$ is a \emph{regular} CW complex, according to the usual terminology (see e.g.~\cite{Piccinini, hatcher}).
Moreover, there is an obvious identification between $|K^n|$ and the $n$-skeleton $|K|^n$ of $|K|$ as a CW complex. We stress that the construction just described is independent of
the choice of the labellings introduced above. More precisely, we understand that geometric $n$-simplices of $|K|$ do \emph{not} come with 
 preferred identifications  with the standard
simplex $\Delta^n$. 

\begin{rem}\label{rem:descr:punti:geom:real}
Points of $|K|$ admit the following easy description: if $\sigma\in I_A$ is a simplex with vertices $A=\{a_0,\ldots,a_n\}$ and $t_0,\ldots,t_n \in [0,1]$ are such that
$t_0+\ldots + t_n=1$, then we define the point 
$$
(\sigma,t_0a_0+\ldots+ t_na_n)\in |K|
$$
as follows: after labelling the $i$-th vertex of $\Delta^n$ by $a_i$, 
we set $(\sigma,t_0a_0+\ldots+ t_na_n)=[(\sigma,(t_0e_0+\ldots+t_ne_n))]$. This definition formalizes the obvious fact that points in the geometric realization
of a simplex $\sigma\in I_A$ correspond to convex combinations of vertices in $A$. This description of points of $|K|$ is almost unique:
indeed, 
$$
(\sigma,t_0a_0+\ldots+ t_na_n)=(\sigma',t'_0a'_0+\ldots+ t'_ma'_m)
$$
if and only if, after setting $\overline{A}=\{a_i\, |\, t_i\neq 0\}$, $\overline{A}'=\{a'_i\, |\, t_i\neq 0\}$, then
$\overline{A}=\overline{A}'$, $\partial_{\overline{A}} \sigma=\partial_{\overline{A}'}\sigma'$, and $t_i'=t_j$ whenever $a_i'=a_j$.
\end{rem}

Of course, geometric realization is functorial:
any simplicial map $f\colon K\to K'$ induces a continuous map
$$
|f|\colon |K|\to |K'|
$$
such that
$$
|f|((\sigma,t_0a_0+\ldots +t_na_n))=(f(\sigma),t_0f(a_0)+\ldots+t_nf(a_n))
$$
for every point $(\sigma,t_0a_0+\ldots +t_na_n)$ of $|K|$.
Observe that, even if $f$ is not injective on the vertices $\{a_0,\ldots,a_n\}$ of $\sigma$, then the symbol
$t_0f(a_0)+\ldots+t_nf(a_n)$ still represents a convex combination of the vertices of $f(\sigma)$. With a slight abuse, we will often denote by the symbol $f$ the map $|f|$,
and we will say that a map between the geometric realizations of multicomplexes is \emph{simplicial}
if it is induced by a simplicial map. 

\section[Multicomplexes and other simplicial structures]{Multicomplexes,  simplicial complexes, $\Delta$-complexes, and simplicial sets}\label{comparison:subsec}
In this section we briefly review the relationship between multicomplexes and other well-known simplicial structures, in order to help the reader
to get familiar both with the (many) similarities between multicomplexes and very well-studied objects, and with the (small but crucial to our purposes) 
peculiarities of multicomplexes.
For a deeper discussion of the other simplicial structures we refer the reader to~\cite{eil-zil, Piccinini, may, GJ}.

A $\Delta$-complex is a graded set $Z=\bigcup_{i\in \mathbb{N}} Z_i$ together with a collection of boundary maps 
$$\partial_i^q\colon Z_q\to Z_{q-1}\, ,\qquad  i=0,\ldots,q$$ that satisfy the identity $$\partial_i^q\partial_j^{q+1}=\partial_{j-1}^q\partial_i^{q+1}$$ for every
$i<j$, $0\leq 1\leq q$, $0\leq j\leq q+1$. Elements in $Z_q$ are called \emph{$q$-simplices}, and for every $s\in Z_q$ the element $\partial_i s \in Z_{q-1}$
is the \emph{$i$-th face} of $s$. If $s\in Z_q$, for $0\leq i\leq q$ one can also define the $i$-th vertex $v_i(s)$ of $s$ by setting
$$
v_i(s)=\partial_q^q\partial_{q-1}^{q-1}\cdots\partial_{i+1}^{i+1}\partial^i_0\partial^{i-1}_0\cdots \partial^1_0 s\ .
$$
With this definition, simplices of a $\Delta$-complex come with an ordering on their vertices.
As already mentioned, we are interested in forgetting this ordering. We say that a $\Delta$-complex is \emph{unordered} if
 there exists a group homomorphism $\theta^q\colon \mathfrak{S}_{q+1}\to \mathfrak{S}(Z_q)$ from the group of permutations
 of $\{0,\ldots,q\}$ to the group of permutations of the set $Z_q$ which is equivariant with respect to boundary maps, i.e.~is such that $\partial_{\tau(i)}(\theta^q(\tau)(s))=\theta^{q-1}(\tau_i)(\partial_{i}(s))$, where 
$\tau_i\colon \{0,\ldots,q-1\}\to \{0,\ldots,q-1\}$ is obtained by ``rescaling'' $\tau$ after removing $i$ from its domain and $\tau(i)$ from its target set.
It is immediate to check that $v_i(\theta^q(\tau)(s))=v_{\tau(i)}(s)$ for every $s\in Z_q$, $0\leq i\leq q$. 

Simplicial complexes may be considered as unordered $\Delta$-complexes in which simplices have distinct vertices and are completely determined by their vertices.
Indeed, usually a simplicial complex with vertex set $V$ is defined as a subset $S$ of $\mathcal{P}_f(V)$ such that $A\in S$ implies $B\in S$ for every $B\subseteq A$.
From such a datum one can construct a $\Delta$-complex $Z$ as follows: $Z_q=\{(v_0,\ldots,v_q)\, |\, v_i\neq v_j\ {\rm for}\ i\neq j,\, \{v_0,\ldots,v_q\}\in S\}$, and
$\partial_i^q(v_0,\ldots,v_q)=(v_0,\ldots,\widehat{v}_i,\ldots,v_q)$.
Moreover, by setting $\theta^q(\tau)(v_0,\ldots,v_q)= (v_{\tau(0)},\ldots,v_{\tau(q)})$ one gets an action
of $\mathfrak{S}_{q+1}$ on $Z_q$, which endows $Z$ with the structure of an unordered $\Delta$-complex. It is easily checked
that $v_i$ is indeed the $i$-th vertex of $(v_0,\ldots,v_q)$ according to the definition of $i$-th vertex given above, so
every simplex has distinct vertices. Moreover, every ordered $(q+1)$-tuple of elements in $Z_0$ is the ordered set of vertices of at most
one element of $Z_q$, so simplices are indeed determined by their vertices.

We can now formalize the statement that multicomplexes are just unordered $\Delta$-complexes in which each simplex has distinct vertices. As such, they provide
a class of objects which generalize simplicial complexes without allowing all the phenomena that can occur in generic $\Delta$-complexes. Indeed,
from any given multicomplex $(V,I,\Omega)$ one may construct an unordered $\Delta$-complex $Z$ by setting 
$$Z_q=\{(\sigma,(v_0,\ldots,v_{q}))\, |\, \sigma\in I_A, \, A=\{v_0,\ldots,v_{q}\}\in\mathcal{P}_f(V),\, |A|=q+1\}\ ,$$
$$ \partial_i ((\sigma,(v_0,\ldots,v_{q}))=(\partial_B\sigma,(v_0,\ldots,\widehat{v}_i,\ldots,v_q)),\quad
{\rm where}\ B=\{v_0,\ldots,\widehat{v}_i,\ldots,v_q\}\ ,
$$
$$
\theta^q(\tau)((\sigma,(v_0,\ldots,v_q)))=(\sigma,(v_{\tau(0)},\ldots,v_{\tau(q)}))\qquad \textrm{for\ every}\ \tau\in\mathfrak{S}_{q+1}\ .
$$
On the other hand, let $Z$ be a given unordered $\Delta$-complex in which each simplex has distinct vertices. Then, one may construct an associated multicomplex
$(V,I,\Omega)$ by setting $V=Z_0$, and putting in $I_A$ one simplex for every $\mathfrak{S}_q$-orbit 
of simplices in $Z_q$ with vertex set equal to $A$. Boundary maps of $\Omega$ may then be deduced from the boundary maps of $Z$, thanks to the fact that
each simplex in $Z$ has distinct vertices, and that boundary maps of the $\Delta$-complex are equivariant (in the sense explained above) with respect to the action of $\mathfrak{S}_{q+1}$ on $Z_q$.

When developing the abstract theory of $\Delta$-complexes, some difficulties arise from the fact that simplicial maps cannot shrink the dimension of simplices. These difficulties may be overcome
by introducing degenerate simplices, thus getting a richer structure in which simplices of dimension $n$ may correspond to geometric objects of dimension $m<n$. This approach leads to the theory of simplicial
sets,
which are now quickly reviewing.
Simplicial sets are $\Delta$-complexes endowed with the extra datum of degeneracy maps $s_i^q\colon Z_q\to Z_{q+1}$, $0\leq i\leq q$,  such that the following conditions hold:
$$
\begin{array}{ll}
s_is_j=s_{j+1}s_i & {\rm for}\ i\leq j\ , \\
\partial_is_j=s_{j-1}\partial_i & {\rm for}\ i<j\ ,\\
\partial_js_j={\rm Id}=\partial_{j+1}s_j\ , &\\
 \partial_i s_j =s_j\partial_{i-1} & {\rm for}\ i>j+1\ .
 \end{array}
 $$
 A simplex of a simplicial set is \emph{degenerate} if it is the image of a simplex under some degeneracy operator. A simplicial set is \emph{symmetric}
 if it is an unordered $\Delta$-complex, and if the action of $\mathfrak{S}_{q+1}$ on $Z_q$ is equivariant  with respect to degeneracy operators. 
 A simplicial map between simplicial sets is just a map of graded sets of degree 0 which commutes with boundary and degeneracy operators. When dealing
 with symmetric simplicial sets, we also ask simplicial maps to be equivariant with respect to the action of $\mathfrak{S}_{q+1}$ on $q$-simplices.
 For those who are used to the category-theoretic definition of simplicial sets, just as simplicial sets may be viewed as functors from the category of finite ordinals (with monotone maps)
 to the category of sets, we have that symmetric simplicial sets can de defined as functors from the category of finite cardinals to the category of sets. For a thorough description of this approach we refer 
 the reader to~\cite{grandis1, grandis2, rosick-symm}.
 
It is well known that to every (unordered) $\Delta$-complex $(V,I,\Omega)$ it is possible to associate a (symmetric) simplicial set by suitably adding degenerate simplices. In particular,
it is possible to associate to a multicomplex a symmetric simplicial set $Z$, which is defined as follows. Since we need to allow degenerate simplices 
(which necessarily have non-distinct vertices) we set
$$
Z_q=\{(\sigma,(v_0,\ldots,v_{q}))\, |\, \sigma\in I_A\in\mathcal{P}_f(V), \, A=\{v_0,\ldots,v_{q}\}\} 
$$ 
(note that in the above formula we do not require the $v_i$ to be pairwise distinct). 
Then we set
$$ \partial_i ((\sigma,(v_0,\ldots,v_{q}))=(\partial_B\sigma,(v_0,\ldots,\widehat{v}_i,\ldots,v_q)),\quad
{\rm where}\ B=\{v_0,\ldots,\widehat{v}_i,\ldots,v_q\}\ ,
$$
$$
s_i(\sigma,(v_0,\ldots,v_{q}))=(\sigma,(v_0,\ldots,v_i,v_i,\ldots,v_q))\ ,
$$
$$
\theta^q(\tau)((\sigma,(v_0,\ldots,v_q))=(\sigma,(v_{\tau(0)},\ldots,v_{\tau(q)}))\qquad \textrm{for\ every}\ \tau\in\mathfrak{S}_{q+1}\ .
$$
We thus obtain a symmetric simplicial set in which every \emph{non-degenerate} simplex  has distinct vertices.
Conversely, it is immediate to check that if every non-degenerate simplex of a symmetric  simplicial set has distinct vertices, then the face of every non-degenerate
simplex is itself non-degenerate, so one can delete degenerate simplices from the structure thus obtaining an unordered $\Delta$-complex in which every simplex has distinct vertices,
i.e.~a multicomplex. The constructions just described are one the inverse of the other, so a multicomplex may be identified with a symmetric simplicial set in which
every non-degenerate simplex has distinct vertices. Under this identification, it is immediate to check that simplicial maps between multicomplexes
correspond to simplicial maps between symmetric simplicial sets (recall that such maps must be equivariant with respect to the action of $\mathfrak{S}_{q+1}$ on $q$-simplices).

We can sum up the discussion above in the following:

\begin{prop}
The following categories are equivalent:
multicomplexes, unordered $\Delta$-complexes in which every simplex has distinct vertices, symmetric simplicial sets in which every non-degenerate simplex has distinct vertices.
\end{prop}

The geometric realization of $\Delta$-complexes and of simplicial sets is constructed in the very same way as the geometric realization of multicomplexes:
one takes one geometric simplex for every combinatorial simplex of the structure, and glue them according to the boundary (and degeneracy, in the case of simplicial sets) maps
(see e.g.~\cite{may} for a thorough discussion of this construction for simplicial sets).
In particular, since every degenerate simplex is glued to one of its faces via an affine projection, 
only non-degenerate simplices give rise to actual geometric simplices in the realization. Let us warn the reader about a small subtlety: with the usual
definitions, the geometric realization
of a multicomplex $K$ differs from the geometric realization of the associated (ordered!) simplicial set $Z$, in that every $q$-simplex  in $K$ gives rise to exactly $(q+1)!$ non-degenerate
simplices in $Z$, and the geometric realization of each simplex of $Z$ admits a canonical identification with the standard simplex $\Delta^n$. On the contrary, 
each $n$-cell of $|K|$ is  identified with $\Delta^n$ only \emph{up to affine isomorphisms of $\Delta^n$}. 
On the other hand, 
in the geometric  realization for \emph{unordered} $\Delta$-complexes and for \emph{symmetric} simplicial sets one asks that $q$-simplices
in the same $\mathfrak{S}_{q+1}$-orbit  be identified  to each other. With this definition,
the geometric realization of a multicomplex coincides with the geometric realization of the corresponding unordered $\Delta$-complex and of the corresponding
symmetric simplicial set. In fact, the geometric
realization of a multicomplex is a \emph{regular unordered $\Delta$-complex}, according to the definition given by Hatcher in~\cite[page 533]{hatcher}
(where \emph{regular} stands for ``with embedded closed cells'', i.e.~with simplices with distinct vertices).

\section{Simplicial (bounded) (co)homology}\label{sec:simpl:coom}
In this section we introduce the simplicial (co)homology of a multicomplex, and we compare it with the singular (co)homology of its geometric realization. The definitions
we are going to give are very natural and extend (or restrict) to the context of multicomplexes well-known definitions that are usually introduced in the theory of
simplicial complexes or of simplicial sets. However, since we will make an extensive use of the explicit chain complexes computing simplicial (co)homology,
we need to set the notation and state some fundamental results for later reference. 

\subsection{Simplicial (co)chains}
Let $K=(V,I,\Omega)$ be a multicomplex, and let $R$ be a ring with unity (in our cases of interest, we will  have either $R=\mathbb{Z}$ or $R=\mathbb{R}$).

The complex $(C_*(K;R),\partial)$ of simplicial chains over $K$ with coefficients in $R$ is simply the chain complex associated to the ordered simplicial set associated to $K$
(see Section~\ref{comparison:subsec}). Therefore, for every $n\in\mathbb{N}$ the module $C_n(K;R)$ is the free $R$-module generated by the set
$$
\{(\sigma,(v_0,\ldots,v_{n}))\, |\, \sigma\in I_A, \, A=\{v_0,\ldots,v_{n}\}\in\mathcal{P}_f(V)\}\ , 
$$ 
while the boundary operator $\partial_n\colon C_n(K;R)\to C_{n-1}(K;R)$ is the $R$-linear extension of the map
$$
(\sigma,(v_0,\ldots,v_{n}))\mapsto \sum_{i=0}^{n} (-1)^i (\partial_{B_i}\sigma,(v_0,\ldots,\widehat{v}_i,\ldots,v_n))\ ,
$$ 
where $B_i = \{v_0, \ldots, \hat{v}_i, \ldots, v_n\}$. Every element of $C_n(K;R)$ of the form 
$$
(\sigma,(v_0,\ldots,v_{n}))\, \qquad  \sigma\in I_A, \ A=\{v_0,\ldots,v_{n}\}\in\mathcal{P}_f(V)
$$ 
will be called an \emph{algebraic simplex} (or, when there is no danger of ambiguity, simply \emph{simplex}).
It is easy to check that $\partial_n\circ\partial_{n+1}=0$ for every $n\geq 1$ (where we understand that $\partial_0=0$).
The simplicial homology $H_*(K;R)$ of $K$ is then the homology of the complex $(C_*(K;R),\partial_*)$. 
As usual, by taking duals one may define the cohomology of $K$: namely, the complex of simplicial cochains
$(C^*(K;R),\delta^*)$ is defined by setting $C^n(K;R)=\Hom(C_n(K;R),R)$ and by letting the coboundary operator $\delta^n\colon C^n(K;R)\to C^{n+1}(K;R)$
be the dual of the map $\partial_{n+1}\colon C_{n+1}(K;R)\to C_n(K;R)$. The simplicial cohomology $H^*(K;R)$ of $K$ is then the cohomology of the
complex $(C^*(K;R),\delta^*)$.


\subsection{The $\ell^1$-norm on chains and the $\ell^\infty$-norm on cochains}
When $R=\mathbb{Z},\mathbb{R}$, one may also define \emph{bounded} simplicial cochains of $K$ as follows. 
Let $c=\sum_{i\in I} a_i \psi_i\in C_n(K;R)$ be a chain written in reduced form (i.e.~suppose that $I$ is a finite set, and the $\psi_i$ are algebraic simplices such that $\psi_i\neq\psi_j$ for $i\neq j$). Then we define
the $\ell^1$-norm of $c$ by setting
$$
\|c\|_1=\sum_{i\in I} |a_i|\ .
$$
In this way, the module $C_n(K;R)$ is endowed with a norm, which restricts to a norm on the submodule of cycles $Z_n(K;R)$, which in turn descends to  a quotient seminorm
(still denoted by $\|\cdot \|_1$) on $H_n(K;R)$. (Observe that, if $R=\mathbb{Z}$, then for every $\alpha \in H_n(K;\mathbb{Z})$ and every $\lambda\in\mathbb{Z}$
we have $\|\lambda \cdot \alpha\|_1\leq |\lambda|\cdot \|\alpha\|_1$, but the inequality may be strict; nevertheless, we will  call the map
$\|\cdot \|_1\colon H_n(K;\mathbb{Z})\to \mathbb{Z}$ a seminorm).

Dually, one can consider the submodule of \emph{bounded cochains} $C^n_b(K;R)$ of $C^n(K;R)$ given by those cochains which are continuous with respect to the $\ell^1$-norm of $C_n(K;R)$,
and endow this space with the dual norm $\|\cdot \|_\infty\colon C^n_b\colon (K;R)\to R$. For notational reasons, we assign an $\ell^\infty$-norm also to cochains
which are not bounded, by setting $\|\varphi\|_\infty=+\infty$ for every $\varphi\in C^n(K;R)\setminus C^n_b(K;R)$. It is readily seen that
$$
\|\varphi\|_\infty=\sup \{|\varphi(s)|, \ s\ \textrm{algebraic\ simplex\ in}\ C_n(K;R)\}
$$
for every
$\varphi\in C^n(K;R)$.
Moreover, the coboundary operator of $C^*(K;R)$ restricts to $C^*_b(K;R)$, which is thus a complex. The cohomology of $C^*_b(K;R)$ is called the \emph{(simplicial) bounded cohomology} of $K$, and it is denoted by $H^*_b(K;R)$. The $\ell^\infty$-norm on 
$C^n_b(K;R)$
restricts to a norm on the submodule of cocycles $Z^n_b(K;R)$, which in turn descends to  a quotient seminorm
(still denoted by $\|\cdot \|_\infty$) on $H^n_b(K;R)$. In the same way one may define an $\ell^\infty$-seminorm (taking values
in $[0,+\infty]$) on $H^n(K;R)$.

(Bounded) simplicial (co)homology is functorial: every simplicial map between multicomplexes induces a chain map on (bounded) (co)chains, which induces in turn a map
in (bounded) (co)homology. In fact, using that a simplicial map takes every single simplex to one single simplex, it is readily seen that simplicial maps between multicomplexes
induce norm non-increasing maps in homology (endowed with the $\ell^1$-norm) and in bounded cohomology (endowed with the $\ell^\infty$-norm).

\subsection{Reduced chains and alternating cochains}\label{subsec:reduced:alternating}
By definition, every $n$-simplex of a multicomplex $K$ gives rise to $(n+1)!$ algebraic simplices in $C_n(K;R)$, and to many algebraic simplices in
$C_m(K;R)$ for every $m>n$. It is sometimes useful to reduce the number of generators for simplicial chains (of course, without altering the resulting homology
modules) by discarding all chains that are degenerate, according to a suitable definition that we are now going to give. 

For every $n\in\mathbb{N}$, we denote by $D_n(K;R)$ the submodule of $C_n(K;R)$ generated by the set
\begin{align*}
& \left\{(\sigma,(v_0,\ldots,v_n))- \varepsilon(\tau) (\sigma,(v_{\tau(0)},\ldots,v_{\tau(n)}))\ |\ \tau\in\mathfrak{S}_{n+1}\right\} \\ \cup\  & \left\{(\sigma,(v_0,\ldots,v_n))\, |\, v_i=v_j\ {\rm for\ some}\ i\neq j\right\}\ ,
\end{align*}
where we denote by $\varepsilon(\tau)=\pm 1$ the sign of the permutation $\tau$. We then define the module of \emph{reduced} $n$-chains
$C_n(K;R)_{\red}$ as the quotient $$C_n(K;R)_{\red}=C_n(K;R)/D_n(K;R)\ .$$ It is readily seen that the boundary of a degenerate chain is still degenerate,
so the boundary operators of the complex $C_*(K;R)$ descend to boundary operators (still denoted by $\partial_*$) on $C_*(K;R)_{\red}$, which is therefore endowed
with the structure of a complex. Observe that to every $n$-simplex of $K$ there is now associated a unique algebraic simplex of $C_n(K;R)_{\red}$, and that
the elements of $C_n(K;R)_{\red}$ arising from the $n$-simplices of $K$ provide a basis of $C_n(K;R)_{\red}$ (which, in particular, is still free).
In order to ease notation, from now on we will often denote an element of $C_n(K;R)$ and its class in $C_n(K;R)_{\red}$ by the same symbol (however, it will always be clear whether
we are working in $C_*(K;R)$ or in $C_*(K;R)_{\red}$). For example, 
the identity 
$$
(\sigma,(v_0,v_1))=-(\sigma,(v_1,v_0))
$$
holds in $C_1(K;R)_{\red}$ (and does not hold in $C_1(K;R)$).

\begin{rem}
 If $R=\mathbb{R}$ (or, more in general, if $2$ is invertible in $R$), then the submodule $D_n'(K;R)$ of $C_n(K;R)$ generated by the set
$$\left\{(\sigma,(v_0,\ldots,v_n))- \varepsilon(\tau) (\sigma,(v_{\tau(0)},\ldots,v_{\tau(n)}))\ |\ \tau\in\mathfrak{S}_{n+1}\right\}$$ automatically contains every element of the form
$\alpha=(\sigma,(v_0,\ldots,v_n))$, where $v_i=v_j$
for some $i\neq j$. When $R=\mathbb{Z}$, for such an element $\alpha$ we just have $2\alpha \in D_n'(K;R)$, from which we cannot deduce that $\alpha\in D_n'(K,R)$ too.  
\end{rem}

Coming to cohomology, one can consider the dual of $C_n(K;R)_{\red}$, which may be identified with
the subspace $C^n(K;R)_{\alt}\subseteq C^n(K;R)$ of cochains which vanish on $D_n(K;R)$. Such cochains are called \emph{alternating}, because 
an element $\varphi\in C^n(K;R)$ lies in $C^n(K;R)_{\alt}$ if and only if it satisfies the equality
$$
\varphi(\sigma,(v_{\tau(0)},\ldots,v_{\tau(n)}))=\varepsilon(\tau) \varphi(\sigma,(v_0,\ldots,v_n))
$$
for every $\tau\in\mathfrak{S}_{n+1}$. The differential  preserves alternating cochains, so $(C^*(K;R)_{\alt},\delta^*)$ is a complex. We will be interested also in the complex
$(C^*_b(K;R)_{\alt},\delta^*)$ 
of alternating bounded cochains, which is defined just by setting $C^n_b(K;R)_{\alt}=C^n_b(K;R)\cap C^n(K;R)_{\alt}$ (recall that $\delta^*$ also preserves boundedness of cochains).
The $\ell^1$-norm on simplicial chains induces a quotient seminorm (which is in fact a norm) on reduced simplicial chains, and the $\ell^\infty$-norm on (bounded) cochains
restricts to a norm on alternating (bounded) cochains; we will still denote these norms by $\|\cdot \|_1$ and $\|\cdot\|_\infty$, respectively.

The following theorem ensures that the $\ell^1$-seminorms induced on $H_*(K;R)$ by $C_*(K;R)$ and by $C_*(K;R)_{\red}$ coincide, and that the same
holds for the $\ell^\infty$-seminorms induced on $H^*_b(K;R)$ by $C^*_b(K;R)$ and by $C^*_b(K;R)_{\alt}$.

\begin{thm}\label{reduced:equivalence}
The chain projection $C_*(K;R)\to C_*(K;R)_{\red}$ induces an isometric isomorphism between $H_*(K;R)$ and the homology of the complex
$C_*(K;R)_{\red}$. The inclusions of complexes $C^*(K;R)_{\alt}\to C^*(K;R)$ (resp.~$C_b^*(K;R)_{\alt}\to C_b^*(K;R)$) induce
isometric isomorphisms between the cohomology of the complex $C^*(K;R)_{\alt}$ (resp.~$C_b^*(K;R)_{\alt}$) and
$H^*(K;R)$ (resp.~$H_b^*(K;R)$).
\end{thm}
\begin{proof}
The fact that the maps described in the statement are isomorphisms is well known for simplicial sets, whence for multicomplexes, since the (co)homology of a multicomplex
is just the (co)homology of the associated simplicial set. However, we are interested here in showing that these isomorphisms are isometric, which is a bit more delicate
(see Remark~\ref{noisog}). To this aim we retrace
 an argument originally due to Eilenberg, referring the reader e.g.~to~\cite[Theorem~9.1]{eil-sing} and \cite[Theorems~6.9 and 6.10]{E-Steenrod} for the details.
Eilenberg's proof concerns in fact simplicial complexes, and uses the fact (which still holds in multicomplexes) that simplices have distinct vertices.


Let us consider the projection $p_*\colon C_*(K;R)\to C_*(K;R)_{\red}$,
let us fix an auxiliary total ordering of the vertices of $K$, and
let $s_*\colon C_*(K;R)_{\red}\to C_*(K;R)$
be the unique $R$-linear map such that the following conditions hold: $s_n(\sigma,(v_0,\ldots,v_n))=0$ whenever at least two of the $v_i$ coincide;
$s_n(\sigma,(v_0,\ldots,v_n))=\varepsilon(\tau) (\sigma,(v_{\tau(0)},\ldots,v_{\tau(n)}))$, where $\tau\in \mathfrak{S}_{n+1}$ is the unique
permutation such that $v_{\tau(0)}<v_{\tau(1)}<\ldots<v_{\tau(n)}$. It is readily seen that $s_*$ is a well-defined chain map. By definition, the composition $p_*\circ s_*$
is equal to the identity of $C_*(K;R)_{\red}$. Moreover,
it can be shown that $s_*\circ p_*$ is homotopic to the identity of $C_*(K;R)$ (via a bounded homotopy with respect to the $\ell^1$-norms), and this implies at once that
$p_*$ induces an isomorphism on homology. On the other hand, the inclusion of alternating cochains into generic ones is just the dual map of the projection $p_*$.
As a consequence, such inclusion is also a homotopy equivalence of complexes (where homotopies can be chosen to be bounded with respect to the $\ell^\infty$-norm).
The conclusion follows from the fact that both $p_*$ and $s_*$ (and, therefore, also their dual maps) are norm non-increasing.
\end{proof}


\begin{rem}\label{noisog}
The fact that the homology (resp.~(bounded) cohomology) computed by
reduced chains (resp.~by alternating cochains) is isomorphic to the homology (resp.~(bounded) cohomology) 
of ordinary chains (resp.~(bounded) cochains)
holds true also for $\Delta$-complexes and simplicial sets.

On the contrary, the fact that reduced chains induce on $H_*(K;R)$ the same $\ell^1$-seminorm as ordinary chains is due to the fact that simplices
in a multicomplex have distinct vertices. This assumption is indeed necessary, as the following example shows. Let $K$ be the unordered $\Delta$-complex obtained from a 
$3$-simplex $\Delta$ with vertices $v_0,v_1,v_2,v_3$ via the following identifications: the face spanned by $v_1,v_2,v_3$ is identified with the face spanned
by $v_0,v_2,v_3$ via the affine isomorphism sending $v_1$ to $v_2$, $v_2$ to $v_3$ and $v_3$ to $v_0$;
 the face spanned by $v_0,v_1,v_3$ is identified with the face spanned
by $v_0,v_1,v_2$ via the affine isomorphism sending $v_0$ to $v_1$, $v_1$ to $v_2$ and $v_3$ to $v_0$. If $\sigma$ is the algebraic simplex 
$(\Delta,(v_0,v_1,v_2,v_3))$, then it is readily seen that $\partial \sigma$ is  degenerate but different from $0$. As a consequence,
the class $[\sigma]$ of $\sigma$ in $C_3(K;\mathbb{Z})_{\red}$ is a cycle. Let $\alpha\in H_3(K;\mathbb{Z})$ be the class of  $[\sigma]$. 
By construction, the seminorm of $\alpha$ induced by $C_*(K;\mathbb{Z})_{\red}$ is not bigger than 1 (and, since $\alpha\neq 0$, it is equal to 1).
On the other hand, using that $\partial \sigma \neq 0$ it is not difficult to show that $\alpha$ cannot be represented by any cycle consisting
of a single algebraic simplex in $C_3(K;R)$. This shows that the seminorms induced on $H_3(K;\mathbb{Z})$ by $C_*(K;\mathbb{Z})$ and
 by $C_*(K;\mathbb{Z})_{\red}$ are distinct.
\end{rem}

Thanks to Theorem~\ref{reduced:equivalence}, we will simply denote by $H_*(K;R)$ (resp.~$H^*_{(b)}(K;R)$) the homology of the complex $C_*(K;R)_\red$ (resp.~$C^*_{(b)}(K;R)_{\alt}$),
i.e.~we will feel free to study (bounded) (co)homology by working directly with reduced chains and with alternating cochains, when useful.
Also observe that simplicial maps between multicomplexes take degenerate chains to degenerate chains, so they induce (norm non-increasing) maps on reduced chains and on alternating cochains. 

A classical and fundamental  result in the (co)homology theory of simplicial complexes and simplicial sets is that the simplicial (co)homology of such objects is canonically isomorphic to
the singular (co)homology of their geometric realizations. 
The same result also holds for multicomplexes. More precisely, 
let 
\begin{displaymath}
\phi_* \colon C_{*}(K ; R) \rightarrow C_{*}(\lvert K \rvert ; R)
\end{displaymath}
be the $R$-linear chain map
sending every element $(\sigma,(v_0,\ldots,v_n))\in C_*(K;R)$ to the singular simplex 
$$
\Delta^n\to |K|\, ,\quad (t_0,\ldots,t_n)\mapsto (\sigma,t_0v_0+\ldots+\ldots t_nv_n)\ ,
$$
and denote by $\phi^*\colon C^*(|K|;R)\to C^*(K;R)$ the dual chain map induced by $\phi_*$ on cochains. 

\begin{thm}\label{simpl-hom-eq-sing-one}
For any multicomplex $K$, the homomorphisms
\begin{displaymath}
H_{*}(K; R) \rightarrow H_{*}(\lvert K \rvert; R)\, ,\qquad
H^*(|K|;R)\to H^*(K;R)
\end{displaymath}
induced by $\phi_*$ and $\phi^*$, respectively, 
are isomorphims in every degree.
\end{thm}
\begin{proof}
The  standard argument for the isomorphism between 
simplicial and singular cohomology for $\Delta$-complexes (see e.g.~\cite[Theorem~2.27]{hatcher}) applies \emph{verbatim} to multicomplexes.
\end{proof}

Of course, Theorem~\ref{simpl-hom-eq-sing-one} cannot hold for \emph{bounded} cohomology: if $K$ is a finite multicomplex, then every simplicial cochain on $K$ is bounded, while there may well be
singular classes on $|K|$ which do not admit any bounded representative (we refer the reader to Section~\ref{bounded:singular:sec} for the definition of bounded cohomology
of a topological space).  
One of the key results on the bounded cohomology
of multicomplexes establishes a suitable version of Theorem~\ref{simpl-hom-eq-sing-one} for the notable class of \emph{complete} multicomplexes, that will be introduced and studied in Chapter~\ref{chap2:hom} (see Theorem~\ref{isometry-lemma-intro}).

\subsection{The relative case}

Let $K$ be a multicomplex and let $L$ be a submulticomplex of $K$. The simplicial inclusion $L\hookrightarrow K$ induces chain maps $C_*(L;R)\to C_*(K;R)$ and 
$C_*(L;R)_{\red}\to C_*(K;R)_{\red}$, and, after identifying (reduced) chains on $L$ with their images in $C_*(K;R)$ (or in $C_*(K;R)_{\red}$), one can define as usual relative chains for the pair $(K,L)$ by setting
$$
C_*(K,L;R)=C_*(K;R)/C_*(L;R)\, ,\quad C_*(K,L;R)_{\red}=C_*(K;R)_{\red}/C_*(L;R)_{\red}\ .
$$
The boundary operator for absolute (reduced) chains induces a boundary operator for relative (reduced) chains, and we denote by
$H_*(K,L;R)$ the resulting homology. Theorem~\ref{reduced:equivalence} also holds in the relative context (with the very same proof as for the absolute case),
so by the symbol $H_*(K,L;R)$ we may equivalently denote the homology of chains or of reduced chains.

The short exact sequence of complexes 
\begin{displaymath}
\xymatrix{
0 \ar[r]  & C_{*}(L; {R}) \ar[r]^{i} & C_{*}(K; R) \ar[r]^{j} & C_{*}(K,L; R) \ar[r]  & 0,
}
\end{displaymath}
where $i$ is the inclusion and $j$ the quotient map, induces 
the long exact sequence 
\begin{align}\label{long-ex-seq-multi}
\begin{split}
\cdots \longrightarrow &\  H_{n}(L; R) \ 
\xrightarrow{H_n(i_{n})}
\ H_{n}(K; R) \ 
\xrightarrow{H_n(j_{n})}
\ H_{n}(K,L; R)\  \xrightarrow{\partial_n}\\
\xrightarrow{\partial_n} & \ H_{n-1}(L; R) \ \longrightarrow \ldots  
\end{split}
\end{align}

The inclusion of $L$ in $K$ also defines a restriction map $C^*_{(b)}(K;R)\to C^*_{(b)}(L,R)$, whose kernel
$C^*_{(b)}(K,L;R)$ is canonically identified with the dual (or the topological dual, in the case of bounded cochains)
of $C_*(K,L;R)$. After endowing $C^*_{(b)}(K,L;R)$ with the restriction of the differential of $C^*_{(b)}(K;R)$ we can compute the cohomology
$H^*_{(b)}(K,L;R)$
of the resulting complex, which is by definition the relative (bounded) cohomology of the pair $(K,L)$. Just as for homology, we can compute (bounded) cohomology of pairs
using alternating cochains, and we have a long exact sequence
\begin{equation}\label{long-ex-seq-multi-coh}
\xymatrix@C = 1.4em{
\cdots\ar[r] & H^n_{(b)}(K,L;R) \ar[r] & H^{n}_{(b)}(K; R) \ar[r] & H^{n}_{(b)}(L; R) \ar[r] & 
H^{n+1}_{(b)}(K,L; R) \ar[r] & \cdots
}
\end{equation}

As in the absolute case, if we denote by 
\begin{displaymath}
\phi_* \colon C_{*}(K, L ; R) \rightarrow C_{*}(\lvert K \rvert, \lvert L \rvert; R)
\end{displaymath}
the $R$-linear map 
sending (the class of) every algebraic simplex $(\sigma,(v_0,\ldots,v_n))\in C_{n}(K,L;R)$ to (the class of) the singular simplex 
$$
\Delta^n\to |K|\, ,\quad (t_0,\ldots,t_n)\mapsto (\sigma,t_0v_0+\ldots+\ldots t_nv_n)\ ,
$$
and  by $\phi^*\colon C^*(|K|,|L|;R)\to C^*(K,L;R)$ its dual chain map, then we have the following

\begin{thm}\label{simpl-hom-eq-sing-one-rel}
For any pair of multicomplexes $(K,L)$, the homomorphisms
\begin{displaymath}
H_{*}(K, L ; R) \rightarrow H_{*}(\lvert K \rvert, \lvert L \rvert; R)\, ,\qquad
H^*(|K|,|L|;R)\to H^*(K,L;R)
\end{displaymath}
induced by $\phi_*$ and $\phi^*$, respectively, 
are isomorphims in every degree.
\end{thm}

\section{Group actions on multicomplexes}\label{quotient:sec}
Let $\Gamma$ be a group acting simplicially on a multicomplex $K=(V,I,\Omega)$ , i.e.~suppose that we are given a group homomorphism
$\rho\colon \Gamma \to {\rm Aut}(K)$, where ${\rm Aut}(K)$ is the group of simplicial automorphisms of $K$. In general, in order to give sense to the notion of quotient
multicomplex $K/\Gamma$ one needs to subdivide $K$ (see Remark~\ref{freeno}). However, things get much better for special classes of group actions.

\begin{defn}
A simplicial action $\Gamma\actson K$  is called \emph{$0$-trivial}
if $g(v)=v$ for every vertex $v$ of $K$ and every $g\in\Gamma$, i.e.~if $\Gamma$ acts trivially on the $0$-skeleton of $K$.
\end{defn}

If an action $\Gamma\actson K$ is $0$-trivial, then we can define a quotient multicomplex $K/\Gamma=(V',I',\Omega')$ as follows:
$V'=V$, $(I')^n=I^n/\Gamma$ (i.e. $n$-simplices of $K/\Gamma$ are in bijection with $\Gamma$-orbits
of $n$-simplices of $K$), and $\partial_{A,B} [\sigma]=[\partial_{A,B} \sigma]$ for every $[\sigma] \in I'_A$ (where for every simplex $\sigma$ of $K$ we denote by $[\sigma]$ the equivalence class of $\sigma$ with respect to the action
of $\Gamma$). It is immediate to check that the triple
$(V',I',\Omega')$ indeed defines a multicomplex, and that the following proposition holds:

\begin{prop}\label{0trivial:prop}
Let $\Gamma\actson K$ be a $0$-trivial simplicial action. Then the map
$$
\pi\colon K\to K/\Gamma\, ,\qquad \pi(\sigma)=[\sigma]
$$
is a non-degenerate simplicial map. Moreover, if we denote by $|g|\colon |K|\to |K|$ (resp.~by $|\pi|\colon |K|\to|K/\Gamma|$) the map induced by $g\in\Gamma$ (resp.~by $\pi$)
on geometric realizations, then for every $g\in\Gamma$ the following diagram commutes:
$$
\xymatrix{|K| \ar[dr]_{|\pi|} \ar[rr]^{|g|} & & |K| \ar[dl]^{|\pi|} \\
& |K/\Gamma| &
}
$$
In particular, $|K/\Gamma|$ is naturally homeomorphic to $|K|/\Gamma$. 
\end{prop}

\begin{rem}\label{freeno}
Proposition~\ref{0trivial:prop} cannot hold in general, even when restricting to free actions.
For example, let $K$ be the $1$-dimensional multicomplex having exactly 2 vertices and 2 edges (so that $|K|$ is homeomorphic to $S^1$), let $\Gamma=\mathbb{Z}_2$ and suppose that the non-trivial
element $g_0$ of $\Gamma$ acts on $K$ by switching the vertices and the edges of $K$. Then $|g_0|$ acts on $|K|$ as the antipodal map, so $|K|/\Gamma$ is still homeomorphic to $S^1$. The cellular structure
of $|K|$ projects onto a cellular structure for $|K|/\Gamma$ having just one vertex and one edge. Such a structure cannot correspond to any multicomplex structure, since edges in (the geometric
realization of) a multicomplex always have distinct endpoints.
\end{rem}

\chapter{The singular multicomplex}\label{sing:mult:chap}

To our purposes, the most important example of multicomplex is the \emph{singular multicomplex} associated to a topological space. 
The name \emph{singular multicomplex} evokes the  (\emph{total}) \emph{singular complex} in the theory of $\Delta$-complexes and simplicial sets: 
if $X$ is a topological space,
the singular complex $\mathcal{S}(X)$ of $X$ is the simplicial set having as simplices the singular simplices with values in $X$ (see e.g.~\cite{eil-sing, eil-zil,  milnor-geom, Piccinini, GJ}).
Roughly speaking, the singular multicomplex $\mathcal{K}(X)$ is the multicomplex having as simplices
the singular simplices in $X$ with distinct vertices, up to affine symmetries. Therefore, the singular multicomplex differs from the singular complex
both beacuse of the requirement that singular simplices be injective on vertices, and because the geometric simplices of the singular multicomplex
do not come with a preferred ordering of their vertices.

Let us now be more precise. We denote by ${S}_n(X)$ the set of singular simplices with values in $X$, and 
for every $\sigma \in S_n(X)$ we define the set of vertices of $\sigma$ as the image of the set of vertices of $\Delta^n$ via $\sigma$.
Then, we denote by $S_n^0(X)$ the subset of 
$S_n(X)$ given by the singular simplices whose set of vertices consists of exactly $(n+1)$ points (i.e.~$S_n^0(X)$ is the set of singular $n$-simplices
that are injective on the vertices of $\Delta^n$).
We say that two simplices $\sigma,\sigma'\in S_n^0(X)$ are equivalent
if $\sigma'=\sigma\circ \tau$, where $\tau$ is an affine diffeomorphism of $\Delta^n$ into itself, and we denote by
$\overline{S}_n^0(X)$ the set of equivalence classes of elements of $S_n^0(X)$. Observe that, if $\overline{\sigma}\in \overline{S}_n^0(X)$, then 
we may define the set of vertices of $\overline{\sigma}$ as the set of vertices of any of its representatives. 
We are now ready to define the \emph{singular multicomplex} 
$$
\mathcal{K}(X)=(V,I,\Omega)
$$
of $X$ 
as follows: $V=X$, i.e.~the vertices of $\mathcal{K}(X)$ are just the points of $X$; for every subset $A\subseteq V=X$ such that $|A|=n+1$, the set $I_A$ consists 
of the elements of $\overline{S}_n^0(X)$ having $A$ as set of vertices; finally, if $\overline{\sigma}\in I_A$ and $B\subseteq A$, then we choose a representative $\sigma$ of $\overline{\sigma}$,
and we define $\partial_{A,B}\overline{\sigma}$ as the equivalence class of the unique face of $\sigma$ having $B$ as set of vertices. 
It is easy to check that this construction indeed defines a multicomplex.

The geometric realization of $\mathcal{K}(X)$ is a CW complex whose $0$-cells are in bijection with the points of $X$. Moreover, there is a natural map
\begin{displaymath}
S_X \colon \lvert \mathcal{K}(X) \rvert \rightarrow X
\end{displaymath}
which is defined as follows. Recall that a point of $|\mathcal{K}(X)|$ is represented by a pair $(\overline{\sigma},t_0x_0+\ldots+t_nx_n)$,
where $\overline{\sigma}\in \overline{S}_n^0(X)$ and $x_0,\ldots,x_n$ are the vertices of $\overline{\sigma}$ (see Remark \ref{rem:descr:punti:geom:real}). We choose the representative $\sigma$ of $\overline{\sigma}$
sending the $i$-th vertex of $\Delta^n$ to $x_i$, and we set
$$
S_X((\overline{\sigma},t_0x_0+\ldots+t_nx_n))=\sigma (t_0 e_0 + \ldots + t_n e_n)\ .
$$
It is easy to check that $S_X$ is well defined and continuous. We will often call the map $S_X$ the \emph{natural projection} of $|\calK(X)|$ onto $X$. When the space $X$ is clear from the context, we will often denote
$S_X$ simply by $S$.

Let now $K$ be a multicomplex and let $X$ be a topological space. If $f\colon |K|\to X$ is any continuous map such that $f|_A$ is injective for every $A\subseteq V$ such that $I_A\neq \emptyset$, 
then we can define a non-degenerate simplicial map ${}^t f\colon K\to \calK(X)$ as follows: 
if $\overline{\sigma}$ is an $n$-simplex of $K$ and $\eta\colon \Delta^n\to |\sigma|$ is a characteristic map for $\sigma$, then ${}^t f(\sigma)=\overline{f\circ\sigma}$,
where $\sigma$ is any representative of $\overline{\sigma}$. 
On the other hand, if $g\colon K\to \calK(X)$ is a non-degenerate simplicial map, then we can define a continuous map ${}^t g\colon |K|\to X$ by setting ${}^t g=S_X\circ |g|$ (and, for every $A\subseteq V$ such that $I_A\neq \emptyset$, the
map ${}^t g$ is automatically injective on $A$). 


The map ${}^t f$ (resp.~${}^t g$) is called the \emph{transpose} of $f$ (resp.~of $g$). It is straightforward to check that transposition satisfies the following 
identities:
$$
{}^t({}^t f)=f\, ,\qquad {}^t({}^t g) =g\, \ .
$$
In particular, if $X=|K|$, the identity ${\rm Id}\colon |K|\to |K|$ induces a non-degenerate simplicial map ${}^t{\rm Id}\colon K\to\calK(|K|)$. Since
$$
S_{|K|}\circ |{}^t{\rm Id}|= {}^t({}^t{\rm Id})={\rm Id}\ ,
$$
the map ${}^t{\rm Id}\colon K\to\calK(|K|)$ is injective, and will be called the \emph{natural embedding} of $K$ into $\calK(|K|)$.

Our next goal is to show that, at least for reasonable topological spaces, the natural projection $S\colon |\mathcal{K}(X)|\to X$
is a weak homotopy equivalence. 

\section{The weak homotopy type of the singular multicomplex}\label{Sec:weak:hom:type:K(X)}
In order to prove that the natural projection $S\colon |\mathcal{K}(X)|\to X$
is a weak homotopy equivalence we first need to select a suitable class of topological spaces to work with. 


\begin{defn}\label{good:def}
 Let $X$ be a topological space. We say that $X$ is \emph{good} if:
 \begin{enumerate}
  \item $X$ is semilocally simply connected (i.e.~any path connected component of $X$ admits a universal covering);
  \item every non-empty path connected finite subset of $X$ is  a singleton.
 \end{enumerate}
\end{defn}

Of course, locally contractible spaces satisfy condition (1) in the above definition, while it is an easy exercise to show that condition (2) is satisfied e.g.~by every topological space with 
the separation property $T_1$. In particular,
CW complexes (with no restrictions on the cardinality of cells) are good spaces.

 As anticipated in the introduction, we have the following:
 
\begin{thm:repeated:homotopy-weak-intro}
Let $X$ be a good topological space. Then, the natural projection
\begin{displaymath}
S \colon |\mathcal{K}(X)| \rightarrow X
\end{displaymath}
is a weak homotopy equivalence.
\end{thm:repeated:homotopy-weak-intro}

Since CW complexes are good spaces, thanks to Whitehead Theorem we immediately have the following:

\begin{cor}\label{weak-CW}
 Let $X$ be a CW complex. Then, the natural projection
\begin{displaymath}
S \colon |\mathcal{K}(X)| \rightarrow X
\end{displaymath}
is a homotopy equivalence.
\end{cor}

The singular multicomplex $\calK(X)$ is constructed by gluing (classes of) singular simplices that are injective on vertices. By considering all possible singular simplices,
one gets instead the 
classical \emph{singular simplicial set} $\calS(X)$ associated to  $X$ (see e.g.~\cite{eil-sing, milnor-geom, Piccinini, may} for the formal definition and the main properties of $\calS(X)$).
Also the geometric realization of the simplicial set $\calS(X)$ comes equipped with a natural projection onto $X$, and a fundamental classical result ensures that such projection is a weak homotopy equivalence
(see e.g.~\cite[Theorem~4]{milnor-geom} and \cite[Theorem~4.5.30]{Piccinini}). Moreover, as we will see at the end of the section, $\calK(X)$ embeds into $\calS(X)$, hence in order to prove 
Theorem~\ref{homotopy-weak-intro} it would be sufficient to show that $|\calK(X)|$ is a deformation retract of $|\calS(X)|$ (at least for good spaces). We will obtain this result as a corollary of Theorem~\ref{homotopy-weak-intro}.
But writing down an explicit  retraction of
 $|\calS(X)|$ onto $|\calK(X)|$ turned out to be so difficult that we decided to prove Theorem~\ref{homotopy-weak-intro} via a different strategy (which is inspired by Milnor's approach to the same problem in the context of simplicial sets~\cite{milnor-geom}).
 
Our proof of Theorem~\ref{homotopy-weak-intro} consists of two steps: we first prove that 
  $S$ induces an isomorphism at the level of fundamental groups, and then we show that $S$ may be lifted to a  weak homotopy equivalence between the universal coverings of  $|\mathcal{K}(X)|$ and $X$. 
  Since coverings induce isomorphisms on higher homotopy groups, these facts suffice to prove that $S$ 
is a weak homotopy equivalence.  In order to show that the universal coverings of  $|\mathcal{K}(X)|$ and $X$ are weakly homotopy equivalent, we will make use of Hurewicz Theorem \cite[Corollary 4.33]{hatcher},
which states that a map between simply connected spaces is a weak homotopy equivalence provided it induces an isomorphism on homology with integral coefficients.

Without loss of generality, we may (and henceforth we do) suppose that $X$ is path connected. We first characterize 
 the universal covering of $\lvert \mathcal{K}(X) \rvert$, showing at the same time that $S$ induces an isomorphism on fundamental groups. 
 We denote by $\pi \colon \widetilde{X} \rightarrow X$ the universal covering of $X$.
 Let $\widetilde{\mathcal{K}(X)}$ be the submulticomplex of $\mathcal{K}(\widetilde{X})$ such that a simplex $\overline{\sigma}$ of $\mathcal{K}(\widetilde{X})$ belongs to
 $\widetilde{\mathcal{K}(X)}$ if and only if $\overline{\pi\circ \sigma}$ belongs to $\mathcal{K}(X)$ for any choice of a representative $\sigma$ of $\overline{\sigma}$. In other words,
 $\widetilde{\mathcal{K}(X)}$ contains (the classes of) the singular simplices in $\widetilde{X}$ whose vertices project onto distinct points of $X$.

 There is an obvious simplicial map $\widetilde{\mathcal{K}(X)}\to \mathcal{K}(X)$ which takes the class of a singular simplex $\sigma\colon \Delta^n\to\widetilde{X}$
 to the class of its projection $\pi\circ\sigma$ on $X$. This map is well defined exactly thanks to the definition of $\widetilde{\mathcal{K}(X)}$. We denote by
 $\widetilde{\pi}\colon \left|\widetilde{\mathcal{K}(X)}\right| \to |\mathcal{K}(X)|$ the induced  map on geometric realizations. We also denote by 
 $$
 \widetilde{S}\colon \left|\widetilde{\mathcal{K}(X)}\right|\to \widetilde{X}
 $$
 the restriction of the natural projection $|\mathcal{K}(\widetilde{X})|\to \widetilde{X}$.

 \begin{lemma}
 The space $\left|\widetilde{\mathcal{K}(X)}\right|$ is path connected and simply connected.
 \end{lemma}
 \begin{proof}
 Since $\left|\widetilde{\mathcal{K}(X)}\right|$ is a CW complex, it is sufficient to show that every pair of vertices are connected by an edge, and every simplicial path
 is null homotopic. Let $\widetilde{x}_0, \widetilde{x}_1$ be distinct vertices of $\left|\widetilde{\mathcal{K}(X)}\right|$, i.e.~points of $\widetilde{X}$. 
 Since $\widetilde{X}$ is path connected, we can choose a path $\gamma\colon [0,1]\to \widetilde{X}$ joining $\widetilde{x}_0$ and $\widetilde{x}_1$. 
 If $\pi(\widetilde{x}_0)\neq \pi(\widetilde{x}_1)$, such a path defines a $1$-simplex joining $\widetilde{x}_0$ and $\widetilde{x}_1$ in 
 $\left|\widetilde{\mathcal{K}(X)}\right|$. Suppose now  $\pi(\widetilde{x}_0)= \pi(\widetilde{x}_1)$. The path $\pi\circ \gamma$ cannot be constant, because otherwise the image of $\gamma$ would be contained in $\pi^{-1}(\pi(\widetilde{x}_0))$, which is discrete, so $\gamma$ itself should be constant, a contradiction. Therefore, the image of $\gamma$ contains points that do not project onto $\pi(\widetilde{x}_0)=\pi(\widetilde{x}_1)$, and we can express $\gamma$ as the concatenation of two paths that define $1$-simplices of $|\widetilde{\mathcal{K}(X)}|$. The concatenation of such $1$-simplices
 joins $\widetilde{x}_0$ and $\widetilde{x}_1$ in $|\widetilde{\mathcal{K}(X)}|$, and this concludes the proof of the fact that $|\widetilde{\mathcal{K}(X)}|$
 is path connected.
 
 Let now $Z$ be a finite multicomplex such that $|Z|$ is homeomorphic to $S^1$, and let $\gamma \colon Z\to \widetilde{\mathcal{K}(X)}$ be a non-degenerate simplicial map.
 As discussed above, in order to show that $\left|\widetilde{\mathcal{K}(X)}\right|$ is simply connected it is sufficient to show that the geometric realization
 $|\gamma|$ of $\gamma$ is null-homotopic. Of course, if $X$ consists of a single point there is nothing to prove, so since $X$ is good and path connected
 we can assume that $X$ contains infinitely many points.  In particular, there exists a point $\widetilde{x}_0$ in $\widetilde{X}$ whose projection on
 $X$ is not contained in $\pi(|\gamma|(Z^0))$. If $D$ is the multicomplex obtained by coning $Z$ over a point $O$ (so that $|D|$ is homeomorphic to
 the $2$-dimensional disc), using the fact that $\widetilde{X}$ is simply connected we can extend $\widetilde{S}\circ |\gamma|\colon |Z|\to \widetilde{X}$
to a map $f\colon |D|\to \widetilde{X}$ sending $O$ to $\widetilde{x}_0$. By construction, the restriction of $\pi\circ f$ to every simplex of $D$ is injective
on vertices, so $f$ defines a transpose simplicial map ${}^tf\colon D\to {\mathcal{K}(\widetilde{X})}$. 
By construction, the map ${}^t f$ takes values in $\widetilde{\mathcal{K}({X})}\subseteq {\mathcal{K}(\widetilde{X})}$, and its 
geometric realization  extends
$|\gamma|$ to $|D|$, thus showing that $|\gamma|$ is null-homotopic in $|\mathcal{K}(\widetilde{X})|$. 
 \end{proof}
  
 \begin{prop}\label{fund-group:prop}
 The map $\widetilde{\pi}\colon \left|\widetilde{\mathcal{K}(X)}\right| \to |\mathcal{K}(X)|$ is a universal covering. Moreover, the map
 $S \colon |\mathcal{K}(X)| \rightarrow X$ induces an isomorphism at the level of fundamental groups.
 \end{prop}
 \begin{proof}
 Let $\Gamma\cong \pi_1(X)$ be the group of deck transformations of the covering $\pi\colon \widetilde{X}\to X$. Then $\Gamma$ acts simplicially on $\widetilde{\mathcal{K}(X)}$ as follows:
 if $\overline{\sigma}$ is a simplex of $\widetilde{\mathcal{K}(X)}$ represented by $\sigma \colon \Delta^n \to \widetilde{X}$, then $\gamma\cdot \overline{\sigma}$ is the class of $\gamma\circ \sigma$. Since $\gamma$ is a deck transformation, it is easily seen that this action is well defined. Moreover,
 an element  $\gamma\in \Gamma$ leaves a simplex of $\widetilde{\mathcal{K}(X)}$ invariant only if $\gamma={\rm Id}$. Indeed, being free on $\widetilde{X}$, the action of $\Gamma$ is free
 on the vertices of $\widetilde{\mathcal{K}(X)}$. Moreover, if $\gamma\cdot\overline{\sigma}=\overline{\sigma}$ for some simplex $\overline{\sigma}$, then $\gamma$ permutes the vertices of $\overline{\sigma}$; but the vertices
 of $\overline{\sigma}$ project to distinct vertices of $X$, so they are pairwise $\Gamma$-non-equivalent. This implies that $\gamma$ fixes each vertex of $\overline{\sigma}$, so it is the identity. The induced action on the
 geometric realization $\left|\widetilde{\mathcal{K}(X)}\right|$ is therefore \emph{wandering} (i.e.~every point of $\left|\widetilde{\mathcal{K}(X)}\right|$ admits a neighbourhood which is disjoint form the union of
 all its translates). This implies that the quotient map with respect to the $\Gamma$-action on $\left|\widetilde{\mathcal{K}(X)}\right|$ is a covering. In order to conclude that $\widetilde{\pi}$ is a covering
 we now just need to observe that $\widetilde{\pi}\colon \left|\widetilde{\mathcal{K}(X)}\right|\to
 |{\mathcal{K}(X)}|$  is open (since both $\left|\widetilde{\mathcal{K}(X)}\right|$
 and $|{\mathcal{K}(X)}|$ are endowed with the weak topology), and that $\widetilde{\pi}(x)=\widetilde{\pi}(y)$ if and only if $x$ and $y$ lie in the same $\Gamma$-orbit. 
 
 Finally, since $\left|\widetilde{\mathcal{K}(X)}\right|$
 is simply connected, the map $\widetilde{\pi}$ is a universal covering.
  
 The fact that $S$ induces an isomorphism on fundamental groups is now an easy consequence of the fact the the following diagram commutes
 \begin{displaymath}
\xymatrix{
\left|\widetilde{\mathcal{K}(X)}\right| \ar[r]^-{\widetilde{S}} \ar[d]_-{\tilde{\pi}} & \widetilde{X} \ar[d]^-{\pi} \\
\lvert \mathcal{K}(X) \rvert \ar[r]_-{S} & X\ ,
}
\end{displaymath}
together with the fact that
$\widetilde{S}$ commutes with the actions of $\Gamma$ on $\left|\widetilde{\mathcal{K}(X)}\right|$ and $\widetilde{X}$, respectively.
  \end{proof}

Our next goal is to prove that the map $\widetilde{S}$ induces an isomorphism on singular homology groups with integral coefficients. 
To this aim we first introduce the $\Delta$-complex associated to a singular chain, referring the reader
to Section~\ref{comparison:subsec} for the definition of $\Delta$-complex.

Let $X$ be a topological space and let $c=\sum_{j=1}^k a_j\sigma_j\in C_n(X;\mathbb{Z})$ be an integral  singular $n$-dimensional chain with values in $X$. 
We construct a $\Delta$-complex $Z=\bigcup_{i\in\mathbb{N}} Z_i$ as follows. We set $Z_i=\emptyset$ for $i>n$. 
For $i=n$ we set $Z_n=\{\sigma_j\, |\, j=1,\ldots,k\}$. 
For $i\leq n$, we define $Z_i$ to be the set of all the singular $i$-faces of elements in $Z_n$, where the singular $i$-faces of a singular $n$-simplex $\sigma$ are the compositions
$\sigma\circ h$, where $h\colon \Delta^i\to \Delta^n$ is any order-preserving affine inclusion of the standard simplex $\Delta^i$ onto a face of $\Delta^n$ (of course,
vertices of standard simplices are canonically ordered). The boundary map $\partial^i_k\colon Z_i\to Z_{i-1}$ just sends every singular simplex in $Z_i$ to its $k$-th face. Finally, we denote
by $|Z|$ the geometric realization of $Z$. For every $j=1,\ldots,k$, there is a canonical characteristic map $\eta_j\colon \Delta^n \to |Z|$ 
from the standard simplex to the geometric simplex
of $|Z|$ associated to $\sigma_j$ (such a map may not be injective, e.g.~when $\sigma_j$ is not injective on vertices). By construction, 
the $\sigma_j$ glue up into a continuous map from $|Z|$ to $X$, i.e.~there exists
a continuous map $f\colon |Z|\to X$ such that $\sigma_j=f\circ \eta_j$ for every $j=1,\ldots,k$. If we set $c_Z=\sum_{j=1}^k a_j\eta_j \in C_n(|Z|;\mathbb{Z})$, then
$C_n(f)(c_Z)=c$. Moreover, if $c$ is a cycle, then also $c_Z$ is.

\begin{Proposizione}\label{prop-iso-hom}
The map
$\widetilde{S}\colon \left|\widetilde{\mathcal{K}(X)}\right|\to \widetilde{X}$
induces isomorphisms on homology groups with integral coefficients in every degree.
\end{Proposizione}
\begin{proof}
Let us denote by $H_n(\widetilde{S})\colon H_n\left( \left| \widetilde{\mathcal{K}(X)}\right|;\mathbb{Z}\right)\to H_n(\widetilde{X};\mathbb{Z})$ 
the map induced by $\widetilde{S}$ on singular homology in degree $n$. Proposition~\ref{fund-group:prop} implies that 
$H_n(\widetilde{S})$ is an isomorphism for $n=0,1$, so we may suppose $n\geq 2$.
Recall that $\pi\colon \widetilde{X}\to X$ is the covering projection, and denote by $H_n(\phi_n)\colon H_n\left(\widetilde{\mathcal{K}({X})};\mathbb{Z}\right)\to H_n\left(\left|\widetilde{\mathcal{K}({X})}\right|;\matZ\right)$ the isomorphism between simplicial and singular
homology described in Theorem~\ref{simpl-hom-eq-sing-one}.

We first show that
$H_n(\widetilde{S})$ is surjective. 
Of course we can suppose that $X$ contains at least two points.
Let us take a singular cycle $c = \sum_{j=1}^k a_{j} \sigma_{j} \in \, C_{n}(\widetilde{X}; \mathbb{Z})$.
If the vertices of each $\sigma_j$ project to pairwise distinct points of $X$ (i.e.~if the class
$\overline{\sigma}_j$ of each $\sigma_j$ defines a simplex in $\widetilde{\calK(X)}$), then there is nothing to prove: indeed, to every $\sigma_j$ we can associate the
element 
$$
\widehat{\sigma}_j=(\overline{\sigma}_j,(v_0,\ldots,v_n))\in C_n(\widetilde{\calK(X)};\matZ)\ ,
$$
where $v_i$ is the $i$-th vertex of $\sigma_j$;
if $\widehat{c}=\sum_{j=1}^k a_j\widehat{\sigma}_j \in C_n(\widetilde{\calK(X)};\matZ)$, then 
it readily follows from the definitions that $\widehat{c}$ is a cycle, and
$$
c=C_n(\widetilde{S})(\phi_n(\widehat{c}))\ .
$$
In particular, the class of $c$ lies in the image of $H_n(\widetilde{S})$. 

Therefore, in order to prove that $H_n(\widetilde{S})$ is surjective we are left to show that every class $\alpha\in H_n(\wdtX;\matZ)$ is represented by a cycle in which the vertices of every singular simplex 
project to distinct points of $X$. Let  $c = \sum_{j=1}^k a_{j} \sigma_{j} \in \, C_n(\widetilde{X}; \mathbb{Z})$ be a representative of $\alpha$, and let
$Z,c_Z$ and $f\colon |Z|\to \widetilde{X}$ be the $\Delta$-complex, 
the singular cycle and the map defined above, so that $C_n(f)(c_Z)=c$. Let $Z'$ be the second barycentric subdivision of $Z$, which is now a simplicial complex (see \cite[Exercise~23, page~133]{hatcher}), 
and denote
by $c'$ (resp.~$c'_{Z}$) the second barycentric subdivision of $c$ (resp.~of $c_Z$). Let also $V\subseteq |Z|$ be the set of vertices of $|Z'|$ (where we understand the canonical identification
between $|Z|$ and $|Z'|$). Using that $X$ is good, it is not difficult to show that we can homotope the map $f\colon |Z|\to \widetilde{X}$ to a map $f'$ such that $\pi\circ f'$ is injective on $V$: indeed, since
we are assuming that $X$ contains at least two points and is path connected, it contains  infinitely many points; we can then homotope $(\pi\circ f)|_V$ to an injective map, extend such homotopy
using the homotopy extension property of the CW pair $(|Z|,V)$, and then homotope $f$ to the required map $f'$ thanks to the homotopy lifting property for coverings. 
By construction, each simplex of $c'_Z$ has vertices on distinct points  of $V$, so the vertices of each simplex of the cycle $c'=C_n(f')(c'_Z)$ project to distinct points of $X$. Since $\alpha=[c]=[c']$, this concludes the
proof of the surjectivity of $H_n(\widetilde{S})$.

Let us now prove the injectivity of $H_n(\widetilde{S})$. Let $\alpha\in \ker H_n(\widetilde{S})\subseteq H_n\left(\left|\widetilde{\calK(X)}\right|;\matZ\right)$. By Theorem~\ref{simpl-hom-eq-sing-one} 
we may assume that $\alpha$ is represented by the cycle $c=\phi_n(c_s)$, where 
$$
c_s=\sum_{j=1}^k a_j (\overline{\sigma}_j,(v_0^j,\ldots,v_n^j))\ \in\ C_n\left(\widetilde{\calK(X)};\matZ\right)
$$
is a {simplicial} cycle.
By definition, we have
\begin{equation}\label{bordo:eq}
C_n(\widetilde{S})(c_s)=\sum_{j=1}^k a_j\sigma_j\ ,
\end{equation}
where
$\sigma_j\colon \Delta^n\to\wdtX$ is the unique singular simplex belonging to the class $\overline{\sigma}_j$ and having $v_i^j$ as $i$-th vertex.
Since $\alpha\in \ker H_n(\widetilde{S})$, we then have
\begin{equation}\label{n2:eq}
\sum_{j=1}^k a_{j} \sigma_{j} = \partial\left(\sum_{h=1}^l b_{h} \psi_{h}\right)\ ,
\end{equation}
where $\psi_{h}$ is an $(n+1)$-dimensional simplex in $\wdtX$ for every $h=1,\ldots,l$. 
If $\pi\circ \psi_h$ has distinct vertices in $X$ for every $h$, then just as in the argument for surjectivity we can define a simplicial chain
$\sum_{h=1}^l b_h \widehat{\psi}_h\in C_{n+1}\left(\widetilde{\calK(X)};\matZ\right)$ such that 
$$
c_s=\sum_{j=1}^k a_{j} (\overline{\sigma}_j,(v_0^j,\ldots,v_n^j)) = \partial\left(\sum_{h=1}^l b_{h} \widehat{\psi}_{h}\right)\ ,
$$
and this implies that $c_s$ is null-homologous, hence $\alpha=0$. 
Therefore, we are left to show that in equality~\eqref{n2:eq} we can assume that $\pi\circ\psi_h$ has distinct vertices for every $h=1,\ldots,l$.

Let $Z,c_Z$ and $f\colon |Z|\to \widetilde{X}$ be the $\Delta$-complex, the singular chain and the map associated to $\sum_{h=1}^l b_h\psi_h$. 
We will subdivide $Z$ and $c_Z$ 
so to obtain a simplicial chain whose image via (a map homotopic to) $f$ provides the desired singular $(n+1)$-chain.
Let us say that an edge of $Z$ is \emph{special} if its endpoints are sent by $\pi\circ f$ to the same point of $X$,
and that a $k$-dimensional simplex of $Z$, $0\leq k\leq n+1$, is special if it contains a special edge (in particular,
$0$-simplices are never special). 

We first observe that no special $n$-simplex can appear in $\partial c_Z\in C_n(Z;\matZ)$. 
In fact, if $\sigma_Z$ is a special simplex of $Z$, then 
the vertices of $C_n(f)(\sigma_Z)$ do not project to pairwise distinct points of $X$, so $C_n(f)(\sigma_Z)$ cannot appear
among the $\sigma_j$. But
$$
C_n(f)(\partial c_Z)=\partial (C_n(f)(c_Z))=\partial \left(\sum_{h=1}^l b_h\psi_h\right)=\sum_{j=1}^k a_j\sigma_j\ ,
$$
hence $\sigma_Z$ cannot appear in $\partial c_Z$.

We now subdivide $Z$ by adding the barycenters of all special
simplices, and by subdividing only special simplices as follows: 
starting from dimension 1 and ending in dimension $(n+1)$, we inductively replace each special simplex with the cone on its barycenter
of its (already subdivided) boundary. 
We call $Z'$ the resulting $\Delta$-complex, and $c'_Z$ the simplicial chain obtained by subdividing $c_Z$ accordingly. 
By construction, $\partial c'_Z$ is a subdivision of $\partial c_Z$. But in $\partial c_Z$ no special simplex appears,
so 
$$
\partial c'_Z=\partial c_Z\ .
$$
Let $|L|\subseteq |Z|$ be the subcomplex given by the union of the geometric realizations of \emph{non}-special simplices (so that
$|L|$ contains all the vertices of $|Z|$), and let $V_0\subseteq X$ be the image of the set of vertices of $|Z|$ via $\pi\circ f$. 
Let also $W\subseteq |Z|$ be the set of vertices of $|Z'|$ that are not vertices of $|Z|$ (i.e.~the set of added barycenters).
Using that $X$
is good, we may homotope $f|_{|L|\cup W}$ to a map $f'\colon |L|\cup W\to \widetilde{X}$ with the following properties:
$f|_{|L|}=f'|_{L}$, $(\pi\circ f)|_W$ is injective, and $(\pi\circ f)(W)\cap V_0=\emptyset$. Using the homotopy extension property for CW pairs,
we can then extend this homotopy to a homotopy between $f$ and a map $f''\colon |Z|\to \widetilde{X}$ such that 
$f''|_{|L|\cup W}=f'$. With this choice the chain $C_{n+1}(f'')(c'_Z)\in C_{n+1}(\wdtX;\matZ)$ is a linear combination of singular simplices
each of which projects to a singular simplex with distinct vertices in $X$. Moreover,  $f''$ and $f$ coincide on the geometric realization of
non-special simplices and  no special simplex appears in $\partial c_Z$, hence
\begin{align*}
\partial( C_{n+1}(f'') (c'_Z))& =C_{n+1}(f'')(\partial c'_Z)=C_{n+1}(f'')(\partial c_Z)=C_{n+1}(f)(\partial c_Z)\\ & =\partial\left(\sum_{h=1}^l b_h\psi_h\right)=\sum_{j=1}^k a_j\sigma_j\ ,
\end{align*}
where the last equality is due to the definition of $c_Z$.
This concludes the proof.

\end{proof}

\begin{rem}
 Our Proposition~\ref{prop-iso-hom} is very similar in spirit to~\cite[Remark 1]{Kuessner}, where it is claimed that the inclusion into $C_*(X;R)$ 
 of the chain complex generated by singular simplices  in $X$ with distinct vertices  induces
 an isometric isomorphism in homology. When working with real coefficients, this fact may be deduced (at least for good spaces)
 via a duality argument from Theorem~\ref{isometria-good-K(X)}.
 Moreover, by putting together Proposition~\ref{prop-iso-hom}
 and Theorem~\ref{simpl-hom-eq-sing-one} one easily checks that, for good spaces, this inclusion induces an isomorphism also when working with integral coefficients. However,
 the $\ell^1$-seminorm induced in homology by the chain complex generated over $\mathbb{Z}$ by singular simplices with distinct vertices can be strictly bigger
 than the usual $\ell^1$-seminorm of $H_*(X;\mathbb{Z})$. For example, the  fundamental class $[S^1]\in H_1(S^1;\mathbb{Z})$ of the circle satisfies $\|[S^1]\|_1=1$,
 but the smallest representative of $[S^1]$ given by a linear combination of simplices with distict vertices has $\ell^1$-norm equal to 2.
\end{rem}

We are now ready to conclude the proof of Therem~\ref{homotopy-weak-intro}. We have already observed that the map $S\colon |\calK(X)|\to X$ induces isomorphisms on $i$-th homotopy groups for $i=0,1$. 
For $n\geq 2$, coverings induce isomorphisms on $n$-th homotopy groups, so thanks to the commutative diagram
\begin{displaymath}
\xymatrix{
\left|\widetilde{\mathcal{K}(X)}\right| \ar[r]^-{\widetilde{S}} \ar[d]_{\tilde{\pi}} & \widetilde{X} \ar[d]^{\pi} \\
\lvert \mathcal{K}(X) \rvert \ar[r]_-{S} & X
}
\end{displaymath}
we only need to prove that $\widetilde{S}$ is a weak homotopy equivalence. Now the conclusion follows from Hurewicz Theorem \cite[Corollary 4.33]{hatcher}, since
we know that $\widetilde{\calK(X)}$ and $\wdtX$ are simply connected (see Proposition~\ref{fund-group:prop}), and that $\widetilde{S}$ induces an isomorphism on singular homology with integral
coeffients in every degree (see Proposition~\ref{prop-iso-hom}).

\begin{quest}\label{question:whe}
Is it true that the map $S\colon |\calK(X)|\to X$ is a weak homotopy equivalence for every topological space (i.e.~even without the assumption that $X$ is good)? 
The singular set $\mathcal{S}(X)$ is always weakly homotopy equivalent to $X$, so there is no apparent obstruction to an affirmative answer to our question. In any case, 
in order to apply the homotopy theory of multicomplexes to bounded cohomology we will need to provide an explicit combinatorial description of the homotopy groups of $\calK(X)$ that does not
seem possible for pathological spaces. We refer the reader to Remarks~\ref{nospecial:rem} and~\ref{controesempioa:rem} for more speculations on this issue.
\end{quest}

As mentioned above, we can now prove that the singular multicomplex $\calK(X)$ is homotopy equivalent to the singular simplicial set $\calS(X)$ (at least for good spaces).
Just as for the singular multicomplex, the singular set $\calS(X)$ comes equipped with a continuous map $j \colon |\calS(X)|\to X$, where $|\calS(X)|$ denotes the geometric
realization of $\calS(X)$ (see e.g.~\cite{milnor-geom}). The multicomplex $|\calK(X)|$ may be realized as a subcomplex of $|\calS(X)|$ as follows. Let us fix an arbitrary ordering on $X$, now considered as a set.
Let $\overline{\tau}$ be an $n$-simplex of $\calK(X)$, so that $\overline{\tau}$ is the equivalence class of a singular simplex with values in $X$. In such equivalence class there exists a unique
singular simplex $\tau\colon \Delta^n\to X$ which is increasing on the set of vertices of $\Delta^n$. We then set $i(\overline{\tau})=\tau$, and we observe that this formula defines 
a cellular map (still denoted by $i$) from $|\calK(X)|$ to $|\calS(X)|$. In fact, if we endow $\calK(X)$ with the structure of simplicial set generated by its structure as a multicomplex,
then $i\colon \calK(X)\to\calS(X)$ is a simplicial embedding.

\begin{cor}\label{homKS}
 Let $X$ be a good topological space. Then the simplicial embedding
 $$
 i\colon |\calK(X)|\to |\calS(X)|
 $$
 is a homotopy equivalence.
\end{cor}
\begin{proof}
 Let us consider the commutative diagram
 $$
 \xymatrix{
 |\calK(X)| \ar[rr]^{i} \ar[dr]_{S}& & |\calS(X)|\ar[ld]^{j}\\
 &\ X \ . &
 }
 $$
 We have proved in Therem~\ref{homotopy-weak-intro} that the map $S$ is a weak homotopy equivalence, while it is well known that $j$ is a weak homotopy equivalence (see e.g.~\cite[Theorem 4]{milnor-geom}).
 As a consequence, the map $i$ is also a weak homotopy equivalence. Since both $|\calK(X)|$ and $|\calS(X)|$ are CW complexes, the conclusion follows from Whitehead Theorem. 
\end{proof}

\section{(Lack of) functoriality}\label{Sec:no:functorial}
Unfortunately, the singular multicomplex $\calK(X)$ lacks the functoriality properties enjoyed by the singular complex $\calS(X)$. In fact,
if $f\colon X\to Y$ is a continuous map and $\sigma\colon \Delta^n\to X$ is a singular simplex, then of course $f\circ\sigma$ is a singular simplex
with values in $Y$, so that $f$ induces a simplicial map $\calS(f)\colon \calS(X)\to \calS(Y)$. However, even if $\sigma$ is injective on the vertices of $\Delta^n$,
it may well be that $f\circ \sigma$ is not so. Therefore, $f$ does not induce any meaningful map from $\calK(X)$ to $\calK(Y)$, unless it is itself injective on each path connected
component of $X$.  

Nevertheless, we are able to prove that, at least for good spaces, any map $X\to Y$ induces a well-defined homotopy class of maps $|\calK(X)|\to |\calK(Y)|$:

\begin{prop}\label{commutative:homotopy}
 Let $X,Y$ be good topological spaces and let $f\colon X\to Y$ be a continuous map. Then there exists a continuous map
 $K(f)\colon |\calK(X)|\to |\calK(Y)|$ such that the following diagram commutes up to homotopy:
 $$
 \xymatrix{
 |\calK(X)|\ar[r]^{K(f)}\ar[d]_{S_X} & |\calK(Y)|\ar[d]^{S_Y}\\
 X \ar[r]^f & Y\ .
 }
 $$
 \end{prop}
\begin{proof}
 Let $i_X\colon |\calK(X)|\to |\calS(X)|$ (resp~$i_Y\colon |\calK(Y)|\to |\calS(Y)|$) be the embedding introduced at the end of the previous section, and recall from Corollary~\ref{homKS}
 that $i_X$ and $i_Y$ are homotopy equivalences.
 As mentioned above, the map
 $f\colon X\to Y$ induces a simplicial map $S(f)\colon |\calS(X)|\to |\calS(Y)|$ which makes the following diagram commute:
 $$
 \xymatrix{
 |\calS(X)|\ar[r]^{S(f)}\ar[d]_{j_X} & |\calS(Y)|\ar[d]^{j_Y}\\
 X \ar[r]^f & Y\ .
 }
 $$
 Now the conclusion follows by setting  $$K(f)\colon |\calK(X)|\to |\calK(Y)|\, ,\qquad  K(f)=\rho_Y\circ S(f)\circ i_X\ ,$$ where $\rho_Y\colon |\calS(Y)|\to |\calK(Y)|$ is a homotopy inverse
 of $i_Y$. In fact, we have
\begin{align*} 
f \circ S_X &= f \circ j_X \circ i_X \\
&= j_Y \circ S(f) \circ i_X \\
&\simeq j_Y \circ i_Y \circ \rho_Y \circ S(f) \circ i_X \\
&= S_Y \circ \rho_Y \circ S(f) \circ i_X \\
&= S_Y \circ K(f).
\end{align*}
\end{proof}

\chapter{The homotopy theory of complete multicomplexes}\label{chap2:hom}
This chapter is devoted to the study of the homotopy theory of multicomplexes.  
We have already seen in Section~\ref{Sec:weak:hom:type:K(X)} that, at least when $X$ is a good topological space, the singular multicomplex of $X$ has the same weak homotopy type of $X$, 
thus providing a simplicial model for $X$. We will see that the singular multicomplex is a \emph{complete} multicomplex. Completeness for multicomplexes is an analogue of the Kan property for simplicial
sets, and it is particularly useful in the study of homotopy groups of geometric realizations. Indeed, the combinatorial description of homotopy groups 
of Kan simplicial sets is a classical topic in the theory of simplicial sets (see e.g.~\cite{Piccinini, may, GJ}),
and in this chapter we develop a similar theory for multicomplexes. 


Sometimes a complete multicomplex $K$ can be unnecessarily big (this is the case, for example, of the singular multicomplex of a topological space). Thus it is often
useful to look for a minimal submulticomplex of $K$ which carries all the homotopy of the original multicomplex. Such a multicomplex is called \emph{minimal}. The formal definition and the geometric properties
of minimal multicomplexes are very close to the definition and the properties of minimal simplicial sets
(see again~\cite{Piccinini, may, GJ}).



To any complete and minimal multicomplex $K$ there is associated an aspherical multicomplex $A$
having the same fundamental group as $K$. The construction of $A$ is explicit, and exhibits $A$ as a quotient of $K$ by a group of simplicial automorphisms. 
When applied to the minimal multicomplex associated to the singular multicomplex $\calK(X)$ of a topological space $X$, this construction becomes particularly interesting,
both because it provides an explicit description of an Eilenberg-MacLane space associated to a given topological space (see the end of this chapter), and because it 
plays a fundamental role in our study of the bounded cohomology of topological spaces (see Chapter~\ref{chap2:chap}).

\section{Complete multicomplexes}
In this section we  introduce the notion of \textit{complete} multicomplex. 
As already mentioned, completeness may be interpreted as 
the counterpart, in the theory of multicomplexes, of the Kan property for simplicial sets.

\begin{Definizione}\label{completeness:defn}
A multicomplex $K$ is  \textit{complete} if the following condition holds. Let $f \colon \lvert \Delta^{n} \rvert \rightarrow \lvert K \rvert$ be a continuous map  whose restriction to the boundary 
$f \vert_{\partial |\Delta^{n}|} \colon \partial |\Delta^{n}| \rightarrow |K|$ is a simplicial embedding. Then $f$ is homotopic relative to $\partial |\Delta^{n}|$ to a simplicial embedding $f' \colon |\Delta^{n}| \rightarrow |K|$.
\end{Definizione}

\begin{rem}\label{kan:rem}
Let us denote by $\Lambda^n$ the $n$-dimensional \emph{simplicial horn}, i.e.~the simplicial complex obtained by removing one $n$-dimensional face from $\partial \Delta^{n+1}$. 
If  the removed facet is opposite to the $i$-th vertex of $\Delta^{n+1}$, then we  denote the resulting simplicial horn by the symbol $\Lambda^{n}_{i}$.
We  endow $\Lambda^n$ with the induced structure of multicomplex (which coincides with its structure as a simplicial complex) and of (ordered) simplicial set (which is obtained
by adding all the degenerate simplices generated by the geometric faces of $\Lambda^n$). Recall that a simplicial set $Z$ satisfies the Kan condition (or, for brevity, \emph{is Kan})
if, for every $n\in\matN$ and $k \in \matN$, every simplicial map $f\colon \Lambda^{n}_{k}\to Z$ may be extended to a simplicial map $F\colon \Delta^{n+1}\to Z$.  

The Kan condition seems to be strictly related to an analogue of completeness for simplicial sets. Indeed, it turns out that a simplicial set $Z$ is Kan if and only if the following 
strong completeness condition holds:
for every $n\in\matN$ and every continuous map $f\colon |\Delta^n|\to |Z|$ which is simplicial on $\partial |\Delta^n|$, there exists a simplicial map
$f'\colon |\Delta^n|\to |Z|$ which agrees with $f$ on $\partial |\Delta^n|$ and is homotopic to $f$ relative to $\partial |\Delta^n|$. The ``if'' part of this statement is an easy exercise
(strong completeness allows to extend a simplicial map defined on $\Lambda^n$ to a null-homotopic simplicial map defined on the whole of $\partial \Delta^{n+1}$;
a further application of strong completeness then provides the needed extension of this map to a simplicial map defined on $\Delta^{n+1}$); the ``only if'' part is highly non-trivial,
and turns out to be very useful when developing the homotopy theory of Kan sets (see e.g.~\cite[Theorem 4.5.27]{Piccinini}). 

While (strong) completeness involves continuous maps on geometric realizations in its definition, the Kan condition is purely combinatorial, and usually much easier to check.
Therefore, it may be interesting to find, in the context of multicomplexes, an analogue of the Kan condition which is equivalent to completeness. 

Let $K$ be a multicomplex with associated (ordered) simplicial set $Z$. Unfortunately, completeness for $K$ does \emph{not} imply strong completeness for $Z$: if two edges of $K$ join the distinct vertices $v,w$
(and this phenomenon may well occur in complete multicomplexes),
then a simplicial map $\Lambda^1\to K$ having these edges as image does not extend to a simplicial map on $\Delta^2$, since non-degenerate simplices of $Z$ must have distinct vertices. Therefore, completeness for $K$ is \emph{not} equivalent to the Kan condition for $Z$ (and the Kan condition for $Z$ seems to be too restrictive to our purposes: if 
$K=\calK(X)$, then the associated simplicial set $Z$ is not Kan, while it is complete as we will prove in Theorem \ref{K(X)-compl}). We could therefore define the Kan condition for multicomplexes as follows: a multicomplex $K$ is e-Kan
(where ``e'' stands for ``embedding'')
if, for every $n\in\matN$, every simplicial \emph{embedding} $f\colon \Lambda^n\to K$ may be extended to a simplicial map (which is automatically an embedding)
$F\colon \Delta^{n+1}\to K$.  The singular multicomplex $\calK(X)$ satisfies the e-Kan condition for every topological space $X$. Moreover, the same argument showing that strong completeness implies the Kan condition for simplicial sets shows that completeness implies the e-Kan condition for multicomplexes (see also the proof of Lemma~\ref{quoz-compl-lemma1} below). Unfortunately, at the moment we are not able to show that the e-Kan condition implies (and therefore is equivalent to) completeness for multicomplexes.
\end{rem}

Before investigating more closely the features of completeness, let us 
discuss some  examples of complete multicomplexes.

\begin{Esempi}\label{ex-compl-multi}
\hspace{2em}
\begin{enumerate}
\item
Let $K$ be a complete connected one-dimensional multicomplex. Then $K$ is a segment.
Indeed, $K$ must contain at least two vertices (because otherwise it would be either empty or a point). 
Let us suppose by contradiction that $K$ contains at least three vertices $v_{0}, v_{1} $and $v_{2}$. Since $K$ is connected, up to reordering these vertices we can suppose that
$K$ contains a $1$-simplex $e_1$ joining $v_1$ to $v_2$ and a $1$-simplex $e_2$ joining $v_2$ to $v_3$. The concatenation of $e_1$ and $e_2$ provides a continuous map $|\Delta^1|\to |K|$
which is a simplicial embedding on $\partial |\Delta^1|$, so completeness ensures the existence of a $1$-simplex $e_3$ joining $v_1$ to $v_3$ and homotopic
to $e_1*e_2$ relative to the endpoints. The loop given by the concatenation of $e_1,e_2$ and (the inverse of) $e_3$ is null-homotopic, and can be realized as a simplicial embedding
of $\partial \Delta^2$ in $K$. By completeness, this implies that $K$ contains at least one $2$-dimensional simplex. 
This contradicts the fact that $K$ is $1$-dimensional, thus showing that $K$ contains exactly 2 vertices. 

Let us finally prove that the vertices of $K$ are joined by a unique edge. In fact,
if $K$ contained at least two edges, then we could take a reduced (i.e.~non-backtracking) simplicial path of length 3 in $|K|$. By completeness, this path should be homotopic relative to its endpoints
to a single  edge of $K$. But reduced paths in a graph are never homotopic relative to their endpoints, unless they coincide. This shows that $K$ must be a segment.
\item More in general, it can be easily proved by induction that, if $K$ is a complete and connected multicomplex with vertex set $V$, then for any $A\in \mathcal{P}_f(V)$ there is a simplex in $K$
with vertex set $A$. In particular, if $V$ is infinite then $K$ is infinite-dimensional. 
\item As a consequence of the previous item, 
if a connected complete multicomplex $K$ is a simplicial complex, then it equals the full simplicial complex on the vertex set of $K$. 
\end{enumerate}
\end{Esempi}

The following result will prove useful later:

\begin{Lemma}\label{quoz-compl-lemma1}
Let $K$ be a complete multicomplex. Let $f \colon (\Delta^{n})^{1} \rightarrow K$ be a simplicial embedding such that the restriction of $f$ to each triangular loop is null-homotopic. 
Then, $f$ extends to a simplicial embedding of the whole $\Delta^{n}$.
\end{Lemma}
\begin{proof}
By definition of completeness, the statement holds if $n=2$. 

Let $f$ be a map as in the statement, and 
 assume by induction that the thesis is true for $(n-1)$-simplices. 
 Then we can extend $f$ to a simplicial embedding $\phi\colon \Lambda^{n-1}\to K$ of the $(n-1)$-dimensional simplicial horn $\Lambda^{n-1}\subseteq \partial \Delta^n$ (see Remark~\ref{kan:rem} 
 for the definition of simplicial horn). Let $F\subseteq \partial \Delta^n$ be the facet of $\Delta^n$ not contained in $\Lambda^{n-1}$, and let $r\colon |F|\to |\Lambda^{n-1}|$ 
 be any homeomorphism which is the identity on $\partial |F|=|F|\cap|\Lambda^{n-1}|$ (recall that a facet of a simplex is a codimension one face of the simplex).
The map $\phi\circ r\colon |F|\to |K|$ restricts to a simplicial embedding of $\partial F$, so by completeness
there exists a simplicial embedding $\iota\colon F\to K$ which is homotopic to $\phi\circ r$ relative to the boundary of $F$. We can now glue $\phi$ and $\iota$ along
$\partial F$ to obtain a simplicial embedding $\psi\colon\partial \Delta^n\to K$. It is not difficult to check that this embedding is null-homotopic, so again by completeness
we can extend $\psi$ to the desired embedding of $\Delta^n$ into $K$.
\end{proof}

\section{Special spheres}
Our next goal is the study of homotopy groups of (complete) multicomplexes. To this aim we will be interested in representing elements of the $n$-th homotopy group of (the geometric realization)
of a multicomplex as simplicial maps defined on a specific realization of the $n$-sphere as a multicomplex. 

\begin{defn}\label{Def-special-Sphere-cap-1}
 The $n$-dimensional \emph{special sphere} is an $n$-dimensional multicomplex $K=(V,I,\Omega)$ which is defined as follows: $V$ is a set of cardinality $(n+1)$; $I^n=I_V=\{\Delta^n_n,\Delta^n_s\}$, and
 $I_A$ consists of a single simplex $\Delta_A$ for every $A\subseteq V$, $A\neq V$. The boundary maps in $\Omega$ are the only possible ones: for every $B\subseteq A\subseteq V$,
 we have $\partial_{A,B}=\id$ if $A=B=V$, $\partial_{V,B}(\Delta^n_{n})=\partial_{V,B}(\Delta^n_s)=\Delta_B$ if $B\neq V$, and $\partial_{A,B}(\Delta_A)=\Delta_B$ in the other cases. 
 In other words, $K$ consists of a \emph{northern} $n$-simplex $\Delta^n_n$ and
 a \emph{southern} $n$-simplex $\Delta^n_s$ which are glued along their common boundary. It follows that $|K|$ is homeomorphic to $S^n$. The subspace
 $S^{n-1} \cong |K^{n-1}|=\partial |\Delta^n_n|=\partial |\Delta^n_s|\subseteq |K|\cong S^n$ is the \emph{equator} of $K$. We will usually denote (any multicomplex isomorphic to) a special $n$-sphere by the symbol $\dot{S}^n$.
 We will also denote by $j_{n,s}\colon |\Delta^n_n|\to |\Delta^n_s|$, $j_{s,n}\colon |\Delta^n_s|\to |\Delta^n_n|$ the unique simplicial maps that are the identity on
 the equator of $\dot{S}^n$.
 
 When dealing with homotopy groups, we  need to take care of orientations: a map from a space $S$ homeomorphic to the $n$-sphere defines an element of the $n$-th homotopy group
 only when a pointed identification of $S$ with $S^n$ is fixed. Therefore, henceforth we fix an ordering $s_0<\ldots<s_n$ 
 on the set of vertices $V(\dot{S}^n)$ of $\dot{S}^n$, 
 and an identification $(S^n,*)\cong (|\dot{S}^n|,s_0)$  which takes the basepoint $*$ of $S^n$ into
 $s_0$. 
 \end{defn}

The easy proof of the following lemma is left to the reader. 

\begin{Lemma}\label{lemma-the-following-3-cond-equivalent}
Let $(X,x_0)$ be a pointed topological space,
let $(\dot{S}^n,s_0)$ be a pointed special sphere, and let 
$$
f\colon (|\dot{S}^n|,s_0)\to (X,x_0)
$$
be a continuous map. Let also $f_n$ (resp.~$f_s$) be the restriction of $f$ to $|\Delta^n_n|$ (resp.~$|\Delta^n_s|$).
Then, the following facts are equivalent:
\begin{itemize}
\item[(i)] The class of $f$ in $\pi_{n}(X, x_{0})$ is null.
\item[(ii)] The maps $f_n\colon |\Delta^n_n|\to X$ and $f_s\circ i_{n,s}\colon |\Delta^n_n|\to X$ are homotopic relative to
$\partial |\Delta^n_n|$.
\item[(iii)] The map $f$ extends to an $(n+1)$-topological disc bounded by $|\dot{S}^n|$. 
\end{itemize}
\end{Lemma}

We now prove that singular multicomplexes are complete. Hence, by Theorem~\ref{homotopy-weak-intro}, for
any good space $X$ there always exists a complete multicomplex $\mathcal{K}(X)$ that is weakly homotopic to $X$.

\begin{thm}\label{K(X)-compl}
Let $X$ be a good topological space. Then
the singular multicomplex $\mathcal{K}(X)$ is complete.
\end{thm}
\begin{proof}
Let $F \colon \lvert \Delta^{n} \rvert \rightarrow \lvert \mathcal{K}(X) \rvert$ be a continuous map whose restriction 
$f=F|_{\lvert \partial \Delta^{n} \rvert}$ is a simplicial embedding. Let also $S \colon \lvert \mathcal{K}(X) \rvert \rightarrow X$ be the natural projection. 
The composition $\sigma=S\circ F$ is a singular simplex in $X$ with distinct vertices. Its class (with respect to the action of $\mathfrak{S}_{n+1}$ on singular $n$-simplices with distinct vertices)
is therefore a simplex $\overline{\sigma}$ of $\calK(X)$, whose boundary coincides with the image of $f$. 
Therefore, if $\sigma'\colon |\Delta^n|\to |\calK(X)|$ is the unique characteristic map of $\overline{\sigma}$ which coincides with $f$ on  $\partial |\Delta^n|$,
by construction $\sigma'$ is a simplicial embedding such that
$S\circ \sigma'=S\circ F$. In order to conclude we are left to show that $\sigma'$ is homotopic to $F$ relative to $\partial |\Delta^n|$.
To this aim we can exploit Lemma~\ref{lemma-the-following-3-cond-equivalent} as follows.
Let $g\colon |\dot{S}^n|\to |\calK(X)|$ be the map obtained by setting $g=F$ on the northern emisphere of $|\dot{S}^n|$ and
$g=\sigma'$ on the southern emisphere of $|\dot{S}^n|$. Since $S\circ F=S\circ \sigma'$ we have that $S\circ g$ is null-homotopic in $X$. But Theorem~\ref{homotopy-weak-intro}
ensures that $S$ is a weak homotopy equivalence,
so $g$ is null-homotopic in $|\calK(X)|$, and this implies in turn that $F$ and $\sigma'$ are homotopic relative to the boundary of $|\Delta^n|$.
\end{proof}

We are now going to show that, for complete multicomplexes, the homotopy groups of the geometric realization are completely carried by simplicial special spheres. Just as for Kan simplicial sets (see for instance \cite{kan-comb, may, GJ})
this allows us to approach the study of the homotopy type of a complete multicomplex via combinatorial methods. We first need the following definitions, that will be used throughout the whole paper.

\begin{defn}\label{special:def}
Let $K$ be a multicomplex. Then, two $n$-dimensional simplices $\Delta_{1}$ and $\Delta_{2}$ of $K$ are said to be 
\textit{compatible} if the following conditions hold: the vertex sets of $\Delta_1$ and of $\Delta_2$ coincide; moreover, for every
proper subset $B$ of such vertex set, we have $\partial_B \Delta_1=\partial_B \Delta_2$. Equivalently,
$\Delta_1$ and $\Delta_2$ are compatible if 
$|\partial \Delta_{1}| = |\partial \Delta_{2}|$ in the geometric realization $|K|$. 
Being compatible is obviously an equivalence relation, and
following~\cite{Grom82}, for every simplex $\Delta_1$ of $K$ we  denote by $\pi(\Delta_1)$ the set of simplices of $K$ that are compatible with $\Delta_1$.

Two $n$-dimensional simplices $\Delta_{1}$ and $\Delta_{2}$ of $K$ are \textit{homotopic} if they are compatible and the
following condition holds: Let $i_j\colon |\Delta^n|\to |\Delta_j|$ be characteristic maps for $\Delta_1,\Delta_2$ such that
$i_1|_{\partial |\Delta^n|}=i_2|_{\partial |\Delta^n|}$. Then $i_1$ is homotopic to $i_2$ relative to ${\partial |\Delta^n|}$.
Being homotopic is also an equivalence relation among simplices of $K$.

Let $\Delta_{1}$ and $\Delta_{2}$ be two compatible $n$-simplices in $K$ and let 
$(x_0,\ldots,x_{n+1})$ be an ordering  of their vertices. We now define the (pointed) special sphere $\dot{S}^n(\Delta_1,\Delta_2)$ associated to $\Delta_1$ and $\Delta_2$ as follows: 
$$
\dot{S}^n(\Delta_1,\Delta_2)\colon (\dot{S}^n,s_0)\to (K,x_0)
$$
is the unique simplicial map such which is order-preserving on vertices and sends $\Delta^n_s$ to $\Delta^n_1$ and $\Delta^n_n$ to $\Delta^n_2$.
We will usually denote by $\dot{S^n}(\Delta_1,\Delta_2)$ also the induced map on geometric realizations, that may be interpreted as a map
$(S^n,*)\to (|K|,x_0)$, and thus defines an element of $\pi_n(|K|,x_0)$.

It readily follows from Lemma~\ref{lemma-the-following-3-cond-equivalent} that the compatible simplices $\Delta_1,\Delta_2$ are homotopic if and only if the associated special sphere
$\dot{S}^n(\Delta_1,\Delta_2)$ defines the trivial element of $\pi_n(|K|,x_0)$.
\end{defn}

The following result provides a nice combinatorial description of the homotopy groups of complete multicomplexes (compare with \cite[Section~I.11, page~60]{GJ} and \cite[Definition~3.6]{may}, where the case of simplicial sets is addressed).

\begin{thm}\label{complete:special}
Let $K$ be a complete multicomplex, and
 let $\Delta_0$ be an $n$-simplex of $K$. Also fix an ordering on the vertices of $\Delta_0$ and denote by $x_0$ the minimal vertex of $\Delta_0$.
 The map
 $$
 \Theta\colon \pi(\Delta_0)\to \pi_n(|K|,x_0)\, ,\qquad
 \Theta(\Delta)=\left[\dot{S}^n(\Delta_0,\Delta)\right]
 $$
 is surjective, and $\Theta(\Delta)=\Theta(\Delta')$ if and only if $\Delta$ is homotopic to $\Delta'$.
\end{thm}
\begin{proof}
 Let $g\colon (S^n,*)\to (|K|,x_0)$ be a continuous map. We are going to define a continuous map $g'\colon (|\dot{S}^n|,s_0)\to (|K|,x_0)$  representing the same element of $\pi_n(|K|,x_0)$ as $g$, and sending the southern
 hemisphere of $\dot{S}^n$ onto $\Delta_0$. On $|\Delta^n_s|$, 
 the map $g'$ is just the simplicial isomorphism  between $\Delta^n_s$ and $\Delta_0$ which preserves the ordering on vertices. On the northern hemisphere, $g'$ is defined as follows.
 We take an affine copy $|\overline{\Delta}|$ of the $n$-simplex inside $|\Delta^n_n|$, in such a way that one vertex of $|\overline{\Delta}|$ coincides with $x_0$, and all the other vertices
 of $|\overline{\Delta}|$ lie in the internal part of $|\Delta^n_n|$. We then define $X$ as the topological space obtained from $|\Delta^n_n|$ by collapsing $\partial |\overline{\Delta}|$. It is readily seen
 that $X$ is homeomorphic to the wedge of an $n$-simplex $\Delta'$ and an $n$-sphere $\overline{S}$ (see Figure~\ref{special:fig}). 
   \begin{figure}
 \begin{center}
 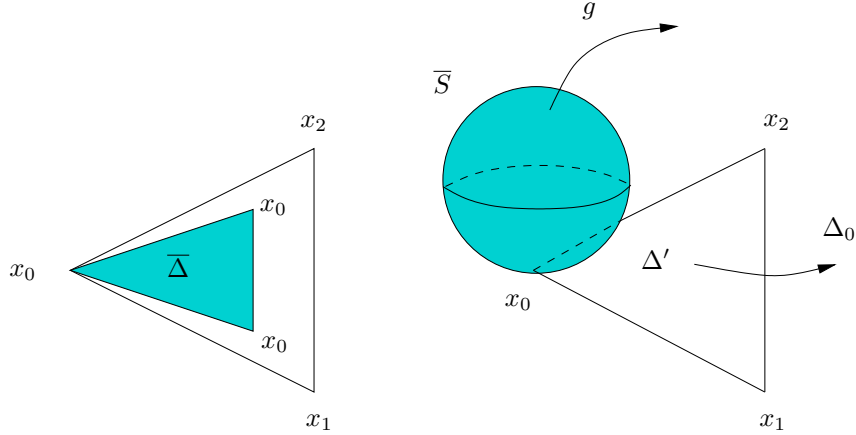
   \caption{The construction of the map $g'$. We denote vertices by the name of their images in $|K|$.}\label{special:fig}
\end{center}
   \end{figure}

 We then define a map  $g''\colon X\to |K|$ as the wedge of the order-preserving affine identification between $\Delta'$ and $\Delta_0$ (where we order the vertices of $\Delta'$ 
as the corresponding ones in $\Delta^n_n$) and the map $g$ on $\overline{S}$, where we are identifying $\overline{S}$ with $S^n$ via an orientation-preserving homeomorphism which takes $*$ to $s_0$ (observe that
$\overline{S}$ inherits an orientation from the orientation of $|\Delta^n_n|$, whence of $\overline{\Delta}$). Finally, $g'\colon |\Delta^n_n|\to |K|$ is the composition of $g''$ with the quotient map
$|\Delta^n_n|\to X$.

By completeness, there exists a simplicial embedding $\Delta^n_n\to \Delta_1$ onto an $n$-simplex of $K$ which preserves the ordering on vertices and is  homotopic to $g'$ relative to
$\partial |\Delta^n_n|$. It is now easy to check that the simplicial special sphere $\dot{S}^n(\Delta_0,\Delta_1)$ defines the same element as $g$ in $\pi_n(|K|,x_0)$. In fact,
by construction such simplicial sphere is homotopic relative to $s_0$ to the wedge of the map $g$ and a map which coincides with $\Delta_0$ both on the nothern and on the southern hemisphere. This latter map
is clearly null-homotopic, and this concludes the proof of the surjectivity of $\Theta$.

The fact that $\Theta(\Delta)=\Theta(\Delta')$ if and only if $\Delta$ is homotopic to $\Delta'$ is an easy consequence of Lemma~\ref{lemma-the-following-3-cond-equivalent}.
\end{proof}

\begin{rem}\label{nospecial:rem}
 As explained in Question~\ref{question:whe}, at the moment we are not able to prove (or disprove) that the singular multicomplex $\calK(X)$ is weakly homotopy equivalent
 to $X$ for \emph{every} topological space $X$. However, even if this fact were true, other difficulties could arise when dealing with exotic topological spaces. 
 If $X$ is a good space, by combining Theorem~\ref{homotopy-weak-intro} with Theorem~\ref{complete:special} we can conclude that every element
 of $\pi_n(X)$, $n\geq 1$, is represented by a special sphere in $\calK(X)$. This result will play a fundamental role in the sequel, and 
 definitely cannot be true if we remove the assumption that $X$ is good. Indeed, let $X$ be a finite topological space 
 having the same weak homotopy type of the $2$-dimensional sphere $S^2$ (the fact that such a topological space exists is well known, see e.g.~\cite{mccord}). Let $n$
 be the number of points of $X$ (it is known that the smallest number of points in a finite topological space weakly homotopy equivalent to $S^m$ is $2m+2$, hence
 we may assume $n=6$). Then the set of vertices of the singular multicomplex $\calK(X)$ consists of $n$ points. In particular, if $k \geq n$, then 
  $\calK(X)$ cannot contain any  $k$-dimensional special sphere. But $\pi_k(X)=\pi_k(S^2)\neq 0$ for every $k \geq 2$~\cite{homotopy}, hence there is no hope
  that special spheres could carry the homotopy groups  $\pi_k(X)$ for $k\geq n$. 
    \end{rem}

\section{Simplicial approximation of continuous maps}

It is well known that, in the context of simplicial complexes and simplicial sets, up to suitably subdividing the domain, continuous maps can be assumed to be simplicial up to homotopy (see for instance \cite[Theorem~2C.1]{hatcher} and \cite{zeeman} for simplicial complexes and \cite[Section~4]{jardine} for simplicial sets).

Analogous results can be recovered for multicomplexes, and  by exploiting completeness we can obtain even stronger results.
Indeed, Proposition~\ref{approx-simpl} provides a controlled simplicial approximation for maps between multicomplexes, provided that the target is complete. In order to approximate continuous maps 
with simplicial ones, one usually needs to subdivide each simplex in the domain a number of times that depends on the simplex (and on the  map). 
This implies in particular that working with simplicial (co)chains rather than with singular ones has usually an uncontrolled cost in terms of the $\ell^1$ (resp.~$\ell^\infty$) seminorms
on homology (resp.~cohomology). A key property of complete multicomplexes is that they allow us to shift from singular (co)chains to simplicial ones without increasing
these seminorms (see in particular Theorem~\ref{isometry-lemma-intro}).

\begin{prop}\label{approx-simpl}
Let $L,K$ be multicomplexes, suppose that $K$ is complete, and let $f \colon \lvert L \rvert \rightarrow \lvert K \rvert$ be a continuous map
that is simplicial on $L^{0}$, i.e.~that maps vertices of $L$ to vertices of $K$. Suppose that, for every simplex $\sigma$ of $L$,
the map $f$ is injective on the vertices of $\sigma$, and that 
 $f$ is simplicial on a  submulticomplex $L_1$ of $L$.
 Then, there exists a non-degenerate simplicial map $f'\colon L\to K$
 whose geometric realization is homotopic to $f$ relative to $V(L)\cup L_1$
 (in particular, $f'|_{V(L)\cup L_1}=f|_{V(L)\cup L_1}$).
  \end{prop}
\begin{proof}
We 
proceed
by induction on the skeleta of $L$. More precisely, for every $n\in\matN$ we  define maps $f'_n\colon |L|\to |K|$ with the following properties:
\begin{enumerate}
\item $f'_n$ is simplicial on the $n$-skeleton of $L$;
\item $f'_n$ is homotopic to $f'_{n-1}$ relative to $|L|^{n-1}\cup L_1$  (in particular, $f'_n|_{V(L)\cup L_1}=f'_{n-1}|_{V(L) \cup L_1}=f|_{V(L)\cup L_1}$ and 
$f'_n|_{|L|^{n-1}}=f'_{n-1}|_{|L|^{n-1}}$).
\end{enumerate}
We begin by setting $f'_0=f$. We then suppose that a map $f'_n$ as above is given, and we construct $f'_{n+1}$.
First of all we set $f'_{n+1}=f_n$ on $|L|^n$. Then we extend $f'_{n+1}$ to the $(n+1)$-skeleton as follows.
Let $\sigma$ be an $(n+1)$-simplex of $L$.
If $\sigma$ is a simplex of $L_1$, then we just set $f'_{n+1}(\sigma)=f(\sigma)$
(since this does not cause any ambiguity, we are now using the same symbol for a simplicial map and for the continuous map it induces on geometric realizations).
Otherwise, let $\alpha\colon |\Delta^{n+1}|\to |\sigma|$ be a characteristic map for $\sigma$. 
Since $f'_n$ is simplicial on $L^n$ and injective on the set  of vertices of $\sigma$, the restriction of
$f'_n\circ \alpha$ to $\partial |\Delta^{n+1}|$ is a simplicial embedding. By completeness of $K$, 
$f'_n\circ \alpha$ is homotopic relative to $\partial |\Delta^{n+1}|$ to a simplicial embedding of $\Delta^{n+1}$ onto a simplex $\sigma'$ of $K$. We then define
$f'_{n+1}(\sigma)=\sigma'$.
After repeating this procedure for all the $(n+1)$-simplices of $L$, we are left with a simplicial map $f'_{n+1}\colon |L|^{n+1}\to |K|$ which is homotopic
to $f'_n|_{ |L|^{n+1}}$ relative to  $|L|^{n}\cup |L_1|$. Using the homotopy extension property for CW pairs we can now extend $f'_{n+1}$ to the whole of $|L|$ in such a way that
conditions (1) and (2) above hold. 

We can now define the desired map $f'$ just by setting $f'(\sigma)=f_{n} (\sigma)$ for every $n$-simplex $\sigma$ of $L$. The fact that $f'$ is simplicial and non-degenerate is obvious,
and the fact that $f'$  
 is homotopic to $f$ relative to $V(L)\cup L_1$  follows from the fact that each $f'_n$ is, together with the fact that $f'_k|_{|L|^n}=f'_n|_{|L|^n}$ 
 for every $k\geq n$ (see e.g.~\cite[Proposition 11.2]{Strom}).
\end{proof}

We are now interested in turning continuous homotopies between simplicial maps into simplicial homotopies, at least when the target is a complete multicomplex.
To this aim we first define a structure of multicomplex on the product of a multicomplex times the interval. It is well known that products are more easily defined in the context of ordered
simplicial structures, like ordered simplicial sets. Nevertheless, it is easy to adapt to multicomplexes the definition of product between simplicial complexes 
(that, just as multicomplexes, are unordered simplicial structures). Our construction closely follows~\cite[Lemma 19.1]{munkres}. 

Let $I=[0,1]$. 
Let $\sigma_n$ denote the standard simplex $\Delta^n$, endowed with its natural structure of multicomplex (in particular, the vertices of $\sigma_n$ are \emph{not} ordered).
Of course, in order to define a multicomplex $K\times I$ it is sufficient to define a structure of multicomplex (i.e.~a suitable subdivision into non-degenerate simplices, which we will simply call \emph{triangulation}) 
on the product $\sigma_n \times I$ in such a way that the induced structure on each $(\partial_\alpha \sigma_n)\times I$,
where $\partial_\alpha \sigma_n$ denotes a facet of $\sigma_n$,
coincides with the structure of $\sigma_{n-1}\times I$. Indeed, we will endow $\sigma_n \times I$ with a structure of simplicial complex. By gluing the simplicial complexes
arising from the simplices of $K$ we will then obtain the multicomplex $K\times I$.

We proceed by induction. For $n=0$ we just define $\sigma_0\times I$ as the first barycentric subdivision of $I$.
We then suppose to have already triangulated 
$\sigma_{n-1}\times I$ in such a way that $\sigma_{n-1}\times \{0\}$ and $\sigma_{n-1}\times \{1\}$ are not subdivided (i.e.~they still appear as simplices of the multicomplex structure on
$\sigma_{n-1}\times I$). We then consider the triangulation of the geometric boundary $\partial (\sigma_{n}\times I)$ given by the union of the triangulations
of $(\partial_\alpha \sigma_{n})\times I\cong \sigma_{n-1}\times I$, where $\partial_\alpha\sigma_n$ varies among all the facets of $\sigma_n$, and of $\sigma_n\times\{0\}$ and $\sigma_n\times \{1\}$. 
We define the desired triangulation of $\sigma_n\times I$ by taking the cone of the triangulation of $\partial (\sigma_{n}\times I)$ over an internal point.

\begin{defn}\label{product:def}
 Let $K$ be a multicomplex. We define the multicomplex $K\times I$ as the multicomplex 
 obtained by gluing a copy of the multicomplex $\sigma\times I$ described above for every simplex $\sigma$ of $K$,
 according to the boundary maps of $K$ (and to the structure of $\sigma\times I$ as a multicomplex). 
 The geometric realization of $K\times I$ may be identified with $|K|\times [0,1]$, and 
 there are simplicial embeddings
 $$
i_0,i_1\colon K\hookrightarrow K\times I
 $$
 that induce the identifications $|K|\cong |K|\times \{0\}$, $|K|\cong |K|\times \{1\}$, respectively. We denote by $K\times \{0\}$, $K\times \{1\}$ the images of these embeddings.
By definition, they are submulticomplexes of $K\times I$ isomorphic to $K$. 

It is easily seen that a simplicial map $f\colon K\to L$ between multicomplexes gives rise to a simplicial map $f\times {\rm Id}\colon K\times I\to L\times I$ whose geometric realization
is just the usual product $f\times {\rm Id}\colon |K|\times [0,1]\to |L|\times [0,1]$. In particular, if $K_0$ is a submulticomplex of $K$, then the inclusion $K_0\hookrightarrow K$
induces a natural inclusion of $K_0\times I$ in $K\times I$. We will often denote just by $K_0\times I$ the submulticomplex of $K\times I$ obtained as the image of this inclusion.
\end{defn}

\begin{defn}
 Let $f_0,f_1\colon L\to K$ be simplicial maps between multicomplexes. We say that a simplicial map $F\colon L\times I\to K$ is a \emph{simplicial homotopy} between $f_0$ and $f_1$ (and that 
 $f_0,f_1$ are simplicially homotopic) if $F\circ i_0=f_0$, $F\circ i_1=f_1$, where $i_j\colon L\to L\times I$ is the canonical embedding onto $L\times \{j\}$ described in Definition~\ref{product:def}.
 We say that the homotopy $F$ is non-degenerate if $F$ is non-degenerate as a simplicial map (which implies that $f_0,f_1$ are non-degenerate too).
\end{defn}

As we may expect, simplicially homotopic maps induce the same maps on simplicial (bounded) (co)homology:

\begin{lemma}\label{homotopy-homology}
 Let $f_0,f_1\colon L\to K$ be simplicially homotopic simplicial maps between multicomplexes, and let 
 $$\begin{array}{ll} H_n(f_i)\colon &H_n(L;R)\to H_n(K;R)\ ,\\ H^n(f_i)\colon & H^n(K;R)\to H^n(L;R)\ ,\\
 H^n_b(f_i)\colon & H^n_b(K;R)\to H^n_b(L;R) \end{array} $$ be the induced maps on (bounded) (co)homology. 
 Then 
 $$
 H_n(f_0)=H_n(f_1)\, ,\quad
 H^n(f_0)=H^n(f_1)\, ,\quad
 H^n_b(f_0)=H_b^n(f_1)\, .
 $$
\end{lemma}
\begin{proof}
 As usual, it is sufficient to show that the inclusions $i_0,i_1\colon L\hookrightarrow L\times I$ induce chain maps
 $C_*(i_0),C_*(i_1)\colon C_*(L;R)\to C_*(L\times I;R)$ that are algebraically homotopic via a homotopy which is bounded (with respect
 to the $\ell^1$-norm) in every degree. Such a homotopy may be constructed by sending each simplex $\sigma$ of $L$ to the sum (with signs)
 of the simplices triangulating the subset $\sigma\times I\subseteq L\times I$ (see \cite[Theorems 12.5, 13.3 and 19.2]{munkres} for the details). 
\end{proof}

\begin{rem}\label{homotopy-uniform}
 In the proof of the previous lemma we have observed  that 
 simplicially homotopic maps 
 are algebraically homotopic via a homotopy which is bounded in every degree.
 In fact, in the sequel we will be working with (possibly infinite) families of homotopic maps, and in that context it will prove useful to know that
 algebraic homotopies coming from simplicial homotopies have uniformly bounded norms. Indeed, the proof of Lemma~\ref{homotopy-homology} (compare with \cite[Theorems 12.5 and 19.2]{munkres}) makes clear that the following holds: 
 let $C_n$ be the number of $n$-simplices in the the multicomplex $\sigma_{n-1}\times I$ described above, where again $\sigma_{n-1}$ denotes the standard simplex $\Delta^{n-1}$ with its natural structure of multicomplex.
 Then, for every pair
 $f_0,f_1\colon L\to K$ of simplicially homotopic simplicial maps between multicomplexes, the induced maps $f_i^*\colon C^*_b(K)\to C^*_b(L)$, $i=0,1$, on bounded cochains are algebraically homotopic
 via a homotopy $T^*\colon C^*_b(K)\to C^{*-1}_b(L)$ such that $\|T^n\|_\infty\leq C_n$.
\end{rem}

\begin{defn}
A multicomplex $K$ is \emph{large} if every connected component of $K$ contains infinitely many vertices. 
\end{defn}

One of the main features of completeness is that it allows to turn topological homotopies into simplicial homotopies:

\begin{Lemma}[Homotopy Lemma]\label{homotopy-lemma}
Let $L,K$ be multicomplexes, and suppose that $K$ is large and complete.
Let $f_0,f_1\colon L\to K$ be non-degenerate simplicial maps that induce homotopic maps $|L|\to |K|$ on geometric realizations. 
Then $f_0,f_1$ are simplicially homotopic via a non-degenerate homotopy. 
\end{Lemma}
\begin{proof}
Let $F\colon |L|\times I \to |K|$ be a continuous homotopy between the geometric realizations of $f_0$ and $f_1$. 
The structure of $L\times I$ as a multicomplex is such that no simplex of $L\times I$ can have vertices both on $L\times \{0\}$ and on $L\times \{1\}$. Therefore,
 by exploiting as usual the homotopy extension
property for CW pairs, the largeness of $K$ and  the fact that $f_0,f_1$ are non-degenerate, we can modify $F$ (without altering its behaviour on $|L|\times \{0,1\}$) 
in such a way that, if $\sigma$ is a simplex of $L\times I$, then the restriction of $F$ to the vertices
of $\sigma$ is injective. 
By applying Proposition~\ref{approx-simpl} (with $f=F$ and $L_1=L\times\{0\}\cup L\times\{1\}$) we obtain
a non-degenerate simplicial map $F'\colon L\times I$ which agrees with $f_i$ on $L\times \{i\}$. This concludes the proof. 
\end{proof}

\section{Minimal multicomplexes}\label{minimal:section}
In this section we will introduce the notion of \textit{minimal multicomplex}.
In complete multicomplexes, every simplicial embedding of $\partial \Delta^n$ which continuously extends over $|\Delta^n|$ may be deformed into
a simplicial embedding of $\Delta^n$ itself. In minimal multicomplexes such a deformation, when possible, leads to a \emph{unique} choice for the resulting simplicial embedding
(see Definition~\ref{minimal:def:1} for a precise formulation of this condition).
In some sense, minimal multicomplexes do not contain redundant simplices, i.e.~simplices that do not enrich the homotopy type of the geometric realization.  
We will show how to replace a complete multicomplex $K$ with a submulticomplex $L$ which is both complete and minimal, without affecting 
the homotopy type of the geometric realization (see Theorem \ref{exist-min}). This fact will prove very useful: indeed, we have already seen in Theorems \ref{homotopy-weak-intro} and \ref{K(X)-compl} that every good topological space $X$ is weakly homotopy equivalent
to the singular multicomplex $\calK(X)$, which is complete. Replacing $\calK(X)$ with a complete minimal submulticomplex will allow us to
dramatically reduce 
the number of simplices  needed to exhibit a multicomplex with the same homotopy type as $X$. 

The notion of miminality for multicomplexes is 
 very similar to the classical definition of minimality for simplicial sets (and for $\Delta$-complexes), and Theorem~\ref{exist-min} extends to multicomplexes
 the classical result that every Kan complex contains a minimal complex with the same homotopy type (see for 
instance \cite[Section~4]{eil-zil} and~\cite[Definition 9.1 and Theorem 9.5]{may}):


\begin{Definizione}\label{minimal:def:1}
A multicomplex $K$ is said to be \emph{minimal} if the following condition holds. Let $f \colon \lvert \Delta^{n} \rvert \rightarrow \lvert K \rvert$ be a continuous map  whose restriction to the boundary 
$f \vert_{\partial |\Delta^{n}|} \colon \partial |\Delta^{n}| \rightarrow |K|$ is a simplicial embedding. Then $f$ is homotopic relative to $\partial |\Delta^{n}|$ to at most one simplicial embedding $f' \colon |\Delta^{n}| \rightarrow |K|$.
In particular, $K$ is minimal and complete if $f$ is homotopic relative to $\partial |\Delta^{n}|$ to exactly one simplicial embedding $f' \colon |\Delta^{n}| \rightarrow |K|$.
\end{Definizione}

The following lemma readily follows from the definitions (see Definition~\ref{special:def} for the notion of homotopic simplices): 

\begin{lemma}\label{caratterizzazione-minimality-per-confronto-ammissibili}
 Let $K$ be a multicomplex. Then $K$ is minimal if and only if it does not contain any pair of distinct homotopic simplices.
\end{lemma}

In the sequel we will make use of the following elementary fact about minimal multicomplexes:

\begin{Lemma}\label{min-allbutoneface}
Let $K$ be a minimal multicomplex and let $\sigma$ and $\tau$ be two $n$-simplices of $K$  that share all but at most one facets. 
Then, the set of facets of  $\sigma$ and  $\tau$ coincide, i.e.~$\sigma$ and $\tau$ are compatible. 
\end{Lemma}
\begin{proof}
Let us order the facets of $\sigma$ and $\tau$ in such a way that $\partial_i\sigma=\partial_i\tau$ for every $i=0,\ldots,n-1$, while possibly
$\partial_{n} \sigma\neq \partial_{n} \tau$. 
By construction, the $(n-1)$-simplices $\partial_{n} \sigma$ and $\partial_{n} \tau$ are homotopic. Thus, they coincide
 by Lemma \ref{caratterizzazione-minimality-per-confronto-ammissibili}.
\end{proof}

\begin{Esempio}\label{ex-not-minimal}
Of course, every simplicial complex is minimal as a multicomplex. Also the simplest multicomplex that is not a simplicial complex is  minimal:
if the multicomplex $L$ consists only of two vertices joined by two distinct edges, then $L$ is minimal (but not complete, see  Example~\ref{ex-compl-multi} (1)).
By contrast, the simplicial cone $C(L)$ over $L$ is not minimal: its geometric realization $|C(L)|$ is homeomorphic to a disc, whose boundary may be identified with $|L|$, 
so the two edges of $L$ are homotopic in $C(L)$. 
\end{Esempio}

The following result, which is an immediate corollary of 
Theorem~\ref{complete:special}, provides a combinatorial description of the homotopy groups of complete and minimal multicomplexes. It provides an analogue for multicomplexes
of a classical result for simplicial sets (see e.g.~\cite[Section I.11, page~ 60]{GJ}).

\begin{thm}\label{complete:minimal:special}
Let $K$ be a complete minimal  multicomplex, and
 let $\Delta_0$ be an $n$-simplex of $K$. Also fix an ordering on the vertices of $\Delta_0$ and denote by $x_0$ the minimal vertex of $\Delta_0$.
 The map
 $$
 \Theta\colon \pi(\Delta_0)\to \pi_n(|K|,x_0)\, ,\qquad
 \Theta(\Delta)=\left[\dot{S}^n(\Delta_0,\Delta)\right]
 $$
 is a bijection.
\end{thm}

Being a bijection, the map $\Theta$ described in the previous theorem induces a group structure on $\pi(\Delta_0)$. This group structure will be described in terms of the action
of particular groups of simplicial automorphisms of $K$ in the next chapter (see e.g.~Lemma~\ref{phi:hom}). 

Theorem~\ref{complete:minimal:special} suggests that it would be useful 
to replace (possibly in a canonical way) a complete multicomplex with a complete and minimal one, 
of course without affecting its homotopy type. The following theorem provides a statement in this direction, and generalizes to the context
of multicomplexes analogous results for simplicial sets (see e.g.~\cite[Theorems 9.5 and 9.8]{may}).

\begin{thm}\label{exist-min}
Let $K$ be a  complete multicomplex. Then, there exists a submulticomplex $L\subseteq K$ with the following properties:
\begin{enumerate}
 \item $L$ is complete and minimal;
 \item $L$ has the same vertices as $K$;
 \item there exists a simplicial retraction $r\colon K\to L$;
 \item the geometric realization of $r$ realizes $|L|$ as a strong deformation retract of $|K|$ (in particular, the inclusion of $|L|$ in $|K|$ is a homotopy equivalence).
\end{enumerate}
Moreover, such a submulticomplex $L$ is unique up to simplicial isomorphism.
\end{thm}

\begin{proof}
We construct $L$ inductively on the dimension of simplices as follows. The set of vertices of $L$ is equal to the set of vertices of $K$, i.e.~$L^0=K^0$. Once $L^n$ has been constructed,
we define $L^{n+1}$ by adding to $L^n$ one $(n+1)$-simplex for each homotopy class of $(n+1)$ simplices of $K$ whose facets are all contained in $L^n$. With a slight abuse,
we denote by $i_n\colon L^n\to K^n$ both the simplicial inclusion of $L^n$ in $K^n$ and the induced inclusion on geometric realizations.

Let us first prove that $|L|$ is a strong deformation retract of $|K|$. 
By~\cite[Proposition 11.2]{Strom} it is sufficient to construct,
for every $n\in\matN$, a map $r_n\colon |K|\to |K|$
and a homotopy $h_n\colon |K|\times [0,1]\to |K|$ between $r_n$ (at time $0$) and $r_{n+1}$ (at time $1$) satisfying the following properties: $r_0$ is the identity of $|K|$;
$r_n|_{|K^n|}$ is a simplicial retraction of $|K^n|$ onto $|L^n|$; $r_{n+1}|_{|K^n|}=r_n|_{|K^n|}$; $h_{n+1}(x,t)=r_{n}(x)$ for every $x\in |K^{n}|$.

For $n=0$ we set $r_0={\rm Id}$. Having defined $r_n$ and $h_{n-1}$, we then set $r_{n+1}|_{|K^n|}=r_n$ and we define $r_{n+1}$ on 
$(n+1)$-simplices as follows.
Let $\sigma$ be an $(n+1)$-simplex
of $K$.
Our inductive hypothesis ensures that the restriction of $r_n$ 
to ${\partial |\sigma |}$ is a simplicial embedding with values in $|L|$ (recall that $r_n=r_0={\rm Id}$ on the $0$-skeleton of $K$).
By completeness of $K$ and by definition of $L$, there exists a homotopy $h_\sigma\colon |\sigma|\times [0,1]\to |K|$
such that $h_\sigma (x,0)=x$ for every $x\in |\sigma|$, $h(x,t)=x$ for every $x\in\partial |\sigma|$, and
$h_\sigma(\cdot, 1)$ is an affine isomorphism between $|\sigma|$ and an $(n+1)$-simplex $r_{n+1}(\sigma)$ of $L$. 
Observe that, by construction, $r_{n+1}(\sigma)=\sigma$ if $\sigma$ is already a simplex of $L$
(and in this case the homotopy  $h_\sigma$ may be chosen to be constant). 
After repeating this construction for every $(n+1)$-simplex of $K$ we end up with well-defined maps $r_{n+1}\colon K^{n+1}\to L^{n+1}$ and $h_{n}\colon |K^{n+1}|\times [0,1]\to |L^{n+1}|$, which satisfy all the required
properties, except that they are not defined on the whole of $|K|$ and $|K|\times [0,1]$. 
However, 
thanks to the homotopy extension property  for CW pairs, we can then extend these maps to the required maps $r_{n+1}\colon |K|\to |L|$ and 
$h_{n+1}\colon |K|\times [0,1]\to |K|$. This concludes the construction of the strong deformation retraction $r$. Also observe that $r$ is simplicial.


 In order to conclude we now need to show that $L$ is minimal and complete (and unique up to isomorphism).
 Minimality of $L$ is obvious: If two simplices are homotopic in $L$, then \emph{a fortiori} they are homotopic in $K$. But by construction $L$ does not contain any pair of distinct simplices
 that are homotopic in $K$, and this implies that $L$ is minimal. In order to show that $L$ is complete we make use of the retraction we have just constructed: if $f\colon |\Delta^n|\to |L|$ is a continuous
 map whose restriction to $\partial |\Delta^n|$ is a simplicial embedding, then by completeness of $K$ there exists a simplicial embedding $f'\colon |\Delta^n|\to |K|$ which is homotopic (in $|K|$) to $f$ relative
 to $\partial |\Delta^n|$. Since $f(\partial |\Delta^n|)\subseteq |L|$, by definition of $L$ we can choose $f'$ to have image in $|L|$. In order to conclude we now need to observe that, by composing the homotopy in $|K|$
 between $f$ and $f'$ with the retraction $r$, we obtain a homotopy between $f$ and $f'$ (relative to $\partial |\Delta^n|$) in $|L|$. This concludes the proof that $L$ is complete. 
 
 Finally, observe that, if $L_1$ and $L_2$ are complete and minimal multicomplexes as in the statement, then by composing the inclusion of $L_1$ in $K$ with the retraction of
 $K$ to $L_2$ we obtain a homotopy equivalence which is a bijection on vertices. Uniqueness of $L$ up to simplicial isomorphisms is then a consequence
 of Proposition~\ref{uniq:prop} below. 
\end{proof}

\begin{prop}\label{uniq:prop}
Let $L_{1}$ and $L_{2}$ be complete  minimal multicomplexes. Let $f \colon L_{1} \rightarrow L_{2}$ be a simplicial map which is bijective on the 0-skeletons. 
Also suppose that the geometric realization  $\lvert {f} \rvert\colon |L_1|\to |L_2|$ of $f$ is a homotopy equivalence. Then $f$ is a simplicial isomorphism.
\end{prop}
\begin{proof}
We prove by induction on $n\in\matN$ that $f$ is bijective on the set of $n$-simplices. The case $n=0$ holds by hypothesis, so we may suppose that
$f$ restricts to an isomorphism between $L_1^{n-1}$ and $L_2^{n-1}$.  Let $\sigma,\sigma'$ be $n$-simplices of $L_1$ such that $f(\sigma)=f(\sigma')$. Since $f$ is injective on
$L_1^{n-1}$ we have that $\sigma,\sigma'$ are compatible. Moreover, the special sphere $\dot{S}^n(\sigma,\sigma')$ is taken by $f$ onto a null-homotopic special sphere. Since $f$ is a  homotopy equivalence, by Lemma~\ref{lemma-the-following-3-cond-equivalent}
this implies that $\sigma$ and $\sigma'$ are homotopic, whence equal to each other by minimality of $L_1$. This shows that the restriction of $f$ to $L_1^n$ is injective. 

Let now $\sigma$ be an $n$-simplex of $L_2$ 
and let 
$j_2\colon \partial\Delta^n\to L_2$ be a  simplicial isomorphism between
$\partial\Delta^n$ and $\partial\sigma$. Since $f$ restricts to an isomorphism between $L_1^{n-1}$ and $L_2^{n-1}$, there exists a simplicial embedding 
$j_1\colon \partial\Delta^n\to L_1$ such that $j_2=f\circ j_1$. Since $|f|$ is a  homotopy equivalence, from the fact that $\partial\sigma$ bounds a simplex in $L_2$
we deduce that the geometric realization of $j_1$ extends to a continuous map on $|\Delta^n|$. By completeness of $L_1$, this implies that there exists a simplex $\Delta_0$ such that $f(\partial \Delta_0)=\partial \sigma$.
Let now $x_0$ be a vertex of $\Delta_0$. Being a  homotopy equivalence, the map $|f|$ induces an isomorphism between $\pi_n(|L_1|,x_0)$ and $\pi_n(|L_2|,f(x_0))$. By Theorem~\ref{complete:minimal:special}, this implies in turn that
$f$ induces a bijection between $\pi(\Delta_0)$ and $\pi(\sigma)$ (recall that $\pi(\Delta)$ is the set of simplices compatible with $\Delta$). In particular, $\sigma$ lies in the image of $f$, and this concludes the proof.
\end{proof}

Theorem~\ref{exist-min} shows that, if $K$ is a complete multicomplex, then minimal and complete submulticomplexes of $K$ 
whose inclusion in $K$ defines a homotopy equivalence 
are unique up to isomorphism. Therefore, with a slight abuse we will often refer to such a submulticomplex as to  \emph{the} minimal multicomplex associated to $K$.

\begin{rem}
 Suppose that $K$ is large. Then, thanks to Lemma~\ref{homotopy-lemma} the statement of Theorem~\ref{exist-min} may be strengthened by requiring that $L$ is a simplicial deformation
 retract of $K$, i.e.~that there exists a non-degenerate simplicial retraction $r\colon K\to L$ which is simplicially homotopic to  the identity of $K$. However, we will not use this fact in the sequel.
\end{rem}

We will be particularly interested in studying the minimal multicomplex associated to the singular multicomplex of a topological space. Therefore, we introduce the following:

\begin{defn}
Let $X$ be a good topological space. We know from Theorem~\ref{K(X)-compl} that the multicomplex $\calK(X)$ is complete. We then denote by $\calL(X)$ the minimal multicomplex associated to $\calK(X)$. 
\end{defn}

\begin{cor}\label{cor-assoc-min-e-compl-to-X}
Let $X$ be a good topological space, 
and let $i\colon |\calL(X)|\to |\calK(X)|$ be the inclusion. Then the composition
$$
S \circ i \colon |\calL(X)| \to X
$$
is a weak homotopy equivalence. 
Moreover, if $X$ is a CW complex, then $S \circ i$ is a homotopy equivalence.
\end{cor}
\begin{proof}
Theorems~\ref{homotopy-weak-intro} and~\ref{exist-min} imply that
the natural
projection $S \colon|\calK(X)|\to X$ 
and the inclusion $i \colon |\calL(X)|\to |\calK(X)|$ are weak homotopy equivalences. This proves the first statement. The second statement now follows from
Whitehead Theorem.
\end{proof}

Just as the singular multicomplex $\calK(X)$, also $\calL(X)$ enjoys some functorial properties.
If $X$ is a good topological space, with a slight abuse  we denote by $S_X\colon |\calL(X)|\to X$ the restriction of the natural projection
$S_X\colon|\calK(X)|\to X$.

\begin{prop}\label{commutative:homotopy:L}
 Let $X,Y$ be good topological spaces and let $f\colon X\to Y$ be a continuous map. Then there exists a continuous map
 $L(f)\colon |\calL(X)|\to |\calL(Y)|$ such that the following diagram commutes up to homotopy:
 $$
 \xymatrix{
 |\calL(X)|\ar[r]^-{L(f)}\ar[d]_-{S_X} & |\calL(Y)|\ar[d]^-{S_Y}\\
 X \ar[r]_-f & Y\ .
 }
 $$
  \end{prop}
\begin{proof}
Just set $L(f)=S_Y\circ K(f)\circ i_X$, where $i_X\colon |\calL(X)|\to |\calK(X)|$ is the inclusion, and the map
$K(f)\colon \calK(X)\to\calK(Y)$ is provided by Proposition~\ref{commutative:homotopy}.
\end{proof}

\section{Aspherical multicomplexes}\label{aspherical:sec}
We have seen in Theorem \ref{exist-min} that to every complete multicomplex $K$ it is possible to associate a complete and minimal multicomplex $L$, which is a retract of $K$. A peculiar characteristic
of bounded cohomology is that maps that induce an isomorphism on the fundamental group also induce isometric isomorphisms on bounded cohomology in every degree. In particular, 
bounded cohomology is not sensitive  to higher homotopy groups. In order to prove these facts we will need to work  with aspherical multicomplexes.
In fact, we need to be able to associate to a complete and minimal multicomplex $L$ an aspherical multicomplex $A$, together with a natural quotient map $L\to A$
which induces an isomorphism on fundamental groups of the geometric realizations. 

\begin{Definizione}
Let $A$ be a multicomplex. We say that $A$ is \textit{aspherical} if its geometric realization is an Eilenberg-MacLane space of type $(\pi_{1}(\lvert A \rvert), 1)$.
\end{Definizione}

We have the following characterization of complete, minimal and aspherical multicomplexes:

\begin{prop}\label{aspherical:char}
 Let $A$ be a large and connected multicomplex. Then $A$ is complete, minimal and aspherical if and only if the following conditions hold:
 \begin{enumerate}
  \item For every pair of distinct vertices $v_0,v_1$ of $A$ and every continuous path $\gamma\colon [0,1]\to |A|$ with $\gamma(0)=v_0$, $\gamma(1)=v_1$, there exists a unique
  simplicial embedding $\gamma'\colon \Delta^1\to A$ which is homotopic to $\gamma$ relative to the endpoints.
  \item Let  $n\geq 2$, let $(\Delta^n)^1$ be the $1$-skeleton of $\Delta^n$  
  and let $f\colon (\Delta^n)^1\to A$ be a simplicial embedding such that the restriction of $f$ to each triangular loop is null-homotopic. Then there exists
  a unique simplicial embedding $\overline{f}\colon \Delta^n\to A$ extending $f$.
 \end{enumerate}
\end{prop}
\begin{proof}
 We first proof the ``only if'' implication, which is much easier. In fact, condition (1) is an immediate consequence of completeness and minimality of $A$. Moreover, by Lemma~\ref{quoz-compl-lemma1}
 completeness of $A$ implies the existence of the simplicial embedding required in (2). In order to prove the uniqueness of such an embedding we argue by contradiction.
 Let  $\Delta_1,\Delta_2$ be distinct $n_0$-simplices of $A$ 
 sharing the same $1$-skeleton. By requiring that these simplices have the least possible dimension, we can assume that (2) holds in every dimension $n<n_0$. In particular, every codimension one face of
 $\Delta_1$ is also a codimension one face of $\Delta_2$, i.e.~$\Delta_1$ and $\Delta_2$ are compatible. By minimality of $A$, this implies in turn that $\Delta_1$ and $\Delta_2$ are \emph{not} homotopic. By
Lemma~\ref{lemma-the-following-3-cond-equivalent}, the embedded sphere $|\Delta_1|\cup |\Delta_2|$ is the support of a homotopically non-trivial $n_0$-sphere in $|A|$, and this contradicts the fact that $A$ is aspherical.

Let us now suppose that conditions (1) and (2) hold, and prove that $A$ is complete, minimal and aspherical. 
We first prove asphericity by showing that the universal covering $\widetilde{|A|}$ of $|A|$ is contractible.
Observe first that the universal covering of (the geometric realization) of a multicomplex admits a natural structure of multicomplex. We will denote by $\widetilde{A}$
the multicomplex such that $|\widetilde{A}|=\widetilde{|A|}$, and by
$p\colon |\widetilde{A}|\to |A|$ the covering projection (which is a simplicial map).

We claim that  $\widetilde{A}$ is a simplicial complex. In fact, assume by contradiction that there exist distinct $n$-simplices
$\widetilde{\Delta}_1,\widetilde{\Delta}_2$ having the same set of vertices, and denote by $\Delta_i$ the projection of $\widetilde{\Delta}_i$
in $A$, $i=1,2$. 
Since $p$ is locally injective, we have $\Delta_1\neq \Delta_2$. Thus,
if $n=1$, the simplices $\Delta_1,\Delta_2$ would be homotopic (and distinct) in $A$, against condition (1).
If $n\geq 2$, then 
the argument just described proves that the $1$-skeleton of $\widetilde{\Delta}_1$ should coincide with the $1$-skeleton of $\widetilde{\Delta}_2$. Therefore, 
$\Delta_1$ and $\Delta_2$ would be distinct simplices of $A$ sharing the same $1$-skeleton, and this contradicts condition (2). 

Let us now prove that an $(n+1)$-tuple $(a_0,\ldots,a_n)$ of pairwise distinct vertices of $\widetilde{A}$
spans an $n$-simplex if and only if the vertices $p(a_0),\ldots, p(a_n)$ of $A$ are pairwise distinct.
The ``only if'' implication is obvious: since the covering projection $p\colon\widetilde{A}\to A$ is simplicial and non-degenerate, the image of the set of vertices
of an $n$-simplex of $\widetilde{A}$ is the set of vertices of an  $n$-simplex of $A$, so it consists of $(n+1)$ distinct vertices of $A$.

In order to prove the converse implication, 
let $(a_0,\ldots,a_n)$ be an $(n+1)$-tuple of vertices of $\widetilde{A}$ such that $p(a_i)\neq p(a_j)$ for every $i\neq j$. Since $|A|$ is connected, also $\widetilde{|A|}$ is. 
If $n=1$, this means that we can fix a path $\gamma\colon [0,1]\to \widetilde{|A|}$ 
joining $a_0$ and $a_1$. By (1), the composition $p\circ \gamma$ is homotopic relative to the endpoints to a $1$-simplex of $A$. The lift of this $1$-simplex
with initial point $a_0$ provides a $1$-simplex of $\widetilde{A}$ with vertices $a_0,a_1$. 

Suppose now $n\geq 2$. We have just proved that for every $0\leq i<j\leq n$ there exists a $1$-simplex in $\widetilde{A}$ with vertices $a_i,a_j$. Therefore,
we can construct a simplicial embedding $\widetilde{f}\colon (\Delta^n)^1\to \widetilde{A}$ such that the images of the vertices of $\Delta^n$ are exactly the $a_i$. 
The composition $f=p\circ\widetilde{f}\colon  (\Delta^n)^1\to A$ is now 
a simplicial embedding such that the restriction of $f$ to each triangular loop (being the projection of a loop in $|\widetilde{A}|$, which is simply connected) is null-homotopic. 
By condition (2), we can extend $f$ to a simplicial embedding $g\colon \Delta^n\to A$. 
It is now easy to check that $g$ lifts to a simplicial embedding  $\widetilde{g}\colon \Delta^n\to \widetilde{A}$ with vertices $a_0,\ldots,a_n$. Therefore, the points
$a_0,\ldots,a_n$ are the vertices of a simplex of $\widetilde{A}$.  We have thus proved that
an $(n+1)$-tuple $(a_0,\ldots,a_n)$ of pairwise distinct vertices of $\widetilde{A}$
spans an $n$-simplex if and only if the vertices $p(a_0),\ldots, p(a_n)$ of $A$ are pairwise distinct.

We are now ready to prove that $\widetilde{|A|}$ is contractible. Since $\widetilde{|A|}$ is a simply connected CW complex, by Hurewicz Theorem it is sufficient
to show that $H_n(\widetilde{|A|};\mathbb{Z})=0$. Since singular homology and simplicial homology are canonically isomorphic, we can equivalently show
that $H_n(\widetilde{A};\mathbb{Z})=0$ for every $n\geq 2$. Recall that simplices of $\widetilde{A}$ are completely determined by their vertices, so we can  identify $C_n(\widetilde{A};\mathbb{Z})$
with the free $\mathbb{Z}$-module generated by the $(n+1)$-tuples of the form $(a_0,\ldots,a_n)$, where the set $\{a_0,\ldots,a_n\}$  is the set of vertices of a simplex
of $\widetilde{A}$. Let then
$$
z=\sum_{j=1}^k b_j (a_0^j,\ldots,a_n^j)
$$
be an $n$-cycle in $C_n(\widetilde{A};\mathbb{Z})$, $n\geq 2$, $b_j \in \, \matZ$. Since $A$ is large,  there exists a vertex $q$ of $A$ such that $q\neq p(a_i^j)$ for every $i=0,\ldots,n$, $j=1,\ldots,k$. Let $\overline{a}$ be a fixed lift of $q$ to $\widetilde{A}$. 
The description of $\widetilde{A}$ given above shows that, for each $j=1,\ldots,k$,
the set $\{\overline{a},a_0^j,\ldots,a_n^j\}$ spans a simplex of $\widetilde{A}$. Therefore, the sum
$$
c=\sum_{j=1}^k b_j (\overline{a},a_0^j,\ldots,a_n^j)
$$
defines an element of $C_{n+1}(\widetilde{A};\mathbb{Z})$. It is now readily seen that $\partial c=z$. We have thus shown that $H_n(\widetilde{A};\mathbb{Z})=0$ for every $n\geq 2$, and this implies in turn that $\widetilde{A}$ is contractible, and that $A$ is aspherical.
 
 In order to conclude the proof of the proposition we are left to show that $A$ is complete and minimal. So let $f\colon |\Delta^n|\to |A|$ be a continuous map
 which restricts to a simplicial embedding of $\partial |\Delta^n|$. We need to show that $f$ is homotopic relative to $\partial \Delta^n$ to a unique simplicial embedding
 $g\colon \Delta^n\to A$. If $n=1$, this is just  a restatement of condition (1). If $n\geq 2$, since $p$ is a non-degenerate simplicial map which is also locally injective, the map $f$ lifts to a map $\widetilde{f}\colon |\Delta^n|\to \widetilde{A}$
 which still restricts to a simplicial embedding of $\partial |\Delta^n|$. Being the lifts of $n+1$ distinct points in $A$, the vertices of $\widetilde{f}$ span a simplex of $\widetilde{A}$,
 so the restriction of $\widetilde{f}$ to $\partial |\Delta^n|$ extends to a simplicial embedding $\widetilde{g}\colon \Delta^n\to \widetilde{A}$. Since $|\widetilde{A}|$ is contractible,
 the maps $\widetilde{f}$ and $\widetilde{g}$ are homotopic relative to $\partial |\Delta^n|$. Therefore, the map $g=p\circ \widetilde{g}\colon \Delta^n\to A$
 is a simplicial embedding and is homotopic to $f$ relative to $\partial |\Delta^n|$. If $g'\colon \Delta^n\to A$ were another simplicial embedding homotopic to $f$ relative to $\partial |\Delta^n|$, 
 then the images of $g$ and $g'$ would lift to distinct simplices of $\widetilde{A}$ having the same set of vertices, against the fact that $\widetilde{A}$ is a simplicial complex.
 This concludes the proof.
\end{proof}

We are now able to describe how a large, complete, minimal and aspherical multicomplex $A$ can be canonically associated to a complete and minimal multicomplex $L$. 
We suppose until the end of the section that $L$ is a large, connected, complete and minimal  multicomplex, and we construct the desired multicomplex $A$ as follows.
The $1$-skeleton of $A$ just coincides with the  $1$-skeleton of $L$ (in particular, $L$ and $A$ have the same set of vertices). Then, if $T\subseteq A^1$ is 
 the $1$-skeleton of an $n$-simplex of $L$, $n\geq 2$, we add to $A$ one (and exactly one) $n$-simplex having $T$ as $1$-skeleton. In other words, for every  $n\geq 2$
the set of $n$-simplices of $A$ is given by the equivalence classes of $n$-simplices of $L$, where two $n$-simplices
are considered equivalent if they share the same $1$-skeleton. It is an easy exercise to check that $A$ is indeed a multicomplex, and that the map
which sends every $n$-simplex of $L$, $n\geq 2$, to its class in $A$ extends the identity $L^1\to A^1$ to a non-degenerate simplicial map $\pi\colon L\to A$.
We  call $A$ the \emph{aspherical quotient} of $L$.

\begin{thm}\label{aspherical:thm}
Let $L$ be a complete, minimal and large multicomplex, and let $A$ be the aspherical quotient of $L$. Then:
\begin{enumerate}
\item the projection $\pi\colon |L|\to |A|$ induces an isomorphism on fundamental groups;
\item $A$ is complete, minimal and aspherical.
\end{enumerate}
\end{thm}
\begin{proof}
The fact that $\pi\colon |L|\to |A|$ induces an isomorphism on fundamental groups is an immediate consequence of the fact that
$L$ and $A$ have the same $1$-skeleton, and a triangular loop in $L^1=A^1$ bounds a $2$-simplex in $L$ if and only if it bounds a $2$-simplex in $A$.  

In order to prove that $A$ is complete, minimal and aspherical we will check that $A$ satisfies the conditions described in Proposition~\ref{aspherical:char}.

Let  $v_0,v_1$ be distinct vertices of $A$, and let $\gamma \colon [0,1]\to |A|$ be a path joining $v_0$ and $v_1$. Up to homotopy (relative to the endpoints) we can suppose
that $\gamma$ is supported on $A^1=L^1$. Since $L$ is complete and minimal, the path $\gamma$ (now considered as a map to $|L|$) is homotopic relative to the endpoints (in $|L|$) to a unique
$1$-simplex of $L$. However, 
since $\pi\colon |L|\to |A|$ induces an isomorphism on fundamental groups, two paths in $L^1=A^1$ are homotopic relative to the endpoints in $|L|$ if and only if they are homotopic relative
to the endpoints in $|A|$. This shows that $\gamma$ is homotopic relative to the endpoints in $A$ to a unique $1$-simplex of $A$, i.e.~that condition (1) of Proposition~\ref{aspherical:char} is satisfied.

Let us now prove that also condition (2) of Proposition~\ref{aspherical:char} holds. To this aim, let $n\geq 2$, let $(\Delta^n)^1$ be the $1$-skeleton of $\Delta^n$  
  and let $f\colon (\Delta^n)^1\to A$ be a simplicial embedding such that the restriction of $f$ to each triangular loop is null-homotopic in $A$. We must show that there exists
  a unique simplicial embedding $\overline{f}\colon \Delta^n\to A$ extending $f$. The uniqueness of such an embedding readily follows from the very definition of $A$: indeed, by construction
  there do not exist in $A$ distinct simplices sharing the same $1$-skeleton. Therefore, we are left to show that a simplicial extension of $f$ to $\Delta^n$ exists.
  Since $A^1=L^1$, we can consider $f$ as a map with values in $L$, and since $\pi\colon L\to A$ induces an isomorphism on fundamental groups, the restriction of $f$ to each triangular loop is null-homotopic also in $|L|$. 
  Since $L$ is complete, by Lemma~\ref{quoz-compl-lemma1} we can extend $f$ to a simplicial embedding $f'\colon \Delta^n\to L$. The composition $\overline{f}=\pi \circ f' \colon \Delta^n\to A$ provides the desired extension
  of $f$, and this concludes the proof.
\end{proof}
  \medskip
 
 Recall that, for every good topological space, we have introduced the singular multicomplex $\mathcal{K}(X)$, which turned out to be complete, and
 the minimal complete multicomplex $\mathcal{L}(X)$ associated to $\mathcal{K}(X)$. 
 
 \begin{defn}\label{asphericalX:def}
 Let $X$ be a good topological space. Then we denote by $\mathcal{A}(X)$ the aspherical quotient of $\mathcal{L}(X)$. By Theorem~\ref{aspherical:thm}, the multicomplex
 $\mathcal{A}(X)$ is complete, minimal, and aspherical. 
 \end{defn}

We summarize some properties of $\calK(X)$, $\calL(X)$ and $\calA(X)$ in the following:

\begin{prop}\label{riassunto:prop}
Let $X$ be a good topological space. Then the simplicial retraction $r\colon \calK(X)\to \calL(X)$ is a homotopy equivalence, and the simplicial projection
$\pi\colon \calL(X)\to \calA(X)$ induces an isomorphism on fundamental groups.
\end{prop}

Before concluding the chapter, let us point out how the constructions introduced so far allow us to give an interesting description of a classifying map for a simplicial complex.
Suppose that $X$ is the geometric realization of a simplicial complex $T$.
The natural projection $S\colon |\calK(X)|\to X$ is a homotopy equivalence (see Corollary~\ref{weak-CW}). 
Moreover, the multicomplex $\calK(X)$ contains a submulticomplex $\calK_T(X)\cong T$ whose simplices are the equivalence classes of the affine parametrizations of simplices of $T$. 
This fact may be exploited to construct
a map $i\colon X\to |\calK(X)|$ such that $S\circ i\colon X\to X$ is the identity of $X$. 
Being a right homotopy inverse of $S$, the map $i$ is itself a homotopy equivalence.
Since $|\calA(X)|$ is aspherical, Proposition~\ref{riassunto:prop} implies
that the composition
$$
\xymatrix{
X\ar[r]^-{i} & |\calK(X)|\ar[r]^-{r} & |\calL(X)|\ar[r]^-{\pi} & |\calA(X)|
}
$$
is a classifying map for $X$.

\part{Multicomplexes, bounded cohomology and simplicial volume}

\chapter{Bounded cohomology of multicomplexes}\label{chap2:chap}
In this chapter we deepen our study of the simplicial (bounded) cohomology of multicomplexes. 
In particular, we state and prove the so-called \textit{Isometry Lemma} \ref{isometry-lemma-intro}. 
This result establishes an isometric isomorphism between the singular and the simplicial bounded cohomology of a complete multicomplex. 
By applying the Isometry Lemma to the singular multicomplex $\calK(X)$ and to the aspherical multicomplex $\calA(X)$ associated
to a good topological space $X$, we will then be able to show that the bounded cohomology of a topological space only depends on its fundamental
group. 

In particular, the bounded cohomology of a simply connected space vanishes. This result plays a fundamental role in the theory of bounded cohomology of topological spaces.
It was stated without any assumption on the space involved in~\cite[Section 3.1]{Grom82}, where Gromov deduced it from an argument based on the theory of multicomplexes. 
A completely different proof of this result was given by Ivanov (first for countable CW complexes~\cite{Ivanov}, then for any topological space~\cite{ivanov3}). 

Our proofs come back to Gromov's original approach. Following (but modifying, sometimes in a substantial way) some ideas introduced by Gromov, we will introduce and study peculiar groups of automorphisms of complete and minimal multicomplexes, 
paying a particular attention to their amenability. 
By making use of the invariance of bounded cohomology with respect to weak homotopy equivalences, we will prove that
 the bounded cohomology of any topological space only depends on its fundamental
group. As mentioned in the introduction, the techniques developed here can be easily exploited to provide a proof of this fact that works for any \emph{good} topological space and does not use Ivanov's result on weak homotopy equivalences (see~\cite{Marco:tesi}, where such a proof is written down in detail).

\section{$\ell^1$-homology and bounded cohomology of topological spaces}\label{bounded:singular:sec}
Let $X$ be a topological space, and let $R=\matZ,\R$. Recall that $C_* (X;R)$ (resp. $C^*(X;R)$) denotes the usual
complex of singular chains (resp.~cochains) on $X$ with coefficients in $R$, and
$S_i (X)$ is the set of singular $i$--simplices
in $X$. We also
regard $S_i (X)$ as a subset of $C_i (X;R)$, so that
for any cochain $\varphi\in C^i (X;R)$ it makes sense to consider its restriction
$\varphi|_{S_i (X)}$. For every $\varphi\in C^i (X;R)$, we set
$$
\|\varphi\|_\infty  = \sup \left\{|\varphi (s)|\ |\ s\in S_i (X)\right\}\in [0,\infty].
$$
We denote by $C^*_b (X;R)$ the submodule of bounded cochains, i.e.~we set
$$C^*_b (X;R)= \left\{\varphi\in C^* (X;R)\ | \ \|\varphi\|<\infty\right\}\ .$$ Since
the differential takes bounded cochains to bounded cochains, $C^*_b (X;R)$
is a subcomplex of $C^*(X;R)$.
We denote by $H^*(X;R)$ (resp.~$H_b^*(X;R)$) 
the homology of the complex $C^*(X;R)$ (resp.~$C_b^*(X;R)$).
Of course, $H^*(X;R)$ is the usual singular cohomology module of $X$ with coefficients in $R$, while $H_b^*(X;R)$
is the \emph{bounded cohomology module} of $X$ with coefficients in $R$. 

The norm on $C^i (X;R)$ descends (after taking
the suitable restrictions) to a seminorm on each of the modules 
$H^*(X;R)$, $H_b^*(X;R)$. More precisely, if
$\varphi\in H$ is a class in one of these modules, which is obtained as a quotient
of the corresponding module of cocycles $Z$, then we set 
$$
\|\varphi\|_\infty=\inf \left\{\|\psi\|\ |\  \psi\in Z,\, [\psi]=\varphi\ {\rm in}\ H\right\}.
$$
This seminorm may take the value $\infty$ on elements in $H^*(X;R)$
and may be null on non-zero elements in $H_b^*(X;R)$.

The inclusion of bounded cochains into possibly unbounded cochains
induces the \emph{comparison map}
$$
c^* \colon  H_b^*(X;R)\to H^*(X;R)\ .
$$

The $\ell^\infty$-norm on singular cochains arises as the dual norm of an $\ell^1$-norm on chains. 
In fact, for every $n\geq 0$ one can put on the space $C_n(X;R)$ the $\ell^1$-norm
$$
\left\| \sum_{i \in I} a_i s_i\right\|_1 =\sum |a_i|
$$
(here $I$ is a finite set, $s_i\in S_n(X)$ for every $i\in I$, and $s_i\neq s_j$ if $i\neq j$). 
This norm descends to a seminorm $\|\cdot \|_1$ on $H_n(X;R)$. While $C^n(X;R)$ is the algebraic dual of $C_n(X;R)$, 
the module of bounded cochains $C^n_b(X;R)$ coincides with the \emph{topological} dual of $C_n(X;R)$, i.e.~with the set
of functionals on $C_n(X;R)$ that are continuous with respect to the $\ell^1$-norm just described. Moreover, the $\ell^\infty$-norm of an element
$\varphi\in C^n_b(X;R)$ coincides with the dual norm of $\varphi$ as an element of the topological dual of $C_n(X;R)$.

Henceforth, unless otherwise stated, all the (bounded) (co)homology modules will be understood with real coefficients. Therefore, we will omit the coefficients from our notation.

The duality pairing between $C^n_b(X)$ and $C_n(X)$ induces the \emph{Kronecker product}
$$
\langle \cdot,\cdot \rangle\colon H_b^n(X)\times H_n(X)\to \R\ .
$$
An easy argument based on Hahn-Banach Theorem implies the following result (see e.g.~\cite[Lemma 6.1]{frigerio-libro}):

\begin{lemma}\label{lemma:duality}
 Let $n\in\mathbb{N}$ and take  $\alpha\in H_n(X)$. Then 
 $$
\|\alpha\|_1=\max \{ \langle\beta,\alpha\rangle\, |\,  \beta\in H_b^n(X),\,  \|\beta\|_\infty \leq 1\}\ .
$$
\end{lemma}

\begin{cor}\label{vanish-cor}
 Suppose that $H^n_b(X)=0$. Then
 $$
 \|\alpha\|_1=0
 $$
 for every $\alpha\in H_n(X)$.
\end{cor}

Any continuous map $f\colon X\to Y$ induces norm non-increasing maps 
$$H^n_b(f)\colon H^n_b(Y)\to H^n_b(X), \, \, \, H_n(f)\colon H_n(X)\to H_n(Y).$$
Moreover, homotopic maps induce the same map both on singular homology and on bounded cohomology.
Using these facts, it is not difficult to prove that
a homotopy equivalence induces isometric isomorphisms in every degree both on singular homology (endowed with the $\ell^1$-norm) and on bounded cohomology.
Indeed, a recent result by Ivanov shows that even \emph{weak} homotopy equivalences induce isometric isomorphisms on bounded cohomology:

\begin{thm}[{\cite[Corollary 6.4]{ivanov3}}]\label{weak:iso:thm}
Let $f\colon X\to Y$ be a weak homotopy equivalence. Then the map
$$
H^n_b(f)\colon H^n_b(Y)\to H^n_b(X)
$$
is an isometric isomorphism for every $n\in\mathbb{N}$.
\end{thm}

As an immediate consequence of Theorems~\ref{homotopy-weak-intro}  and~\ref{weak:iso:thm} we have the following:

\begin{thm}\label{iso-KX}
 Let $X$ be a good space and let $S\colon |\calK(X)|\to X$ be the natural projection. Then the induced map 
 $$
 H^n_b(S)\colon H^n_b(X)\to H^n_b(|\calK(X)|)
 $$
 is an isometric isomorphism for every $n\in\mathbb{N}$.
\end{thm}

\section{Bounded cohomology of complete multicomplexes}\label{isometry:lemma:sec}
We have seen in Theorem~\ref{simpl-hom-eq-sing-one} that the simplicial (co)homology of a  multicomplex is canonically isomorphic to
the singular (co)homology of its geometric realization. We have also observed that this result cannot be true in general for bounded cohomology, 
since the fact that a singular coclass admits a bounded representative cannot be easily translated into an elementary property of the corresponding simplicial coclass.
Nevertheless, things get much better for complete multicomplexes: the main goal of this section is the proof of Gromov's \emph{Isometry Lemma} \ref{isometry-lemma-intro}, which states that 
the bounded cohomology (with real coefficients) of a large complete multicomplex is isometrically isomorphic to the bounded cohomology of its geometric realization.

Let $K$ be a complete multicomplex, recall that we have a natural chain map
\begin{displaymath}
\phi_* \colon C_{*}(K) \rightarrow C_{*}(\lvert K \rvert)
\end{displaymath}
sending every element $(\sigma,(v_0,\ldots,v_n))\in C_*(K)$ to the singular simplex 
$$
\Delta^n\to |K|\, ,\quad (t_0,\ldots,t_n)\mapsto (\sigma,t_0v_0+\ldots+ t_nv_n)\ ,
$$
and denote by $\phi^*\colon C_b^*(K)\to C_b^*(|K|)$ the induced dual chain map from simplicial bounded cochains to singular bounded cochains.

\begin{thm:repeated:isometry-lemma}[Isometry Lemma]
Let $K$ be a large and complete multicomplex. Then, the map
\begin{displaymath}
H^n_b(\phi^*) \colon H_{b}^{n}(\lvert K \rvert) \rightarrow H_{b}^n(K)
\end{displaymath}
is an isometric isomorphism for every $n\in\mathbb{N}$.
\end{thm:repeated:isometry-lemma}

\begin{proof}
We may suppose that $K$ is path connected.
Since the map $H^n_b(\phi^n)$ is obviously norm non-increasing, it is sufficient to construct a norm non-increasing chain map
$\psi^*\colon C^{*}_b(K)\to C^{*}_b(|K| )$ that induces the inverse of $H^n_b(\phi^*)$ on bounded cohomology. In fact, 
we will be able to define such a map only on $C^{n\leq N}_*(K)$, where $N$ is any arbitrary natural number. However, by choosing $N>n$, this certainly suffices to
deduce that $H^n_b(\phi^n)$ is indeed an isometric isomorphism for every $n\in\mathbb{N}$.

There is no straightforward formula for $\psi_*$, since a singular simplex with values in $|K|$ does not determine any 
simplex of $K$. In fact,
the first part of the proof will be devoted to the construction of a map which associates to a singular simplex in $|K|$ a simplicial chain in $C_*(K)$.

In order to simplify notation, for every $n\in\mathbb{N}$ we denote by $\Delta^n$ (rather than by $|\Delta^n|$) the geometric realization of the standard simplex.
Let us fix an integer $N\gg 0$ and, for every $n\leq N$,
let $M_n=\frac{(N+1)!}{(N - n)!}$. We denote by
$\sigma_{\mu}^{n} \colon \Delta^{n}  \rightarrow \Delta^{N}$, $\mu=1,\ldots,M_n$,  all possible linear isomorphisms of $\Delta^{n}$ onto an $n$-face of $\Delta^{N}$.
We then define the submulticomplex
$$
\tilde{K}_N\ \subseteq\ \calK(|K|\times \Delta^N) 
$$
as follows: the class of a singular simplex $\rho\colon \Delta^n\to |K|\times \Delta^N$ belongs to $\tilde{K}_N$ if and only if $\pi\circ\rho=\sigma_\mu^n$ for some
$\mu=1,\ldots,M_n$, where $\pi\colon |K|\times \Delta^N\to \Delta^N$ is the projection on the second factor. Observe that, if $\pi\circ\rho=\sigma_\mu^n$, then $\rho$ is automatically injective on vertices,
so it lies indeed in $\calK(|K|\times \Delta^N)$.



We will construct the required (partial) chain map 
$\psi^*\colon C^{*}_b(K)\to C^{*}_b(|K| )$ as
the composition of two chain maps
\begin{displaymath}
\theta^*\colon C_b^{n\leq N} (K)\to C_{b}^{n \leq N} (\tilde{K}_{N})\, ,\quad 
A^* \colon C_{b}^{n \leq N} (\tilde{K}_{N}) \rightarrow C_{b}^{n \leq N}(\lvert  K \rvert)\ .
\end{displaymath} 
We first describe the map $A^*$. 
To every singular simplex $\tau \colon \Delta^{n} \rightarrow \lvert K \rvert$, one can associate the average of the simplices $\Delta_{\mu} \subset \tilde{K}_{N}$, $\mu = 1, \cdots , M_{n} = \frac{(N+1)!}{(N-n)!}$, 
where $\Delta_\mu$ is such that the following diagram commutes
\begin{displaymath}
\xymatrix{
& \lvert K \rvert \times \Delta^{N}  \ar[d]^{\tilde{\pi}} \ar[dr]^{\pi} \\
\Delta^{n} \ar[ru]^{\Delta_{\mu}} \ar[r]^{\tau} \ar@/_1pc/[rr]_{\sigma_{\mu}^{n}} & \lvert K \rvert & \Delta^{N} 
}
\end{displaymath}
(here $\tilde{\pi}$ denotes the projection on the first factor).
By dualizing this definition, one then obtains the map
 $A^*$.
More formally, we define $A_*\colon C_*(|K|)\to C_*(\tilde{K}_N)$ to be the $\R$-linear chain map such that
$$
A_n(\tau)=\frac{1}{M_n}\sum_{\mu=1}^{M_n} \left(\overline{\tau\times\sigma^n_\mu},((\tau(e_0),\sigma^n_\mu(e_0)),\ldots,(\tau(e_n),\sigma_\mu^n(e_n))\right)\ ,
$$
where $\overline{\tau\times\sigma^n_\mu}$ denotes the equivalence class of $\tau\times\sigma^n_\mu\colon \Delta^n\to |K|\times \Delta^N$. We then
define $A^n$ as the dual map of $A_n$. Observe that $A_n$ does not increase the $\ell^1$-norm of chains, so $A^n$ does not increase the $\ell^\infty$-norm of cochains.

Let us now turn to the construction of $\theta^*$. The natural projection $|\calK(|K|\times \Delta^N)|\to |K|\times \Delta^N$ restricts to a map
$\lvert \tilde{K}_{N} \rvert \rightarrow \lvert K \rvert\times \Delta^N$. By taking the composition of this map with the projection onto $|K|$, we thus have a map
$$
f\colon |\tilde{K}_N|\to |K|\ .
$$

Observe that $f$ is very far from being simplicial, so it does not induce any map on simplicial (co)homology. 
However, by making use of the completeness of $K$ (and of the peculiar definition of $\tilde{K}_N$)
we will now show that $f$ is homotopic to a non-degenerate simplicial map 
$\tilde{f} \colon  \tilde{K}_{N}  \rightarrow K $. To this aim we exploit Proposition~\ref{approx-simpl}, which applies to maps
that are injective on the set of vertices of any simplex of the domain. 

Every vertex of $\tilde{K}_{N}$ can be described as a pair $(x, i)$, where $x \in \, \lvert K \rvert$ and $i$ is a vertex of $\Delta^{N}$. It readily follows that two vertices of $\tilde{K}_{N}$ are 
joined by a $1$-simplex if and only if their second components are distinct. We then define an equivalence relation on the set of vertices of $\tilde{K}_{N}$ as follows: 
\begin{displaymath}
(x, i) \sim (y, j) \mbox{ if and only if } i = j\ .
\end{displaymath}
 Thus, $(x, i)$ is equivalent to $(y, i)$ if and only if there is no $1$-simplex joining $(x, i)$ and $(y, j)$. 

We can now define $\tilde{f}$ as follows.
Since $|K|$ is path connected and has infinitely many vertices, thanks to the homotopy extension property for CW complexes we can homotope 
$f$ to a map $\bar{f}\colon |\tilde{K}_N|\to |K|$ 
 sending all the vertices which lie in the same equivalence class to the same 
vertex of $K$, with the 
additional requirement that two distinct equivalence classes are sent to different vertices of $K$. By construction, if $B\subseteq V(\tilde{K}_N)$ is the set of vertices of some simplex
of $\tilde{K}_N$, then $\bar{f}|_B$ is injective. Therefore,
Proposition~\ref{approx-simpl} allows us to homotope $\bar{f}$ to a non-degenerate simplicial map 
\begin{displaymath}
\tilde{f} \colon  |\tilde{K}_{N}|  \rightarrow  |K| 
\end{displaymath}
without subdividing the domain. We then define 
$$
\theta^*\colon C_b^{n\leq N} (K) \to C_{b}^{n \leq N} (\tilde{K}_{N})
$$
as the map induced by $\tilde{f}$ on bounded cochains. Observe that $\theta^*$ is norm non-increasing in every degree, so the composition
$$
\psi^*=A^*\circ \theta^*\colon  C_b^{n\leq N} (K) \to  C_b^{n\leq N} (|K|)
$$
is norm non-increasing. In order to conclude we are now left to show that both compositions
$\phi^*\circ \psi^*$ and $\psi^*\circ\phi^*$ induce the identity on bounded cohomology.

Let us first analyze the map $H^n_b(\psi^*) \circ H^n_b(\phi^*)$, which, by construction, is given by the composition
\begin{equation}\label{comp:1}
\xymatrix{
H^n_b(|K|)\ar[rr]^{H^n_b(\phi^*)} & &    H^n_b(K)\ar[rr]^{H^n_b(\theta^*)} & & H^n_b(\tilde{K}_N)\ar[rr]^{H^n_b(A^*)} & & H^n_b(|K|)\ .
}
\end{equation}
Let us denote by $\Phi^*\colon C^*_b(|\tilde{K}_N|)\to C^*_b(\tilde{K}_N)$ the natural map from singular
cochains to simplicial cochains, which is constructed for $\tilde{K}_N$ in the very same way as it is for $K$. 

Since $\theta^*$ is induced by the (simplicial) map $\tilde{f}\colon \tilde{K}_N\to K$,
the diagram
$$
\xymatrix{
H^n_b(K)\ar[rr]^-{H^n_b(\theta^*)} & & H^n_b(\tilde{K}_N)\\
H^n_b(|K|)\ar[rr]_-{H^n_b(\tilde{f}^*)} \ar[u]^{H^n_b(\phi^*)} & & H^n_b(|\tilde{K}_N|)\ar[u]_{H^n_b(\Phi^*)}
}
$$
commutes. Now homotopic maps induce the same morphism on singular bounded cohomology, so also the diagram
$$
\xymatrix{
H^n_b(K)\ar[rr]^-{H^n_b(\theta^*)} & & H^n_b(\tilde{K}_N)\\
H^n_b(|K|)\ar[rr]_-{H^n_b({f}^*)} \ar[u]^{H^n_b(\phi^*)} & & H^n_b(|\tilde{K}_N|)\ar[u]_{H^n_b(\Phi^*)}
}
$$
commutes, and the composition described in~\eqref{comp:1} coincides with the composition
$$
\xymatrix{
H^n_b(|K|)\ar[rr]^{H^n_b({f}^*)} & &  H^n_b(|\tilde{K}_N|)\ar[rr]^{H^n_b(\Phi^*)} & &H^n_b(\tilde{K}_N)\ar[rr]^{H^n_b(A^*)} & & H^n_b(|K|)\ .
}
$$
Now it is immediate to check that the composition $f_*\circ \Phi_*\circ A_*\colon C_n(|K|)\to C_n(|K|)$ is equal to the identity,
and this concludes the proof that $H^n_b(\psi^*) \circ H^n_b(\phi^*)$ is the identity of $H^n_b(|K|)$.

Let us now deal with the map $H^n_b(\phi^*) \circ H^n_b(\psi^*)$, which is obtained by taking the map induced in bounded cohomology by the dual of the composition
$$
 \xymatrix{
 C_*(K)\ar[r]^{\phi_*} & C_*(|K|) \ar[r]^{A_*} & C_*(\tilde{K}_N) \ar[r]^{\theta_*} & C_*(K)\ .
 }
$$

We first define a submulticomplex $Z$ of $\tilde{K}_N$ as follows: a simplex $\overline{\tau\times \sigma^n_\mu}$ of $\tilde{K}_N$ belongs
to $Z$ if and only if the singular simplex $\tau\colon \Delta^i\to |K|$ is a characteristic map for a simplex of $K$, i.e.~if and only if
it is of the form $\phi_*((\eta,(v_0,\ldots,v_i))$ for some simplex $\eta$ of $K$ and some ordering $(v_0,\ldots,v_i)$ of the vertices of $\eta$. 

By definition, the composition $A_*\circ \phi_*\colon C_*(K)\to C_*(\tilde{K}_N)$ takes values into $Z$, so 
the map $H^n_b(\phi^*) \circ H^n_b(\psi^*)$ is induced by the composition
\begin{equation}
 \xymatrix{
 C_*(K)\ar[r]^{\phi_*} & C_*(|K|) \ar[r]^{A_*} & C_*(Z) \ar[r]^{\theta_*} & C_*(K) 
 }
\end{equation}
as well (where we denote by the same symbol both $\theta_*$ and its restriction to $C_*(Z)$). Observe now that the map $f\colon |\tilde{K}_N|\to |K|$ restricts to a non-degenerate simplicial map
on $|Z|$: indeed, if $\overline{\tau\times \sigma^n_\mu}$ is such that $\tau=\phi_*((\eta,(v_0,\ldots,v_i)))$, then  $f(\overline{\tau\times \sigma^n_\mu})=\eta$. 
Recall now that the (geometric realizations of the) maps $f$ and $\widetilde{f}$ are homotopic, so also their restrictions to $|Z|$ are. Since $K$ is complete, the maps $f,\widetilde{f}\colon |Z|\to |K|$ 
are \emph{simplicially} homotopic (see Lemma~\ref{homotopy-lemma}),
hence the maps
$$
H^n_b(f^*)\colon H^n_b(K)\to H^n_b(Z)\, ,\qquad
H^n_b(\theta^*)\colon H^n_b(K)\to H^n_b(Z)
$$
coincide (see Lemma~\ref{homotopy-homology}). As a consequence, the map $H^n_b(\phi^*) \circ H^n_b(\psi^*)$ is also induced by the composition 
$$
 \xymatrix{
 C_*(K)\ar[r]^{\phi_*} & C_*(|K|) \ar[r]^{A_*} & C_*(Z) \ar[r]^{f_*} & C_*(K)\ .
 }
$$
It is immediate to check that this composition is the identity, and this concludes the proof. 
\end{proof}


The following result  provides an important application of Theorem~\ref{isometry-lemma-intro}.
It implies that one can compute the bounded cohomology of a good topological space just by considering simplices with distinct vertices (in fact,  the 
 cochain complex $C^*_b(\calK(X))$ of simplicial bounded cochains on $\calK(X)$ may be isometrically identified with the complex 
obtained by dualizing the complex of singular chains involving only simplices with distinct vertices).
This apparently innocuous result 
is a fundamental step towards Gromov's proof of the Mapping Theorem (see also Remark~\ref{controesempioa:rem} for a discussion of possible counterexamples to this result in the context of non-good topological spaces).


\begin{Teorema}\label{isometria-good-K(X)}
Let $X$ be a good space. Then the 
composition
$$
\xymatrix{
H^n_b(X)\ar[rr]^{H^n_b(S)} & & H^n_b(|\calK(X)|)\ar[rr]^{H^n_b(\phi^n)} & & H^n_b(\calK(X))
}
$$
is an isometric isomorphism for every $n\in\mathbb{N}$.
\end{Teorema}
\begin{proof}
 The conclusion follows from Theorems~\ref{iso-KX} and~\ref{isometry-lemma-intro}.
\end{proof}

\begin{rem}
Theorem~\ref{isometria-good-K(X)} could be obtained even without exploiting the invariance of bounded cohomology
with respect to weak homotopy equivalences. In fact,  the proof of the Isometry Lemma can be easily adapted
to show directly that the bounded cohomology of $X$ is isometrically isomorphic to the simplicial bounded cohomology
of $\calK(X)$ via the natural map described above. We refer the reader to~\cite{Marco:tesi} for further details on this issue.
\end{rem}

\begin{rem}\label{controesempioa:rem}
 It is worth mentioning that we do not expect Theorem~\ref{isometria-good-K(X)} to hold for \emph{any} topological space.
 If this were the case, then we would be able to prove the following striking result: If $G$ is a finitely presented group, then
 the bounded cohomological dimension of $G$ (with real coefficients) is finite. (To our purposes, we can define the bounded cohomology of $G$ as the singular
 bounded cohomology of any aspherical CW complex with fundamental group isomorphic to $G$. The fact that this definition is well posed readily descends from the fact that homotopy equivalences
 induce isometric isomorphisms on bounded cohomology in every degree.)
 
 In fact, if $G$ is a finitely presented group, then $G\cong \pi_1(Z)$ for some 
  finite topological space $Z$ (indeed, it is easy to construct  a finite simplicial complex with fundamental group
  isomorphic to $G$, and every finite simplicial complex is weakly homotopy equivalent to a finite topological space~\cite{mccord}). 
Moreover, by Theorem~\ref{mapping1_intro} proved
below we have $H^n_b(G)\cong H^n_b(Z)$ for every $n\in\mathbb{N}$. However, if $m$ is the number of points of $Z$,
then $\dim \calK(Z)=m-1$. Since the simplicial bounded cohomology of $\calK(Z)$ may be computed via alternating cochains, this implies 
that $H^{n}_b(\calK(Z))=0$ for every $n\geq m$. Therefore, if $H^{n}_b(\calK(Z))$ were isomorphic to $H^n_b(Z)\cong H^n_b(G)$ for every $n\geq m$, then we could conclude
that the bounded cohomological dimension of $G$ is finite.  
  
  If $X$ is a non-good topological space for which Theorem~\ref{isometria-good-K(X)} fails, then either $H^n_b(X)$ is not isometrically isomorphic
  to $H^n_b(|\calK(X)|)$, or $H^n_b(|\calK(X)|)$ is not isometrically isomorphic to $H^n_b(\calK(X))$ (or both). By Theorem~\ref{iso-KX}, in the first case 
  the natural projection $S\colon |\calK(X)|\to X$ would not be a weak homotopy equivalence. Observe however that the hypothesis that $X$ is good was used
  both in the proof that (for good spaces) the natural projection is a weak homotopy equivalence, and in the proof that (for good spaces)
  the simplicial and the singular bounded cohomology of $\calK(X)$ are isometrically isomorphic.
\end{rem}

With a slight abuse, let us denote by $S\colon |\calL(X)|\to X$ the restriction of the natural projection $S\colon |\calK(X)|\to X$ to $|\calL(X)|$, and by
$\phi^*\colon C^*_b(|\calL(X)|)\to C^*_b(\calL(X))$ the usual restriction of singular cochains to simplicial cochains. 

\begin{Teorema}\label{isometria-good-L(X)}
Let $X$ be a good space. Then the 
composition
$$
\xymatrix{
H^n_b(X)\ar[rr]^{H^n_b(S)} & & H^n_b(|\calL(X)|)\ar[rr]^{H^n_b(\phi^*)} & & H^n_b(\calL(X))
}
$$
is an isometric isomorphism for every $n\in\mathbb{N}$.
\end{Teorema}
\begin{proof}
 The conclusion follows from Theorem~\ref{iso-KX} and Corollary~\ref{cor-assoc-min-e-compl-to-X}.
\end{proof}

Let us mention another corollary of Theorem~\ref{isometry-lemma-intro}:

\begin{cor}\label{cor:equiv:hom:induces:iso:simpl:bc}
Let $f \colon L \rightarrow K$ be a simplicial map between complete and large multicomplexes. Suppose that the geometric realization of $f$ is a homotopy equivalence. 
Then $$H^*_b(f) \colon H^n_b(K) \rightarrow H^n_b(L)$$ is an isometric isomorphism.
\end{cor}
\begin{proof}
It suffices to consider the following commutative diagram
$$
\xymatrix{
H^*_b(|K|) \ar[r]^{H^*_b(|f|)} \ar[d] & H^n_b(|L|) \ar[d] \\
H^*_b(K) \ar[r]_{H^*_b(f)} & H^n_b(L)
}
$$
and observe that the vertical arrows are isometric isomorphisms by Theorem \ref{isometry-lemma-intro}, while the top horizontal arrow
is an isometric isomorphism because $|f|$ is a homotopy equivalence.
\end{proof}

When $f\colon L\to K$ and its homotopy inverse $g\colon K\to L$ are both non-degenerate, then the previous corollary directly follows from the Homotopy Lemma~\ref{homotopy-lemma},
which ensures that $f\circ g $ and $g\circ f$ are \emph{simplicially} homotopic to the identity of $L$ and of $K$, respectively.

\section{Amenable groups of simplicial automorphisms}
We have seen in Sections~\ref{minimal:section} and~\ref{aspherical:sec} that to every good space $X$ there are associated a complete multicomplex $\calK(X)$, a minimal complete multicomplex $\calL(X)$ and
a minimal, complete and aspherical multicomplex $\calA(X)$. Moreover, $\calL(X)$ is a retract (hence, a quotient) of $\calK(X)$, and $\calA(X)$ is a quotient of $\calL(X)$.
In this section we will prove that $\calA(X)$ is in fact obtained from $\calL(X)$ by taking the quotient with respect to an interesting simplicial action on $\calL(X)$.
As far as the induced action on cochains of bounded degree is concerned, this action turns out to be equivalent to the action by an amenable group. Since amenable groups
are invisible to bounded cohomology, this fact will be the key ingredient  to prove Gromov's Mapping Theorem, which (in a simplified version) asserts that 
any continuous map which induces an isomorphism on fundamental groups also induces an isometry on singular bounded cohomology (see Theorem~\ref{mapping1_intro} proved below). Even if our strategy follows~\cite[Section 3.3]{Grom82}, our definitions sometimes differ from Gromov's  (see e.g.~Remark~\ref{rem:ip:agg:gruppi:gromov}), and many of our proofs
are not even sketched in Gromov's paper.

As usual, we will not use distinct symbols for a simplicial map between multicomplexes and the induced map on geometric realizations. By \emph{simplicial action} of a group on a multicomplex $K$
we will understand a representation of a group into the group of simplicial automorphisms of $K$ (or, equivalently, into the group of 
homeomorphisms of $|K|$ induced by simplicial maps). We are now ready to introduce some groups that will play an important role in the sequel.

\begin{Definizione}\label{definizioneGamma}
Let $K$ be a multicomplex. Then, we define $\Gamma \coloneqq \Gamma(K)$ to be the group of all simplicial automorphisms that are homotopic to the identity relative to the $0$-skeleton. 

Moreover, for every $i \geq 1$ we define the subgroup
$$
\Gamma_{i}=\{g\in \Gamma\, |\,  g|_{K^i}={\rm Id}_{K^i}\}\ .
$$
\end{Definizione}

\begin{rem}\label{rem:ip:agg:gruppi:gromov}
Gromov's definition of the groups $\G$ and $\G_i$ is slightly different from ours: in~\cite[page 59]{Grom82}, elements of $\G$ (hence of the $\G_i$) are not required to be homotopic
to the identity relative to $K^0$ (while they are still required to be homotopic to the identity). However, without our further assumptions we are not able to prove Lemma~\ref{phi:hom}, hence Corollary~\ref{corC}, which 
states that the quotient $\G_1/\G_i$ is amenable for every $i\in\mathbb{N}$.
Indeed, Lemma~\ref{phi:hom} is false even if elements of $\Gamma$ are required to fix $K^0$ and to be homotopic to the identity, possibly without being homotopic to the identity \emph{relative to $K^0$}
(see Remark~\ref{nohomorem} and Lemma~\ref{homaction}).

The  amenability of  $\G_1/\G_i$ is a key ingredient in the proof of
Theorem~\ref{teoD}, on which Gromov's Mapping Theorem~\ref{mapping-intro} is based.
\end{rem}

Of course, for every $i\geq 1$ the subgroup $\Gamma_i$ is normal in $\Gamma$. Recall that, if $\Delta_0$ is an $(i+1)$-simplex of $K$, then
$\pi(\Delta_{0})$ the set of all  $(i+1)$-simplices compatible with $\Delta_{0}$. The following important result shows that the groups $\Gamma_i$ act as transitively as possible
on $(i+1)$-simplices.

\begin{Lemma}\label{lemmaB}
Let $K$ be a complete and minimal multicomplex. Let $i\geq 1$, let us fix an arbitrary $(i+1)$-simplex $\Delta_0$, and let $\Delta\in \pi(\Delta_{0})$. Let 
$F$ be a facet of $\Delta_0$ (i.e.~a codimension-1 face of $\Delta_0$). Then, there exists an element $g\in \Gamma_i$ such that the following conditions hold:
$g(\Delta_0)=\Delta$, and $g(\Delta')=\Delta'$ for every $m$-simplex $\Delta'$, $m>i$, which does not contain $F$. 
\end{Lemma}
\begin{proof}
Let $f_{0} \colon \Delta_{0} \rightarrow \Delta$ be a linear isomorphism which keeps fixed the vertices of $\Delta_{0}$
(hence,
the whole boundary of $\Delta_{0}$). Then, keeping fixed the $i$-skeleton of $K$, we can extend $f_{0}$ to a map $\tilde{f} \colon K^{i} \cup \Delta_{0} \rightarrow K$ which sends $K^{i} \cup \Delta_{0}$ isomorphically to 
$K^{i} \cup \Delta$. Let  now $K'$ be the submulticomplex of $K^i$ obtained by removing from $K^i$ the (internal part of the) $i$-simplex $F$. We first prove that
$\tilde{f}$ is homotopic to the inclusion map $K^{i} \cup \Delta_{0} \hookrightarrow K$ relative to $K'$.

Let $v_0$ be the vertex of $\Delta_0$ opposite to $F$, and denote by $\Lambda^{i+1}_{0}$ the simplicial horn obtained by removing  the internal parts
of $F$ and of $\Delta_0$ from $\Delta_0$. 
Then $|\Lambda^{i+1}_{0}|$ is a strong deformation retract of $|\Delta_0|$.
We denote by $H^{\Delta_0}\colon |\Delta_0|\times [0,1]\to |\Delta_0|$ the homotopy realizing the strong deformation retraction.
Observe now that, being $\Delta$ and $\Delta_0$ compatible, $|\Lambda^{i+1}_{0}|$ is contained in $|\Delta|$ too. Moreover,
 arguing in the same way as for $\Delta_0$,
we can construct a homotopy (relative to $|\Lambda^{i+1}_{0}|$)
$H^\Delta\colon |\Delta|\times [0,1]\to |\Delta|$ between the identity and a retraction of $|\Delta|$ onto $|\Lambda^{i+1}_{0}|$.  
We can define the desired homotopy $h$ between $\tilde{f}$ and the inclusion to be the following one:
\begin{displaymath}
h \colon \lvert K^{i} \cup \Delta_{0} \rvert \times I \rightarrow \lvert K \rvert
\end{displaymath}
\begin{displaymath}
h(x, t) = \begin{cases} x &\mbox{if } x \in \, \lvert K^{i} \setminus \Delta_{0} \rvert\ , \\
H^{\Delta_{0}}(x, 2t) & \mbox{if } t \in \left[0, \frac{1}{2} \right] \mbox{ and } x \in \, \Delta_{0}\ , \\
H^{\Delta}(x, 2 - 2t) & \mbox{if } t \in \left[\frac{1}{2}, 1 \right] \mbox{ and } x \in \, \Delta_{0}\ . \end{cases}
\end{displaymath}
This proves that
$\tilde{f}$ is homotopic to the inclusion map $K^{i} \cup \Delta_{0} \hookrightarrow K$ relative to $K'$.

Observe now that the 
inclusion $i \colon \lvert K^{i} \cup \Delta_{0} \rvert \hookrightarrow \lvert K \rvert$ extends to the identity map $\textrm{Id}_{\lvert K \rvert}$. Then, by the Homotopy Extension Property for CW-pairs \cite[Proposition~0.16]{hatcher}, 
applied to the pair $(\lvert K \rvert, \lvert K^{i} \cup \Delta_{0} \rvert)$, we know that there exists
a homotopy $H \colon |K| \times I \rightarrow |K|$ extending $h$. Let $\tilde{f}_1=H(\cdot,1)\colon |K|\to |K|$, and let $K''$ be
the submulticomplex of $K$ obtained by removing from $K$ the $i$-simplex $F$ and all the simplices containing $F$.
By looking at the explicit 
description of $H$ given in  \cite[Proposition~0.16]{hatcher}, one realizes that $H$ may be assumed to be constant
on $K''$. 
In particular, $\tilde{f}_1$ is equal to the identity  on $K''$, hence it is simplicial on $K''\cup \Delta_0$. Since $K$ is complete, we can now invoke
Proposition~\ref{approx-simpl} to homotope $\tilde{f}_1$ (relative to $K''\cup \Delta_0$) to a non-degenerate simplicial map $g\colon K\to K$ which coincides with $\tilde{f}_1$ on $K''\cup\Delta_0$.
By construction, the map $g$ is homotopic to the identity relative to $K^0$, sends $\Delta_0$ to $\Delta_1$, and restricts to the identity on every $m$-simplex of $K$ 
not containing $F$, $m>i$.

In order to conclude the proof we are left to show that $g$ is a simplicial automorphism. 
By construction, $g$ restricts to a bijection of the $0$-skeleton of $K$.
Moreover, being homotopic to the identity, it is a homotopy equivalence. Then the conclusion follows from Proposition~\ref{uniq:prop}.
  \end{proof}

\begin{cor}\label{transitive-on-compatible}
Let $K$ be a complete and minimal multicomplex.
Then, two $(i+1)$-simplices of $K$ are compatible if and only if they lie in the same $\Gamma_i$-orbit. In other words, if $\Delta_0$ is any $(i+1)$-simplex, then
$\Gamma_i$ acts transitively on $\pi(\Delta_0)$.
\end{cor}
\begin{proof}
By definition of $\Gamma_i$, two $(i+1)$-simplices lying in the same $\Gamma_i$-orbit are necessarily compatible, so the 
 statement is an obvious consequence of Lemma~\ref{lemmaB}.
\end{proof}

In fact, we can strengthen the previous corollary as follows:

\begin{Proposizione}\label{action-1-skeleton}
Take integers $i,n\geq 1$ such that $n\geq i+1$,
let $K$ be a complete and minimal multicomplex, and let $\Delta,\Delta'$ be two $n$-simplices of $K$. Then
$\Delta$ and $\Delta'$ lie in the same $\Gamma_i$-orbit if and only if they share the same $i$-skeleton. 
\end{Proposizione}
\begin{proof}
Let $i\geq 1$ be fixed.
It is clear that if two simplices are $\Gamma_{i}$-equivalent, then they share the same $i$-skeleton, so we need to prove that $n$-simplices with the same $i$-skeleton 
lie in the same $\Gamma_i$-orbit.

Let $i\geq 1$ be fixed. We prove by induction on $n\geq i+1$ the following slightly stronger claim: if $\Delta,\Delta'$ are $n$-simplices sharing the same $i$-skeleton, and $F$ is a fixed facet of $\Delta$,
then there exists an element $g\in\Gamma_i$ such that the following conditions hold: 
\begin{enumerate}
\item
$g(\Delta)=\Delta'$;
\item
let $m\geq n$ and suppose that $\Delta''\neq \Delta$ is an $m$-simplex of $K$ such  that $\Delta''\cap \Delta=\overline{F}$ is a common facet of 
$\Delta$ and of $\Delta'$ distinct from $F$; 
then $g$ restricts to the identity of $\Delta''$.
\end{enumerate}

If $n=i+1$, then $\Delta$ and $\Delta'$ are compatible (i.e.~they share the same set of facets), so the conclusion readily follows from Lemma~\ref{lemmaB}.
Let us now suppose that the claim is true for a fixed $n\geq i+1$, 
and take two $(n+1)$-simplices $\Delta,\Delta'$
sharing the same $i$-skeleton. 
We fix an ordering $v_0,\ldots,v_{n+1}$ of the vertices of $\Delta$ such that $F$ is the facet of $\Delta$ opposite to $v_{n+1}$, and 
for every $(n+1)$-simplex $\overline{\Delta}$ with the same set of vertices as $\Delta$ we denote by $\partial_j \overline{\Delta}$ the facet
of $\overline{\Delta}$ opposite to $v_j$. We also choose the ordering of the $v_i$ so that the facets shared by $\Delta$ and $\Delta'$ come first: in other words, we suppose that there exists
$0\leq j_0\leq n+1$ such that $\partial_j \Delta=\partial_j \Delta'$ for $0\leq j<j_0$, and $\partial_j\Delta\neq \partial_j \Delta'$ for $j_0\leq j\leq n$ (observe that we are not making any assumption
whether $F$ is also a facet of $\Delta'$ or not). 
Observe that, if $j_0=0$ (i.e.~if $\Delta$ and $\Delta'$ have no facet in common), then condition (2) automatically holds. Otherwise, 
we denote by
$\Delta''\neq \Delta$ a generic $m$-simplex of $K$, $m\geq n+1$, such  that $\Delta''\cap \Delta=\partial_j \Delta$ for some 
$0\leq j<j_0$.

We construct elements $g_j\in\Gamma_i$, $j=0,\ldots,n$ such that $g_0$ takes $\partial_0 \Delta$ to $\partial_0 \Delta'$, $g_1$ takes $\partial_1 (g_0\Delta)$ to $\partial_1 \Delta'$
without affecting $\partial_0 g_0(\Delta)=\partial_0 \Delta'$, 
$g_2$ takes $\partial_2 (g_1g_0\Delta)$ to $\partial_2 \Delta'$
without affecting $\partial_0 g_1g_0(\Delta)=\partial_0\Delta'$ and  $\partial_1 g_1g_0(\Delta)=\partial_1\Delta'$, and so on.
For $j=0,\ldots, j_0-1$ we have $\partial_j \Delta=\partial_j \Delta'$, and we simply set $g_j={\rm Id}_K$.
The first non-trivial transformation comes into play when settling the $j_0$-th facet. 
Let us now apply 
the inductive hypothesis by letting $\partial_{j_0} \Delta$ play the role of $\Delta$, $\partial_{j_0} \Delta'$ the role of $\Delta'$, and $\partial_{j_0} \Delta\cap \partial_{j_0+1} \Delta$
the role of $F$. Then, there 
 exists $g_{j_0}\in\Gamma_i$ such that $g_{j_0}(\partial_{j_0} \Delta)=\partial_{j_0}\Delta'$ and
 $\partial_j g_{j_0}(\Delta)=\partial_j \Delta=\partial_j \Delta'$ for every $j<j_0$. Moreover,
 since $\Delta''\cap \Delta=\partial_j \Delta$, $j<j_0$, the set $\Delta''\cap \partial_{j_0} \Delta$ 
 is equal to a common facet of $\partial_{j_0} \Delta$ and of $\partial_{j_0} \Delta'$ distinct from $\partial_{j_0} \Delta\cap \partial_{j_0+1} \Delta$.
 Therefore, we may also assume that
 $g_{j_0}(\Delta'')=\Delta''$. We set $\Delta_{j_0}=g_{j_0} (\Delta)=g_{j_0}g_{j_0-1}\cdots g_0(\Delta)$.

Now we can proceed as above to send $\partial_{j_0+1}\Delta_{j_0}$ to  $\partial_{j_0+1}\Delta'$,
without affecting $\Delta''$ and $\partial_j\Delta_{j_0}=\partial_j\Delta'$, $j=0,\ldots,j_0$:
to this aim, we apply the inductive hypothesis by letting $\partial_{j_0+1} \Delta$ play the role of $\Delta$, $\partial_{j_0+1} \Delta'$ the role of $\Delta'$, and $\partial_{j_0+1} \Delta\cap \partial_{j_0+2} \Delta$
the role of $F$. Hence, there 
 exists $g_{j_0+1}\in\Gamma_i$ such that $g_{j_0+1}(\partial_{j_0+1} \Delta)=\partial_{j_0+1}\Delta'$, 
 $\partial_j g_{j_0+1}(\Delta)=\partial_j \Delta=\partial_j \Delta'$ for every $j<j_0+1$, and $g_{j_0+1}(\Delta'')=\Delta''$. 
 We set $\Delta_{j_0+1}=g_{j_0+1} (\Delta_{j_0})=g_{j_0+1}g_{j_0}\cdots g_{0}(\Delta)$.

After $n$ steps, we end up with an $(n+1)$-simplex $\Delta_n=g_n\cdots g_0(\Delta)$ such that
$\partial_j\Delta_n=\partial_j \Delta'$ for every $j=0,\ldots,n$, and $g_n\cdots g_0(\Delta'')=\Delta''$. Observe that we cannot transform $\partial_{n+1} \Delta_n$ into $\partial_{n+1} \Delta'$
by following the same strategy as above, since no facet of $\partial_{n+1}\Delta_n$ can be moved without affecting $\partial_j \Delta_n$ for some $j<n+1$. 
Anyway, by minimality of $K$, Lemma~\ref{min-allbutoneface} implies that the equality $\partial_{n+1} \Delta_n=\partial_{n+1} \Delta'$ must automatically hold.
Therefore, $\Delta_n$ is compatible with $\Delta'$, so by Lemma~\ref{lemmaB} there exists $g_{n+1}\in \Gamma_n<\Gamma_i$ such that $g_{n+1}(\Delta_n)=\Delta'$ and
$g_{n+1}$ fixes every simplex of $K$ not containing $\partial_{n+1} \Delta_n$.
We then have
$$
\Delta'=g_{n+1}\cdots g_0 (\Delta)\ .
$$
Moreover, observe that $\Delta''$ cannot contain $\partial_{n+1} \Delta_n$, because otherwise 
it would contain the vertex $v_j$ for every $j=0,\ldots, n$, against the fact that
 $\Delta''\cap \Delta=\partial_j \Delta$ for some $j<j_0\leq n+1$.
 It follows
that 
$$
\Delta''=g_{n+1}(\Delta'')=g_{n+1}\cdots g_0 (\Delta'')\ .
$$
This concludes the proof.

\end{proof}

Recall from Section~\ref{quotient:sec} that the quotient of a multicomplex with respect to a simplicial action
which is trivial on the $0$-skeleton is still a multicomplex. In fact, Proposition~\ref{action-1-skeleton} shows that, at least when $K$ is large, complete and minimal, the quotient $K/\Gamma_1$ just coincides with 
the aspherical multicomplex associated to $K$ introduced in Section~\ref{aspherical:sec}:

\begin{cor}\label{Aquotient}
 Let $K$ be a large, complete and minimal multicomplex, and let $\Gamma_1=\Gamma_1(K)$. Let also $A$ be the aspherical quotient of $K$. Then $K/\Gamma_1$ is canonically isomorphic to $A$.
\end{cor}



\begin{rem}\label{rem:embedding:no:simpl:set}
As already mentioned in the introduction, it is very natural to wonder whether
Gromov's Mapping Theorem could be proved by resorting to the classical theory of simplicial sets, rather than 
to the quite exotic theory of multicomplexes. 
Unfortunately, this seems very hard. 
A first issue regards the fact that, in order to effectively exploit the action of $\Gamma$ on the simplices of $\calK(X)$ (or of $\calL(X)$), we definitely need to
be able to construct elements in $\G$ which act on some simplices of $\calK(X)$ as an orientation-reversing affine automorphism
(see the conclusion of the proof of the Vanishing Theorem~\ref{vanishing1_intro} at  the very end of Chapter~\ref{vanishing-thm:chap}, or Corollary~\ref{scambi}). To this aim, it is crucial that simplices of $\calK(X)$ correspond to singular simplices in $X$ \emph{up to affine automorphisms}, and this implies
in turn that simplices of $\calK(X)$ must come with no orderings on their vertices. Due to this fact, one could try to work with \emph{unordered} simplicial sets, as
described e.g.~in~\cite{grandis1, grandis2, rosick-symm}. However, also the fact that the simplices of $\calK(X)$ have distinct vertices plays a fundamental role in our arguments, as the following discussion shows.

When working with a pointed (Kan) simplicial set $(\calS,x_0)$, the $n$-th homotopy group of $(|\calS|,x_0)$ is usually studied via the analysis
of \emph{spherical} simplices, that can be defined as follows: a simplex $\Delta$ of $\calS$ is spherical if every face of $\Delta$ is the degenerate
$(n-1)$-simplex supported on the singleton $x_0$. In other words, $\Delta$ is spherical if $|\partial \Delta|=\{x_0\}$, i.e.~if the geometric
realization of $\Delta$ induces a map of pairs $(|\Delta|,|\partial \Delta|)\to (|\calS|,x_0)$. Indeed, when $\calS$ is Kan and minimal, 
the elements of the $n$-th homotopy group of $|\calS|$ bijectively correspond to the spherical simplices of $\calS$ based at $x_0$
(see e.g.~\cite[Section~I.11, page 60]{GJ} and \cite[Definition~3.6]{may}).

Recall now from Corollary~\ref{homKS} that $\calK(X)$ (hence $\calL(X)$) may be identified with a subcomplex
of $\calS(X)$. However, $\calS(X)$ contains many more simplices than $\calK(X)$ and $\calL(X)$:
for example, in positive dimension spherical simplices cannot belong to $\calK(X)$. 
When we tried to adapt Gromov's ideas to the context of simplicial sets,
this raised several difficulties that we were not able to overcome.
Indeed, let us choose  a minimal Kan subcomplex $\calM(X)$ of $\calS(X)$ such that the inclusion $|\calM(X)|\hookrightarrow |\calS(X)|$ is a homotopy equivalence,
and let $\G^S$ be the group of simplicial automorphisms of $\calM(X)$ which are homotopic to the identity relative to the $0$-skeleton. Also let $x_0$ be a vertex of $\calM(X)$
(which is also a point of $X$).
Since $\calM(X)$ is Kan and minimal, the group $\pi_n(|\calM(X)|,x_0)\cong \pi_n(X,x_0)$ is in bijection with the set of spherical simplices
of $\calM(X)$ based at $x_0$. Moreover, every element of $\G^S$, being homotopic to the identity relative to $x_0$, fixes every spherical simplex based at $x_0$.
This implies that no analogue of Corollary~\ref{transitive-on-compatible} or of Proposition~\ref{action-1-skeleton} can hold for the action of $\G^S$ on $\calM(X)$. In particular, the projections of all the spherical simplices
of $\calM(X)$ to the quotient $\calM(X)/\G^S$ remain pairwise distinct. Therefore, either $\calM(X)/\G^S$ is not aspherical, or it is not minimal.
In any case, the realization of an aspherical simplicial set with the same fundamental group of $X$ as a quotient of $\calM(X)$ by a simplicial action 
seems hard to achieve. This fact represents a major difficulty in the approach to Gromov's arguments via simplicial sets rather than via multicomplexes.
\end{rem}

As we will see (and is very well known), amenable groups are invisible to bounded cohomology, so if $\Gamma_1$ were amenable we could conclude that the bounded cohomology of a large, complete and aspherical multicomplex
$K$ should coincide with the one of its aspherical quotient. However, $\Gamma_1$ is \emph{not} amenable in general. Nevertheless, we will now show that $\Gamma_1/\Gamma_i$ is solvable (hence, amenable) for every $i\geq 1$. This will prove sufficient to show that the bounded
cohomology of $K$ coincides with the one of $A$ in every degree smaller than $i$, hence in every degree, thanks to the arbitrariness of $i$.

We begin with the following result, which 
establishes a relationship between the action of $\Gamma_{i-1}$ on $K$ and the description of $\pi_i(|K|)$ via the use of special spheres. Let $i\geq 2$,
and fix a set of representatives $\{\Delta_\alpha\}_{\alpha\in J}$ for the action of $\Gamma_{i-1}$ on the set of $i$-simplices of $K$. Also fix an ordering on the vertices of
$\Delta_\alpha$ for every $\alpha\in J$, and denote by $p_\alpha$ the first vertex of $\Delta_\alpha$.
For every
$\alpha\in J$ and 
for every $\gamma \in  \Gamma_{i-1}$ we can define a special sphere $\dot{S}^i_\alpha(\gamma)\colon (\dot{S}^i,s_0)\to (K,p_\alpha)$ by setting
$$
\dot{S}^i_\alpha(\gamma)=\dot{S}^i(\Delta_\alpha,\gamma(\Delta_\alpha))\ 
$$
(see Definition~\ref{special:def} for the definition of $\dot{S}^i(\Delta_1,\Delta_2)$ for any pair $\Delta_1,\Delta_2$ of compatible $i$-simplices).


\begin{lemma}\label{phi:hom}
Let $i\geq 2$. The map
$$
\phi_\alpha\colon \Gamma_{i-1}\to \pi_i(|K|,p_\alpha)\, ,\quad 
\phi_\alpha(\gamma)=\left[\dot{S}^i_\alpha(\gamma)\right]
$$
is a group homomorphism.
\end{lemma}
\begin{proof}
Take elements $\gamma_1,\gamma_2\in\Gamma_{i-1}$. By exploiting the very definition of sum  in $\pi_i(|K|,p_\alpha)$ it is readily checked that
\begin{align*}
\phi_\alpha(\gamma_1\gamma_2)&=\left[\dot{S}^i (\Delta_\alpha,\gamma_1\gamma_2(\Delta_\alpha))\right]=
\left[\dot{S}^i (\Delta_\alpha,\gamma_1(\Delta_\alpha))\right]+ \left[\dot{S}^i (\gamma_1(\Delta_\alpha),\gamma_1\gamma_2(\Delta_\alpha))\right]\\ &=
\phi_\alpha(\gamma_1)+\left[\gamma_1\circ \dot{S}^i (\Delta_\alpha,\gamma_2(\Delta_\alpha))\right]\ .
\end{align*}
Recall now that, as a map from $|K|$ to itself, every element of $\Gamma_{i-1}$ is homotopic to the identity relative to the $0$-skeleton of $|K|$, hence relative to
$p_\alpha$. As a consequence we have
$$
\left[\gamma_1\circ \dot{S}^i (\Delta_\alpha,\gamma_2(\Delta_\alpha))\right]
=\left[\dot{S}^i (\Delta_\alpha,\gamma_2(\Delta_\alpha))\right]=\phi_\alpha(\gamma_2)\ ,
$$
and this concludes the proof.
\end{proof}

The direct product of the homomorphisms $\phi_\alpha$
defines a homomorphism 
\begin{displaymath}
\phi^{(i)} \colon \Gamma_{i-1} \rightarrow \prod_{\alpha \in \, J} \pi_{i}(\lvert K \rvert, p_{\alpha})\ .
\end{displaymath}

\begin{lemma}\label{kerphi}
Let $K$ be a minimal multicomplex and let $i\geq 2$. Then $\ker \phi^{(i)}=\Gamma_i$.
\end{lemma}
\begin{proof}

Notice that $\gamma \in \, \ker(\phi_{\alpha})$ if and only if  the special sphere
$\dot{S}^i(\Delta_\alpha,\gamma(\Delta_\alpha))$ is homotopically trivial. By Lemma \ref{lemma-the-following-3-cond-equivalent}, this holds
if and only if $\gamma(\Delta_\alpha)$ is homotopic to $\Delta_\alpha$, hence, by minimality of $K$, if and only if $\gamma(\Delta_\alpha)=\Delta_\alpha$. 
 This proves that the kernel 
of $\phi_{\alpha}$ coincides with the stabilizer of $\Delta_{\alpha}$ in $\Gamma_{i-1}$. However, the kernel of $\phi_\alpha$ is obviously normal in
$\Gamma_{i-1}$, so it must coincide with the stabilizer of any simplex in the orbit of $\Delta_\alpha$.  Since the union of the orbits of the $\Delta_\alpha$, $\alpha\in J$, is just
the set of all $i$-simplices of $K$, we can conclude that
\begin{displaymath}
\ker\left(\phi^{(i)}\right) = \bigcap_{\alpha \in J} \ker(\phi_{\alpha}) = \Gamma_{i}\ .
\end{displaymath}
\end{proof}

Let $i\geq 2$. Being the direct product of $i$-th homotopy groups, the group $\prod_{\alpha \in \, J} \pi_{i}(\lvert K \rvert, p_{\alpha})$ is abelian. Therefore, from the previous lemma we immediately deduce the following:

\begin{cor}\label{corabelian}
Let $K$ be minimal and let $i\geq 2$. Then the group $\Gamma_{i-1}/\Gamma_{i}$ is abelian.
\end{cor}

\begin{cor}\label{corC}
Let $K$ be minimal. Then for every $i\geq 2$ the group $\Gamma_1/\Gamma_i$ is solvable, hence amenable.
\end{cor}
\begin{proof} We have
\begin{displaymath}
\Gamma_{1} \slash \Gamma_{i} \unrhd \Gamma_{2} \slash \Gamma_{i} \unrhd \cdots \unrhd \Gamma_{i} \slash \Gamma_{i} = \{ 1 \}.
\end{displaymath}
Moreover, for every $j=1,\ldots,i-1$ we have
\begin{displaymath}
\frac{\Gamma_{j-1} \slash \Gamma_{i}}{\Gamma_{j} \slash \Gamma_{i}} \cong \Gamma_{j-1} \slash \Gamma_{j}\ ,
\end{displaymath}
which is abelian by Corollary~\ref{corabelian}.
\end{proof}

\section{Bounded cohomology is determined by the fundamental group}\label{bd:fg:sc}
As anticipated above, 
our next goal is to prove that, if $K$ is large, complete and minimal,
then the quotient map $K\to K/\Gamma_1$ on the aspherical quotient 
induces an isometric isomorphism in bounded cohomology. We begin with the following result, which concerns simplicial actions of amenable groups. 
Recall that, if $G$ acts simplicially on a multicomplex $K$, we denote by $C_b^*(K)^G\subseteq C_b^*(K)$ the subcomplex of simplicial $G$-invariant
bounded cochains. 

\begin{Teorema}\label{teor-amen-action}
Let $G\actson K$ be a group action on a multicomplex $K$. For every $k\in\mathbb{N}$, let $G_k$ denote the subgroup of $G$ acting trivially on $K^k$, and suppose
that the quotient $G/G_k$ is amenable for every $k\in\mathbb{N}$.
Then, the inclusion
$$
\iota^*\colon C^*_b (K)^G \hookrightarrow C^*_b(K)
$$
induces an isometric embedding  
$$
H^*_b(\iota^*) \colon H^*_b(C^*_b(K)^G)\hookrightarrow H^*_b(K)\ .
$$
If we assume in addition that for every $\gamma\in G$ the simplicial automorphism $\gamma\colon K\to K$ is simplicially homotopic to the identity, 
then the isometric
embedding 
$$
H^*_b(\iota^*) \colon H^*_b(C^*_b(K)^G)\hookrightarrow H^*_b(K)\ .
$$
is also surjective (i.e.~it is an isometric isomorphism).
\end{Teorema}
\begin{proof}
In order to prove the first statement of the theorem, it is sufficient
to construct, for every $k\in\mathbb{N}$, a norm non-increasing (partial) chain map 
$$
A^i\colon C^i_b(K)\to C^i_b (K)^G\, ,\quad i\leq k\, ,
$$
such that  $A^i\circ\iota^i$ is the identity of $C^i_b(K)^G$ for every $i\leq k$.

Let $k\in\mathbb{N}$, and set $G^i=G/G_i$ for every $i\leq k$.  
Let $m_k\colon \ell^\infty(G^k)\to\R$ be a left-invariant mean on $G^k$. For every $i\leq k$, the group $G^i=G/G_i\cong (G/G_k)/(G_i/G_k)$ is a quotient of $G^k$, so we may endow
$G^i$ with the left-invariant mean $m_i$ induced by $m_k$ (see e.g.~\cite[Proposition 3.4]{frigerio-libro}). 

The action of $G_i$ on $C_i(K)$ is trivial, so the quotient $G^i$ naturally acts on $C_i(K)$, hence on $C^i_b(G)$. 
Therefore, if $\sigma$ is any algebraic simplex in $C_i(K)$ and $\phi\in C^i_b(K)$, then the map 
$$G^i\to\R\, ,\qquad \gamma\mapsto \phi(\gamma^{-1}\sigma)$$ 
is well defined and bounded, i.e.~it 
lies in $\ell^\infty(G^i)$. 
We may thus set
$$
A^i(\phi)(\sigma)=m_i (\gamma\mapsto \phi(\gamma^{-1}\sigma))\ ,
$$
and extend $A^i(\phi)$ to the whole $C_i(K)$ by linearity. The left-invariance of $m_i$ ensures that 
$A(\phi)$ is $G^i$-invariant, hence $G$-invariant.  Moreover, for every $\phi\in C^i_b(K)$
and every algebraic $i$-simplex $\sigma$ we have
$$
|A^i(\phi)(\sigma)|\leq \sup_{\gamma\in G} |\phi(\gamma^{-1}\sigma)|\leq \|\phi\|_\infty\ ,
$$
so $A^i$ is norm non-increasing. Each $A^i$ is linear by definition, 
while the fact that $A^i$, $0\leq i\leq k$, is a (partial) chain map may be checked by using that, by construction, the mean $m_i$ is the push-forward of the mean $m_{i+1}$ under the quotient map
$G^{i+1}\to G^i$ for every $i\leq k$ (see again~\cite[Proposition 3.4]{frigerio-libro}). 
We are thus left to show that $A^i\circ \iota^i$ is the identity of $C^i_b(K)^G$,
i.e.~that, if $\phi$ is a $G$-invariant bounded $i$-cochain, then $A^i(\phi)=\phi$. But this descends from the fact that the mean of a constant function is equal to the unique value
it takes. 

In order to prove the second statement, for every $\gamma \in G$ let us  denote by $t_\gamma^*\colon C^*_b(K)\to C^*_b(K)$ the
map on bounded cochains induced by $\gamma$.
Recall that, if
$\gamma\colon K\to K$ is simplicially homotopic to the identity of $K$, then there exists an algebraic homotopy
$
T^*_\gamma\colon C^*_b(K)\to C^{*-1}_b(K)
$
such that
\begin{equation}\label{homo:eq}
\delta^{*-1}T^*_\gamma+T^{*+1}_\gamma\delta^* =t_\gamma^*-{\rm Id}^*\ ,
\end{equation}
and
$$
\|T^n_\gamma\|_\infty\leq C_n
$$
for some universal constant only depending on the degree (see Remark~\ref{homotopy-uniform}). Moreover, the map $T^n_\gamma$ only depends on the behaviour of $\gamma$ on the $(n-1)$-skeleton of $K$.  

Let now $i\leq k$, 
take $\phi\in C^i_b(K)$, and let $\sigma$ be an algebraic $(i-1)$-simplex. Then the map
$$
G^{i-1}\to \R\, ,\qquad 
\gamma\mapsto T^i_{\gamma^{-1}}(\phi)(\sigma)
$$
is well defined and bounded (by the value
$C_i\|\phi\|_\infty$), so we can define the real value
$$
T^i(\phi)(\sigma)=m_{i-1} (\gamma\mapsto T^i_{\gamma^{-1}}(\phi)(\sigma))\ .
$$
Also observe that, thanks to the properties of means, $|T^i(\phi)(\sigma)|\leq C_i \|\phi\|_\infty$, so 
by extending linearly $T^i(\phi)$ over $C_{i-1}(K)$ we obtain an element $T^i(\phi)\in C^{i-1}_b(K)$.

By evaluating equality~\eqref{homo:eq} on $\phi\in C^i_b(K)$ and $\sigma\in C_{i-1}(K)$ as above we now get
$$
\delta^{i-1}(T^i_{\gamma^{-1}}(\phi))(\sigma)+T^{i+1}_{\gamma^{-1}}(\delta^i(\phi))(\sigma) =t_{\gamma^{-1}}^i(\phi)(\sigma)-\sigma=\phi(\gamma^{-1}\sigma)-\sigma\  .
$$
Both the right and the left hand side only depend on the behaviour of $\gamma$ on $K^i$, so we can consider both sides as functions defined over $G^i$, and by
taking the average with respect to $m_i$ of the right and of the left hand sides we then get
$$
\delta^{i-1}(T^i(\phi))(\sigma)+T^{i+1}(\delta^i(\phi))(\sigma) =A^i(\phi)(\sigma)-\sigma\ .
$$
We have thus shown that the (partial) chain map $ A^i\colon C^i_b(K)\to C^i_b(K)$, $i\leq k$, is homotopic to the identity via a (partial) bounded homotopy. As a consequence,
the map induced by $\iota^*$ on bounded cohomology is surjective for every $i<k$, and thanks to the arbitrariness of $k$ this concludes the proof.
\end{proof}

\begin{Corollario}\label{cor-amen-action}
Let $\Gamma\actson X$ be a $0$-trivial action of an amenable group $\Gamma$ on a multicomplex $K$,
and let $\pi\colon K\to K/\Gamma$ be the corresponding projection (see Section~\ref{quotient:sec}). Then the map
\begin{displaymath}
H^n_b(\pi)\colon H^n_b(K/\Gamma)\to H^n_b(\Gamma)
\end{displaymath} 
induced by $\pi$ on bounded cohomology is an isometric embedding for every $n\in\mathbb{N}$. 
If we assume in addition that 
 every $\gamma\in \G$ is simplicially homotopic to the identity, then 
$H^n_b(\pi)$ is an isometric isomorphism.
\end{Corollario}
\begin{proof}
 We have an obvious isometric chain isomorphism between the complex $C^n_b(K/\Gamma)$ and the complex $C^n_b(K)^\Gamma$, so the conclusion follows from
 Theorem~\ref{teor-amen-action}.
 \end{proof}

We are now ready to prove the main result of this section.

\begin{thm}\label{teoD}
Let $K$ be a large, complete and minimal multicomplex, let $A$ be the aspherical quotient associated to $K$, and let $\pi\colon K\to A$
be the natural projection of $K$ to $A$. Then the induced maps 
\begin{displaymath}
\begin{array}{c}
H^n_b(\pi)
\colon H^{n}_{b}(A) \rightarrow H^{n}_{b}(K)\ ,\\
H^n_b(\pi)
\colon H^{n}_{b}(|A|) \rightarrow H^{n}_{b}(|K|)
\end{array}
\end{displaymath}
are isometric isomorphisms for every $n\in\mathbb{N}$.
\end{thm}
\begin{proof}
By Theorem~\ref{isometry-lemma-intro}, it is sufficient to show that the map $\pi$ induces isometric isomorphisms in every degree on \emph{simplicial} bounded cohomology.

Recall from  Corollary~\ref{Aquotient} that $A$ is canonically isomorphic to the quotient $K/\Gamma_1$. Moreover, 
there is an obvious isometric chain isomorphism between the complex $C^n_b(A)=C^n_b(K/\Gamma_1)$ and the complex $C^n_b(K)^{\Gamma_1}$, and under this identification
the projection $\pi$ induces the inclusion $C^*_b (K)^{\Gamma_1} \hookrightarrow C^*_b(K)$ on bounded cochains. Finally, recall 
from Corollary \ref{corC} that the group $\Gamma_{1} \slash \Gamma_{m}$ is amenable for every $m\in\mathbb{N}$, where $\Gamma_m$ is the subgroup of $\Gamma_1$ acting trivially
on $K^m$. Moreover, every element in $\Gamma_1$ is homotopic to the identity, hence it is simplicially homotopic to the identity thanks to the fact that $K$ is large and complete 
(see Lemma~\ref{homotopy-lemma}). Therefore, the conclusion follows from Theorem~\ref{teor-amen-action}.
\end{proof}

Recall from Definition~\ref{asphericalX:def} that to every good topological space there is canonically associated a complete, minimal and aspherical
multicomplex $\calA(X)$.
The following corollary shows that the bounded cohomology of  $X$ is isometrically isomorphic to the simplicial
bounded cohomology of $\calA(X)$. Since the simplicial bounded cohomology of $\calA(X)$ is isomorphic to the singular bounded cohomology of $|\calA(X)|$, 
and $|\calA(X)|$ is an aspherical CW complex with the same fundamental group of $X$, this already implies that the bounded cohomology of $X$ only depends
on the fundamental group of $X$ (see Theorem~\ref{mapping1_intro} below for a precise formulation of this statement).

\begin{Corollario}\label{cor-remarkE}
Let $X$ be a good space. Then 
$H^n_b(X)$ is canonically isometrically isomorphic to $H^n_b(\calA(X))$ for every $n\in\mathbb{N}$.
\end{Corollario}
\begin{proof}
On the one hand,  Theorem~\ref{isometria-good-K(X)} shows that $H^n_b(X)$ is isometrically isomorphic to  $H^n_b(\calK(X))$. On the other hand, if 
$
r\colon \calK(X) \to \calL(X) 
$
and $\pi\colon \calL(X)\to \calA(X)$ are the retraction and the projection described in the previous chapter, then
Corollary~\ref{cor:equiv:hom:induces:iso:simpl:bc} and Theorem~\ref{teoD} imply that both $$H^n_b(r)\colon H^n_b(\calL(X))\to H^n_b(\calK(X))$$ and 
$$H^n_b(\pi)\colon \calA(X)\to \calL(X)$$ are isometric isomorphisms. 
The conclusion follows.
\end{proof}

As already observed, Corollary~\ref{cor-remarkE} implies that the bounded cohomology of a good space $X$ only depends on its fundamental group. We provide a precise formulation of this fact in the next result.
In fact, as anticipated in the introduction, building on the fact that weak homotopy equivalences induce isometric isomorphisms in bounded cohomology, we may even get rid of the assumption that the topological spaces we consider are good.

\begin{thm:repeated:mapping1_intro}
 Let $X$ and $Y$ be path connected topological spaces, and let $f\colon X\to Y$ be a map inducing an isomorphism on fundamental groups. Then the map
 $$
H^n_b( f)\colon H^n_b(Y)\to H^n_b(X)
 $$
is an isometric isomorphism for every $n\in\mathbb{N}$.
\end{thm:repeated:mapping1_intro}
\begin{proof}
Let us first assume that $X$ and $Y$ are good.
Then Proposition~\ref{commutative:homotopy:L} implies that there exists a map $L(f)\colon |\calL(X)|\to |\calL(Y)|$ such that the following diagram commutes up to homotopy:
$$
\xymatrix{
|\calL(X)| \ar[rr]^-{L(f)} \ar[d]_-{S_{X}} & & |\calL(Y)| \ar[d]^-{S_Y} \\
X \ar[rr]_-{f} & & Y\ }
$$
(in particular, the map $L(f)$ induces an isomorphism on fundamental groups).
The vertical arrows of this diagram are weak homotopy equivalences, hence by Theorem~\ref{weak:iso:thm}
we are reduced to show that the map 
$$H^n_b(L(f))\colon H^n_b(|\calL(Y)|)\to H^n_b(|\calL(X)|)$$ 
is an isometric isomorphism for every $n \in \matN$.

Let us consider the aspherical multicomplexes $\calA(X)$ and $\calA(Y)$ associated to $\calL(X)$ and $\calL(Y)$ respectively, and recall from Proposition~\ref{riassunto:prop}
that the  projections $\pi_X\colon |\calL(X)| \to |\calA(X)|$ and $\pi_Y \colon |\calL(Y)| \to |\calA(Y)|$
induce isomorphisms on fundamental groups. 
The general theory of Eilenberg-MacLane spaces provides a map $A(f) \colon |\calA(X)| \to |\calA(Y)|$ such that the following diagram commutes up to homotopy:
$$
\xymatrix{
|\calL(X)| \ar[rr]^-{L(f)} \ar[d]_-{\pi_X } && |\calL(Y)| \ar[d]^-{\pi_Y} \\
|\calA(X)| \ar[rr]_-{A(f)} && |\calA(Y)|\ .}
$$
From the commutativity (up to homotopy) of the diagram we deduce that $A(f)$ induces an isomorphism on fundamental groups. Since $|\calA(X)|$ and $|\calA(Y)|$ are aspherical CW complexes,
this implies that $A(f)$ is a homotopy equivalence, hence it induces isometric isomorphisms on bounded cohomology in every degree. The conclusion that also $L(f)$ induces
an isometric isomorphism on bounded cohomology now follows from Theorem~\ref{teoD}, which asserts that the vertical arrows induce isometric isomorphisms on bounded cohomology.

This concludes the proof of the theorem under the assumption that $X$ and $Y$ are good.

In the general case, let $\calS(X)$ and $\calS(Y)$ be the singular simplicial sets associated to $X$ and $Y$, respectively, and let
$\calS(f)\colon |\calS(X)|\to |\calS(Y)|$ be the (simplicial) map associated to $f$. If we denote by $j_X\colon |\calS(X)|\to X$, $j_Y\colon |\calS(Y)|\to Y$ the natural
projections, then $j_Y\circ \calS(f)=f\circ j_X$.  Moreover, $j_X$ and $j_Y$ are weak homotopy equivalences, hence they induce isometric isomorphisms on bounded cohomology in every degree.
Finally, $\calS(f)$ induces an isomorphism on fundamental groups. Since $|\calS(X)|$ and $|\calS(Y)|$ are good spaces, we may thus conclude that in the following diagram
$$
\xymatrix{
H^n_b(Y) \ar[rr]^-{H^n_b(f)} \ar[d]_-{H^n_b(j_Y )} && H^n_b(X)) \ar[d]^-{H^n_b(j_X)} \\
H^n_b(|\calS(Y)|) \ar[rr]_-{H^n_b(S(f))} && H^n_b(|\calS(X)|)\ ,}
$$
the vertical arrows and the bottom horizontal arrow are isometric isomorphisms. Thus also the top horizontal arrow is an isometric isomorphism, and this concludes the proof.
\end{proof}

In fact, a stronger result holds: a continuous map $f\colon X\to Y$ induces 
isometric isomorphisms on bounded cohomology in every degree provided that the induced map on fundamental groups is an epimorphism with amenable kernel.
To this fact, which is usually known as \emph{Gromov's Mapping Theorem}, will be devoted the next chapter.

\section{Multicomplexes and relative bounded cohomology}\label{relative:sec}
In this section we briefly discuss how the theory developed so far could be adapted to the relative case, in order to produce some tools
for the study of the simplicial volume of manifolds with boundary. 

Let $M$ be an oriented compact manifold with boundary, and let $n=\dim M$. It is well known that $H_n(M, \partial M;\mathbb{Z})$ is isomorphic to $\mathbb{Z}$, and generated
by a preferred element $[M,\partial M]_\mathbb{Z}$, which is called the (integral) \emph{fundamental class} of $M$. As usual, one can define
the (real) fundamental class $[M,\partial M]\in H_n(M, \partial M;\mathbb{R})=H_n(M, \partial M)$ as the image of $[M,\partial M]_\mathbb{Z}$ via the change of coefficients
map $H_n(M, \partial M;\mathbb{Z})\to H_n(M, \partial M)$. The $\ell^1$-norm on $C_n(M)$ descends to a seminorm (which is in fact a norm) on $C_n(M,\partial M)$,
which defines in turn a seminorm $\|\cdot \|_1$ on $H_n(M,\partial M)$. The simplicial volume of $M$ is then defined by setting
$$
\|M\|=\|[M,\partial M]\|_1\ .
$$
Just as in the closed case, there exist very few exact computations for the simplicial volume of manifolds with boundary.
We refer the reader e.g.~to~\cite{FP,BFP,BFP2} for some results and speculations on this topic. A variation of the simplicial volume of compact manifolds with boundary, called \emph{ideal simplicial volume}, has been recently defined by the authors in~\cite{FMo}.  While keeping the most important topological and geometric
features of the classical simplicial volume, the ideal simplicial volume should be a bit easier to compute exactly (at least, it is so in some concrete examples). 

Of course, in order to study the simplicial volume of manifolds with boundary via the use of multicomplexes, the first step should consist in understanding
how to compute the relative bounded cohomology (or the $\ell^1$-seminorm on singular homology) of a pair of spaces in terms
of the simplicial relative bounded cohomology (or the $\ell^1$-seminorm on simplicial homology) of a suitably defined (pair of) multicomplex(es).

Probably the first question  that arises in this context is whether Gromov's and Ivanov's results on the coincidence of bounded cohomology of a space with the bounded cohomology of its fundamental group may be extended to the relative case. To this aim, some attempts were made to adapt Ivanov's techniques also to the study of the bounded cohomology
of pairs of spaces. 
We now briefly describe the results obtained via this strategy, which were proved in~\cite{Park,FP2,Blank} (see also~\cite{BBFIPP}).

Let us first focus on the easier case when the pair of spaces $(X,Y)$ is such that
$X$ and $Y$ are both path connected, and the inclusion $Y\hookrightarrow X$ induces an injective map on fundamental groups. In this case, to the pair $(X,Y)$ one
can associate the pair of groups $(\pi_1(X),\pi_1(Y))$, where $\pi_1(Y)$ may be considered as a subgroup of $\pi_1(X)$. Under these assumptions, 
a straightforward application of the Five Lemma shows
that  $H^n_b(X,Y)$ is  isomorphic to $H^n_b(\pi_1(X),\pi_1(Y))$ (see~\cite{Park, FP2}). However, some difficulties arise when trying to promote
the isomorphism to an isometry. In order to adapt the killing homotopy procedure described by Ivanov to the relative case, some unpleasant 
(and conjecturally unnecessary) additional hypothesis is needed (see~\cite[Remark 4.9]{FP2}).
Namely, 
Pagliantini and the first author proved in \cite[Theorem 1.8]{FP2} that $H^n_b(X,Y)$ is  isometrically isomorphic to $H^n_b(\pi_1(X),\pi_1(Y))$ provided
 that the map $\pi_n(Y)\to \pi_n(X)$ induced by the inclusion is an isomorphism for every $n\geq 2$.

Some more care is needed when $Y$ is not path connected (but still every path connected component of $Y$ is $\pi_1$-injective in $X$).
A suitable notion of bounded cohomology for pairs $(\Gamma, \{H_i\}_{i})$, where $\{H_i\}_i$ is a family of subgroups of $\Gamma$,
was defined by Mineyev and Yaman
  in \cite{Yaman} (see also~\cite{Franceschini2}). Building on this definition (and on the theory of fundamental groupoids) Blank extended the results of~\cite{FP2}
  to the case when $Y$ is not necessarily path connected (but the inclusion of every path connected component of $Y$ in $X$ still induces an injective map 
  on fundamental groups and an isomorphism on higher homotopy groups,
see  \cite[Theorem~9.11]{Blank} for a precise statement).

One can get rid of all these assumptions
in the case when every component of $Y$ has an amenable fundamental group. In this case, it is proved in~\cite{BBFIPP} that the relative bounded
cohomology $H^n_b(X,Y)$ is isometrically isomorphic to the absolute bounded cohomology $H_b^n(X)$ for every $n\geq 2$, and from this fact one can easily deduce 
that, if every path connected component of $Y$ is $\pi_1$-injective in $X$, then $H^n_b(X,Y)$ is also isometrically isomorphic to the bounded cohomology of the corresponding
pairs of groups
(see also~\cite{KimKue} for an alternative proof of this result, which exploits multicomplexes).

It is worth noting that Gromov stated in~\cite{Grom82} a \emph{Relative Mapping Theorem} in which no hypothesis on the spaces involved appear. Namely,
he stated the following:

\begin{thm}[{\cite[page 57]{Grom82}}]\label{rel:gro}
Let $f\colon (X,Y)\to (X',Y')$ be a map of pairs which is bijective on the set of path connected components, i.e.~it induces bijections
$\pi_0(X)\cong \pi_0(X')$, $\pi_0(Y)\cong \pi_0(Y')$. Furthermore, let $f$ induce an isomorphism on fundamental groups on every component of $X$ and
on every component of $Y$. Then the induced map
$$
H^n_b(f)\colon H^n_b(X',Y')\to H^n_b(X,Y)
$$
is an isometric isomorphism for every $n\in\mathbb{N}$.
\end{thm}
(In fact, Gromov's statement is more general, since it allows the map $f$ to induce epimorphisms with an amenable kernel on 
the fundamental groups of all the components of $X$ and $Y$.) 
Gromov asserts that Theorem~\ref{rel:gro} may be proved 
by adapting to the relative case
his proof of the absolute Mapping Theorem. We are quite optimistic about the possibility of exploiting the theory developed in this paper to provide a detailed proof of Theorem~\ref{rel:gro}, following Gromov's suggestion.

For related results on  relative bounded cohomology  obtained via the theory of multicomplexes,
we refer the reader to~\cite{Kuessner}.

\chapter{The Mapping Theorem}\label{chap:theorems}

Gromov's Mapping Theorem 
asserts that a continuous map between two topological spaces induces an isometric isomorphism in bounded cohomology provided that the induced map on fundamental groups is a surjection with amenable kernel.
It was stated without any assumption on the spaces involved in~\cite[Section 3.1]{Grom82}. 
A completely different proof of the Mapping Theorem was given by Ivanov in~\cite{Ivanov} for spaces that are homotopy equivalent to countable CW complexes.
Ivanov recently extended his result to any topological space in~\cite{ivanov3}. Just as for the proof that the bounded cohomology of simply connected spaces vanishes,
Ivanov's argument exploits tools of homological algebra (see also the discussion at the beginning of Chapter~\ref{chap2:chap}). On the contrary, our approach follows Gromov's original ideas,
and it is based on the theory of multicomplexes.

We have seen in the previous chapter that, if $f\colon X\to Y$ is a map between path connected topological spaces inducing an isomorphism on fundamental groups,
 then the induced map
 $
H^n_b( f)\colon H^n_b(Y)\to H^n_b(X)
 $
is an isometric isomorphism for every $n\in\mathbb{N}$.
As anticipated in the introduction, Gromov's Mapping Theorem strengthens this result as follows:


\begin{thm:repeated:mapping-intro}[Gromov's Mapping Theorem]
Let $X_1,X_2$ be path connected topological spaces, and let $f \colon X_{1} \rightarrow X_{2}$ be a continuous map.
Suppose that 
the induced homomorphism $f_{*} \colon \pi_{1}(X_{1}) \rightarrow \pi_{1}(X_{2})$ is an epimorphism with an amenable kernel. 
Then, for every $n\in\mathbb{N}$ the induced map  $H^n_b(f) \colon H_{b}^{n}(X_{2}) \rightarrow H_{b}^{n}(X_{1})$ is an isometric isomorphism. 
\end{thm:repeated:mapping-intro}

As we did in the previous chapter, we will first prove Theorem~\ref{mapping-intro} under the additional assumption that $X_1$ and $X_2$ are good,
and we will then deduce the general case from the fact that weak homotopy equivalences induce isometric isomorphisms on bounded cohomology. 

\section{The group $\Pi(X, X_0)$}\label{sec:mapping:pi(U,V)}


Let us fix a path connected good topological space $X$.
Recall from Definition~\ref{asphericalX:def} that to $X$ there is associated a complete, minimal and aspherical multicomplex $\calA(X)$.
In order to prove the Mapping Theorem we need to introduce a group action on $\calA(X)$ which will prove useful also to other purposes.  

\begin{Definizione}
Let $X_0$ be a subset of $X$. We define the set $\Omega(X,X_0)$ as follows. An element of $\Omega(X,X_0)$ is a family of paths $\{\gamma_x\}_{x\in X_0}$ satisfying the following conditions:
\begin{enumerate}
 \item each $\gamma_x\colon [0,1]\to X$
is a continuous path such that $\gamma_x(0)=x$ and $\gamma_x(1)\in X_0$;
\item  for all but finitely many $x\in X_0$, the path $\gamma_x$ is constant (i.e.~$\gamma_x(t)=x$ for every $t\in [0,1]$);
\item the map 
$$
X_0\to X_0\ ,\qquad x\mapsto \gamma_x(1)
$$
is a bijection of $X_0$ onto itself (which, by (2), is a permutation of $X_0$ with finite support).
\end{enumerate}
Therefore, an element of $\Omega(X,X_0)$ can be considered as a path of maps $I_t\colon X_0\to X$, $t\in [0,1]$ such that $t\mapsto I_t(x_0)$ is continuous for every $x_0\in X_0$
and constant for all but finitely $x_0\in X_0$. Note however that the map $X_0\times [0,1]\to X$ defined by $(x_0,t)\mapsto I_t(x_0)$ need not be continuous
if $X_0$ is not discrete.

By definition,  only a finite number of paths in a fixed element of $\Omega(X,X_0)$ is non-constant, so we will often describe an element
of $\Omega(X,X_0)$ just by listing its non-constant elements. In other words, an element of $\Omega(X,X_0)$ will be often considered as a finite
set of continuous paths $\gamma_i\colon [0,1]\to X$, $i=1,\ldots,n$, such that
\begin{itemize}
\item[(i)] $\gamma_{i}(0) \neq \gamma_{j}(0), \gamma_{i}(1) \neq \gamma_{j}(1) \mbox{ for } i \neq j$;
\item[(ii)] $\{\gamma_{1}(0), \cdots , \gamma_{n}(0)\} = \{\gamma_{1}(1), \cdots , \gamma_{n}(1)\}\subseteq X_0$.
\end{itemize}
There is an obvious multiplication in $\Omega(X,X_0)$, given by the usual concatenation of paths: we denote by $I_t*I'_t\in \Omega(X,X_0)$ the family of paths
obtained by concatenating the paths in $I_t$ with the paths in $I_t'$ (where, as usual, paths in $I_t$ come before paths in $I_t'$). With this multiplication, $\Omega(X,X_0)$ is just a semigroup. 

In order
to obtain a group we consider the set $\Pi(X,X_0)$ of homotopy classes of elements of $\Omega(X,X_0)$, where two elements $\{\gamma_x\}_{x\in X_0}$, $\{\gamma'_x\}_{x\in X_0}$ of $\Omega(X,X_0)$ are said to be homotopic if 
$\gamma_x$ is homotopic to $\gamma_x'$ relative to the endpoints for every $x\in X_0$ (in particular, $\gamma_x(1)=\gamma_x'(1)$ for every $x\in\ X_0$). It is now immediate to check that the multiplication on $\Omega(X,X_0)$
induces a multiplication on $\Pi(X,X_0)$, which endows $\Pi(X,X_0)$ with the structure of a group.
In order to avoid a heavier notation, we will sometimes denote an element of $\Pi(X, X_{0})$ just by specifying the list of homotopically non-trivial paths $\{\gamma_{1}, \cdots, \gamma_{n}\}$
in one of its representatives.
\end{Definizione}

\begin{Esempio}
If $X_{0} = \{x_{0}\}$ is a single point, then $\Pi(X, X_{0}) = \pi_{1}(X, x_{0})$ is the fundamental group of $X$ with basepoint $x_0$.
\end{Esempio}

More in general, there exists an obvious injective homomorphism
\begin{displaymath}
\bigoplus_{x \in \, X_{0}} \pi_{1}(X, x) \hookrightarrow \Pi(X, X_{0}) \ .
\end{displaymath}
Via this inclusion, henceforth we will consider $\bigoplus_{x \in \, X_{0}} \pi_{1}(X, x)$ as a subgroup of $\Pi(X,X_0)$.

\section{The action of $\Pi(X,X)$ on $\calA(X)$}\label{actiononA}
Let $X$ be a path connected good topological space.
We are now going to describe a simplicial action of $\Pi(X,X)$ on $\calA(X)$ that will play an important role in the proofs of the Mapping Theorem \ref{mapping-intro} and of the Vanishing Theorem \ref{vanishing1_intro}.

Let us fix an element $g\in\Pi(X,X)$, and let $\{\gamma_x\}_{x\in X}$ be a representative of $g$.
Recall that the $1$-skeleton of $\calA(X)$ is canonically identified with the $1$-skeleton of the complete and minimal multicomplex $\calL(X)$, whose set of vertices coincides with $X$ (considered as a set). 
In particular, the set of vertices of $\calA(X)$ is canonically identified with $X$. We observed in Proposition~\ref{morph-pi-to-permfinite} that every element of $\Pi(X,X)$ induces a permutation (with finite support) of $X$.
Therefore, we can define the action of $g$ on the $0$-skeleton of $\calA(X)$ as the permutation induced by $g$ on $X$.

Let now $e$ be a $1$-simplex of $\calA(X)$, let $v_0,v_1$ be the vertices of $e$, and let us fix an affine parametrization
$\gamma_e\colon [0,1]\to |e|$ of $|e|$. 
If $S\colon |\calK(X)|\to X$ is the natural projection, then (after considering $e$ as a simplex of $\calL(X)\subseteq \calK(X)$) we can take the concatenation of paths
$\gamma'\colon [0,1]\to X$ given by
$$
\gamma'=\gamma_{v_0}^{-1} * (S\circ \gamma_e) * \gamma_{v_1}\ ,
$$


where as usual we denote by $\gamma^{-1}$ the path $\gamma^{-1}(t)=\gamma(1-t)$. Observe that $\gamma'(0)=\gamma_{v_0}(1)$ and $\gamma'(1)=\gamma_{v_1}(1)$, so the endpoints of $\gamma'$ coincide with
the images of the endpoints of $e$ via the action of $g$ on the $0$-skeleton of $\calA(X)$ (in particular, they are distinct). 
Also recall that $1$-simplices of $\calL(X)$ (hence, of $\calA(X)$) bijectively correspond
to homotopy classes relative to endpoints of paths in $X$ with distinct endpoints (this readily follows from the very definition of $\calL(X)$, together with the fact
that the natural projection $|\calL(X)|\to X$ is a weak homotopy equivalence). Therefore, we can define $g\cdot \gamma_e$ as the homotopy class  of $\gamma'$ relative to the endpoints. It is easy to check that 
this definition does not depend on the numbering of the vertices of $e$. Moreover, the action of $g$ on the edges of $\calA(X)$ indeed  extends the action on vertices, and for every $g,g'\in \Pi(X,X)$ the action of
$gg'$ is equal to the composition of the actions of $g$ and of $g'$. 

Let us now extend the action of $g$ on the whole of $\calA(X)$. Let us first check that, if $\gamma\colon \partial \Delta^2\to \calA(X)$ is a null-homotopic simplicial embedding (i.e.~a null-homotopic triangular loop), 
then $g\circ \gamma$ is also a null-homotopic triangular loop. To this end, let us denote by $\iota \colon \partial \Delta^2 \to \calA(X)$ the simplicial embedding corresponding to $g \circ \gamma$. Then, if $\{v_0,v_1,v_2\}=\gamma((\Delta^2)^0)$
and $e_i$ is the edge of $\calA(X)$ opposite to $v_i$ (with the orientation induced by $\gamma$), then by construction the projection $S \circ \iota$ 
of the triangular loop $g \circ \gamma$ on $X$ is freely homotopic 
to the loop
$$\alpha = \left(\gamma_{v_0}^{-1} * (S\circ \gamma_{e_2}) * \gamma_{v_1}\right) * \left(\gamma_{v_1}^{-1} * (S\circ \gamma_{e_0}) * \gamma_{v_2}\right) *
\left(\gamma_{v_2}^{-1} * (S\circ \gamma_{e_1}) * \gamma_{v_0}\right)\ ,$$
hence
to $S\circ (\gamma_{e_2}*\gamma_{e_0}*\gamma_{e_1})=S\circ \gamma$. But $\gamma$ is homotopically trivial, so also $S\circ\iota\colon \partial \Delta^2\to X$ is homotopically trivial. 
By Proposition~\ref{riassunto:prop}, this implies that 
$\iota$ is null-homotopic in $\calA(X)$, as claimed. 

Observe now that Proposition~\ref{aspherical:char} implies that, if $f\colon (\Delta^n)^1\to \calA(X)$ is a simplicial embedding 
of the $1$-skeleton of $\Delta^n$, $n\geq 2$, 
such that the restriction of $f$ to each triangular loop is null-homotopic, then there exists
  a unique simplicial embedding $\overline{f}\colon \Delta^n\to A$ extending $f$. This result implies at once that the action of $g\in \Pi(X,X)$ on $\calA(X)^1$ can be uniquely extended to a 
  non-degenerate simplicial map $\psi(g)\colon \calA(X)\to\calA(X)$. The uniqueness of the extension also ensures that 
  $\psi(gg')=\psi(g)\psi(g')$ for every $g,g'\in\Pi(X,X)$. In particular, $\psi(g)$ is an automorphism of $\calA(X)$ for every $g\in\Pi(X,X)$. 
  
  \begin{thm}\label{actionPi}
   Let $X$ be a good space, and let
   $$
   \psi\colon \Pi(X,X)\to \aut(\calA(X))
   $$
   be the action described above. Then $\psi(g)$ is simplicially homotopic to the identity for every $g\in \Pi(X,X)$. 
  \end{thm}
\begin{proof}
 Since $\calA(X)$ is an aspherical CW complex, 
 from the general theory of Eilenberg-MacLane spaces we know that $\psi(g)$ is topologically homotopic to the identity provided that
 it induces the identity on $\pi_1(|\calA(X)|,x_0)$ for some $x_0\in |\calA(X)|$ (see for instance  \cite[Theorem 7.2]{whitehead} or \cite[Proposition 1B.9]{hatcher}).

 Without loss of generality we may assume that $X$ is path connected and contains at least two points, so that $\calA(X)$ is large. In particular, 
 if $g=\{\gamma_x\}_{x\in X}$, then there exist two points $x_0,x_1\in X$ such that $\gamma_{x_i}$ is constant for $i=0,1$. This implies that $\psi(g)$ acts as the identity
 on every edge of $\calA(X)$ having $x_0,x_1$ as endpoints. Since $\calA(X)$ is complete, we know that every element of $\pi_1(|\calA(X)|,x_0)$ may be represented by a pair of edges
 with endpoints $x_0,x_1$ (see Theorem~\ref{complete:special}). Therefore, $\psi(g)$ acts as the identity on $\pi_1(|\calA(X)|,x_0)$,
 and this proves that $\psi(g)$ is topologically homotopic to the identity.
 
  Finally, the Homotopy Lemma~\ref{homotopy-lemma} implies that any automorphism of $\calA(X)$ which is topologically homotopic to the identity is also simplicially homotopic
 to the identity.
\end{proof}

We now concentrate our attention on the action on $\calA(X)$ of the subgroup $\bigoplus_{x\in X} \pi_1(X,x)$ of $\Pi(X,X)$. To this aim we first fix some notation.
For every $x\in X$ we set $G_x=\pi_1(X,x)$. Let us fix a basepoint $\overline{x}\in X$, let $G=G_{\overline{x}}$ and let $N$ be a normal subgroup of $G$. 
It is well known that, for every $x\in X$, there exists an isomorphism $G\cong G_x$, which is canonical up to conjugacy. Since
$N$ is normal, this implies that for every $x\in X$ there exists a well-defined isomorphic image $N_x$ of $N$ in $G_x$ (so that $N_{\overline{x}}=N$). We then set
$$
\widehat{N}=\bigoplus_{x\in X} N_x\ .
$$
In particular, $\widehat{G}=\bigoplus_{x\in X} G_x=\bigoplus_{x\in X} \pi_1(X,x)$. 

Recall that the natural projections $|\calL(X)|\to X$, $|\calL(X)|\to|\calA(X)|$ induce isomorphisms on fundamental groups, so for every $x\in X$ there exists a canonical identification
between $\pi_1(X,x)$ and $\pi_1(|\calA(X)|,x)$ (where we denote by $x\in |\calA(X)|$ the vertex of $\calA(X)$ corresponding to $x$). Therefore, 
for every $x\in X$ we will denote by $N_x$ also the subgroup of $\pi_1(|\calA(X)|,x)$ corresponding to $N_x<\pi_1(X,x)$.
For brevity, we will also say that a loop $\gamma\colon S^1\to |\calA(X)|$ is in $N$ if it is freely homotopic to a representative of an element in $N_x$ for some (and hence, every) vertex $x$ of $\calA(X)$. 
Using that $N$ is normal it is readily seen that a loop $\gamma\colon S^1\to |\calA(X)|$
is in $N$ if and only if for every $p\in S^1$ with $\gamma(p)\in \calA(X)^0$ the induced pointed map $\gamma\colon (S^1,p)\to (|\calA(X)|,\gamma(p))$ defines an element in $N_{\gamma(p)}$.

If $g$ is an element of $\widehat{N}$, then $\psi(g)\colon \calA(X)\to\calA(X)$ restricts to the identity of the $0$-skeleton of $\calA(X)$. In other words, the action of $\widehat{N}$ on
$\calA(X)$ is $0$-trivial according to the terminology introduced in Section~\ref{quotient:sec}. In particular, the quotient $\calA(X)/\widehat{N}$ admits a natural structure of multicomplex.
Moreover, the set of vertices of $\calA(X)/\widehat{N}$ is canonically identified with the set of vertices of $\calA(X)$, hence with $X$.
We denote by $\pi\colon \calA(X)\to \calA(X)/\widehat{N}$ the (simplicial) quotient map.

We are now going to prove that the (geometric realization) of the multicomplex  $\calA(X)/\widehat{N}$ is a $K(G/N,1)$-space. We begin with the following:

\begin{lemma}\label{actionNhat}
 Let $e,e'$ be edges of $\calA(X)$ such that there exists $g\in\widehat{N}$ such that $g\cdot e=e'$ (in particular, $e$ and $e'$ share the same endpoints $x_0,x_1$).
 Then, for every $g_0\in N_{x_0}$ there exists a unique element $g_1\in N_{x_1}$ such that 
 $$
 (g_0,g_1)\cdot e =e'\ .
 $$
\end{lemma}
\begin{proof}
 Let $\beta\colon [0,1]\to X$ (resp.~$\beta'\colon [0,1]\to X$) be the path associated to $e$ (resp.~to $e'$) such that $\beta(0)=x_0$, $\beta(1)=x_1$
 (resp.~$\beta'(0)=x_0$, $\beta'(1)=x_1$), and
 let $\alpha_{0}$  be a representative of $g_{0}$. Let also $\alpha_1\colon [0,1]\to X$ be a loop based at $x_1$, and denote by $g_1$ the class of $\alpha_1$
 in $\pi_1(X,x_1)$. 
  Then $(g_0,g_1)\cdot e=e'$ if and only if the paths 
 $
 \alpha_0^{-1}* \beta * \alpha_1
 $
 and $\beta'$ 
 are homotopic in $X$ relative to the endpoints, i.e.~if and only if 
 $$
 (\beta')^{-1}*\alpha_0^{-1}* \beta * \alpha_1
 $$
 defines the trivial element of $\pi_1(X,x_1)$. This condition holds if and only if 
 $$
 g_1=[\alpha_1]=[\beta^{-1}*\alpha_0*\beta']\ \textrm{in}\ \pi_1(X,x_1)\ ,
 $$
 and this concludes the proof.
\end{proof}
 
 \begin{lemma}\label{charinN}
  Let $e,e'$ be edges of $\calA(X)$ sharing the same endpoints and denote by $e*e'$ the loop obtained as one of the possible concatenations of $e$ and $e'$. 
  Then $\pi(e)=\pi(e')$ if and only if $e*e'$ is in $N$.
 \end{lemma}
\begin{proof}
Let $x_0,x_1$ be the endpoints of $e$ and $e'$. Let also $\beta$ (resp.~$\beta'$) be a path with values in $X$ in the class of $e$ (resp.~$e'$) and such that
$\beta(0)=x_0$, $\beta(1)=x_1$ (resp.~$\beta'(0)=x_0$, $\beta'(1)=x_1$).
By definition, $\pi(e)=\pi(e')$ if and only if the exists an element
$\{\gamma_x\}_{x\in X}\in\widehat{N}$
such that the path
$\gamma_{x_0}^{-1} * \beta * \gamma_{x_1}$ is homotopic relative to the endpoints to $\beta'$, or, equivalently, 
the path $\gamma_{x_0}^{-1}*\beta*\gamma_{x_1}*(\beta')^{-1}$ is null-homotopic relative to $x_0$. It is readily seen that this is in turn equivalent to the fact that
the class of the path $\beta* (\beta' )^{-1}$ is in $N_{x_0}$, and this concludes the proof. 
\end{proof}

\begin{lemma}\label{sameprojection}
 Let $\gamma,\gamma'$ be pointed non-degenerate simplicial paths in $\calA(X)$ of the same length, i.e.~let
 $\gamma,\gamma'\colon (S,v_0)\to (\calA(X),x_0)$ be non-degenerate simplicial maps, where
 $S$ is a multicomplex homeomorphic to $S^1$.
 If $\pi\circ \gamma=\pi\circ \gamma'$, then (the geometric realization of) $\gamma'*\gamma^{-1}$ lies in $N_{x_0}$.
 \end{lemma}
\begin{proof}
Let us fix a circular ordering $v_0,\ldots,v_{n-1}$ on the vertices of $S$. Since $\pi\circ\gamma=\pi\circ\gamma'$ and the projection $\pi$ induces
an identification between the $0$-skeleton of $\calA(X)$ and the $0$-skeleton of the quotient, we can set $x_i=\gamma(v_i)=\gamma'(v_i)$.
For every $i=0,\ldots,n-1$ let $e_i=\gamma([v_i,v_{i+1}])$ (resp.~$e'_i=\gamma'([v_i,v_{i+1}])$), where we take indices mod $n$. Let us also orient $e_i$ and $e_i'$ from
$x_i$ to $x_{i+1}$. 

By Lemma~\ref{charinN}, for every $i=0,\ldots,n-1$ the loop $e_i'*e_i^{-1}$ is in $N$, so for every $i=0,\ldots,n-1$ also the loop
$$
\alpha_i=e_0*e_1*\ldots*e_{i-1}*(e_i'*e_i^{-1})*e_{i-1}^{-1}*\ldots* e_0^{-1}
$$
is in $N$. But the concatenation
$$
\alpha_0*\alpha_1*\ldots*\alpha_{n-1}
$$
defines the same element of $\pi_1(|\calA(X)|,x_0)$ as the loop $\gamma'*\gamma^{-1}$. This concludes the proof.
\end{proof}

\begin{lemma}\label{trianglesinN}
 Let $\gamma$ be a triangular loop in $\calA(X)$, i.e.~a simplicial embedding $\gamma\colon \partial \Delta^2\to \calA(X)$, and suppose that $\pi\circ\gamma$ bounds a $2$-dimensional simplex of $\calA(X)/\widehat{N}$. Then
$\gamma$ is in $\widehat{N}$.
\end{lemma}
\begin{proof}
By definition, the hypothesis implies that there exists a triangular loop $\gamma'\colon \partial \Delta^2\to \calA(X)$ which bounds a $2$-dimensional simplex of $\calA(X)$ and is such that
$\pi\circ\gamma'=\pi\circ\gamma$. By Lemma~\ref{sameprojection} we now obtain that $\gamma*(\gamma')^{-1}$ is in $\widehat{N}$. But $\gamma'$ is homotopically trivial in $|\calA(X)|$,
so $\gamma$ is in $\widehat{N}$.
\end{proof}

We are now ready to compute the fundamental group of $|\calA(X)/\widehat{N}|$.

\begin{prop}\label{fund:group:quotient}
 Let $x_0$ be a vertex of $\calA(X)$. Then the projection $\pi\colon \calA(X)\to \calA(X)/\widehat{N}$
 induces an epimorphism
 $$
 \pi_*\colon \pi_1(|\calA(X)|,x_0)\to \pi_1(|\calA(X)/\widehat{N}|,x_0)
 $$
 with kernel 
 $$
 \ker \pi_*= N_{x_0}\ .
 $$
 In particular, $\pi_1(|\calA(X)/\widehat{N}|,x_0)$ is canonically isomorphic to $\pi_1(X,x_0)/N_{x_0}\cong G/N$. 
\end{prop}
\begin{proof}
 The fact that $\pi_*$ is surjective easily follows from the fact that each element of $\pi_1(|\calA(X)/\widehat{N}|,x_0)$ is represented
 by a simplicial loop contained in $\calA(X)/\widehat{N}$, toghether with the fact that each such simplicial loop lifts to a simplicial loop in $\calA(X)$.
 
Let now $g\in N_{x_0}$. Since $\calA(X)$ is large and complete, by Theorem~\ref{complete:special} we have $g=[e'*e^{-1}]$, where $e$ and $e'$ are two edges of $\calA(X)$ with endpoints $x_0,x_1$
 (and $x_1$ is some vertex of $\calA(X)$ distinct from $x_0$). By Lemma~\ref{charinN} we have $\pi(e)=\pi(e')$, and this implies at once that $g$ lies in the kernel of $\pi_*$.

On the other hand, let now $g$ be an element in $\ker\pi_*$. We may suppose that $g$ is represented by a non-degenerate simplicial loop $\gamma\colon (S,v_0)\to (\calA(X),x_0)$, where $S$ is a multicomplex
 homeomorphic to $S^1$. The combinatorial description of the fundamental group of simplicial complexes obviously applies also to multicomplexes, so from the fact that
 $\gamma'=\pi\circ \gamma$ is null-homotopic in $|\calA(X)/\widehat{N}|$ we deduce that there exists a finite sequence of non-degenerate simplicial loops $\gamma'_0,\ldots,\gamma'_i,\ldots,\gamma_k'=\gamma'$ 
 such that $\gamma_0'$ is the constant path at $x_0$ and $\gamma_{i+1}'$ is obtained from $\gamma_i'$ via one
 of the following moves:
 \begin{enumerate}
  \item removal from $\gamma_i'$ of a subpath of the form $e*e^{-1}$, where $e$ is an edge of $\calA(X)/\widehat{N}$, or viceversa;
  \item if $e_0,e_1,e_2$ are the edges of a $2$-simplex of $\calA(X)/\widehat{N}$, replacement of a subpath of $\gamma_i'$ of the form
  $e_0*e_1$ with the remaining edge $e_2$ (endowed with the obvious orientation), or viceversa.
 \end{enumerate}
Let us now prove that for every $i=0,\ldots, k$, there exists a lift $\alpha_i$ of $\gamma_i'$ to $\calA(X)$ such that $\alpha_i$ belongs to $N_{x_0}$. 
We argue by induction on $i=0,\ldots,k$, the case $i=0$ being obvious. So suppose that $\gamma_i'$ lifts to a loop $\alpha_i$ in $N_{x_0}$.

If $\gamma_{i+1}'$ is obtained from $\gamma_i'$ by adding  a subpath of the form $e*e^{-1}$, then we can simply add to $\alpha_i$ a subpath of the form $f*f^{-1}$,
where $f$ is an edge lifting $e$, thus obtaining a lift $\alpha_{i+1}$ of $\gamma_{i+1}'$ still lying in $N_{x_0}$. Let us now suppose that
$\gamma_{i+1}'$ is obtained from $\gamma_i'$ by removing a subpath of the form $e*e^{-1}$, and let $f_1*f_2$ be the subpath of $\alpha_i$ which lifts
$e*e^{-1}$. We obviously have $\pi(f_1)=\pi(f_2)=e$, so from Lemma~\ref{charinN} we deduce that $f_1*f_2$ is a loop in $\widehat{N}$. 
Therefore, we can remove the path $f_1*f_2$ from $\alpha_i$, thus obtaining a lift $\alpha_{i+1}$ of $\gamma_{i+1}'$ still lying in $N_{x_0}$.
A similar argument (using now Lemma~\ref{trianglesinN}) describes how to construct a lift of $\gamma_{i+1}'$ in $N_{x_0}$ from a lift of $\gamma_i'$ in $N_{x_0}$ when
$\gamma_{i+1}'$ is obtained from $\gamma_i'$ via a move of type (2).

We have thus shown that our loop $\pi\circ\gamma=\gamma'=\gamma_k'$ lifts to a loop $\overline{\gamma}$ in $N_{x_0}$. By construction we have $\pi\circ\gamma=\pi\circ\overline{\gamma}$,
so Lemma~\ref{sameprojection} ensures that the concatenation $\gamma*\overline{\gamma}^{-1}$ also lies in $N_{x_0}$. Therefore, $\gamma$ lies in $N_{x_0}$ too, which concludes the proof.
 \end{proof}

We are now ready to prove the following:

\begin{thm}\label{quoziente:che:vogliamo}
The multicomplex  $\calA(X)/\widehat{N}$ is complete, minimal and aspherical.
Therefore, the CW complex $|\calA(X)/\widehat{N}|$ is a $K(G/N,1)$-space.
\end{thm}
\begin{proof}
 We prove that $\calA(X)/\widehat{N}$ satisfies the conditions described in Proposition~\ref{aspherical:char}, which we recall here for the convenience of the reader:
 \begin{enumerate}
  \item For every pair of distinct vertices $x_0,x_1$ of $\calA(X)/\widehat{N}$ and every continuous path $\gamma\colon [0,1]\to |\calA(X)/\widehat{N}|$ with $\gamma(0)=x_0$, $\gamma(1)=x_1$, there exists a unique
  simplicial embedding $\gamma'\colon \Delta^1\to \calA(X)/\widehat{N}$ which is homotopic to $\gamma$ relative to the endpoints.
  \item Let  $n\geq 2$, let $(\Delta^n)^1$ be the $1$-skeleton of $\Delta^n$  
  and let $f\colon (\Delta^n)^1\to \calA(X)/\widehat{N}$ be a simplicial embedding such that the restriction of $f$ to each triangular loop is null-homotopic. Then there exists
  a unique simplicial embedding $\overline{f}\colon \Delta^n\to \calA(X)/\widehat{N}$ extending $f$.
 \end{enumerate}
So let $x_0,x_1$ and  $\gamma$ be as in (1). 
Using that the map $$\pi_*\colon \pi_1(|\calA(X)|,x_0)\to \pi_1(|\calA(X)/\widehat{N}|,x_0)$$ is surjective it is not difficult to show that there exists a path $\widetilde{\gamma}\colon [0,1]\to |\calA(X)|$
such that $\widetilde{\gamma}(0)=x_0$, $\widetilde{\gamma}(1)=x_1$ and $\gamma$, $\pi\circ \widetilde{\gamma}$ lie in the same homotopy class relative to the endpoints. By completeness
of $\calA(X)$, there exists an edge $e$ of $\calA(X)$ homotopic to $\widetilde{\gamma}$ relative to the endpoints, and by construction the simplex $\pi(e)$
provides an edge of $\calA(X)/\widehat{N}$ which is homotopic to $\gamma$ relative to the endpoints. 
 
Suppose now that $e,e'$ are homotopic edges of $\calA(X)/\widehat{N}$ with endpoints $x_0,x_1$, and let $\widetilde{e},\widetilde{e}'$ be edges of $\calA(X)$ such that $\pi(\widetilde{e})=e$, $\pi(\widetilde{e}')=e'$. Since $e,e'$ are homotopic,
if we orient $\widetilde{e},\widetilde{e}'$ from $x_0$ to $x_1$, then
the projection $\pi\circ (\widetilde{e}'*\widetilde{e}^{-1})$ is null-homotopic. By Proposition~\ref{fund:group:quotient}, this implies that $\widetilde{e}'*\widetilde{e}^{-1}$ lies in $N$, so by Lemma~\ref{charinN} $\pi(\widetilde{e}')=\pi(\widetilde{e})$, i.e.~$e=e'$. 
This concludes the proof of (1).

Let now $f\colon (\Delta^n)^1\to \calA(X)/\widehat{N}$ be as in (2). We fix an ordering $v_0,\ldots,v_n$ of the vertices of $\Delta^n$, and we set $x_i=f(v_i)$. We also denote by $e_{i,j}$ the
edge $f([v_i,v_j])$, oriented from $x_i$ to $x_j$. We first fix an arbitrary lift $\widetilde{e}_{i,i+1}\subseteq \calA(X)$ of $e_{i,i+1}$ for every $i=0,\ldots,n-1$. Then for every $i<j$ we define $\widetilde{e}_{i,j}$
as the unique edge of $\calA(X)$ homotopic relative to the endpoints to the concatenation $\widetilde{e}_{i, i+1}*\ldots*\widetilde{e}_{j-1, j}$ (such an edge exists and is unique by completeness an minimality of $\calA(X)$). We also define
$\widetilde{f}\colon (\Delta^n)^1\to \calA(X)$ by setting $f([v_i,v_j])=\widetilde{e}_{i,j}$ for every $i<j$.  We now prove 
 that $\pi\circ\widetilde{f}=f$. In fact, for every $i<j$ we show
by induction on $k$ that 
$\pi \circ \widetilde{f}([v_i, v_k]) = f([v_i, v_k])$ for every  $k= i+1, \ldots, j$. 
The claim holds by construction for $k=i+1$. Assume that it holds for a fixed $i+1\leq k<j$ and consider the edge $\widetilde{f}([v_i, v_{k+1}]) = \widetilde{e}_{i, k+1}$. By hypothesis $\gamma = \widetilde{e}_{i, k} \ast \widetilde{e}_{k, k+1} \ast \widetilde{e}_{i, k+1}^{-1}$ is a null-homotopic loop in $\calA(X)$. Thus, $\pi \circ \gamma$ is also null-homotopic. This implies that $\pi(\widetilde{e}_{i, k} \ast \widetilde{e}_{k, k+1}) = e_{i, k} \ast e_{k, k+1}$ is a path homotopic to $\pi(\widetilde{e}_{i, k+1})$ relative to its endpoints. Since $\calA(X) /\widetilde{N}$ satisfies (1), the concatenation $e_{i, k} \ast e_{k, k+1}$ can be homotoped relative to its endpoints only to the $1$-simplex $e_{i, k+1}$ and so $\pi(\widetilde{e}_{i, k+1})$ must coincide with $e_{k, k+1}$. This proves that $\pi \circ \widetilde{f}([v_i, v_{k+1}]) = f([v_i, v_{k+1}])$, whence the claim.
Now by construction, the restriction of $\widetilde{f}$ to each triangular loop is null-homotopic in $\calA(X)$,
so by Proposition~\ref{aspherical:char} (applied to $\calA(X)$) there exists a unique simplicial embedding $h\colon \Delta^n\to \calA(X)$ extending $\widetilde{f}$. The composition $\overline{f}=\pi\circ h$ provides the desired extension
of $f$ to a simplicial embedding of $\Delta^n$ into $\calA(X)/\widehat{N}$. 

In order to show that the extension is unique, let $\overline{f}'\colon \Delta^n\to \calA(X)/\widehat{N}$ be an arbitrary extension of $f$, and observe that $\overline{f}'=\pi\circ h'$ for some embedding
$h' \colon \Delta^n \to \mathcal{A}(X)$. Let $e'_{i,j}=h'([v_i,v_j])$, and observe that the argument above shows that $h$ (resp.~$h'$) is uniquely determined by the edges $e_{i,i+1}$ (resp.~$e'_{i,i+1}$), $i=0,\ldots,n-1$.
Therefore, if $g\in \widehat{N}$ is such that $g(e_{i,i+1})=e'_{i,i+1}$ for every $i=0,\ldots,n-1$, then  $h'=g\circ h$, whence $\overline{f}'=\pi\circ h'=\pi\circ h=\overline{f}$, which concludes the proof.

We are thus left to find an element $g\in \widehat{N}$ is such that $g(e_{i,i+1})=e'_{i,i+1}$ for every $i=0,\ldots,n-1$. Since $\pi(e_{i,i+1})=\pi(e'_{i,i+1})$, for every $i$ 
the edges $e_{i,i+1}$ and $e'_{i,i+1}$ are $\widehat{N}$-equivalent. By repeatedly using Lemma~\ref{actionNhat} (with $g_0=1$ when $i=1$), we can then find elements $g_i\in N_{x_i}$, $i=1,\ldots,n$,
such that the following conditions hold: 
\begin{align*}
g_1\cdot e_{0,1}& =e'_{0,1},\\ (g_1,g_2)\cdot e_{1,2}&=e'_{1,2},\\ \ldots  \\(g_{n-1},g_n)\cdot e_{n-1,n}&=e'_{n-1,n}\ .
\end{align*}
Let now $g=(g_1,\ldots,g_n)\in \widehat{N}=\bigoplus_{x\in X} N_x$. By construction we have $g\cdot e_{i,i+1}=e'_{i,i+1}$ for every $i=0,\ldots,n$, and this concludes the proof.
\end{proof}


\begin{rem}
 The arguments developed in this section apply more in general to any complete, minimal and aspherical multicomplex $A$. Indeed, if $A$ is such a multicomplex and $N$ is a normal subgroup
 of $G=\pi_1(|A|,x_0)$ for some vertex $x_0\in V(A)$ of $A$, then one can define as above a subgroup $\widehat{N}$ of $\bigoplus_{v\in V(A)} \pi_1(|A|,v)$. This group
 naturally acts via simplicial isomorphisms on $A$, and the quotient $A/\widehat{N}$ is a minimal, complete and aspherical multicomplex with fundamental group isomorphic to $G/N$.
\end{rem}

\section{Proof of Gromov's Mapping Theorem}
We are now ready to conclude the proof of Gromov's Mapping Theorem
\ref{mapping-intro}. Let $f \colon X_{1} \rightarrow X_{2}$ be a continuous map between path connected  topological spaces, and suppose that the map induces an epimorphism 
$f_*\colon \pi_1(X_1)\to \pi_1(X_2)$ 
with an amenable kernel.

Let us first observe that it is sufficient to consider the case when $X_1,X_2$ are good topological spaces. In fact,
if $\calS(X_i)$ is the simplicial set associated to $X_i$ and $j_i\colon |\calS(X_i)|\to X_i$ is the natural projection, then the map
$f$ induces a (simplicial) map $\calS(f)\colon |\calS(X_1)|\to |\calS(X_2)|$ such that $j_2\circ \calS(f)=f\circ j_1$. 
Since $j_1$ and $j_2$ are weak homotopy equivalences, also the map $\calS(f)$ induces an empimorphism with an amenable kernel on fundamental groups. Moreover,
thanks to Theorem~\ref{weak:iso:thm}, the vertical arrows of the commutative diagram
$$
\xymatrix{ 
H^n_b(|\calS(X_2)|)\ar[rr]^{H^n_b(\calS(f))} & & H^n_b(|\calS(X_1)|)\\
H^n_b(X_2)\ar[rr]^{H^n_b(f)} \ar[u]^{H^n_b(j_2)} & & H^n_b(X_1)\ar[u]^{H^n_b(j_1)}
}
$$
are isometric isomorphisms for every $n\in\mathbb{N}$. Thus, up to replacing $X_i$ with $\calS(X_i)$ for $i=1,2$, we may assume that $X_1$ and $X_2$ are good.

Recall from Proposition~\ref{commutative:homotopy} that 
there exists a continuous map
 $K(f)\colon |\calK(X_1)|\to |\calK(X_2)|$ such that the following diagram commutes up to homotopy:
 $$
 \xymatrix{
 |\calK(X_1)|\ar[r]^{K(f)}\ar[d]_{S_{X_1}} & |\calK(X_2)|\ar[d]^{S_{X_2}}\\
 X_1 \ar[r]^f & X_2\ .
 }
 $$
 ($S_{X_i}$ denotes the natural projection of $|\calK(X_i)|$ onto $X_i$).
Recall from Corollary~\ref{iso-KX} that the maps $S_{X_1}$ and $S_{X_2}$ induce isometric isomorphisms on bounded cohomology. Moreover, homotopic maps induce the same morphism on bounded
cohomology, so it is sufficient to show that the map $K(f)$ induces an isometric isomorphism on singular bounded cohomology. Let $\widetilde{\pi}_i = \pi_i \circ r_i \colon |\calK(X_i)|\to |\calA(X_i)|$ be the canonical projection
of $\calK(X_i)$ on its aspherical quotient, and recall that $\widetilde{\pi}_i$ induces an isomorphism on fundamental groups for $i=1,2$ (see Theorems~\ref{exist-min} and~\ref{aspherical:thm}).
Using the asphericity of $|\calA(X_2)|$ we can construct a map 
$A(f)\colon |\calA(X_1)|\to |\calA(X_2)|$ which (under the identifications $\pi_1(|\calK(X_i)|)\cong \pi_1(|\calA(X_i)|)$) induces the same morphism as $K(f)$ on fundamental groups. The general theory of Eilenberg-MacLane
spaces now ensures that the diagram 
$$
 \xymatrix{
 |\calK(X_1)|\ar[r]^-{K(f)}\ar[d]_{\widetilde{\pi}_{1}} & |\calK(X_2)|\ar[d]^{\widetilde{\pi}_{2}}\\
 |\calA(X_1)| \ar[r]_-{A(f)} & |\calA(X_2)| 
 }
 $$
 commutes up to homotopy (see e.g.~\cite[Proposition 1B.9]{hatcher}). Moreover, we know from Theorem~\ref{teoD} that the projections $\widetilde{\pi}_i$, $i=1,2$, induce isometric isomorphisms on bounded cohomology, so
 we are left to show that $A(f)$ induces an isometric isomorphism on bounded cohomology. Let now $N<\pi_1(X_1)$ be the kernel of the map $f_*\colon \pi_1(X_1)\to \pi_1(X_2)$, and denote by
 $\pi\colon |\calA(X_1)|\to |\calA(X_1)/\widehat{N}|$ the projection introduced above. 
 Under the identifications $\pi_1(X_1)\cong \pi_1(|\calK(X_1)|))\cong \pi_1(|\calA(X_1)|)$, we have 
 \begin{align*}
 N&\cong\ker \big( \pi_*\colon \pi_1(|\calA(X_1)|)\to \pi_1(|\calA(X_1)/\widehat{N}|)\big)\, ,\\
 N&\cong\ker \big( A(f)_*\colon \pi_1(|\calA(X_1)|)\to \pi_1(|\calA(X_2)|)\big) \ .
 \end{align*}
 Since both $|\calA(X_2)|$ and $|\calA(X_1)/\widehat{N}|$ are aspherical, we can then construct a homotopy equivalence $h\colon |\calA(X_1)/\widehat{N}|\to |\calA(X_2)|$ such that the following diagram commutes up to homotopy
$$
\xymatrix{
|\calA(X_1)|\ar[rr]^{A(f)} \ar[rd]_{\pi} & & |\calA(X_2)|\\
& |\calA(X_1)/\widehat{N}|\ar[ur]_{h}
}\ .
$$
Recall now from Theorem~\ref{actionPi} that for every $g\in \widehat{N}$ the action of $g$ on $|\calA(X_1)|$ is homotopic to the identity. Therefore, a straightforward application of Corollary~\ref{cor-amen-action}
implies that the map $\pi$ induces an isometric isomorphism on bounded cohomology. Since homotopy equivalences always induce isometric isomorphisms on bounded cohomology, the same is true also
for $h$, whence for $A(f)$. This concludes the proof of Gromov's Mapping Theorem.

\chapter{The Vanishing Theorem}\label{vanishing-thm:chap}

The Vanishing Theorem establishes a criterion for the vanishing of the comparison map (whence of the simplicial volume, when applied to compact manifolds), in terms of amenable covers of the space.
Just as the Mapping Theorem, it was stated by Gromov in~\cite{Grom82} without any assumption on the topology of the space involved, and proved (with a slightly different formulation, see Remark~\ref{vanishing:literature}) by 
Ivanov in~\cite{Ivanov} (for countable CW complexes)
and in~\cite{ivanov3} (in the general case).


\section{Amenable covers, the Nerve Theorem and the Vanishing Theorem}\label{amenable:vanishing:sec}
We first introduce the fundamental notion of amenable subset of a given topological space.

\begin{Definizione}
Let $X$ be a topological space and let $i\colon Y\hookrightarrow X$ be the inclusion of a subset $Y$ of $X$. Then
$Y$ is \textit{amenable} (in $X$) if for every path-connected component $Y'$ of $Y$ the image of 
$i_{*} \colon \pi_{1}(Y') \rightarrow \pi_{1}(X)$ is an amenable subgroup of $\pi_{1}(X)$.
\end{Definizione}

Let us now fix a topological space $X$ and let $\calU=\{U_i\}_{i\in I}$ be a cover of $X$, i.e.~suppose that $U_i\subseteq X$ for every $i\in I$ and that $X=\bigcup_{i\in I} U_i$. We say that the cover
is \emph{open} if each $U_i$ is open in $X$, and \emph{amenable} if each $U_i$ is amenable in $X$. The multiplicity of $\calU$ is defined by
\begin{align*}
\mult(\calU)&=\sup \left\{n\, |\, \exists\ i_1,\ldots,i_n\in I,\, i_h\neq i_k\ \textrm{for}\ h\neq k,\, U_{i_1}\cap\ldots\cap U_{i_n}\neq\emptyset\right\}\\ &\in\ \mathbb{N}\cup \{\infty\}\ .
\end{align*}

To any cover $\calU=\{U_i\}_{i\in I}$ of $X$ there is associated a simplicial complex $N(\calU)$, called the \emph{nerve} of the cover,
which is defined as follows: the set of vertices of $N(\calU)$ is $I$, and $n+1$ elements $i_0,\ldots, i_n$ of $I$ span a simplex of $N(\calU)$ if and only if 
$$
U_{i_0}\cap\ldots\cap U_{i_n} \neq \emptyset \ .
$$
By definition, $\mult(\calU)=1+\dim N(\calU)$. As usual, we denote by $H^n(N(\calU))$ the simplicial cohomology of $N$ with real coefficients.

Suppose now that $X$ is paracompact, and let $\Phi=\{\varphi_i\}_{i\in I}$ be a partition of unity subordinate to the cover $\calU$. Then the map
$$
f_\Phi\colon X\to |N(\calU)|\, ,\qquad f_\Phi (x)=\sum_{i\in I} \varphi_i(x) \cdot i
$$
is  well defined and continuous. Moreover, if $\Phi'$ is another partition of unity subordinate to $\calU$, then for every $t\in [0,1]$ the convex combination
$$tf_\Phi+(1-t)f_{\Phi'}\colon X\to |N(\calU)|$$ is well defined. As a consequence, $f_\Phi$ and $f_{\Phi'}$ are homotopic, and they induce the same map
$$
\beta^*\colon H^n(|N(\calU)|)\to H^n(X) 
$$
in singular cohomology.  

This chapter is devoted to the proof of the following two results, the first of which was  anticipated in the introduction. We refer the reader to Remark~\ref{vanishing:literature} for a brief discussion of similar statements that are already available in the literature. 
Recall that a topological space is triangulable if it is homeomorphic to the geometric realization of a simplicial complex. 

\begin{Teorema}[Nerve Theorem]\label{van-thm2}
Let $X$ be a triangulable space and let  $\calU$ be an amenable open cover of $X$. Also suppose that, for every finite subset $I_0\subseteq I$,
the intersection $\bigcap_{i\in I_0} U_{i}$ is path connected (possibly empty).
Then for every $n\in\mathbb{N}$ there exists a map $\Theta^n\colon H^n_b(X)\to H^n(N(\calU))$ such that the following diagram commutes:
$$
\xymatrix{
H^n_b(X)\ar[rr]^{c^n}\ar[d]_{\Theta^n} & & H^n(X)\\
 H^n(N(\calU)) \ar[rr]^{\cong} & & H^n(|N(\calU)|)\ar[u]_{\beta^n}\ .
}
$$
\end{Teorema}
In particular, under the hypothesis of Theorem~\ref{van-thm2}, if $n\geq \mult(\calU)$, then the comparison map $c^n\colon H^n_b(X)\to H^n(X)$ vanishes.
In fact, the same conclusion holds even without any assumption on the topological space $X$ or on the connectedness of the $U_i$:

\begin{thm:repeated:vanishing1_intro}[Vanishing Theorem]
Let $X$ be a topological space and let  $\calU$ be an amenable open cover of $X$. Then for every $n\geq \mult(\calU)$ (or, equivalently, for every $n>\dim N(\calU)$)
the comparison map
$$
c^n\colon H^n_b(X)\to H^n(X)
$$
vanishes.
\end{thm:repeated:vanishing1_intro}

\begin{rem}\label{vanishing:literature} 
Theorem~\ref{vanishing1_intro} is originally due to Gromov~\cite[page 40]{Grom82}, while Theorem~\ref{van-thm2} was proved (in a slightly different version) by Ivanov in~\cite[Theorem 6.2]{Ivanov} and in~\cite[Theorem 9.1]{ivanov3} 
(with  weaker hypotheses than ours on the topology of $X$). Ivanov's argument is completely different from ours,
and it  is based on the use of a variation of the Mayer--Vietoris double complex for singular cohomology. As a consequence, it is not clear at first sight whether the factoring map
$H^n(|N(\calU)|)\to H^n(X)$ decribed by Ivanov coincides with the canonical map $\beta^n$ introduced here. Our Theorem~\ref{van-thm2}  suggests that this should indeed be the 
case, and a complete proof of this fact may be found in~\cite{Maffei}.

No assumption on the connectedness of finite intersections of elements of $\calU$ appears in Theorem~\ref{vanishing1_intro}, hence one may wonder whether this requirement could be removed
from the hypotheses of Theorem~\ref{van-thm2}. The following example shows that this is not the case.
Indeed, let $X$ be a finite simplicial complex of dimension $n$. One can construct an amenable 
cover $\calU=\{U_0,\ldots,U_n\}$ by requiring that, for every $i=0,\ldots,n$, the set $U_i$ is the disjoint union of contractible open neighbourhoods of the interiors of the $i$-simplices of $X$, chosen in such a way that
$U_i\cap X^j=\emptyset$ for every $j<i$. This cover is obviously amenable, and $N(\calU)$ is isomorphic to the $n$-simplex $\Delta^n$. Therefore, $H^i(|N(\calU)|)=0$ for every $i>0$, and if Theorem~\ref{van-thm2} were
true for the cover $\calU$, then we would conclude that the comparison map $H^i_b(X)\to H^i(X)$ is null for every $i>0$, which is clearly false in general. Be aware that the assumption that finite intersections of elements
of $\calU$ be connected seems to be missing in~\cite[Theorem 6.2]{Ivanov}, \cite[Theorem 9.1]{ivanov3}.
\end{rem}

Via duality, Theorem~\ref{vanishing1_intro} implies the vanishing of the simplicial volume for closed manifolds admitting amenable covers of small multiplicity:

\begin{thm:repeated:van-cor-intro}
Let $X$ be a topological space admitting an open amenable cover of multiplicity $m$, and let $n\geq m$. Then
$$
\|\alpha\|_1=0
$$
for every $\alpha\in H_n(X)$. 
In particular, if
$M$ is a closed  oriented manifold admitting an open amenable cover $\calU$ such that 
$\mult(\calU)\leq \dim M$,
then
$$
\|M\|=0\ .
$$
\end{thm:repeated:van-cor-intro}
\begin{proof}
Just combine Theorem~\ref{vanishing1_intro} and Corollary~\ref{vanish-cor}.
\end{proof}

We will strengthen the second statement of the previous corollary in Section~\ref{invisible:sec}. In fact, in Theorem~\ref{invisible:new-intro} we will show that, 
 if
 $M$ is a closed  oriented manifold admitting an open amenable cover $\calU$ such that 
$\mult(\calU)\leq \dim M$,
 then $M$ is $\ell^1$-invisible (see Definition~\ref{invisible:def}).
This implies in particular that $\|M\|=0$.

\section{Amenable subgroups of $\Pi(X,X)$ and their action on $\calA(X)$}\label{am:sub:sec}
Before going into the proofs of the Vanishing Theorem and the Nerve Theorem, we need to collect some preliminary results.
We refer the reader to Section~\ref{sec:mapping:pi(U,V)} for the definition of the group $\Pi(X,X_0)$, where $X_0$ is any subset of $X$.

\begin{Definizione}
For any set $X$, we denote by $\simf(X)$ 
the group of permutations of $X$ with finite support, i.e.~the subgroup of the symmetric group $\mathfrak{S}(X)$ given by those permutations which fix all but finitely many elements of $X$. 
\end{Definizione}

The following proposition readily follows from the definitions.
\begin{Proposizione}\label{morph-pi-to-permfinite}
Let $X$ be a topological space and let $X_{0} \subset X$ be any subset. Then, there exists an exact sequence of groups
$$
\xymatrix{
1\ar[r] & \bigoplus_{x \in \, X_{0}} \pi_{1}(X, x) \ar[r] & \Pi(X,X_0) \ar[r]^{\phi} & \simf(X_0)\ ,
}
$$
where the homomorphism $\phi\colon \Pi(X,X_0)\to \simf(X_0)$ associates to the class of an element $\{\gamma_x\}_{x\in X_0}\in\Omega(X,X_0)$ the permutation $x\mapsto \gamma_x(1)$. 

If $X_0$ is contained in a path connected component of $X$, then the map $$\phi\colon \Pi(X, X_{0}) \twoheadrightarrow \simf(X_{0})$$ is surjective, so the sequence above may be extended to a short
exact sequence
$$
\xymatrix{
1\ar[r] & \bigoplus_{x \in \, X_{0}} \pi_{1}(X, x) \ar[r] & \Pi(X,X_0) \ar[r]^{\phi} & \simf(X_0)\ar[r] & 1\  .
}
$$
\end{Proposizione}

If $V\subseteq U\subseteq X$, then the inclusion $(U,V)\hookrightarrow (X,X)$ induces a group homomorphism
$\Pi(U,V)\to \Pi(X,X)$, which is injective if and only if every path connected component of $U$ intersecting $V$ is $\pi_1$-injective in $X$.
In what follows a lot of attention will be payed to the image
of this homomorphism, which we denote henceforth by $\Pi_X(U,V)$. Of course, since $\Pi(X,X)$ acts on $\calA(X)$, also $\Pi_X(U,V)$ does.

\begin{lemma}\label{prop-psi-pi-u-v-amenable}
Let $U$ be an amenable subset of $X$, and take any subset $V\subseteq U$. Then, the subgroup $\Pi_X(U,V)<\Pi(X,X)$ is amenable.
\end{lemma}
\begin{proof}
Recall that every element $g\in \Pi(X,X)$ induces a permutation with finite support $\phi(g)\in\simf(X)$. It is then easy to check that the exact sequence of 
Proposition~\ref{morph-pi-to-permfinite} induces the exact sequence
\begin{displaymath}
1 \rightarrow \bigoplus_{y \in \, V} \mathrm{Im}(\pi_{1}(U, y) \rightarrow \pi_{1}(X, y)) \xrightarrow{j} \Pi_X(U,V) \xrightarrow{\phi} \simf(V) ,
\end{displaymath}
where $j$ is the obvious inclusion (the map $\phi$ is surjective if and only if $V$ is contained in a path connected component of $U$). 
 
Every finitely generated subgroup of $\simf(V)$ is finite, hence amenable. Since locally amenable groups are amenable, this implies that $\simf(V)$ is itself amenable. 
 Moreover, being a direct sum of amenable groups, also the group
\begin{displaymath}
\bigoplus_{y \in \, V} \mathrm{Im}(\pi_{1}(U, y) \rightarrow \pi_{1}(X, y))
\end{displaymath}
is amenable. We have thus described $\Pi_X(U,V)$ as an extension of an amenable group by an amenable group, and this concludes the proof.
\end{proof}

%
%

Let now $X$ be a topological space admitting a triangulation $T$, i.e.~assume that $X$ is equal to the geometric realization $|T|$ of a simplicial complex $T$.
For every vertex $v$ of $T$, the closed star of $v$ in $T$ is defined as the subcomplex of $T$ containing all the simplices containing $v$ (and all their faces). 
By suitably subdividing $T$ we may suppose that the following condition holds (see for instance \cite[Theorem~16.4]{munkres}): for every vertex $v$ of $T$ there exists $i(v)\in I$ such that the closed star of $v$ in $T$ 
is contained in the 
element $U_{i(v)}$ of the cover (of course, the choice of $i(v)$ may be non-unique).

For every $i\in I$ we set 
$$
V_i=\{v\in V(T)\, |\, i(v)=i\}\ .
$$
By construction we have $V_i\subseteq U_i$. Let us now set 
$$
\G=\bigoplus_{i\in I} \Pi_X(U_i,V_i)\ .
$$
The direct sum of amenable groups is amenable, so Lemma~\ref{prop-psi-pi-u-v-amenable} implies that $\G$ is amenable. Also observe that, if $i\neq j$, then $V_i\neq V_j$, so elements
in $\Pi_X(U_i,V_i)$ commute with elements in $\Pi_X(U_j,V_j)$. As a consequence, the group $\G$ naturally sits in $\Pi(X,X)$ as a subgroup, and acts on $\calA(X)$.

We will now construct a copy of $T$ inside the multicomplex $\calA(X)$. 
The multicomplex $\calK(X)$ contains a submulticomplex 
 $\calK_T(X)\cong T$ whose simplices
are the equivalence classes of the affine parametrizations of simplices of $T$. When constructing the submulticomplex $\calL(X)$ of $\calK(X)$, one needs to choose
a representative for every homotopy class (relative to the boundary) of simplices of $\calK(X)$ (see the proof of Theorem~\ref{exist-min}). Of course, we may choose the simplices
of $\calK_T(X)$ as representatives of their homotopy classes, thus obtaining $\calK_T(X)\subseteq \calL(X)$. Now it is obvious that the quotient map 
$\calL(X)\to \calA(X)$ is injective on $\calK_T(X)$. As a consequence, we can realize $T\cong \calK_T(X)$ as a submulticomplex of $\calA(X)$.

Recall that classes in $H^n_b(X)$ may be represented by simplicial cocycles
on $\calA(X)$. More precisely, if $\pi\colon \calK(X)\to \calA(X)$ is the composition of the retraction $\calK(X)\to \calL(X)$ with the projection $\calL(X)\to\calA(X)$,
and $S\colon |\calK(X)|\to X$ is the natural projection, then the composition
$$
\xymatrix{
H^n_b(X)\ar[rr]^{H^n_b(S)} & & H^n_b(|\calK(X)|)\ar[rr]^{H^n_b(\phi^n)} & & H^n_b(\calK(X))
}
$$
and the map $H^n_b(\pi)\colon H^n_b(\calA(X))\to H^n_b(\calK(X))$ are isometric isomorphisms for every $n\in\mathbb{N}$
(see Theorems~\ref{isometria-good-K(X)} and~\ref{aspherical:thm}). In order to save words, henceforth we will denote by
$$
\Psi^*\colon H^*_b(\calA(X))\to H^*_b(X)
$$
the isometric isomorphism obtained by composing $H^*_b(\pi)$ with $(H^*_b(\phi^*)\circ H^*_b(S))^{-1}$. 

\begin{lemma}\label{criterion:equal}
 Let $z\in C^n_b(\calA(X))$ be a bounded cocycle which vanishes on
 $C_n(T)\subseteq C_n(\calA(X))$. Then 
 $$
 c^n(\Psi^n([z]))=0\quad \textrm{in}\ H^n(X)\ .
 $$
\end{lemma}
\begin{proof}
The inclusion $T\hookrightarrow \calA(X)$ induces a restriction map $r^*\colon H^*_b(\calA(X))\to H^*_b(T)$. 
It is easy to check that the diagram
$$
\xymatrix{
H^n_b(\calA(X))\ar[r]^{r^n} \ar[d]^{\Psi^n} & H^*_b(T)\ar[r]^{c^n} & H^n(T)\ar[d]^{\cong}\\
H^n_b(X)\ar[rr]^{c^n} & & H^n(X) 
}
$$
is commutative (the vertical arrow on the right is the usual isomorphism between the simplicial cohomology of $T$ and the singular cohomology of $|T|=X$). 
Of course, if $z\in C^n_b(\calA(X))$ vanishes on $C_n(T)$, then  $r^n(z)=0$, and this concludes the proof.
\end{proof}

\section{Proof of the Vanishing Theorem}\label{van1:sec}
Let us first suppose that $X$ is homeomorphic to the geometric realization of a simplicial complex $T$. We keep the notation of the previous section, we
 suppose that $\mult(\calU)=m$, and we take a bounded class
 $\alpha \in H^n_b(X)$. Since $\Psi^*\colon H^n_b(\calA(X))\to H^n_b(X)$ is an isomorphism, we have $\alpha = \Psi^*(\beta)$ for some 
 $\beta\in H^n_b(\calA(X))$.
By Lemma~\ref{criterion:equal}, it suffices to show that $\beta$ admits a representative that vanishes on every algebraic $n$-simplex
of $C_n(T)\subseteq C_n(\calA(X))$. 


Let now $z\in C^n_b(\calA(X))$ be an alternating cocycle representing $\beta$. Recall that we have an action of the amenable group $\G$ on $\calA(X)$ via simplicial 
automorphisms that are homotopic to the identity. By Theorem~\ref{teor-amen-action} we may suppose that $z$ is $\G$-invariant. Therefore, after possibly alternating it,
we can assume that $z\in C^n_b(\calA(X))_\alt^\G$. 

Let $(s,(x_0,\ldots,x_n))$ be 
an algebraic simplex in $C_n(T)\subseteq C_n(\calA(X))$.
If $x_h=x_k$ for some $h\neq k$, then $z(s,(x_0,\ldots,x_n))=0$ because $z$ il alternating, and we are done. We may thus suppose $x_h\neq x_k$ for $h\neq k$, and denote by
$e_{hk}$ the edge of
$s$ joining $x_h$ with $x_k$, $h,k=0,\ldots,n$. Since $n\geq m=\mult(\calU)$, there exist $h,k\in \{0,\ldots,n\}$, $h\neq k$, such that $i(x_h)=i(x_k)$, i.e. both $x_h$ and
$x_k$ belong to the same $V_i$. Now the closed stars of $x_h$ and of $x_k$ in $T$ are contained in $U_i$ by the assumptions on $\calU$, so the edge $e_{hk}$ of $\calA(X)$ 
(which, since $\calA(X)^1=\calL(X)^1$, is also an edge of $\calL(X)\subseteq \calK(X)$) projects via the map $S\colon |\calK(X)|\to X$ onto (the image of) a path
$\gamma\colon [0,1]\to U_i$ such that $\gamma(0)=x_h$, $\gamma(1)=x_k$. 
Let  us now consider the element $g=\{\gamma,\gamma^{-1}\}\in \Pi_X(U_i,V_i)< \G$. It is immediate to check that $g\cdot s=s$. Moreover, $g\cdot x_h=x_k$, $g\cdot x_k=x_h$, and $g\cdot x_l=x_l$ for every $l\neq h,k$.
As a consequence, from the fact that $z$ is alternating we obtain $z(s,(x_0,\ldots,x_n))=-z(g\cdot (s,(x_0,\ldots,x_n)))$, while from the fact that $z$ is $\G$-invariant we get 
 $z(s,(x_0,\ldots,x_n))=z(g\cdot (s,(x_0,\ldots,x_n)))$. This shows that
$z(s,(x_0,\ldots,x_n))=0$, and concludes the proof of the Vanishing Theorem~\ref{vanishing1_intro} under the assumption that $X$ is triangulable.

Let now $X$ be any topological space, and denote by $\calS(X)$ the singular simplicial set associated to $X$. 
Since the second barycentric subdivision of $\calS(X)$ is a simplicial complex, 
Theorem~\ref{vanishing1_intro} holds for $|\calS(X)|$. Therefore, in order to conclude that
Theorem~\ref{vanishing1_intro} holds for $X$ it is sufficient to prove the following claims: if $X$ admits an amenable cover of multiplicity $m$, then also $|\calS(X)|$ admits an amenable cover
of multiplicity $m$;  if the comparison map $c^n\colon H^n_b(|\calS(X)|)\to H^n(|\calS(X)|)$ vanishes,
then also the comparison map
$c^n\colon H^n_b(X)\to H^n(X)$ vanishes.

Let $j\colon |\calS(X)|\to X$ be the natural projection.
If $\calU$ is a open cover of $X$, then $j^{-1}\calU=\{j^{-1}(U),\, U\in\calU\}$ is an open cover of $|\calS(X)|$ such that $\mult(j^{-1}\calU)=\mult(\calU)$.
 Moreover,
 the map $j\colon |\calS(X)|\to X$ induces an isomorphism on fundamental groups, so 
 if $Z$ is a path connected component of $j^{-1}(U)$ for some $U\in\calU$, then from the commutative diagram
 $$
 \xymatrix{
 Z\ar[r]\ar[d]& |\calS(X)|\ar[d]^{j}\\
 U\ar[r] & X
 }
 $$
 we deduce that the image of $\pi_1(Z)$ in $\pi_1(|\calS(X)|)$ is isomorphic to a subgroup of the image of $\pi_1(U)$ in $\pi_1(X)$. In particular,
 the cover $j^{-1}\calU$ is amenable if $\calU$ is. This proves the first claim. 

The diagram 
 $$
 \xymatrix{H^n_b(X)\ar[d]_{H^n_b(j)} \ar[r]^-{c^n} & H^n(X)\ar[d]^{H^n(j)}\\ H^n_b(|\calS(X)|)\ar[r]_-{c^n} & H^n(|\calS(X)|) }
 $$
 commutes. 
The map $j$ is a weak homotopy equivalence, thus Theorem~\ref{weak:iso:thm} ensures that $H^n_b(j)$ is an isometric isomorphism. 
It is a classical result that also $H^n(j)$ is an isomorphism (see e.g.~\cite[Proposition~4.21]{hatcher}). Therefore, the comparison map
$c^n\colon H^n_b(X)\to H^n(X)$ vanishes if and only if also the comparison map $c^n\colon H^n_b(|\calS(X)|)\to H^n(|\calS(X)|)$ vanishes.
This proves the second claim, thus concluding the proof of Theorem~\ref{vanishing1_intro} for any topological space.

\section{Proof of the Nerve Theorem}
Let us now come back to the case when $X=|T|$ is the geometric realization of a simplicial complex $T$. Let $\calU$ be as in the statement
of Theorem~\ref{van-thm2}, and keep the notation from the previous sections.

Let us first construct the required map $\Theta^n\colon H^n_b(X)\to H^n(N(\calU))$. If $\calU'$ is a locally finite refinement of $\calU$, then we have a natural map
$r\colon N(\calU')\to N(\calU)$. Moreover, from a partition of unity $\Phi'$ relative to $\calU'$ we may construct a partition of unity $\Phi$ relative to $\calU$, such that the  map
$f_\Phi \colon X\to N(\calU)$ associated to $\Phi$ is equal to $r\circ f_{\Phi'}$. As a consequence, up to replacing $\calU$ with $\calU'$,
we may assume that $\calU$ is locally finite. 

Using again that $\calU$ is locally finite, up to subdividing $T$ we may suppose that $V_i\neq \emptyset$ for every $i\in I$.
Recall that there exists an isometric isomorphism $\Psi^*\colon H^*_b(\calA(X))\to H^*_b(X)$, and that
the bounded cohomology of $\calA(X)$ is computed by the complex of $\G$-invariant (alternating) cochains, 
where $\G=\bigoplus_{i\in I} \Pi_X(U_i,V_i)$ as above.
Therefore, in order to define $\Theta^*\colon H^*_b(X)\to H^*(N(\calU))$ it is sufficient to construct  a
chain map $\Omega^*\colon C^*_b(\calA(X))^\G_\alt \to C^*(N(\calU))$.

Thus, let $z\in C^n_b(\calA(X))_\alt^\G$ and take an algebraic  $n$-simplex $({i_0},\ldots,{i_n})$ of $N(\calU)$ (recall that $N(\calU)$ is a simplicial complex
with  $I$ as set of vertices). 
If there exist $0\leq j<h\leq n$ with ${i_j}={i_h}$, then we simply set $\Omega^n(z)({i_0},\ldots,{i_n})=0$. Otherwise,
by definition of nerve of a cover we have $\bigcap_{j=0}^n U_{i_j}\neq \emptyset$, hence we may choose a point
$q\in \bigcap_{j=0}^n U_{i_j}$. Moreover, for every $j=0,\ldots,n$ we may choose a point $v_{i_j}\in V_{i_j}$. 
Since $U_{i_j}$ is path connected, for every $j=0,\ldots,n$,
there exists a path $\alpha_j\colon [0,1]\to U_{i_j}$ such that $\alpha_j(0)=v_{i_j}$, $\alpha_j(1)=q$. For $0\leq j<h\leq n$ we then set
$\alpha_{jh}=\alpha_j*\alpha_h^{-1}\colon [0,1]\to U_{i_j}\cup U_{i_h}$. Since the $v_{i_j}$ are pairwise distinct, by definition of $\calL(X)$ there exists a unique oriented $1$-simplex in $\calL(X)$
which projects onto a path in $X$ that is homotopic to $\alpha_{jh}$ relative to the endpoints. We denote by $e_{jh}\in \calL(X)^1=\calA(X)^1$ this simplex.
It readily follows from the definitions that for every $0\leq j<h<k\leq n$ the loop $\alpha_{jh}*\alpha_{hk}*\alpha_{jk}^{-1}$ is null-homotopic in $X$. 
By Proposition~\ref{riassunto:prop}, this implies that the concatenation of oriented simplices $e_{jh}*e_{hk}*e_{jk}^{-1}$ is null-homotopic in $|\calA(X)|$. As a consequence,
Proposition~\ref{aspherical:char} ensures that there exists a unique $n$-simplex $s$ of $\calA(X)$ whose $1$-skeleton is given by the union of the $e_{jh}$. We then
set 
$$
\Omega^n(z)({i_0},\ldots,{i_n})=z((s,(v_{i_0},\ldots,v_{i_n})))\ .
$$

We need to show that $\Omega^n(z)$ is well defined, i.e.~that different choices for the point $q$, for the points $v_{i_0},\ldots,v_{i_n}$ and for the paths $\alpha_0,\ldots,\alpha_n$ lead
to the same value for $z((s,(v_{i_0},\ldots,v_{i_n})))$. Indeed, let $q'\in \bigcap_{i=0}^n U_i$, 
let $v_{i_j}'\in V_{i_j}$ for every $j=0,\ldots,n$, and let 
   $\alpha_j'\colon [0,1]\to U_{i_j}$ be a path with $\alpha'_j(0)=v'_{i_j}$, $\alpha'_j(1)=q'$. For $0\leq j<h\leq n$ set
$\alpha'_{jh}=\alpha'_j*(\alpha'_h)^{-1}\colon [0,1]\to U_i\cup U_j$. Finally, let  $e'_{jh}$ be the $1$-simplex
of $\calA(X)$ corresponding to $\alpha'_{jh}$, and let $s'$ be the $n$-simplex of $\calA(X)$ whose $1$-skeleton is given by the $e_{jh}'$. 

Since $\bigcap_{j=0}^n U_{i_j}$ is path connected, we can choose a path $\beta\colon [0,1]\to \bigcap_{j=0}^n U_{i_j}$ such that $\beta(0)=q$, $\beta(1)=q'$.
For every $i=0,\ldots,n$, let now $\gamma_j=\alpha_j*\beta*(\alpha'_j)^{-1}$. By construction, the path $\gamma_j$ joins $v_{i_j}$ with $v'_{i_j}$ and is supported in $U_{i_j}$. Therefore, if we set 
$g_j=\{\gamma_j\}$ if $v_{i_j}=v'_{i_j}$, and $g_j=\{\gamma_j,\,\gamma_j^{-1}\}$ if $v_{i_j}\neq v'_{i_j}$, then
$g_j$ is an element of $\Pi(U_{i_j},V_{i_j})$. Using that ${i_j}\neq {i_h}$ if $j\neq h$, we can thus define the element
$g=\oplus_{j=0}^n g_j\in \G$, and
 it is  immediate to check that 
$g\cdot e_{jh}'=e_{jh}$ for every $0\leq j<h\leq n$. Since simplices in $\calA(X)$ are determined by their $1$-skeleton (see Proposition~\ref{aspherical:char}), this implies
$g\cdot s'=s$.
Since $z$ is $\G$-invariant, this implies in turn that
$z((s,(v_{i_0},\ldots,v_{i_n})))=z((s',(v'_{i_0},\ldots,v'_{i_n})))$, i.e.~$\Omega^n(z)$ is well defined.

For later reference we observe that, if $(s,(v_0,\ldots,v_n))\in C_n(T)\subseteq C_n(\calA(X))$ is an algebraic $n$-simplex such that $v_j\in V_{i_j}$ for every $j=0,\ldots,n$, then
\begin{equation}\label{oksus}
\Omega^n(z)(i_0,\ldots,i_n)=z((s,(v_{0},\ldots,v_{n})))\ .
\end{equation}
Indeed, if $i_j=i_h$ for some $j\neq h$, then by definition $\Omega^n(z)(i_0,\ldots,i_n)=0$, while
$z((s,(v_{0},\ldots,v_{n})))=0$ because $z$ is alternating and $\G$-invariant (see the proof of Theorem~\ref{vanishing1_intro}). Otherwise,
our assumptions on the fineness of $T$ imply that $s$ is supported in $\bigcap_{j=0}^n U_{i_j}$. This easily implies that $s$ may be obtained as a simplex
of $\calA(X)$ associated to the $(n+1)$-tuple $(i_0,\ldots,i_n)$ via the construction above, and the conclusion follows. 

Once we know that $\Omega^*\colon C^*_b(\calA(X))^\G_\alt \to C^*(N(\calU))$ is well defined, it is easy to show that it is indeed a chain map:
the fact that $\Omega^n(z)(\partial ({i_0},\ldots,i_{n+1}))=(\Omega^n(\delta z))({i_0},\ldots,{i_{n+1}})$ readily follows from the definitions if ${i_j}\neq {i_h}$ for every $0\leq j<h\leq n+1$,
and from the fact that $z$ is alternating and $\Gamma$-invariant if ${i_j}= {i_h}$ for some $j\neq h$.

We are now left to show that the diagram
\begin{equation}\label{diagram1}
\xymatrix{
H^n_b(X)\ar[rr]^{c^n}\ar[d]_{\Theta^n} & & H^n(X)\\
 H^n(N(\calU)) \ar[rr]_-{H^n(\phi^n)} & & H^n(|N(\calU)|)\ar[u]_{\beta^n}
}
\end{equation}
commutes, where $H^*(\phi^*)$ denotes the classical isomorphism between the simplicial cohomology of $T$ and the singular cohomology of $|T|=X$.

Let $f\colon X\to |N(\calU)|$ be
the unique simplicial map that sends every vertex $v\in V_i$ of $T$ to the vertex $i$ of $N(\calU)$ (this map indeed exists thanks to the assumption that $T$ is sufficiently fine). 
Observe that $f$ is the map associated to the partition of unity
$\Phi=\{\varphi_i\}_{i\in I}$ defined as follows: 
 $\varphi_i\colon X\to \R$ is the unique map which, on every simplex of $T$, affinely extends the characteristic function $\chi_{V_i}$ of $V_i$. Therefore, $\beta^n=H^n(f)$ for every $n\in\mathbb{N}$. 
 
 Let us take an element $\varphi\in H^n_b(X)$. 
 As above, after identifying $H^n_b(X)$ with $H^n_b(\calA(X))$, we can choose 
 an alternating $\G$-invariant representative $z\in C^n_b(\calA(X))$ of $\varphi$. 
   If $c\in C_n(T)$ is  a simplicial cycle (which defines also a singular cycle in $C_n(X)$ and a simplicial cycle in $C_n(\mathcal{A}(X))$, still denoted by $c$), then  we have
 \begin{equation}\label{unau}
 \langle(\beta^n\circ H^n(\phi^n)\circ \Theta^n)(\varphi),[c]\rangle =\langle(H^n(\phi^n)\circ \Theta^n)(\varphi),H_n(f)([c])\rangle\ .
 \end{equation}
 Since $f\colon |N(\calU)|\to X$ is simplicial, the chain $C_n(f)(c)$ is itself simplicial, hence (by identifying as usual simplicial chains/classes with the corresponding singular chains/classes)
 \begin{equation}\label{dueu}
 \langle(H^n(\phi^n)\circ \Theta^n)(\varphi),H_n(f)([c])\rangle=\langle \Theta^n(\varphi),H_n(f)([c])\rangle \ ,
 \end{equation}
 and, putting together~\eqref{unau} and~\eqref{dueu}, 
 $$
 \langle(\beta^n\circ H^n(\phi^n)\circ \Theta^n)(\varphi),[c]\rangle=\langle \Theta^n(\varphi),H_n(f)([c])\rangle\ .
 $$
  
 Let us now look more closely at the right hand side of the last equality.
Since $\Theta^n(\varphi)=[\Omega^n(z)]$, we have
 $$
\langle \Theta^n(\varphi),H_n(f)([c])\rangle=\langle \Omega^n(z),C_n(f)(c)\rangle\ .
 $$
 Moreover, one easily checks that, since $\Psi^n([z]) = \varphi$, we have
 \begin{equation}\label{eq:spiegazione:comparison}
 \langle z, c  \rangle = \langle c^n(\varphi), [c] \rangle \ .
 \end{equation}
 Now, if $(s,(v_0,\ldots,v_n))$ is an algebraic simplex in $C_n(T)$, where $v_j\in V_{i_j}$ for every $j=0,\ldots,n$, then 
 $$
 \langle \Omega^n(z),C_n(f)(s,(v_0,\ldots,v_n))\rangle=\langle \Omega^n(z),(i_0,\ldots,i_n)\rangle=\langle z, (s,(v_0,\ldots,v_n))\rangle\ ,
 $$
 where the first equality is due to the definition of $f$, and the second one to Equation~\eqref{oksus}. Thus
 $$
 \langle(\beta^n\circ H^n(\phi^n)\circ \Theta^n)(\varphi),[c]\rangle=\langle \Omega^n(z),C_n(f)(c)\rangle=\langle z,c\rangle =\langle c^n(\varphi),[c]\rangle\ ,
 $$
 where the last equality is due to \eqref{eq:spiegazione:comparison}. We have thus shown that the coclasses $(\beta^n\circ H^n(\phi^n)\circ \Theta^n)(\varphi)$ and $c^n(\varphi)$ coincide on every simplicial class in $H_n(T)$.
 Thanks to the Universal Coefficient Theorem and to the canonical isomorphism $H^n(T)\cong H^n(|T|)=H^n(X)$, this implies that these coclasses coincide,
 i.e.~the diagram~\eqref{diagram1} commutes. This concludes the proof of Theorem~\ref{van-thm2}.

\part{The simplicial volume of open manifolds}

\chapter{Finiteness and Vanishing Theorems\\ for locally finite homology}\label{finvan:chap}

The third part of this work is devoted to the study of the simplicial volume of open manifolds. The main results we  prove are Gromov's {Vanishing} and {Finiteness Theorems} for locally finite homology,
which provide useful criteria for the finiteness or the vanishing of the simplicial volume of open manifolds. 
As the name suggests, the Vanishing and the Finiteness Theorems for locally finite homology are very close in spirit to Theorem~\ref{vanishing1_intro}, and just as Theorem~\ref{vanishing1_intro} 
they are based on the study of spaces which admit
special covers of small multiplicity. 
The main ingredient of the proof is local diffusion of chains, that was introduced by Gromov in~\cite{Grom82} (see also \cite{KimKue, strz-unp}).

Other results on the vanishing and/or the finiteness of the simplicial volume were obtained e.g.~in \cite{Loeh, Loh-Sauer}. In particular,
L{\"o}h provides in \cite{Loeh} a complete criterion for the finiteness of the simplicial volume of tame open manifolds, in terms of the so-called
\emph{$\ell^1$-invisibility} of the boundary components of the manifolds (see Definition~\ref{invisible:def} and Theorem~\ref{invisible:thm}). 
As a by-product of our results, by comparing
the Finiteness Theorem~\ref{Fin-Theor-intro}   with Theorem~\ref{invisible:thm} we obtain a new sufficient condition for a closed manifold to be $\ell^1$-invisible (see Theorem~\ref{invisible:new-intro}).

\section{The simplicial volume of open manifolds}\label{simplicial:open:sec}
A manifold is
\emph{open} if it is connected, non-compact  and without boundary. If $M$ is an open manifold and $n=\dim M$, then $H_n(M,R)=0$ for every ring with unity $R$, so in order to define the simplicial
volume for open manifolds we need to introduce locally finite homology.

\begin{Definizione}
Let $X$ be a topological space, and let $S_n(X)$ denote the set of singular simplices with values in $X$.
A subset $A \subset S_n(X)$ is \emph{locally finite} if any compact set $K \subseteq X$ intersects the image of only finitely many singular simplices of $A$.

Following \cite[Chapter 5.1]{Lothesis}, a (possibly infinite) singular $n$-chain on $X$ (with coefficients in the ring with unity $R$) is a formal sum $\sum_{\sigma\in S_n(X)} a_\sigma \sigma$, where $a_\sigma$ is an element of $R$. We say that such a chain
is \emph{locally finite} if the set $\{\sigma\in S_n(X)\, |\, a_\sigma\neq 0\}$ is locally finite, and we denote by $C_n^{\lf}(X;R)$ the $R$-module of locally finite chains on $X$. 

The $R$-linear extension of the usual boundary operator sends locally finite chains to locally finite chains, so it makes sense
to define the locally finite homology $H^{\lf}_*(X;R)$ of $X$ as the homology of the complex $C_*^{\lf}(X;R)$.
\end{Definizione}

Henceforth we  denote the vector spaces $C_*^\lf(X;\R)$ and $ H^{\lf}_*(X;\R)$ simply by $C_*^\lf(X)$ and $H_*^\lf(X)$. 

As in the  case of finite chains, $C_n^\lf(X)$ is endowed with the $\ell^1$-norm defined by
$$
\left\|\sum_{\sigma \in S_n(X)} a_\sigma \sigma\right\|_1=\sum_{\sigma \in S_n(X)} |a_\sigma|\ \in\ [0,+\infty]
$$
(which may now take the value $+\infty$). This norm induces
an $\ell^1$-seminorm (with values in $[0,+\infty]$) on $H_*^\lf(X)$, which will still be denoted by $\|\cdot\|_1$.

It is a standard result of algebraic topology (see for instance \cite[Theorem~5.4]{Lothesis}) that, if $X$ is an $n$-dimensional oriented open manifold, then $H^{\lf}_n(X;\mathbb{Z})\cong \mathbb{Z}$ is generated by a preferred
element $[X]_\mathbb{Z}\in H^{\lf}_n(X;\mathbb{Z})$, called the \emph{fundamental class} of $X$. Under the obvious change of coefficients homomorphism
$H^{\lf}_*(X;\mathbb{Z})\to H^{\lf}_*(X)$, the element $[X]_\mathbb{Z}$ is taken  to the \emph{real} fundamental class $[X]\in H^{\lf}_n(X)$ of $X$.

The simplicial volume of $X$ is then defined as follows.
\begin{Definizione}
Let $X$ be an oriented open $n$-manifold. The \emph{simplicial volume} of $X$ is $$\lVert X \rVert \coloneqq \lVert [X] \rVert_{1} = \inf \{\lVert c \rVert_1 \, | \, c \in \, C_{n}^{\lf}(X) \mbox{ is a fundamental cycle of } X\} \in \, [0, +\infty].$$
\end{Definizione}

It is straightforward to realize that when $X$ is  compact we recover the definition of the classical simplicial volume.

\section{Statement of the main results}
Let $X$ be a connected non-compact space. 

\begin{Definizione}
A subset  $U$ of $X$ is  \emph{large} if its complement  in $X$ is relatively compact.
In particular, a large set $U \subset X$ must contain all the ends of $X$.
\end{Definizione}

\begin{Definizione}\label{amenable:infinity}
A sequence of subsets $\{U_{j}\}_{j\in\mathbb{N}}$ of $X$ is said to be \emph{amenable at infinity} if 
there exists a
sequence $\{W_{j}\}_{j\in\mathbb{N}}$ of large open subsets of $X$ such that the following conditions hold:
\begin{enumerate}
\item the family $\{W_{j}\}_{j\in\mathbb{N}}$  is locally finite;
\item $U_{j} \subset W_{j}$ for every $j$;
\item there exists $\bar{j}\in\mathbb{N}$ such that $U_{j}$ is an amenable subset of $W_{j}$ for every $j \geq \bar{j}$.
\end{enumerate}
By (1) and (2), any sequence of subsets which is amenable at infinity is necessarily locally finite.
\end{Definizione}

We now state the main results of the third part of this paper.

\begin{thm:repeated:Van-Theor-intro}[Vanishing Theorem for locally finite homology]
Let $X$ be a connected non-compact triangulable topological space.
Let $\calU=\{U_j\}_{j\in\mathbb{N}}$ be an amenable open cover of $X$ such that each $U_j$ is relatively compact in $X$. 
Also suppose that 
the sequence $\{U_j\}_{j\in\mathbb{N}}$ is amenable at infinity.
Then for every $k\geq \mult(\calU)$ and every $h \in \, H^{\lf}_{k}(X)$ we have $$\lVert h \rVert_1 = 0\ .$$
\end{thm:repeated:Van-Theor-intro}

\begin{thm:repeated:Fin-Theor-intro}[Finiteness Theorem]
Let $X$ be a connected non-compact triangulable topological space.
Let $W$ be a large open subset of $X$, and
let $\calU=\{U_j\}_{j\in\mathbb{N}}$ be an open cover of $W$ such that each $U_j$ is relatively compact in $X$. 
Also suppose that 
the sequence $\{U_j\}_{j\in\mathbb{N}}$ is amenable at infinity (in particular, $\calU$ is locally finite in $X$).
Then for every $k\geq \mult(\calU)$  and every $h \in \, H^{\lf}_{k}(X)$ we have $$\lVert h \rVert_1 <+\infty  .$$
\end{thm:repeated:Fin-Theor-intro}

In this paper we prove Theorems~\ref{Van-Theor-intro} and \ref{Fin-Theor-intro}   
only for triangulable spaces. We believe that these results should hold
for any topological space, but due to the amount of technicalities already involved in their proofs we prefer to confine ourselves to the context of triangulable spaces, which certainly suffices
for many interesting applications. 

The following corollaries provide the main applications of Theorems~\ref{Van-Theor-intro}  and~\ref{Fin-Theor-intro}    to the simplicial volume of open manifolds.

\begin{thm:repeated:va-cor-intro}
Let $M$ be an oriented open triangulable manifold of dimension $m$ and let $\calU=\{U_j\}_{j\in\mathbb{N}}$ be an amenable open cover of $M$ such that each $U_j$ is relatively compact in $M$. Also  suppose that the sequence $\{U_j\}_{j\in\mathbb{N}}$ is amenable at infinity, and that $\mult(\calU)\leq m$. Then
$$
\|M\|=0\ .
$$
\end{thm:repeated:va-cor-intro}

\begin{thm:repeated:fi-cor-intro}
Let $M$ be an oriented open triangulable manifold of dimension $m$. Let $W$ be a large open subset of $M$, and
let $\calU=\{U_j\}_{j\in\mathbb{N}}$ be an open cover of $W$ such that each $U_j$ is relatively compact in $M$. 
Also suppose that 
the sequence $\{U_j\}_{j\in\mathbb{N}}$ is amenable at infinity (in particular, $\calU$ is locally finite in $M$), and that $\mult(\calU)\leq m$.
Then
$$
\|M\|\ <\ +\infty\ .
$$
\end{thm:repeated:fi-cor-intro}

\section{$\ell^1$-homology and invisibility}\label{invisible:sec}
As mentioned above, L{\"o}h described in~\cite{Loeh} a necessary and sufficient condition for a tame manifold $M$ to have finite simplicial volume.
In order to describe L{\"o}h's criterion we first need to introduce $\ell^1$-homology. 

Let $X$ be a topological space. The space $\cl_*(X)$ of $\ell^1$-chains on $X$ is the metric completion of $C_*(X)$ with respect to the $\ell^1$-norm.
More concretely, an element $c\in \cl_n(X)$ is a formal sum
$$
c=\sum_{\sigma\in S_n(X)} a_\sigma \sigma
$$
such that
$$
\sum_{\sigma \in S_n(X)} |a_\sigma|< +\infty\ ,
$$
where $S_n(X)$ is the set of singular simplices with values in $X$, and  $a_\sigma\in \R$ for every $\sigma\in S_n(X)$. The boundary operator
$\partial_n\colon C_n(X)\to C_{n-1}(X)$ is continuous with respect to the $\ell^1$-norm, hence it induces a boundary operator (still denoted by the same symbol)
$\partial_n\colon \cl_n(X)\to \cl_{n-1}(X)$. The $\ell^1$-homology $\hl_*(X)$ is then the homology of the complex $\cl_*(X)$. 
The  inclusion of complexes $C_*(X)\hookrightarrow \cl_*(X)$ induces a map $c_*\colon H_*(X)\to \hl_*(X)$. 

\begin{defn}[\cite{Loeh}]\label{invisible:def}
 Let $M$ be a closed orientable $n$-manifold with real fundamental class $[M]\in H_n(M)$. Then $M$ is \emph{$\ell^1$-invisible} if 
 $$
 c_n([M])=0\quad \textrm{in}\ \hl_n(X)\ .
 $$
 If $M$ is closed and non-orientable, then $M$ is $\ell^1$-invisible if its orientable double covering is so, and if $M$ is compact, without boundary and disconnected, then
 it is $\ell^1$-invisible if every connected component of $M$ is so.
\end{defn}

 \begin{defn}\label{tame:defn}
An open topological manifold $X$ is \emph{tame} if it is homeomorphic to the internal part of a compact manifold with boundary. 
 \end{defn}
 
The following result provides a complete characterization of  tame manifolds having a finite simplicial volume:

\begin{thm}[\cite{Loeh}]\label{invisible:thm}
 Let $M=\inte (\overline{M})$, where $\overline{M}$ is a compact manifold with boundary. Then $\|M\|< +\infty$ if and only if  $\partial \overline{M}$ is $\ell^1$-invisible.
\end{thm}

L{\"o}h describes in~\cite{Loeh} several families of $\ell^1$-invisible manifolds. For example, If $p\colon M \to B$ is a fibration of oriented, closed, connected manifolds
whose fiber $F$ is also an oriented, closed, connected manifold of non-zero
dimension and if $\pi_1(F)$ is amenable, then $M$ is $\ell^1$-invisible
(see~\cite[Example 6.7]{Loeh}). As anticipated in the introduction, here we strengthen this result as follows:

\begin{thm:repeated:invisible:new-intro}
 Let $M$ be a closed oriented triangulable $n$-dimensional manifold admitting an amenable cover $\calU$ such that $\mult(\calU)\leq n$. Then $M$ is $\ell^1$-invisible.
\end{thm:repeated:invisible:new-intro}
\begin{proof}
If $M$ is non-orientable, then the pull-back of $\calU$ defines an amenable cover $\widetilde{\calU}$ of the orientable double covering $\widetilde{M}$ of $M$ such that
$\mult(\widetilde{U})=\mult(\calU)$. Therefore, we may suppose that $M$ is oriented.

Of course we may suppose that $\calU=\{U_1,\ldots,U_k\}$ is finite. Following~\cite[Theorem 5.3]{Loh-Sauer},
we now construct an amenable cover $\calU'$ of $M\times \mathbb{R}$ satisfying the hypothesis of the Theorem \ref{Van-Theor-intro} and such that $\mult(\calU')=\mult(\calU)+1$. 
Let us consider locally finite open covers $\mathcal{R}_1,\ldots,\mathcal{R}_k$ of the real line $\mathbb{R}$ with the following properties: each set in any $\calR_i$ is a bounded interval;  
$\mult(\mathcal{R}_i)= 2$ for every $i=1,\ldots,k$;  $\mult(\mathcal{R}_i \sqcup \mathcal{R}_j) = 3$ for every $1\leq i < j \leq k$ (the existence of such covers 
is observed in~\cite[Theorem~5.3]{Loh-Sauer}).
Then, let us set $$\calU'=\{U_i\times V_\alpha\, |\, U_i\in\calU\, ,\ V_\alpha\in\calR_i\, ,\ i=1,\ldots,k\}\ .$$

Using that $\mult(\calU')=\mult(\calU)+1$ 
(see~\cite[Theorem 5.3]{Loh-Sauer} or the proof of Theorem~\ref{teor-co-am} below), it is straightforward to check that the cover $\calU'$ satisfies the hypothesis of Corollary~\ref{va-cor-intro} (with $M$ replaced by $M\times \mathbb{R}$).  Therefore we have $\|M\times \mathbb{R}\|=0$. 
But $M\times \mathbb{R}\cong \inte(M\times [0,1])$, hence
from L{\"o}h's Theorem~\ref{invisible:thm}
we deduce that the manifold $\partial (M\times [0,1])\cong M\sqcup M$ is $\ell^1$-invisible. Thus $M$ itself is $\ell^1$-invisible, as desired.
\end{proof}

It is maybe worth mentioning that $\ell^1$-invisibility is a very interesting but a bit elusive notion. It is easy to show that the simplicial volume of any closed $\ell^1$-invisible manifold vanishes. Thus, Theorem~\ref{invisible:new-intro} strengthens the second statement of Corollary~\ref{van-cor-intro}.
Moreover, the property of being $\ell^1$-invisible is closed under amenable gluings~\cite{Loeh}.
However, it is still not known whether the vanishing of the simplicial volume is sufficient to ensure $\ell^1$-invisibility for closed manifolds.

Using the techniques developed in this paper, Theorem~\ref{invisible:new-intro}
has been recently generalized by the first author in~\cite{Fri:amenable}: namely, it is now known that, if $X$ is a topological space admitting an amenable open cover of multiplicity $k$,
then for every $n\geq k$ and every $\alpha\in H_n(X)$, the image of $\alpha$ in $H_n^{\ell^1}(X)$ vanishes.

\chapter{Diffusion of chains}\label{diffusion:chap}
Diffusion of chains is a technique introduced by Gromov in~\cite{Grom82}. It has then been exploited for the study of the simplicial volume
e.g.~in \cite{KimKue, strz-unp}. Gromov himself proposed a slightly different approach to diffusion in  \cite{gromov-cycles}, which was then developed by Alpert~\cite{Alpert1, Alpert2}.
In this work we resort to the original definition of diffusion, which perfectly fits with our purposes.

In presence of suitable actions by amenable groups, diffusion allows to
decrease the $\ell^1$-norm of cycles without altering their homology class. 
In this paper we  mainly concentrate
on diffusion of \emph{locally finite} chains, which as far as the authors know is the only technique available to prove  Theorems~\ref{Van-Theor-intro} and~\ref{Fin-Theor-intro}. 
Diffusion of ordinary singular chains (and its applications) may be recovered as a special case of the theory for locally finite chains. In fact, we prefer
to deal first with the case of diffusion of ordinary chains (which applies for example to the study of the simplicial volume of compact manifolds), in order to 
describe the main ideas behind the use of diffusion without begin forced to work in a very technical context. More precisely, following Gromov, in Section~\ref{toy:sec} we will illustrate
how diffusion of finite chains may be exploited to obtain a new proof of (a special case of) Corollary~\ref{van-cor-intro}. The applications of local diffusion of locally finite chains will be discussed
in the next chapters.

\section{Diffusion operators}
Let $\G$ be a group acting on a set $\Lambda$. 
For every function $f\colon \Lambda\to \R$ we denote by $\supp(f)$ the support of $f$, i.e.~the set
$$
\supp(f)=\{x\in \Lambda\, |\, f(x)\neq 0\}\ ,
$$
and
we denote by $\ell_0(\Lambda)$ the set of real functions on $\Lambda$ with finite support. For every $f\in \ell_0(\Lambda)$ we denote by
$$
\|f\|_1=\sum_{x\in \Lambda} |f(x)|
$$
the $\ell^1$-norm of $f$.

If $\mu$ is a probability measure on $\G$ with (at most) countable support, then we will
often consider $\mu$ as a non-negative function $\mu\colon \G\to\R$ such that $\sum_{g\in\G} \mu(g)=1$.

Let us now fix a probability measure $\mu$  on $\G$ with finite support.

\begin{Definizione}
The \emph{diffusion operator} associated to $\mu$ is the $\R$-linear map
\begin{displaymath}
\mu \ast \colon \ell_{0}(\Lambda) \rightarrow \ell_{0}(\Lambda)
\end{displaymath}
defined by
\begin{displaymath}
f(x) \mapsto (\mu \ast f) (x) = \sum_{\gamma \in \, \Gamma} \mu(\gamma) f(\gamma^{-1} x).
\end{displaymath}
(The fact that $\supp(\mu*f)$ is finite easily follows from the fact that $\supp(f)$ and $\supp(\mu)$ are.)
\end{Definizione}

An easy computation shows that
$
\|\mu*f\|_1\leq \|f\|_1
$ for every $f\in \ell_0(\Lambda)$. We 
are  interested in establishing much stronger inequalities under some additional conditions (among them, the amenability of the group $\G$).
 To this aim we  introduce the following definition (see \cite[Section 4.2, page 60]{Grom82}).
 
 \begin{Definizione}
Given an element $\varphi \in \, \Gamma$, we define the \emph{derivative} of $\mu$ at $\varphi$ by setting
\begin{displaymath}
D_{\varphi}\mu (\gamma) = \mu(\gamma \varphi) - \mu(\gamma)\ .
\end{displaymath}
Moreover, given a subset $\Phi$ of $\Gamma$, we set  $$\lVert D_{\Phi} \mu \rVert = \sup_{\varphi \in \, \Phi} \lVert D_{\varphi} \mu \rVert_1\ .$$
\end{Definizione}

If $\Phi$ is a subset of $\G$, then we denote by $\Phi^{-1}$ the set $\{\gamma^{-1}\in\G\, |\, \gamma\in\Phi\}$. 

\begin{Proposizione}\label{Proposition-diff-oper}
Let $f\in\ell_0(\Lambda)$  and let  $x_0\in\G$. 
Let $\Phi\subseteq \G$
 be a finite subset of $\G$ such that
$
\Phi\cdot x_0\supseteq \supp(f)
$.
Then for any probability measure $\mu$ on $\G$ we have
\begin{displaymath}
\lVert \mu \ast f \rVert_1 \leq \Big\lvert \sum_{x \in \, \supp(f)} f(x) \Big\rvert + \lVert D_{\Phi^{-1}} \mu \rVert \cdot \lVert f \rVert_1\ .
\end{displaymath} 
\end{Proposizione}
\begin{proof}
The map
$$q \colon \Gamma \rightarrow \Lambda\, ,\quad q(\gamma) = \gamma \cdot x_{0}$$ 
induces
a push-forward operator: $$q_{*} \colon \ell_{0}(\Gamma) \rightarrow \ell_{0}(\Lambda)\, ,\quad (q_{*}g) (x) = \sum_{\gamma \in \, q^{-1}(x)} g(\gamma)\ ,$$
which is easily seen to be norm non-increasing.

Let us now construct a function $g\in\ell_0(\G)$ such that $\supp(g)\subseteq \Phi$ and $q_*(g)=f$. To this aim, for every $x\in \supp(f)$ we choose $\gamma_x\in \Phi$ such that $\gamma_x(x_0)=x$, and we set
$g(\gamma)=f(x)$ if $\gamma=\gamma_x$ for some $x\in \Lambda$ (of course, if such an $x\in \Lambda$ exists, it is unique), and $g(\gamma)=0$ otherwise. It is now obvious that
$q_*(g)=f$ and $\supp(g)\subseteq \Phi$. Also observe that

\begin{equation}\label{proprietag}
 \Big|\sum_{\lambda\in\G} g(\lambda)\Big|=\Big|\sum_{x\in \, \supp(f)} f(x)\Big|\, ,\qquad
 \|g\|_1=\|f\|_1\ .
\end{equation}

After endowing $\G$ with the action given by left translations, we can let $\mu$ act also on $\ell_0(\G)$, thus defining an operator
$$\mu \ast\colon \ell_0(\G)\to \ell_0(\G)\ .$$
It is not difficult to show that $q_{*}(\mu \ast g) = \mu \ast (q_{*}g)$, i.e.~that diffusion commutes with the push-forward operator $q_{*}$.
This implies in particular that
$$\lVert \mu \ast f \rVert_{1} = \lVert \mu \ast (q_{*}g) \rVert_{1} = \lVert q_{*}(\mu \ast g) \rVert_{1} \leq \lVert \mu \ast g \rVert_{1}\ .$$ 
Now the conclusion follows from the following inequality:
\begin{align*}
\lVert \mu \ast g \rVert_1 &= \sum_{\gamma \in \, \Gamma} \Big\lvert \sum_{\lambda \in \Gamma} \mu(\lambda) g(\lambda^{-1} \gamma) \Big\rvert  \\
&= \sum_{\gamma \in \, \Gamma} \Big\lvert \sum_{\lambda \in \, \Gamma} \mu(\gamma \lambda)g(\lambda^{-1})\Big\rvert \\
&= \sum_{\gamma \in \, \Gamma} \Big\lvert \sum_{\lambda \in \, \Gamma} (\mu(\gamma \lambda) - \mu(\gamma)) g(\lambda^{-1}) + \sum_{\lambda \in \, \Gamma} \mu(\gamma) g(\lambda^{-1}) \Big\rvert \\
&\leq \sum_{\gamma \in \, \Gamma} \sum_{\lambda \in \, \Gamma} \lvert \mu(\gamma \lambda) - \mu(\gamma) \rvert \cdot \lvert g(\lambda^{-1}) \rvert + \Big\lvert \sum_{\gamma \in \, \Gamma} \mu(\gamma) \Big\rvert \cdot 
\Big\lvert \sum_{\lambda \in \, \Gamma} g(\lambda^{-1}) \Big\rvert \\
&=\sum_{\lambda \in \, \Phi^{-1}} \lVert D_{\lambda} \mu \rVert_1  \lvert g(\lambda^{-1}) \rvert + \Big\lvert \sum_{\lambda \in \, \Gamma} g(\lambda) \Big\rvert \\
&\leq \lVert D_{\Phi^{-1}} \mu \rVert \lVert g \rVert_1 + \Big\lvert \sum_{\lambda \in \, \Gamma} g(\lambda) \Big\rvert\\
& =\lVert D_{\Phi^{-1}} \mu \rVert \lVert f \rVert_1 + \Big\lvert \sum_{\lambda \in \, \supp(f)} f(\lambda) \Big\rvert\ ,
\end{align*}
where the last equality is due to~\eqref{proprietag}, the inequality between the third and the fourth lines is due to the fact that
$$
\sum_{\gamma\in\Gamma}\Big\lvert\sum_{\lambda \in \, \Gamma} \mu(\gamma) g(\lambda^{-1}) \Big\rvert=
\sum_{\gamma\in\Gamma} \left(\mu(\gamma) \Big\lvert\sum_{\lambda \in \, \Gamma}g(\lambda^{-1}) \Big\rvert\right)=
\Big\lvert \sum_{\gamma \in \, \Gamma} \mu(\gamma) \Big\rvert \cdot 
\Big\lvert \sum_{\lambda \in \, \Gamma} g(\lambda^{-1}) \Big\rvert \ ,
$$
and the equality between the fourth and fifth lines is due to the fact that $g$ is supported on $\Phi$.

\end{proof}

In order to exploit the previous lemma we need to 
construct probability measures on $\G$ with a small derivative. This can be done under the assumption that $\G$ is amenable:

\begin{Proposizione}[{\cite[Theorem~4.4]{Paterson}}]\label{prop-amen-paterson}
Let $\Gamma$ be an amenable group. 
Let $\Phi$ be a finite subset of $\G$ and take 
$\varepsilon > 0$. Then there exists a probability measure $\mu$ on $\Gamma$ with finite support such that $$\lVert D_{\Phi} \mu \rVert_1 < \varepsilon\  .$$ 
\end{Proposizione}

\begin{proof}
By~\cite[Theorem~4.4]{Paterson}, there exists a probability measure $\mu'$ on $\Gamma$ 
(possibly with infinite support)  such that $$\lVert D_{\Phi} \mu' \rVert_1 < \frac{\varepsilon}{3}\  .$$ 
Since $\mu'$ is a probability measure, there exists a finite subset $\Gamma_{0} \subset \Gamma$ such that
\begin{displaymath}
D=\sum_{\gamma \in \, \Gamma_{0}} \mu'(\gamma) \geq 1 - \frac{\varepsilon}{6}.
\end{displaymath}
Let  $\mu$ be the probability measure defined by

\begin{displaymath}
\mu(\gamma) = \begin{cases} 
\frac{\mu'(\gamma)}{D}& \mbox{ if } \gamma \in \, \Gamma_{0} \\
0 & \mbox{ if } \gamma \not\in \, \Gamma_{0}.
\end{cases}
\end{displaymath}
An easy computation shows that

\begin{align*}
\lVert \mu - \mu' \rVert_1 
\leq 2(1-D)\leq \frac{\varepsilon}{3},
\end{align*}
so for every $\varphi\in \Phi$ we have

\begin{align*}
\|D_{\varphi}\mu (\gamma)\|_1 &= \sum_{\gamma\in\G} |\mu(\gamma \varphi) - \mu(\gamma)|\\ &=\sum_{\gamma\in\G}|\mu(\gamma \varphi) - \mu'(\gamma\varphi)+\mu'(\gamma\varphi)-\mu'(\gamma)+\mu'(\gamma)-\mu(\gamma)|
\\ &\leq \sum_{\gamma\in\G}|\mu(\gamma \varphi) - \mu'(\gamma\varphi)|+\sum_{\gamma\in\G}|\mu'(\gamma\varphi)-\mu'(\gamma)|+\sum_{\gamma\in\G}|\mu'(\gamma)-\mu(\gamma)|\\ &\leq 2\|\mu-\mu'\|_1+\|D_{\Phi}\mu'\|_1\leq\varepsilon\ .
\end{align*}
This concludes the proof.
\end{proof}

We can now put together the previous results to prove the following:
\begin{Corollario}\label{Cor-A-diff-oper}
Let $\Gamma$ be an amenable group acting transitively on a set $\Lambda$,  
let $S\subseteq \Lambda$ be finite, and
let  $\varepsilon > 0$ be given.
Then there exists a probability measure $\mu$ on $\Gamma$ with finite support such that 
for every $f\in \ell_0(\Lambda)$ with $\supp(f)\subseteq S$ the following inequality holds:
\begin{displaymath}
\lVert \mu \ast f \rVert_1 \leq \Big\lvert \sum_{x \in \, S} f(x) \Big\rvert + \varepsilon \cdot \lVert f \rVert_1 \ .
\end{displaymath}
\end{Corollario}
\begin{proof}
Since $\G$ acts transitively, 
we can choose an element $x_0\in S$ and a finite subset $\Phi$ of $\G$ such that 
$$S \subseteq \Phi \cdot x_{0}\ .$$ 
We  
can then apply Proposition \ref{prop-amen-paterson} to the finite subset $\Phi^{-1}$ of $\G$, thus obtaining  
 a probability measure $\mu$ on $\Gamma$ with finite support such that $$\lVert D_{\Phi^{-1}} \mu \rVert \leq \varepsilon\ .$$

Now  Proposition \ref{Proposition-diff-oper} gives
$$
\lVert \mu \ast f \rVert_1  
\leq   \Big|\sum_{x \in \, S} f(x)\Big| + \lVert D_{\Phi^{-1}} \mu \rVert\cdot  \lVert f  \rVert_1 \\
\leq \Big| \sum_{x \in S} f(x)\Big| + \varepsilon \lVert f \rVert_1 \ .
$$
\end{proof}

\begin{cor}\label{oss-scelta-epsilon}
Let $\Gamma$ be an amenable group acting transitively on a set $\Lambda$, and take an element
  $f\in \ell_0(\Lambda)$.
Then for every $\eta>0$ there exists a probability measure $\mu$ on $\Gamma$ with finite support such that 
\begin{displaymath}
\lVert \mu \ast f \rVert_1 \leq  \Big| \sum_{x \in \Lambda} f(x) \Big| + \eta\ .
\end{displaymath}
\end{cor}
\begin{proof}
Of course we may suppose $f\neq 0$. We then
set $\varepsilon = \frac{\eta}{\lVert f \rVert_1 }$ and apply Corollary~\ref{Cor-A-diff-oper}.
\end{proof}

\begin{rem}
 In~\cite[page 61]{Grom82} it is stated that, if an amenable group $\G$ acts transitively on a set $\Lambda$, then for every finite set $S$ and every $\varepsilon>0$ there exists
 a probability measure with finite support such that 
 $$
 \lVert \mu \ast f \rVert_1 \leq  \Big| \sum_{x \in S} f(x) \Big| + \varepsilon
 $$
 for every $f\in\ell_0(\Lambda)$ with $\supp(f)\subseteq S$. The following example shows that this stronger formulation of Corollary~\ref{oss-scelta-epsilon} cannot hold
 in general. Let $\Lambda=\G=\mathbb{Z}$ act  on itself by translations,
 let $S=\{0,1\}$ and
 and let $f_n\colon \mathbb{Z}\to\R$ be the map such that $f_n(0)=n$, $f_n(1)=-n$ and $f_n(m)=0$ for $m\notin\{0,1\}$. Since $\mu*$ is linear and $\sum_{x\in S} f_n(x)=0$ for every $n$,
 if the above inequality were true we would have 
 $$
n\cdot \lVert \mu \ast f_1 \rVert_1= \lVert \mu \ast f_n \rVert_1  \leq  \Big| \sum_{x \in S} f_n(x) \Big| + \varepsilon=\varepsilon
 $$
 for every $n\in\mathbb{N}$, hence $\lVert \mu \ast f_1 \rVert_1=0$ and $\mu*f_1(x)=0$ for every $x\in\mathbb{Z}$. However,
 $\mu*f_1(x)=\sum_{m\in\mathbb{Z}} \mu(m)f_1(x-m)=\mu(x)-\mu(x-1)$, so in order for $\mu*f_1$ to vanish 
 we should have that $\mu$ is constant, against the fact that $\mu$ is a probability measure on an infinite set. 

This shows that  
  the measure $\mu$ provided by Corollary~\ref{oss-scelta-epsilon} necessarily depends on $f$. 
 \end{rem}

\section{Locally finite actions and diffusion}\label{Locally-Finite-Diff}
Let us fix an action $\G\actson \Lambda$ of a group $\G$ on a set $\Lambda$, and let us suppose that the number of the orbits of the action is countable.
We  denote by $\Lambda_s\subseteq \Lambda$, $s\in\mathbb{N}$, the orbits of the action.

\begin{Definizione}\label{locallyfinitefun}
A function $f \colon \Lambda \rightarrow \mathbb{R}$ is \emph{locally finite} 
if $\supp(f)\cap \Lambda_s$ is finite for every $s\in\mathbb{N}$, i.e.~if $f|_{\Lambda_s}\in \ell_0(\Lambda_s)$ for every
$s\in\mathbb{N}$.
We denote by $\lf(\Lambda)$ the set of locally finite functions on $\Lambda$.
\end{Definizione}

\begin{Definizione}
The \emph{support} $\supp(\Gamma\actson \Lambda)$ of the action $\Gamma\actson \Lambda$ is defined by setting 
$$\supp(\Gamma\actson \Lambda) = \{x \in \, \Lambda\ |\ \exists \, \gamma \in \, \Gamma \mbox{ such that } \gamma \cdot x \neq x\} \subseteq \Lambda\ .$$
\end{Definizione}

\begin{Definizione}\label{Def-locally-finite-action}
The action $\G\actson \Lambda$ is \emph{locally finite} (relatively to the family of subgroups $\{\Gamma_{s}\}_{s \in \, \mathbb{N}}$) 
if there exist subgroups $\Gamma_{s} < \Gamma$, $s \in \, \mathbb{N}$
such that the following conditions hold:
\begin{enumerate}
\item
$\Gamma_s$ acts transitively on $\Lambda_s$ for every $s\in\mathbb{N}$;
\item 
the
actions of the groups $\Gamma_{s}$ are \emph{asymptotically disjoint}, i.e.~for every $s \in \, \mathbb{N}$ there exists $ k({s}) \in \, \mathbb{N}$ such that 
$$\left(\Lambda_s\cup\supp(\Gamma_{{s}}\actson \Lambda)\right) \cap \supp(\Gamma_{s'}\actson \Lambda) = \emptyset \, \mbox{ for every } s' \geq k({s}).$$
\end{enumerate}
\end{Definizione}
We are now ready to introduce the notion of  \emph{local} diffusion operator. 
Suppose that the action $\G\actson \Lambda$  is locally finite 
relatively to the sequence of subgroups $\{\Gamma_{s}\}_{s \in \, \mathbb{N}} \leq \Gamma$.

\begin{Lemma}\label{prop-local-diff-pres-orb-pre}
Let $f\colon \Lambda\to\mathbb{R}$ be \emph{any} function, and let $\mu$ be a finitely supported  probability measure on $\G$. Then the map
$$
\mu*f\colon \Lambda\to \R\, ,\quad (\mu*f)(x)=\sum_{\gamma\in\G} \mu(\gamma)f(\gamma^{-1} x)
$$
is well defined.
Moreover, if $f\in \lf(\Lambda)$ then $\mu*f \in \lf(\Lambda)$, and for every $s\in\mathbb{N}$ we have
$$\sum_{x \in \Lambda_s} (\mu \ast f)(x) = \sum_{x \in \Lambda_s} f(x)\ ,$$
$$
\lVert ({\mu} \ast f) \vert_{\Lambda_{s}} \rVert_1 \leq \lVert f \vert_{\Lambda_{s}} \rVert_1\ .
$$
\end{Lemma}
\begin{proof}

Since $\mu$ has finite support,  
the sum $\sum_{\gamma\in\G} \mu(\gamma)f(\gamma^{-1} x)$ is always finite, and this shows that $\mu* f$ is indeed well defined.
Moreover,  it readily follows from the definition that $\supp (\mu *f)\subseteq \supp(\mu)\cdot \supp (f)$, which implies
that $\mu* f$ is locally finite, if $f$ is so.

Suppose now that $f\in \lf(\Lambda)$. Since $\Lambda_s$ is  a $\G$-orbit we readily have
\begin{align*}
\sum_{x \in \, \Lambda_{s}} (\mu \ast f)(x) &= \sum_{x \in \, \Lambda_{s}} \sum_{\gamma \in \, \G} \mu(\gamma) f(\gamma^{-1} x) \\
&= \sum_{\gamma \in \, \G}\sum_{x \in \, \Lambda_{s}}  \mu(\gamma) f(\gamma^{-1} x) \\
&= \sum_{\gamma \in \, \G} \sum_{x \in \, \Lambda_{s}} \mu(\gamma) f(x) \\
&= \left(\sum_{\gamma \in \, \G} \mu(\gamma)\right) \sum_{x \in \, \Lambda_{s}} f(x) \\
&= \sum_{x \in \, \Lambda_{s}} f(x)\ .
\end{align*}
This proves the equality in the statement. For the inequality between the  $\ell^1$-norms we compute
\begin{align*}
& \sum_{x\in \Lambda_s} |(\mu*f)(x)|=\sum_{x\in \Lambda_s}\Big|\sum_{\gamma\in\G} \mu(\gamma) f(\gamma^{-1} x)\Big|\leq \sum_{x\in \Lambda_s} \sum_{\gamma\in\G} \mu(\gamma) |f(\gamma^{-1} x)|\\ &=
\sum_{\gamma\in\G} \mu(\gamma)\left(\sum_{x\in \Lambda_s} |f(\gamma^{-1}x)|\right)=\sum_{\gamma\in\G} \mu(\gamma)\|f|_{\Lambda_s}\|_1=\|f|_{\Lambda_s}\|_1\ .
\end{align*}
\end{proof}

For every $s\in\mathbb{N}$ let $\mu_s$ be  a probability measure on $\G$ such that the support of $\mu_s$ is finite and contained in $\G_s$.
For ease of notation we set
$\overline{\mu}=\{\mu_{s}\}_{s \in \, \mathbb{N}}$.

Take an element $f\in \lf(\Lambda)$. We inductively define the sequence of maps $f_s\colon \Lambda\to \Lambda$, $s\in\mathbb{N}$, by setting
$$
f_{1} = \mu_{1} \ast f\, ,\quad  f_{s} = \mu_{s} \ast f_{s-1}\ \mbox{for every}\ s\in\mathbb{N}
$$
(see Lemma~\ref{prop-local-diff-pres-orb-pre}).
We claim that the value $f_s(x)$ does not depend on $s$ for large $s$. Indeed, for every $x\in \Lambda$ there exists $s\in\mathbb{N}$ such that
$x\in \Lambda_s$. Since the actions of the $\G_s$ are asymptotically disjoint, there exists $k(s)\in\mathbb{N}$ such that $\gamma\cdot x=x$ for every
$\gamma\in\bigcup_{s'\geq k(s)} \G_{s'}$. As a consequence, for every  $s'\geq k(s)$ we have $(\mu_{s'}*{f_{s'-1}})(x)=f_{s'-1}(x)$. We have thus shown that
\begin{equation}\label{fundamental:local}
 f_{s'}(x)=f_{k(s)}(x)\quad \mbox{for every}\ x\in \Lambda_s,\ s'\geq k(s)\ .
\end{equation}
This shows in particular that by setting
$$
(\overline{\mu}*f)(x)=\lim_{s\to \infty} f_s(x)\quad (=f_{k(s)}(x)\ \mbox{if}\ x\in \Lambda_s)
$$
we obtain a well-defined map $\overline{\mu}*f\colon \Lambda\to \R$. 

Let us now show that the function $\overline{\mu}*f$ is locally finite. By~\eqref{fundamental:local}, it is sufficient
to show that $\supp(f_{k(s)})\cap \Lambda_s$ is finite for every $s\in\mathbb{N}$. However, this readily follows from Lemma~\ref{prop-local-diff-pres-orb-pre}, 
 since
the function $f_{k(s)}$
is obtained from $f$ by applying 
a finite number of convolutions with finitely supported measures. 

\begin{defn}
The map
$$
\overline{\mu}*\colon \lf(\Lambda) \rightarrow \lf(\Lambda)\, , \quad f\mapsto \overline{\mu}*f
$$
just introduced is called the 
\emph{local diffusion operator} associated to $\overline{\mu}$.
\end{defn}

We will exploit (local) diffusion operators to construct (locally finite) chains with small $\ell^1$-norm. To this aim we need to extend 
Corollary \ref{oss-scelta-epsilon} to the context we are interested in. Henceforth we fix a local diffusion operator $\overline{\mu}*$ as above.

\begin{Lemma}\label{prop-local-diff-pres-orb}
For every $f\in\lf(\Lambda)$ and every $s\in\mathbb{N}$ we have
\begin{displaymath}
\sum_{x \in \, \Lambda_{s}} (\overline{\mu} \ast f)(x) = \sum_{x \in \, \Lambda_{s}} f(x)\ ,
\end{displaymath}
$$
\lVert (\overline{\mu} \ast f) \vert_{\Lambda_{s}} \rVert_1 \leq \lVert f \vert_{\Lambda_{s}} \rVert_1\ .
$$
\end{Lemma}
\begin{proof}
For every $x\in \Lambda_s$ we have 
$$ \overline{\mu} \ast f (x) = \mu_{k(s)} \ast \cdots \ast \mu_{1} \ast f(x)\ ,$$
so the conclusion readily follows from Lemma~\ref{prop-local-diff-pres-orb-pre}.
\end{proof}

\begin{Proposizione}\label{Prop-locally-finite-diff-op-norm-min-eps}
Let $\Gamma\actson \Lambda$ be a locally finite action, let $f\in \lf(\Lambda)$ and suppose that there exists $\bar{s}\in\mathbb{N}$ such that the following conditions hold:
\begin{enumerate}
 \item 
 $\sum_{x \in \Lambda_s} f(x) = 0$ for every $s \geq \bar{s}$;
 \item the group $\Gamma_{s}$ is amenable for every $s \geq \bar{s}$.
\end{enumerate}
Then, for any arbitrary sequence $\{\varepsilon_{s}\}_{s\geq  {\bar{s}}}$ of positive numbers, there is a local diffusion operator $\overline{\mu} \ast$ such that
\begin{displaymath}
\lVert (\overline{\mu}\ast f) \vert_{\Lambda_{s}} \rVert_1 \leq \varepsilon_{s}
\end{displaymath}
for every $s \geq \bar{s}$.
\end{Proposizione}
\begin{proof}
For every $s<\bar{s}$ we arbitrarily choose a probability measure $\mu_s$ on $\G$ whose support is finite and contained in $\G_s$, and we set as usual
$f_i=\mu_i*\mu_{i-1}*\ldots*\mu_1*f$ for every $i<\bar{s}$.

We will now prove inductively that for every $s\geq \bar{s}$ there exists a finitely supported measure $\mu_s$ on $\G$ such that $\supp(\mu_s)\subseteq \G_s$ and 
$$\lVert ( \mu_{s} \ast f_{s-1}) \vert_{\Lambda_{s}} \rVert_1 \leq \varepsilon_{s}\ ,$$
where $f_{s-1}=\mu_{s-1}*\mu_{s-2}*\ldots*\mu_1*f$.

Let us just consider the base case $s=\bar{s}$ of the induction, the inductive step being identical. 
The map
$f_{\bar{s}-1}$ is locally finite, so its restriction to $\Lambda_{\bar{s}}$ has finite support.  
Our assumptions imply that $\G_{\bar{s}}$ is amenable, hence we can apply Corollary \ref{oss-scelta-epsilon}
to the action $\G_{\bar{s}}\actson \Lambda_{\bar{s}}$ and to the function $f_{\bar{s}-1}|_{\Lambda_{\bar{s}}}$, thus getting a finitely supported probability measure $\mu'_{\bar{s}}$
on $\G_{\bar{s}}$ such that 
\begin{equation}\label{abceq}
\lVert \mu'_{\bar{s}} \ast (f_{\bar{s}-1}|_{\Lambda_{\bar{s}}})  \rVert_1 \leq \varepsilon_{\bar{s}}+ \Big| \sum_{x\in \Lambda_{\overline{s}}} f_{\bar{s}-1}(x)\Big|\ .
\end{equation}
Our assumptions imply that $\sum_{x \in \, \Lambda_{\overline{s}}} f(x) = 0$; moreover, by construction
$f_{\bar{s}-1}=\mu_{\bar{s}-1}*\mu_{\bar{s}-2}*\ldots*\mu_1*f$, so Lemma~\ref{prop-local-diff-pres-orb-pre} implies that 
$\sum_{x\in \Lambda_{\bar{s}}}f_{\bar{s}-1}(x)=0$. Therefore, from~\eqref{abceq} we deduce that
$$\lVert \mu'_{\bar{s}} \ast (f_{\bar{s}-1}|_{\Lambda_{\bar{s}}})  \rVert_1 \leq \varepsilon_{\bar{s}}\ .
$$
If we denote by $\mu_{\bar{s}}$ the unique probability measure on $\G$ obtained by extending $\mu'_{\bar{s}}$, then we have
$$\lVert (\mu_{\bar{s}} \ast f_{\bar{s}-1}) \vert_{\Lambda_{\bar{s}}} \rVert_1=\lVert \mu'_{\bar{s}} \ast (f_{\bar{s}-1}|_{\Lambda_{\bar{s}}})  \rVert_1  \leq \varepsilon_{s}\ .$$
This proves the base case of the induction. The proof of the inductive step is identical, and it is left to the reader.

Let us now set $\overline{\mu}=\{\mu_s\}_{s\in\mathbb{N}}$, take
$s \geq \bar{s}$ and let $k(s)\in\mathbb{N}$ be  such that $\Gamma_{i}$ acts trivially on $\Lambda_{s}$ for every $i \geq k(s)$. Then by repeatedly applying Lemma~\ref{prop-local-diff-pres-orb-pre}
we obtain
\begin{align*}
\lVert (\overline{\mu} \ast f) \vert_{\Lambda_{s}} \rVert_1 &= \| f_{k(s)}|_{\Lambda_s}\|_1  \\
&=  \|(\mu_{k(s)} \ast \cdots \ast \mu_{s}  \ast f_{s-1})|_{\Lambda_s}\|_1 \\
&\leq \lVert (\mu_{s} \ast f_{s-1}) \vert_{\Lambda_{s}} \rVert_1 \\
&\leq \varepsilon_s\ .
\end{align*}
This concludes the proof.
\end{proof}

\section{A toy example}\label{toy:sec}

The local diffusion of chains is very useful to study the behaviour of the simplicial volume of open manifolds. However, 
before applying local diffusion to locally finite chains, 
for the sake of clarity we prefer to deal with the case of ordinary singular chains (which is of use in studying the simplicial volume of closed manifolds). 
We proved in Corollary~\ref{van-cor-intro} that, 
if $X$ be a topological space admitting an open amenable cover of multiplicity $m$, and if $n\geq m$, then
$$
\|\alpha\|_1=0
$$
for every $\alpha\in H_n(X)$. 
In this section we  provide a different proof of this result under the additional hypothesis that $X$ is aspherical. 
In order to avoid the need to restrict to aspherical spaces we should introduce some more technicalities in our argument. Since this section is only meant to illustrate the 
general ideas involving diffusion of chains, we do not believe that treating the general case would be worth the  effort.  Anyway, we point out that
a proof via diffusion of chains of Corollary~\ref{van-cor-intro} in its full generality is now available in~\cite{Fri:amenable} (where it is proved the stronger fact that, if 
$X$ and $\alpha$ are as above, then the image of $\alpha$ in $H_n^{\ell^1}(X)$ is null).

Let $K$ be a multicomplex and let $\Theta(k)$ be the set of all $k$-algebraic simplices of $K$ (see Section~\ref{sec:simpl:coom}).
There exists a natural isometric identification between the space $\ell_0(\Theta(k))$ and the chain module $C_k(K)$ (both endowed with their $\ell^1$-norms),
which
 identifies an element $f\in \ell_0(\Theta(k))$ with the simplicial chain $\sum_{\sigma \in \Theta(k)} f(\sigma)\cdot \sigma$. Therefore, if $\mu$ is a probability measure on $\G$
 with finite support, the operator $\mu*\colon \ell_0(\Theta(k))\to\ell_0(\Theta(k))$ defines a diffusion operator on chains
 $$
 \mu*\colon C_k(K)\to C_k(K)\ .
 $$
An easy computation shows that, if $c = \sum_{\sigma\in\Theta(k)} a_\sigma\cdot \sigma$, then
$$\mu \ast c = \sum_{\sigma \in \, \Theta(k)} \left(\sum_{\gamma \in \, \Gamma} \mu(\gamma) f(\gamma^{-1} \cdot \sigma)\right) \sigma
=\sum_{\gamma\in\G} \mu(\gamma)(\gamma\cdot c)\ .$$
 
The fundamental result of this section is the following:

\begin{Teorema}\label{Thm:pre:van:toy:ex}
Let $K$ be a multicomplex, and
suppose that there exists a group $\Gamma$ of simplicial automorphisms of $K$ which satisfies the following properties:
\begin{itemize}
\item[(i)] $\Gamma$ is amenable;

\item[(ii)] For any $\sigma = (\Delta,(v_0,\ldots,v_k)) \in \, \Theta(k)$ there exists an element $\gamma \in \, \Gamma$ such that $\gamma \cdot \sigma = (\Delta,(v_{\tau(0)},\ldots,v_{\tau(k)}))$, where $\tau$ is an odd permutation.

\item[(iii)] Every automorphism $\gamma \in \, \Gamma$ is simplicially homotopic to the identity.
\end{itemize}
Then, for every $\alpha\in H_k(K)$ we have 
$$
\|\alpha\|_1=0\ .
$$
\end{Teorema}
\begin{proof}
 In order to (isometrically) compute the simplicial homology of $K$  we can restrict to considering chains that are \emph{alternating}, in the following sense: if $c=\sum_{\sigma\in\Theta(k)} a_\sigma\cdot \sigma\in C_k(K)$  and $\sigma=(\Delta,(v_0,\ldots,v_k))$, $\sigma'=(\Delta,(v_{\tau(0)},\ldots,v_{\tau(k)}))$ are algebraic simplices
which can be obtained one from the other via a permutation $\tau$ of the vertices, then $a_\sigma=\varepsilon(\tau)a_{\sigma'}$
(here $\varepsilon(\tau)=\pm 1$ denotes the sign of $\tau$).  
In fact, the linear operator
$\alt_*\colon C_*(K)\to C_*(K)$ such that
$$
\alt(\Delta,(v_0,\ldots,v_k))=\frac{1}{(k+1)!} \sum_{\tau\in\mathfrak{S}_{k+1}} \varepsilon(\tau)(\Delta,(v_{\tau(0)},\ldots,v_{\tau(k)})) 
$$
is well defined and homotopic to the identity. 

Therefore, if $\alpha$ is an element of $H_k(K)$, we may suppose that $\alpha=[c]$ for 
an alternating simplicial cycle
$c = \sum_{\sigma\in\Theta(k)} a_\sigma\cdot \sigma$. 
Since $c$ is a finite linear combination of simplices in $\Theta(k)$, there exist $\G$-orbits 
$ \Theta(k)_1,\ldots, \Theta(k)_s$ for the action of $\G$ on $\Theta(k)$ such that 
$$c = \sum_{j = 1}^{s} c_j\ ,$$
where
$$c_j = \sum_{\sigma\in\Theta(k)_j} a_\sigma \sigma$$
for every $j=1,\ldots,s$.


We now look for a diffusion operator $\mu*\colon C_k(K)\to C_k(K)$ such that $\|\mu \ast c \|_1 < \varepsilon$ and $[\mu \ast c] = [c]$ in $H_k(K)$.
Unfortunately, to this end we cannot directly apply Corollary~\ref{oss-scelta-epsilon} because the action of $\Gamma$ is not transitive on $\Theta(k)$. 
Therefore, rather than a single diffusion operator, we will need to exploit a finite composition of diffusion operators.

Let $\eta = \varepsilon \slash s$. We inductively define probability measures $\mu_1,\ldots,\mu_s$ satisfying the following property:
for every $j=1,\ldots, s$ and every $1\leq i\leq j$
\begin{equation}\label{inductivemu}
\|\mu_j*(\mu_{j-1}*(\ldots*(\mu_1*c_i)))\|_1\leq \eta\ .
\end{equation}
So let $i=j=1$, and
let $f_1\in \ell_0(\Theta(k))$ be the function associated to $c_1$ (by construction, $f_1$ is supported on $\Theta(k)_1$).
Since $\supp(f_1)$ is finite and contained in a single $\G$-orbit, there exist a finite subset $\Phi$ of $\Gamma$ and an algebraic simplex $\sigma_1\in\Theta(k)_1$
such that $\supp(f_1) \subseteq \Phi \cdot \sigma_1$. 
We can then apply Corollary~\ref{oss-scelta-epsilon} to  $f_1$ (now considered as a function on $\Theta(k)_1$), thus
obtaining a  probability measure $\mu_1$ on $\G$ with finite support such that 
\begin{equation}\label{formula:diffusione:toy}
\| \mu_1\ast c_1\|_1=\| \mu_1 \ast f_1 \|_1 \leq \Big\lvert \sum_{\sigma \in \, \Theta(k)_1} f_1(\sigma) \Big\rvert + \eta\ .\end{equation}

Recall now that $c$ is an alternating chain and that $\Gamma$ satisfies property (ii), that is for any algebraic simplex 
$\sigma = (\Delta,(v_0,\ldots,v_k)) \in \, \Theta(k)_1$
there exists an element $\gamma\in \Gamma$ 
such that $\gamma\cdot \sigma = (\Delta,(v_{\tau(0)},\ldots,v_{\tau(k)}))$, where $\tau$ is an odd permutation. This easily implies that 
$\sum_{\sigma \in \, \Theta(k)_1} f_1(\sigma)= 0$ which, together with~\eqref{formula:diffusione:toy}, gives 
\begin{equation}\label{formula:diffusione:toy2}
\| \mu_1 \ast c_1 \|_1 \leq \eta \ .\end{equation}
This settles the case $j=1$. 

Suppose now we have constructed probability measures $\mu_1,\ldots,\mu_l$ satisfying property~\eqref{inductivemu}, and let
$g=\mu_l*(\mu_{l-1}*(\ldots*(\mu_1*f_{l+1})))$. It is readily seen that $\supp(g)$ is finite and contained in $\Theta(k)_{l+1}$. Moreover,
diffusion preserves the alternation of chains, hence the finite chain associated to $g$ is alternating. This allows us to argue as above to 
obtain a probability measure $\mu_{l+1}$ on $\G$ with finite support such that $\|\mu_{l+1}\ast g\|_1\leq \eta$, which implies 
$$
\|\mu_{l+1}*(\mu_{l}*(\ldots*(\mu_1*c_{l+1})))\|_1
=\|\mu_{l+1}*(\mu_{l}*(\ldots*(\mu_1*f_{l+1})))\|_1=
\|\mu_{l+1}\ast g\|_1\leq \eta\ .
$$
Moreover, for every $1\leq i\leq l$, if $c'=\mu_i*(\mu_{i-1}*(\ldots*(\mu_1*c_i)))$ then we know by our inductive hypothesis
that $\|c'\|_1\leq \eta$, hence
$$
\|\mu_{l+1}*(\mu_{l}*(\ldots*(\mu_1*c_{i})))\|_1=\|\mu_{l+1}*(\mu_{l}*(\ldots*(\mu_{i+1}*c')))\|_1\leq \|c'\|_1\leq \eta\ ,
$$
where the second-last inequality is due to the fact that diffusion is always norm non-increasing. This proves the inductive step,
hence inequality~\eqref{inductivemu} for every $j=1,\ldots, s$ and every $1\leq i\leq j$.

We are now able to diffuse the chain $c$ by setting 
$$
c'=\mu_s*(\mu_{s-1}*(\ldots*(\mu_1*c)))\ .
$$
By~\eqref{inductivemu} we have
\begin{align*}
\|c'\|_1&=\Big\|\mu_s*\Big(\mu_{s-1}*\Big(\ldots*\Big(\mu_1*\Big(\sum_{j=1}^s c_j\Big)\Big)\Big)\Big)\Big\|_1\\ &=
\Big\|\sum_{j=1}^s\mu_s*\big(\mu_{s-1}*\big(\ldots*\big(\mu_1* c_j\big)\big)\big)\big\|_1 \\ &\leq
\sum_{j=1}^s \big\|\mu_s*\big(\mu_{s-1}*\big(\ldots*\big(\mu_1* c_j\big)\big)\big)\big\|_1\\ &\leq s\cdot \eta\leq \varepsilon\ .
\end{align*}

We are now left to prove that $c'$ is homologous to $c$. 
To this end we exploit our assumption (iii), i.e.~the fact that each simplicial automorphism $\gamma \in \, \Gamma$ is homotopic to the identity. 
Indeed, if $\mu$ is any probability measure with finite support on $\G$ and $z$ is any cycle in $C_k(X)$, then
$\mu \ast z =  \sum_{\gamma \in \, \Gamma} \mu(\gamma) (\gamma \cdot z)$ is a convex combination of chains of the form $\gamma \cdot z$ 
for suitable chosen elements of $\gamma \in \, \Gamma$. 
Since each $\gamma\in \Gamma$ is simplicially homotopic to the identity, this proves  that $\mu\ast z$ is a cycle homologous to $z$.
An obvious inductive argument now shows that $c'=\mu_s*(\mu_{s-1}*(\ldots*(\mu_1*c)))$ is a cycle homologous to $c$, and this concludes the proof.
\end{proof}

\begin{Corollario}\label{Van:thm:loc}
Let $X$ be an aspherical triangulable space, and assume that $X$
admits an open amenable cover $\calU$ such that $\mult (\calU)=m$. Then, for every $n\geq m$ and every  $\alpha\in H_n(X)$ we have
$$
\|\alpha\|_1=0\ .
$$
In particular, if $M$ is an aspherical triangulable closed oriented $n$-manifold admitting an open amenable cover of multiplicity not bigger than $n$, then
$$
\| M \| = 0\ .
$$
\end{Corollario}
\begin{proof}
Suppose that $X$ is homeomorphic to the geometric realization of the simplicial complex $T$, and
 let $\calU = \{U_j\}_{j \in \, J}$ be an amenable open cover of $X$ as in the statement.
  As discussed in Chapter~\ref{vanishing-thm:chap}, using \cite[Theorem 16.4]{munkres} we may assume that the triangulation $T$ of $X$ is so fine that for every vertex $v$ in $T$ there exists $j(v) \in J$ such that the closed star of $v$ is entirely contained in the element $U_{j(v)}$ of $\calU$. For every $j\in J$ we  set $$V_j = \{v \in T^0\ |\ j(v) = j\}\ .$$ 
  
  Let us consider the singular multicomplex $\calK(X)$ associated to $X$. 
  As observed e.g.~in Section~\ref{van1:sec}, 
  we have a simplicial inclusion $$l \colon T \rightarrow \calK(X)$$ which realizes $X=|T|$ as a subset $|\calK_{T}(X)| \subset |\calK(X)|$. Moreover, if $S \colon |\calK(X)| \rightarrow X$ is the natural projection, then $S \circ |l| = \id_X$. As already observed, we can choose the minimal and complete multicomplex $\calL(X)\subseteq \calK(X)$ associated to $\calK(X)$
  in such a way that $\calK_T(X)\subseteq \calL(X)$. Moreover, since the natural projection is a weak homotopy equivalence, the assumption that $X$
  is aspherical readily implies that also $\calK(X)$, hence $\calL(X)$, is aspherical. Therefore, $\calA(X)=\calL(X)$. 
  
  Take now an element $\alpha\in H_n(X)$. Since the simplicial homology of $T$ is canonically isomorphic to the singular homology of $X$,
  we may choose an alternating cycle $c\in C_n(T)\subseteq C_n(\calA(X))$ such that, if 
  $$
  \phi_*\colon C_*(\calA(X))\to C_*(|\calA(X)|)
  $$
  is the usual inclusion of simplicial chains into singular chains, then 
  $$
  H_n(S_n)(H_n(\phi_n)([c]))=\alpha\ .
  $$
  
We now aim to prove that the element $[c]\in H_n(\calA(X))$ satisfies 
 $\|[c]\|_1=0$.
 Indeed, since both $H_n(S_n)$ and $H_n(\phi_n)$ are norm non-increasing, this would imply that 
\begin{align*}
\| \alpha \| = \| H_n(S_n)(H_n(\phi_n)([c])) \|_1 
\leq \| H_n(\phi_n)([c]) \|_1 
\leq \|[c]\|_1 =0\ ,
\end{align*}
whence the thesis. 

To this end we would like to apply Theorem~\ref{Thm:pre:van:toy:ex}. Thus, we have to construct a group $\Gamma$ of automorphisms of $\calA(X)$ satisfying the conditions of Theorem~\ref{Thm:pre:van:toy:ex}. Such a group has already been described in Chapter~\ref{chap:theorems}. Indeed, following the notations introduced in Sections~\ref{sec:mapping:pi(U,V)} and~\ref{actiononA} and recalling Theorem~\ref{actionPi}, there exists a group homomorphism $$\psi \colon \Pi(X, X) \rightarrow \aut(\calA(X)),$$ such that for every element $g \in \, \Pi(X, X)$ the automorphism $\psi(g)$ is simplicially homotopic to the identity. 

Let $$H = \bigoplus_{j \in \, J} \Pi_{X}(U_j, V_j)$$ be the subgroup of $\Pi(X, X)$ introduced in Chapter~\ref{vanishing-thm:chap} (where it was denoted by the symbol $\G$). 
Since the  cover $\calU$ is amenable, Lemma~\ref{prop-psi-pi-u-v-amenable} implies that $H$ is amenable. We set $\Gamma = \psi(H)$. Being
a homomorphic image of an amenable group, $\Gamma$ is itself amenable. Moreover, by Theorem~\ref{actionPi}, every element of $\Gamma$ is 
simplicially homotopic to the identity. In order to conclude we are left to check that $\Gamma$ satisfies condition (ii) of Theorem~\ref{Thm:pre:van:toy:ex}. However, 
as already showed in the proof of the Vanishing Theorem~\ref{vanishing1_intro}, condition (ii)  follows
from the fact that $\mult(\calU)\leq n$. This concludes the proof.
\end{proof}

\chapter{Admissible submulticomplexes of $\mathcal{K}(X)$}\label{admissible:chap}
If $X$ is a non-compact topological space, the singular multicomplex $\calK(X)$ organizes the family of all possible singular simplices
(with distinct vertices) with values in $X$. While being very useful to describe finite chains in $X$, the multicomplex $\calK(X)$ is not sufficiently sensitive
to the structure of $X$ as a non-compact space in order to allow an effective study of locally finite chains.
Therefore,
following~\cite{Grom82} we  introduce a suitable submulticomplex of $\calK(X)$, whose simplices  have the following property: 
an infinite family of simplices of the submulticomplex leaves every compact subset of $X$ 
provided that the sets of vertices of the simplices in the family do so.

\section{(Strongly) admissible simplices and admissible maps}\label{strong:simpl:maps}
Henceforth we restrict to the setting 
of the Vanishing and the Finiteness Theorems for locally finite homology. Namely, we
suppose that $X$ is a non-compact connected  topological space homeomorphic to the geometric realization of a simplicial complex $L$. 
Moreover, $\calU=\{U_{j}\}_{j\in\mathbb{N}}$ 
will be a locally finite open cover of $X$ such that each $U_i$ is relatively compact in $X$ (these assumptions automatically imply that $L$ is locally finite).

We also assume that
the cover $\{U_{j}\}_{j\in\mathbb{N}}$ is amenable at infinity, i.e.~that for each $j \in \, \mathbb{N}$ there exists a large set $W_{j} \supset U_{j}$ such that 
\begin{itemize}
\item the sequence $\{W_{j}\}_{j \in \, \mathbb{N}}$ is locally finite;
\item for sufficiently large $j \in \, \mathbb{N}$, $U_{j}$ is an amenable subset of $W_{j}$.
\end{itemize}

\begin{rem}
In the assumptions of the Finiteness Theorem, the sequence $\calU=\{U_j\}_{j\in\mathbb{N}}$ is not supposed to be a cover of the whole of $X$. 
Nevertheless, in the very first step of the proof we will extend $\calU$ to an actual cover $\widehat{\calU}$ of $X$, and we will then exploit
the construction we are going to describe (which makes sense only for a cover of $X$) working with $\widehat{\calU}$ rather than with $\calU$. This justifies 
our current assumption that $\bigcup_{j\in\mathbb{N}} U_j=X$. 
\end{rem}

Henceforth, we say that $p\in X$ is a vertex of $X$ if it corresponds to a vertex of $L$ under the identification $X=|L|$. 
As we did in Chapter~\ref{vanishing-thm:chap}, we may assume that the triangulation $L$ of $X$ is so fine that the set of closed stars of vertices refines the open cover given by the $U_{j}$ 
(cfr. \cite[Thm. 16.4]{munkres}). 
We then fix a coloring of the vertices of $X$ adapted to the $U_j$, 
i.e.~a partition $\{V_j\}_{j\in J}$ of the vertices of $X$ such that for every $v\in V_j$ the closed star of
$v$ in $L$ is entirely contained in $U_j$. Since each $U_j$ is relatively compact, the set $V_j$ is finite for every $j \in \, \mathbb{N}$.

As usual, we call \emph{vertices} of a singular simplex $\sigma\colon |\Delta^n|\to X$ the images of the vertices of $\Delta^n$ via $\sigma$. 
In the sequel, we will be mainly interested in singular simplices with vertices in the set of vertices of $X$.

The following definitions play a fundamental role in the study of locally finite chains via the theory of multicomplexes.

\begin{defn}\label{admissible:map}
Let $K$ be a multicomplex. A continuous map $f\colon |K|\to X$ is \emph{admissible} if the following conditions hold:
\begin{enumerate}
\item $f$ maps each vertex of $K$ to a vertex of $X$;
\item  if $\sigma$ is an $m$-dimensional simplex of $K$ with vertices $w_0,\ldots,w_m$
and $f(w_i)\in V_{j(i)}$ for every $i=0,\ldots,m$, then
 $$f(\sigma) \subset \bigcup_{j\in J_f(\sigma)} W_{j}\ ,$$
 where $$J_f(\sigma)=\{j(0),\ldots,j(m)\}\ .$$
(This set may contain strictly less than $(m+1)$ indices, if $j(i)=j(i')$ for some $i\neq i'$.) 
\end{enumerate}
 \end{defn}

A singular simplex $\sigma\colon |\Delta^n|\to X$ is \emph{admissible} if it is admissible when endowing $\Delta^n$ with its natural structure of
multicomplex (in particular, if a singular simplex is admissible, then any of its faces is admissible too). For simplicity, if $\sigma$ is a singular admissible $n$-simplex, then we denote by $J(\sigma)$ the set of colors of the vertices
of $\sigma$ (which, according to Definition~\ref{admissible:map}, should be denoted by $J_\sigma(\Delta^n)$).
By construction, an infinite family of admissible simplices 
leaves every compact subset of $X$ 
provided that the set of vertices of the simplices in the family do so.

\begin{defn}
Let $K$ be a multicomplex and let $f,g\colon |K|\to X$ be admissible maps. An \emph{admissible homotopy} 
between $f$ and $g$ is an ordinary homotopy
$H\colon |K|\times [0,1]\to X$ between $f$ and $g$ such that the following additional conditions hold:
\begin{enumerate}
\item for every vertex $v$ of $K$, the color of $f(v)$ coincides with the color of $g(v)$ (hence, for every simplex $\sigma$ of $K$, the sets
$J_f(\sigma)$ and $J_g(\sigma)$ coincide); 
\item 
for every vertex $v$ of $K$ and every  $t\in [0,1]$, we have $H(v,t)\in U_j$, where $j$ is the color of $f(v)$;
\item for every simplex $\sigma$ of $K$, 
$$
H(|\sigma|\times [0,1])\subseteq \bigcup_{j\in J_f(\sigma)} W_j=\bigcup_{j\in J_g(\sigma)} W_j\ ;
$$
\item
there exists a finite subset $V_0$ of the set of vertices of $K$ such that, if $\Delta$ is a simplex of $K$ with no vertices
in $V_0$, then $H(x,t)=f(x)=g(x)$ for every $x\in |\Delta|$ (henceforth, when this condition holds we will say that the homotopy $H$ has \emph{bounded support}).
\end{enumerate}
If this is the case, we say that $f$ and $g$ are \emph{ad-homotopic}. It is clear from the definition that
being ad-homotopic is an equivalence relation.
\end{defn}

Observe that, if the singular simplex $\sigma'$ is obtained from the singular simplex $\sigma$ by precomposition with an affine automorphism of the standard simplex,
then $\sigma'$ is admissible if and only if $\sigma$ is. Therefore, it makes sense to speak of admissible simplices in $\calK(X)$. 
Just as we did in the case of the singular multicomplex of a topological space, we would like to define a suitable notion of completeness and minimality
for multicomplexes of admissible singular simplices. To this aim we need to introduce a stronger notion of admissible homotopy for singular simplices.

\begin{defn}
Let $\sigma,\sigma'\colon |\Delta^n|\to X$ be admissible simplices. Then $\sigma$ is \emph{strongly ad-homotopic} to $\sigma'$ if 
it is homotopic to $\sigma'$ via an admissible homotopy $H\colon |\Delta^n|\times [0,1]\to X$
 such that $H(x,t)=\sigma(x)=\sigma'(x)$ for every $x\in\partial|\Delta^n|$, $t\in [0,1]$ (i.e.~the homotopy is relative to the boundary).
Of course, being strongly ad-homotopic is an equivalence relation on the set of admissible singular simplices.
We also say that two admissible abstract simplices $[\sigma], [\sigma']$ of $\calK(X)$ are strongly ad-homotopic if they admit strongly ad-homotopic representatives.
\end{defn}

\section{Admissible multicomplexes}

Recall that $\calK(X)$ comes with a natural projection $S\colon |\calK(X)|\to X$. In the sequel we  will sometimes denote simply by $S$
the restriction of this map to any subcomplex of $|\calK(X)|$.

\begin{Definizione}
Let $A$ be a submulticomplex of $\calK(X)$.
Then $A$ is \emph{ad-complete} if the following condition holds:
Let $f \colon \lvert \partial \Delta^{n} \rvert \rightarrow \lvert A \rvert$ be a simplicial embedding, and suppose that the composition
$S\circ f\colon \lvert \partial \Delta^{n} \rvert\to X$ extends to an admissible simplex $\sigma\colon |\Delta^n|\to X$. Then 
the map $f$ extends to a simplicial embedding $f'\colon |\Delta^n|\to |A|$ such that $S\circ f'$ is admissible and strongly ad-homotopic to $\sigma$. 
In other words, $A$ contains
at least one element for every strong ad-homotopy class of admissible simplices of $\calK(X)$ whose boundary lies in $A$. 

The multicomplex $A$  is \emph{ad-minimal} if it contains
at most one element for every strong ad-homotopy class of admissible simplices of $\calK(X)$.

Finally, we say that $A$ is \emph{admissible} if it is ad-complete, ad-minimal and 
all simplices in $A$ are admissible. 
\end{Definizione}

We are now going to define a canonical admissible submulticomplex $\mathcal{AD}_{L}(X)\subseteq \calK(X)$ associated to our fixed triangulation  $L$ of $X$.
As observed in Section~\ref{am:sub:sec}, the multicomplex $\calK(X)$ contains a submulticomplex 
 $\calK_L(X)\cong L$ whose simplices
are the equivalence classes of the affine parametrizations of simplices of $L$. We will simply denote by $L$ the submulticomplex $\calK_L(X)$,
thus realizing 
$L$ as a submulticomplex of $\mathcal{K}(X)$.  

Despite all simplices in $L$ are admissible, the submulticomplex $L\subseteq \calK(X)$ is in general very far from being  ad-complete, hence admissible. 
However, it is possible to enlarge $L$ to an admissible submulticomplex of $\calK(X)$, which will be denoted by  $\mathcal{AD}_{L}(X)$:

\begin{Lemma}
There exists an admissible multicomplex $\mathcal{AD}_{L}(X)\subseteq \calK(X)$ such that $L \subseteq \mathcal{AD}_{L}(X)$
and $ \mathcal{AD}_{L}(X)^{0}=L^0$.
\end{Lemma}
\begin{proof}
We construct $\mathcal{AD}_{L}(X)$ by induction on the dimension of its skeleta. We first set $\mathcal{AD}_{L}(X)^{0} = L^{0}$, and we assume
 the $(n-1)$-skeleton of $\mathcal{AD}_{L}(X)$ has already  been defined. 
We then consider the set of admissible simplices $f\colon |\Delta^n|\to X$ whose restriction to the boundary of $\Delta^n$ 
lifts to a simplicial embedding $\partial \Delta^n\to \mathcal{AD}_{L}(X)^{n-1}$, and we add to $\mathcal{AD}_{L}(X)^{n}$ (the class of) 
one representative for each strong ad-homotopy class of such admissible simplices. We also agree that simplices in $L$ are chosen as preferred representatives of their strong ad-homotopy classes.

It is now easy to check that the submulticomplex
$$
\mathcal{AD}_{L}(X) =\bigcup_{n=0}^\infty \mathcal{AD}_{L}(X)^{n}
$$
satisfies all the conditions of the statement.
\end{proof}

By construction, the restriction of the natural projection $S\colon |\calK(X)|\to X$ to $\mathcal{AD}_{L}(X)$ is an admissible map. Also observe that this projection restricts
to a canonical bijection between the set of vertices
of $\mathcal{AD}_{L}(X)$ and the set of vertices of $X$ as a triangulated space. In the sequel, we will often denote by the same symbol
the vertices which correspond one to the other via this identification. In particular, the coloring of the vertices of $X$ naturally defines
a coloring also of the vertices of $\mathcal{AD}_{L}(X)$.

Recall from Chapter~\ref{chap2:hom} that the homotopy groups of a complete and minimal multicomplex are completely encoded by \emph{special spheres}. 
The following proposition shows that a similar result holds in the context of the ad-complete and ad-minimal multicomplex $\mathcal{AD}_{L}(X)$.
If $\Delta_0$ is a $k$-simplex of $\mathcal{AD}_{L}(X)$, then we denote by $ \pi(\Delta_0)$ the set of simplices of
$\mathcal{AD}_{L}(X)$ that are compatible with $\Delta_0$ (we refer the reader to Definition~\ref{special:def} for the definitions of compatible simplices and special spheres, and for the corresponding notation).

\begin{prop}\label{bijection-ad}
 Let 
 $\Delta_0$ be a $k$-simplex of $\mathcal{AD}_{L}(X)$, $k\geq 1$. Also fix an ordering on the vertices of $\Delta_0$ and denote by $x_0$ the minimal vertex of $\Delta_0$.
 The map
 $$
 \Theta\colon \pi(\Delta_0)\to \pi_k\left(\bigcup_{j\in J(\Delta_0)} W_{j},x_0\right)\, ,\qquad
 \Theta(\Delta)=S_*\left(\left[\dot{S}^k(\Delta_0,\Delta)\right]\right)
 $$
 is well defined and bijective.
\end{prop}
\begin{proof}
 First observe that, for every $\Delta\in \pi(\Delta_0)$, the special sphere $\dot{S}^k(\Delta_0,\Delta)$ is constructed from two admissible simplices
 with vertices in $\bigcup_{j\in J(\Delta_0)} V_j$. Therefore, the composition $S\circ \dot{S}^k(\Delta_0,\Delta)$ has indeed values in $\bigcup_{j\in J(\Delta_0)} W_{j}$,
and the map $\Theta$ is well defined.

The proof of the fact that $\Theta$ is bijective is now analogous to the proof of Theorem~\ref{complete:special}. Indeed, let us set for brevity
$W=\bigcup_{j\in J(\Delta_0)} W_j$, and let us denote by $\iota\colon \Delta^k_s \rightarrow \Delta_{0}$ the simplicial isomorphism  between 
$\Delta_0$ and
the southern hemisphere $\Delta^k_s$ 
 of $\dot{S}^k$  that  preserves the ordering on vertices.
In order to prove that $\Theta$ is surjective, let us fix an element $\alpha\in \pi_k(W,x_0)$.

Arguing as in the proof of Theorem~\ref{complete:special}, we may suppose that $\alpha$ is represented by a continuous map
$g'\colon (|\dot{S}^k|,s_0)\to (W,x_0)$ such that $g'|_{| \Delta^k_s|}=S\circ \iota$.
The restriction of $g'$ to the northern hemisphere $|\Delta^k_n|$ is now an admissible simplex, and 
$g'|_{\partial |\Delta^k_n|}=g'|_{\partial |\Delta^k_s|}=S\circ \iota|_{\partial |\Delta^k_s|}$ factors through a simplicial
embedding of $\partial |\Delta^k_s|$ into $\mathcal{AD}_{L}(X)$. By ad-completeness of $\mathcal{AD}_{L}(X)$,
there exists a simplicial embedding $\psi\colon |\Delta^k_n|\to |\Delta_1|$ onto a $k$-simplex $\Delta_1$ of $\mathcal{AD}_{L}(X)$
such that $S\circ\psi$ is 
 strongly ad-homotopic to $g'|_{|\Delta^k_n|}$. 
By definition, the strongly admissible homotopy between $S\circ\psi$ and $g'|_{|\Delta^k_n|}$ takes place in $W$, and this readily implies that
 $\Theta(\Delta_1)=\alpha$ (see e.g.~Lemma~\ref{lemma-the-following-3-cond-equivalent}). This proves the surjectivity of $\Theta$.

The fact that $\Theta$ is injective is now an easy consequence of the ad-minimality of $\mathcal{AD}_{L}(X)$: if $\Theta(\Delta_1)=\Theta(\Delta_2)$,
then the map $S\circ \dot{S}^k(\Delta_1,\Delta_2)$ is null homotopic in $W$.
Therefore, if $\psi_i\colon |\Delta^k|\to |\Delta_i|$ is the affine isomorphism that preserves the ordering on vertices, then
by Lemma~\ref{lemma-the-following-3-cond-equivalent} the singular simplices $S\circ\psi_1$ and $S\circ\psi_2$ are homotopic relative to $|\partial \Delta^k|$
within $W$, i.e.~they are strongly ad-homotopic. By ad-minimality of $\mathcal{AD}_{L}(X)$ we conclude that $\Delta_1=\Delta_2$, hence $\Theta$ is injective. 
\end{proof}

\section{Group actions on the admissible multicomplex} 
We are now going to define an action of the direct sum of groups
$$
\bigoplus_{j \in \, \mathbb{N}} \Pi(W_{j}, V_{j})
$$
on the $1$-skeleton of $\mathcal{AD}_{L}(X)$. Recall from Section~\ref{sec:mapping:pi(U,V)}
that an element $g\in\Pi(W_{j}, V_{j})$ is given by (the homotopy classes) of a collection of paths $\{\gamma_x\}_{x\in V_j}$ with values in $W_j$ with the following
properties:
$\gamma_x(0)=x$ and $\gamma_x(1)\in V_j$ for every $x\in V_j$, and the map $V_j\to V_j$ given by $x\mapsto \gamma_x(1)$ defines
a permutation of $V_j$.

As we did in the case of aspherical multicomplexes (cfr. Section \ref{sec:mapping:pi(U,V)}), we begin by defining an action 
of $\bigoplus_{j \in \, \mathbb{N}} \Pi(W_{j}, V_{j})$
on the $1$-skeleton of $\mathcal{AD}_{L}(X)$.

Let us fix an element $g=(g_j)_{j\in \mathbb{N}}$ of $\bigoplus_{j \in \, \mathbb{N}} \Pi(W_{j}, V_{j})$, and choose a set of representatives
$\{\gamma_x\}_{x\in V_j}$ for every $g_j\in \Pi(W_{j}, V_{j})$ (this notation is not ambiguous because every $x\in \mathcal{AD}(X)^0$
belongs to exactly one of the $V_j$).
For every $x\in \mathcal{AD}_{L}(X)^0$ we then set $g\cdot x=\gamma_x(1)$. We have thus constructed an action of $\bigoplus_{j \in \, \mathbb{N}} \Pi(W_{j}, V_{j})$
on the $0$-skeleton of $\mathcal{AD}_{L}(X)^0$.

Let now $e$ be a $1$-simplex of $\mathcal{AD}_{L}(X)$ with vertices $v_{0}\in V_i$, $v_{1}\in V_j$ (where possibly $i=j$).
 Let us fix an affine parametrization $\widetilde{\gamma}_{e} \colon [0,1] \rightarrow |e|$ of $|e|$, and set $\gamma_e=S\circ \widetilde{\gamma}_e$. 
 Let us consider the concatenation of paths $\gamma' \colon [0,1] \rightarrow X$ given by
$$
\gamma' = \gamma_{v_{0}}^{-1} \ast  \gamma_{e} \ast \gamma_{v_{1}}.
$$
By construction $\gamma'(0) \in \, V_{i}$ and $\gamma'(1) \in \, V_{j}$. Moreover, the image of $\gamma'$ is contained in $W_i\cup W_j$, hence $\gamma'$ is admissible. Finally,
the strong ad-homotopy class of $\gamma'$  is independent of the choice of the representatives
$\gamma_{v_0}$, $\gamma_{v_1}$. Since $\mathcal{AD}_{L}(X)$ is ad-complete and ad-minimal, there exists a unique $1$-simplex $e'$ of $\mathcal{AD}_{L}(X)$ which is strongly ad-homotopic to
(the class defined by) $\gamma'$, and we set $g\cdot e=e'$. It is immediate to check that this construction indeed defines an action of $\bigoplus_{j \in \, \mathbb{N}} \Pi(W_{j}, V_{j})$
on $\mathcal{AD}(X)^1$.

For every $j\in\mathbb{N}$ we have a natural group homomorphism $$\Pi(U_j,V_j) \to \Pi(W_j,V_j).$$ As in the previous sections,
we denote by the symbol $\Pi_{W_{j}}(U_{j}, V_{j})$ the image of $\Pi(U_{j}, V_{j})$ under this map. By restricting the action just described to the direct sum of the
$\Pi_{W_{j}}(U_{j}, V_{j})$ we obtain the following:

\begin{prop}\label{Cor-az-piUj-Wj-1-skel}
There exists a well-defined action of $$\bigoplus_{j \in \, \mathbb{N}} \Pi_{W_{j}}(U_{j}, V_{j})$$ on the $1$-skeleton of $\mathcal{AD}_{L}(X)$. 
In particular, this action induces a natural homomorphism $$\Phi \colon \bigoplus_{j \in \, \mathbb{N}} \Pi_{W_{j}}(U_{j}, V_{j}) \rightarrow \aut(\mathcal{AD}_{L}(X)^{1}),$$ 
where $\aut(\mathcal{AD}_{L}(X)^{1})$ denotes the group  of simplicial automorphisms of the $1$-skeleton of $\mathcal{AD}_{L}(X)^{1}$. 
\end{prop}

In fact, we are going to prove that an automorphism of $\mathcal{AD}_{L}(X)^1$ arises from an element in $\bigoplus_{j \in \, \mathbb{N}} \Pi_{W_{j}}(U_{j}, V_{j})$ if and only if it is \emph{admissible}
in the following sense:

\begin{Definizione}
 Let $K,K'$ be submulticomplexes of $\mathcal{AD}_{L}(X)$, and let $\varphi\colon K\to K'$ be a simplicial map. We say that $\varphi$ is \emph{admissible} if 
  the map $S|_{|K'|}\circ \varphi\colon |K|\to X$ 
   is ad-homotopic to $S|_{|K|}\colon |K|\to X$.
 
 If $K=K'$ and $\varphi$ is an automorphism, then we say that
$\varphi$ belongs to the group of admissible automorphisms of $K$, which will be denoted by $\aut_{AD}(K)$. (The fact that admissible automorphisms indeed
form a group is immediate.)
 \end{Definizione}
 
 The following lemma readily implies that every admissible  automorphism of $\mathcal{AD}_{L}(X)^1$ is induced by an element of $\bigoplus_{j \in \, \mathbb{N}} \Pi_{W_{j}}(U_{j}, V_{j})$.
 
 \begin{lemma}\label{adm:cammini}
  Let $\varphi\in \aut_{AD}(\mathcal{AD}_{L}(X)^{1})$, and let $H$ be an admissible homotopy between $S|_{|\mathcal{AD}_{L}(X)^{1}|}$
  and $S|_{|\mathcal{AD}_{L}(X)^{1}|}\circ\varphi$. For every vertex $v$ of $\mathcal{AD}_{L}(X)$ let $\gamma_v\colon [0,1]\to X$ be defined by $\gamma_v(t)=H(v,t)$. Then the element
  $\gamma\in \bigoplus_{j \in \, \mathbb{N}} \Pi_{W_{j}}(U_{j}, V_{j})$ defined by the (classes of) the $\gamma_v$ is such that $\Phi(\gamma)=\varphi$.
 \end{lemma}
\begin{proof}
 Since the homotopy $H$ has bounded support, all but a finite number of the $\gamma_v$ are constant. Moreover, 
 by definition of admissible homotopy, if $v\in V_j$ then $\gamma_v$ is supported in $U_j$. Therefore, the paths
 $\gamma_v$ indeed define an element 
 $\gamma\in \bigoplus_{j \in \, \mathbb{N}} \Pi_{W_{j}}(U_{j}, V_{j})$. The fact that $\Phi(\gamma)=\varphi$ now readily follows from the definitions. 
 \end{proof}

 \begin{lemma}\label{1dimensional:case}
  For every $\gamma\in\bigoplus_{j \in \, \mathbb{N}} \Pi_{W_{j}}(U_{j}, V_{j})$ we have $\Phi(\gamma)\in \aut_{AD}(\mathcal{AD}_{L}(X)^1)$.
 \end{lemma}
\begin{proof}
 Let $\gamma$ be defined by the (classes of the)
paths $\{\gamma_x\}_{x\in V_j, j\in \mathbb{N}}$, and let $\varphi=\Phi(\gamma)$. 
We need to prove that the map $S|_{|\mathcal{AD}_{L}(X)^{1}|}\circ \varphi$  is ad-homotopic to the natural projection
 $S|_{|\mathcal{AD}_L(X)^1|}$.
 
By definition, if $v\in V_j$ is a vertex of $X$, then the path $\gamma_v$ is supported in $U_j$. Therefore,
the map 
$$H \colon |\mathcal{AD}_{L}(X)^{0}| \times I \rightarrow X\ , \quad H(v, t) = \gamma_{v}(t)$$ 
is an admissible homotopy between 
$S \vert_{|\mathcal{AD}_{L}(X)^{0}|}$ and $S \circ \varphi \vert_{|\mathcal{AD}_{L}(X)^{0}|} = H(\cdot, 1)$. 

Let now $e$ be a $1$-simplex of $\mathcal{AD}_{L}(X)^{1}$ with endpoints $v_0,v_1$. Let $\widetilde{\gamma}_e\colon [0,1]\to |\mathcal{AD}_{L}(X)^{1}|$ be an affine parametrization of $e$
and let $\gamma_e=S\circ\widetilde{\gamma}_e$. The homotopy $H$ may be extended to a homotopy 
$H'\colon |\mathcal{AD}_{L}(X)^{1}| \times I \rightarrow X$ such that 
 the path 
$\gamma'_e\colon [0,1]\to X$ given by
$\gamma'_e(t)=H'(\widetilde{\gamma}_e(t),1)$ is a reparametrization
of the path $\gamma_{v_0}^{-1}*\gamma_e*\gamma_{v_1}$.

 We already observed that such a path is admissible, and 
$\varphi(e)$ is in fact defined as the unique $1$-simplex of $\mathcal{AD}_{L}(X)^{1}$ lying in the same strong ad-homotopy
class of $\gamma_{v_0}^{-1}*\gamma_e*\gamma_{v_1}$. As a consequence, we can
simultaneously homotope (via an admissible homotopy relative to the endpoints) each path of the form $\gamma_{v_0}^{-1}*\gamma_e*\gamma_{v_1}$ into its representative in 
$\mathcal{AD}_{L}(X)^{1}$. By concatenating these homotopies with  $H'$ we obtain the desired 
admissible homotopy
between $S|_{|\mathcal{AD}_{L}(X)^{1}|}$ and $S |_{|\mathcal{AD}_{L}(X)^{1}|}\circ \varphi$. This concludes the proof.
\end{proof}

Putting together Lemmas~\ref{adm:cammini} and~\ref{1dimensional:case} we get the following:

\begin{prop}
 $$
\Phi\left(\bigoplus_{j \in \, \mathbb{N}} \Pi_{W_{j}}(U_{j}, V_{j})\right)=\aut_{AD}(\mathcal{AD}_{L}(X)^1)\ .
$$
\end{prop}

 We are now going to prove that every admissible automorphism of $\mathcal{AD}_{L}(X)^1$ extends to an admissible automorphism
 of $\mathcal{AD}_{L}(X)$.
A rather disappointing consequence of the non-asphericity of (finite unions of) the $W_j$ is that such extension need not be unique (see Remark~\ref{extension-non-unique}).
We begin with the following:

\begin{lemma}\label{extend:admissible}
 Let $K$ be any submulticomplex of $\mathcal{AD}_{L}(X)$, let $f\colon |\mathcal{AD}_{L}(X)|\to X$ be 
 an admissible map 
 and let 
 $$
 H'\colon |K|\times [0,1]\to X
 $$
 be an admissible homotopy such that $H'(x,0)=f(x)$ for every $x\in |K|$. 
 Let also $A$ be the submulticomplex of $\mathcal{AD}_L(X)$ such that
a  simplex $\Delta$ of $\mathcal{AD}_L(X)$ belongs to $A$ if and only if $H(x,t)=x$ for every $x\in |\Delta\cap K|$ (in particular, if $\Delta\cap K=\emptyset$, i.e.~if no vertex of $\Delta$ lies in $K$,
then $\Delta\subseteq A$).

 Then $H'$ may be extended to an admissible homotopy
 $$
 H\colon |\mathcal{AD}_{L}(X)|\to X
 $$
such that $H(x,0)=f(x)$ for every $x\in |\mathcal{AD}_{L}(X)|$ and $H(x,t)=x$ for every $x\in |A|$.
\end{lemma}
\begin{proof}
 We claim that the standard proof of the homotopy extension property for CW pairs (see e.g.~\cite[Prop. 0.16]{hatcher}) already provides the required homotopy. 
 Indeed, if $H\colon |\mathcal{AD}_{L}(X)|\to X$ is the homotopy described in the reference, then for every  simplex $\sigma$ of $\mathcal{AD}_{L}(X)$
 the inclusion 
 $$
 H(|\sigma|\times [0,1])\subseteq H(|\partial \sigma|\times [0,1])\cup f(|\sigma|)
 $$
 holds. Moreover, if $x\in|A|$, then $H(x,t)=f(x)$ for every $t\in [0,1]$. This already implies that $H$ has bounded support,
 since $H'$ has bounded support.
 
 Using these facts, it is not difficult to show by induction on the dimension of simplices that $H$ is admissible. Indeed, the fact that the restriction
 of $H$ to $|\mathcal{AD}_{L}(X)|^0$ is admissible is an immediate consequence of the description of the behaviour of $H$ on the  vertices 
 of $\mathcal{AD}_{L}(X)$. Let us now suppose that the restriction of $H$ to $|\mathcal{AD}_{L}(X)^k|\times [0,1]$ is admissible, and let $\sigma$ be a $(k+1)$-dimensional simplex
 of $\mathcal{AD}_{L}(X)$ with facets $\tau_0,\ldots,\tau_{k+1}$. Then $J(\tau_i)\subseteq J(\sigma)$ for every $i=0,\ldots,k+1$,
 so thanks to our inductive hypothesis and to the admissibility of $f$ we have
\begin{align*}
 {H}(|\sigma|\times [0,1]) & \subseteq H(|\partial \sigma|\times [0,1])\cup f(|\sigma|)=\left(\bigcup_{i=0}^{k+1} H(|\tau_i|\times [0,1])\right)\cup f(|\sigma|)\\
 &\subseteq \left(\bigcup_{i=0}^{k+1} \bigcup_{j\in J(\tau_i)} W_j\right)\cup \left(\bigcup_{j\in J(\sigma)} W_j\right)
 \subseteq \bigcup_{j\in J(\sigma)} W_j\ .
\end{align*}
This shows that the restriction of $H$ to $|\mathcal{AD}_{L}(X)^{k+1}|\times [0,1]$ is admissible, thus concluding the proof.
\end{proof}

Let $\varphi$ be an element of $\aut_{AD}(\mathcal{AD}_L(X)^1)$, and recall that we aim at extending $\varphi$ to an admissible automorphism of the whole of
$\mathcal{AD}_L(X)$.
Our strategy is straightforward. 
By Lemma~\ref{extend:admissible}, 
the composition $S|_{|\mathcal{AD}_L(X)^1|}\circ\varphi$ extends to an admissible
map $f\colon |\mathcal{AD}_L(X)|\to X$ that is ad-homotopic to $S$.
Since $f$ is admissible, for every abstract $k$-dimensional simplex $\sigma$ in $\mathcal{AD}_{L}(X)$, the restriction of $f$ to $\sigma$ is an admissible singular $k$-simplex,
so we would be tempted to construct the extension $\widetilde{\varphi}$ of $\varphi$ by defining $\widetilde{\varphi}(\sigma)$ to be the class of $f|_{\sigma}$. Unfortunately, not every admissible simplex is represented
by an abstract simplex in the multicomplex $\mathcal{AD}_{L}(X)$. The following proposition takes care of this issue by suitably modifying $f$ within
its ad-homotopy class.

\begin{prop}\label{extension:prop}
Let $K,K'$ be submulticomplexes of $\mathcal{AD}_L(X)$ each of which contains $\mathcal{AD}_L(X)^1$, and let
$\varphi\colon K\to K'$ be an admissible isomorphism between $K$ and $K'$. Let also $H\colon |K|\times [0,1]\to X$ be an admissible homotopy between
$S|_{|K|}$ and $S|_{|K'|}\circ \varphi$, and let $A$ be the submulticomplex of $\mathcal{AD}_L(X)$ such that
a $k$-dimensional simplex $\Delta$ of $\mathcal{AD}_L(X)$ belongs to $A$ if and only if $H(x,t)=x$ for every $x\in |\Delta\cap K|$.

Then, there exists an admissible automorphism $\widetilde{\varphi}\in\aut_{AD}(\mathcal{AD}_L(X))$ extending $\varphi$. Moreover,
the maps $S\circ \widetilde{\varphi}$ and $\widetilde{\varphi}$ are ad-homotopic via a homotopy which is constant on $|A|$.
\end{prop}
\begin{proof}
We first prove the following claim: 
for every $k\geq 1$, there exist a simplicial isomorphism $\varphi_k\colon K\cup \mathcal{AD}_{L}(X)^{k}\to K'\cup \mathcal{AD}_{L}(X)^{k}$
and an admissible map
 $f_k\colon |\mathcal{AD}_{L}(X)|\to X$ such that the following conditions hold: 
 \begin{enumerate}
  \item $\varphi_1=\varphi$ and $f_1$ is ad-homotopic to $S$ relative to $|A|$;
\item for every $k\geq 1$, 
  $
  S|_{|K'\cup \mathcal{AD}_{L}(X)^{k}|}\circ \varphi_k=f_k|_{|K\cup \mathcal{AD}_{L}(X)^{k}|}
  $;
  \item for every $k\geq 2$, the map $f_{k}$ is ad-homotopic to $f_{k-1}$ relative to $|K\cup A\cup \mathcal{AD}_{L}(X)^{k-1}|$;
  (in particular, $f_k$ is ad-homotopic to $f_1$, hence to $S$, for every $k\geq 1$).
 \end{enumerate}
 We set $\varphi_1=\varphi$ and argue by induction on $k\geq 1$. When $k=1$, our claim reduces to the existence of a map
 $f_1\colon |\mathcal{AD}_{L}(X)|\to X$ which extends $S|_{|K'|}\circ \varphi_1$ and is ad-homotopic to $S$ relative to $|A|$.
 We can then set $f_1=H_1(x,1)$, where $H_1\colon |\mathcal{AD}_L(X)|\times [0,1]\to X$ is a homotopy extending $H$ as provided by Lemma~\ref{extend:admissible}.
 
 We may thus suppose that 
maps $f_1,\ldots,f_k,\varphi_1,\ldots,\varphi_k$ as described in the claim exist, and we now need to construct $f_{k+1}\colon |\mathcal{AD}_{L}(X)|\to X$
and $\varphi_{k+1}\colon K\cup \mathcal{AD}_{L}(X)^{k+1}\to K'\cup \mathcal{AD}_{L}(X)^{k+1}$.

Let us first define the simplicial map $\varphi_{k+1}$ as follows.
We require that the restriction of $\varphi_{k+1}$ to $K\cup \mathcal{AD}_{L}(X)^{k}$ is equal to $\varphi_k$. Then,
let $\sigma$ be a $(k+1)$-dimensional simplex of $\mathcal{AD}_{L}(X)$, and let us fix an affine parametrization $\alpha\colon \Delta^{k+1}\to |\sigma|$. 
Since $f_k$ is admissible, the singular simplex $\beta\colon \Delta^{k+1}\to X$ defined by $\beta=f_k\circ \alpha$ is admissible. 
Moreover, since $f_k|_{|\partial \sigma|}=S\circ \varphi_k|_{|\partial \sigma|}$, we have $\beta|_{|\partial \Delta^{k+1}|}=(S\circ \varphi_k\circ \alpha)|_{|\partial \Delta^{k+1}|}$, i.e.~$\beta$
is the composition of the natural projection $S$ with a simplicial embedding of $\partial \Delta^{k+1}$ into $\mathcal{AD}_{L}(X)$. Since $\mathcal{AD}_{L}(X)$
is ad-complete, this implies that the class of $\beta$ is strongly ad-homotopic to a simplex $\sigma'$ of $\mathcal{AD}_{L}(X)$. We then 
set $\varphi_{k+1}(\sigma)=\sigma'$. It is immediate to check that $\varphi_{k+1}$ is indeed a simplicial map.
Moreover, if $\sigma$ is a simplex of 
$K\cup \mathcal{AD}_{L}(X)^{k}$, then from the fact that $S|_{|K'\cup \mathcal{AD}_{L}(X)^{k}|}\circ \varphi_k=f_k|_{|K\cup \mathcal{AD}_{L}(X)^{k}|}$
we deduce that the restriction of $f_k$ to $|\sigma|$ already corresponds to a simplex of $K'\cup \mathcal{AD}_{L}(X)^{k}$, and this readily implies that
$\varphi_{k+1}$ coincides with $\varphi_k$ on $K\cup \mathcal{AD}_{L}(X)^{k+1}$.
We will prove later that $\varphi_{k+1}$ is an isomorphism of $K\cup\mathcal{AD}_{L}(X)^{k+1}$
onto $K'\cup\mathcal{AD}_{L}(X)^{k+1}$.

Let us now construct the required map $f_{k+1}\colon |\mathcal{AD}_{L}(X)|\to X$. 
By definition of strong admissible homotopy between singular simplices, we can
homotope $\beta$ (via an admissible homotopy relative to the boundary of $|\Delta^{k+1}|$) to a singular simplex whose class is equal to $\sigma'$. 
By simultaneously taking into account all the $(k+1)$-dimensional simplices of $\mathcal{AD}_{L}(X)$ we thus get an admissible 
 homotopy
$$
H'_{k+1}\colon |K\cup \mathcal{AD}_{L}(X)^{k+1}| \times I \rightarrow X
$$
between $f_k|_{|K\cup \mathcal{AD}_{L}(X)^{k+1}|}$ and the map $S_{|K'\cup \mathcal{AD}_{L}(X)^{k+1}|}\circ \varphi_{k+1}$.
We have already observed that, if $\sigma$ is a simplex of 
$K\cup \mathcal{AD}_{L}(X)^{k}$, then the restriction of $f_k$ to $|\sigma|$ corresponds to a simplex of $K'\cup \mathcal{AD}_{L}(X)^{k}$. This implies that the homotopy $H'_{k+1}$
may be chosen to be constant on $|\sigma|$. Moreover, if $\sigma$ is a simplex of $A$, then $f_k$ and the restriction of $S$ coincide on $|\sigma|$, so again we may
assume that $H'_{k+1}$ is constant on $|\sigma|$. We may thus assume that the homotopy $H'_{k+1}$ is constant on $ |K\cup A\cup \mathcal{AD}_{L}(X)^{k}|$.

Thanks to Lemma~\ref{extend:admissible} we may extend $H'_{k+1}$ to an admissible homotopy 
$$
H_{k+1}\colon |\mathcal{AD}_{L}(X)| \times I \rightarrow X
$$
relative to $ |K\cup A\cup \mathcal{AD}_{L}(X)^{k}|$
between $f_k$ and a map $f_{k+1}\colon |\mathcal{AD}_{L}(X)|\to X$ extending $S\circ \varphi_{k+1}$, and this concludes the construction of $f_{k+1}$. 

In order to conclude the proof of the claim, we are now left to show that $\varphi_{k+1}$ is an isomorphism of $K\cup\mathcal{AD}_{L}(X)^{k+1}$
onto $K'\cup\mathcal{AD}_{L}(X)^{k+1}$.

We already know that $\varphi_{k+1}|_{K\cup \mathcal{AD}_{L}(X)^k}=\varphi_k$ is an isomorphism,
so it suffices to show that $\varphi_{k+1}$ induces a bijection on simplices of dimension $k+1$. Let $\Delta$ be such a simplex, and let
$\pi(\Delta)$ be the set of simplices of $\mathcal{AD}_{L}(X)$ that are compatible with $\Delta$. As usual, we fix an ordering on the vertices of $\Delta$,
and we denote by $x_0$ the first vertex of $\Delta$. We also set $W=\bigcup_{j\in J(\Delta)} W_j$. Observe that, since the simplices of $\mathcal{AD}_{L}(X)$
are admissible, the composition of the natural projection $S$ with any special sphere constructed via simplices in $\pi(\Delta)$ has values in $W$. Moreover, by construction the coloring
of the vertices of $\varphi_{k+1}(\Delta)$ is the same as the coloring of the vertices of $\Delta$, so the same is true also for special spheres constructed via simplices in $\pi(\varphi_{k+1}(\Delta))$,
and also for the composition of $f_{k+1}$ with special spheres constructed via simplices in $\pi(\Delta)$. 
We may thus consider the maps
$$
\begin{array}{cc}
\Theta_1\colon \pi(\Delta)\to \pi_{k+1}(W,x_0)\, ,\quad &
\Theta_1(\Delta')=[S\circ \dot{S}^{k+1}(\Delta,\Delta')]\ ,\\
F_{k+1}\colon \pi(\Delta)\to \pi_{k+1}(W,\varphi_{k+1}(x_0))\, ,\quad & 
F_{k+1}(\Delta')=[f_{k+1}\circ \dot{S}^{k+1}(\Delta,\Delta')]\ ,\\
\Theta_2\colon \pi(\varphi_{k+1}(\Delta))\to \pi_{k+1}(W,\varphi_{k+1}(x_0))\, ,\quad  &
\Theta_2(\Delta')=[S\circ \dot{S}^{k+1}(\varphi_{k+1}(\Delta),\Delta')]\ .
\end{array}
$$
We know from Proposition~\ref{bijection-ad} that both $\Theta_1$ and $\Theta_2$ are bijective. Let us prove that also $F_{k+1}$ is. 
Indeed, since $\Theta_1$ is bijective, every element of $\pi_{k+1}(W,x_0)$ is of the form
$[S\circ \dot{S}^{k+1}(\Delta,\Delta')]$ for exactly one choice of $\Delta' \in \pi(\Delta)$. Now $f_{k+1}$ is ad-homotopic to $S$,
and this readily implies that, if we denote by $Z$ the submulticomplex of $\mathcal{AD}_{L}(X)$ given by all the simplices of $\pi(\Delta)$
and all their faces, the restrictions $f_{k+1}|_{|Z|}\colon |Z|\to W$ and $S|_{|Z|}\colon |Z|\to W$ are homotopic. 
As a consequence, 
every element of $\pi_{k+1}(W,\varphi_{k+1}(x_0))$ is of the form
$[f_{k+1}\circ \dot{S}^{k+1}(\Delta,\Delta')]$ for exactly one choice of $\Delta' \in \pi(\Delta)$, i.e.~$F_{k+1}$ is bijective.

Let us now consider the map
$$
\Phi_{\Delta}\colon \pi(\Delta)\to \pi(\varphi_{k+1}(\Delta))\, ,\quad \Phi_{\Delta}(\Delta')=\varphi_{k+1}(\Delta')
$$
induced by $\varphi_{k+1}$. By construction we have $\Phi_{\Delta}=\Theta_2^{-1}\circ F_{k+1}$, so $\Phi_{\Delta}$ is bijective. 
Using this we can now prove that $\varphi_{k+1}$ is itself bijective. Take two $(k+1)$-simplices $\Delta_1,\Delta_2$ of $\mathcal{AD}_{L}(X)$
such that $\varphi_{k+1}(\Delta_1)=\varphi_{k+1}(\Delta_2)$. Since $\varphi_{k+1}|_{\mathcal{AD}_{L}(X)^k}=\varphi_k$ is bijective, 
$\Delta_1$ and $\Delta_2$ are compatible, i.e.~$\Delta_2\in\pi(\Delta_1)$ and $\Phi_{\Delta_1}(\Delta_1)=\Phi_{\Delta_1}(\Delta_2)$.
Since $\Phi_{\Delta_1}$ is injective, we thus have $\Delta_1=\Delta_2$. This shows that $\varphi_{k+1}$ is injective. 

Let now $\Delta'$ be a given $(k+1)$-simplex of $\mathcal{AD}_{L}(X)$, and fix an affine parametrization $\psi\colon |\Delta^{k+1}|\to |\Delta'|$ of $\Delta'$. 
Since $\varphi_k$ is an automorphism of $\mathcal{AD}_{L}(X)^k$,
we may consider the simplicial embedding $g\colon \partial \Delta^{k+1}\to \mathcal{AD}_{L}(X)$ defined by $g=\varphi_{k}^{-1}\circ \psi|_{\partial \Delta^{k+1}}$. 
We claim that the map $S\circ g \colon |\partial \Delta^{k+1}|\to X$ may be extended to an admissible singular simplex defined on the whole of $|\Delta^{k+1}|$.
Indeed, if $H_k$ is an admissible  homotopy between $S$ and $f_k$, then we can define a continuous map $$G\colon (|\partial \Delta^{k+1}|\times [0,1])\cup (|\Delta^{k+1}|\times \{1\})\to X$$
by setting $G(x,t)=H_k(g(x),t)$ if $x\in |\partial \Delta^{k+1}|$ and $G(x,1)=S(\psi(x))$ for every $x\in |\Delta^{k+1}|$. 
After fixing an identification of $(|\partial \Delta^{k+1}|\times [0,1])\cup (|\Delta^{k+1}|\times \{1\})$ with $|\Delta^{k+1}|$ itself
(identifying the vertices of $|\partial \Delta^{k+1}|\times \{0\}$ with the vertices of $|\Delta^{k+1}|$) we may now consider $G$ as an extension
of $g$ to $|\Delta^{k+1}|$. Using that $H_k$ is an admissible homotopy, it is immediate to check that $G$ is an admissible simplex.
By ad-completeness of $\mathcal{AD}_{L}(X)$, the map $G$ is strongly ad-homotopic to a singular simplex appearing in $\mathcal{AD}_{L}(X)$. As a consequence, 
there exists a simplex $\Delta$ in $\mathcal{AD}_{L}(X)$ such that $\partial \Delta=\varphi_{k}^{-1}(\partial \Delta')$. Therefore, $\Delta' \in \, \pi(\varphi_{k+1}(\Delta)).$ Now the surjectivity of the map
$$
\Phi_{\Delta} \colon \pi(\Delta) \rightarrow \pi(\varphi_{k+1}(\Delta))
$$
implies that $\Delta'=\varphi_{k+1}(\Delta'')$ for some $\Delta''\in \pi(\Delta)$, and this concludes the proof of the surjectivity of $\varphi_{k+1}$, whence of the claim.

Let us now conclude the proof of the proposition. Observe that from (2) and (3) (or from the very construction of the $\varphi_k$) we have
$\varphi_{k+1}|_{|K\cup \mathcal{AD}_{L}(X)^{k}|}=\varphi_{k}$ for every $k\geq 1$. 
Therefore, there exist a well-defined simplicial automorphism 
$$
\widetilde{\varphi}\colon \mathcal{AD}_{L}(X)\to \mathcal{AD}_{L}(X)
$$
such that $\widetilde{\varphi}|_{|K\cup \mathcal{AD}_{L}(X)^k|}=\varphi_k$ for every $k\geq 1$, and a well-defined continuous map
$$
\widetilde{f} \colon |\mathcal{AD}_{L}(X)|\to X
$$
such that $\widetilde{f}|_{|\mathcal{AD}_{L}(X)^k|}=f_k|_{|\mathcal{AD}_{L}(X)^k|}$ for every $k\geq 1$. Moreover,
since $S|_{|\mathcal{AD}_{L}(X)^k|}\circ \varphi_k=f_k|_{|\mathcal{AD}_{L}(X)^k|}$ we also have
$$
S\circ\widetilde{\varphi}=\widetilde{f}\ ,
$$
so in order to conclude we only need to observe that $\widetilde{f}$ is ad-homotopic to $S$ relative to $|A|$. We know that $f_1$ is ad-homotopic to $S$ relative to $|A|$, and for every $k\geq 1$
the map $f_{k+1}$ is ad-homotopic to $f_k$ relative to $|A\cup\mathcal{AD}_{L}(X)^k|$. By suitably concatenating
the homotopies between $f_k$ and $f_{k+1}$ one can then construct a homotopy between $f_1$ and $\widetilde{f}$. The usual construction
of this homotopy, which is described e.g.~in \cite[Proposition 11.2]{Strom}, also ensures that the homotopy between
$f_1$ and $\widetilde{f}$ may be chosen to be admissible and relative to $|A|$, since the homotopies between $f_k$ and $f_{k+1}$ are.
\end{proof}

\begin{cor}
Every element of $\aut_{AD}(\mathcal{AD}_L(X)^1)$ extends to an element of $\aut_{AD}(\mathcal{AD}_L(X))$.
\end{cor}
\begin{proof}
 Just apply the previous proposition with $K=K'=\mathcal{AD}_L(X)^1$.
\end{proof}

\begin{rem}\label{extension-non-unique}
The extension  
$\widetilde{\varphi}$ of $\varphi$ constructed in the previous proposition need not be unique due to the non-asphericity of (finite unions of) the $W_i$.
Indeed, 
the possible non-triviality of the higher homotopy groups of the $W_i$ 
implies that 
two strongly ad-homotopic $k$-simplices may be joined by inequivalent admissible homotopies $H_{1}, H_{2}$ relative to the boundary $|\partial \Delta^{k}|$. 
Here by inequivalent homotopies we mean that there is no homotopy between $H_{1}$ and $H_{2}$ relatively to $|\partial \Delta^k| \times I$. This clearly affects the construction
of $\varphi_{k+1}$, which, therefore, is not determined by $\varphi_k$. Things would be much easier in an aspherical context: any simplicial group action
on the $1$-skeleton of a complete, minimal and aspherical multicomplex may be canonically extended to an action on the whole multicomplex, as we described in Section~\ref{actiononA}.
\end{rem}

Recall that we are identifying the fixed triangulation $L$ of $X$ with a submulticomplex of $\mathcal{AD}_{L}(X)$. The following corollary should hold for every simplex
of $\mathcal{AD}_L(X)$. However, to our purposes it is sufficient to establish it for simplices of $L$.

\begin{cor}\label{scambi}
 Let $\Delta$ be a $k$-simplex of $L\subseteq \mathcal{AD}_L(X)$ with vertices $v_0,\ldots,v_k$, let $j_{i}\in\mathbb{N}$ be the color of $v_i$ for every $i=0,\ldots,k$, and
 suppose that there exist distinct indices $i_1,i_2\in\{0,\ldots,k\}$ such that  $j_{i_1}=j_{i_2}$.
 Then, there exists an element
 $\varphi\in\aut_{AD}(\mathcal{AD}_L(X))$ such that $$\varphi(\Delta)=\Delta\ ,\quad \varphi(v_{i_1})=v_{i_2}\ , \quad \varphi(v_{i_2})=v_{i_1}\ ,$$ and
 $\varphi(v_i)=v_i$ for every $i\in\{0,\ldots,k\}\setminus \{i_1,i_2\}$.
\end{cor}
\begin{proof}
 Without loss of generality we may suppose $i_1=0$, $i_2=1$. 
 Recall that the closed stars of $v_0$ and $v_1$ in $L$ are contained in $U_{j_0}=U_{j_1}$. As a consequence, if we denote by $e_{ij}$
 the edge of $\Delta$ joining $v_i$ and $v_j$, and by $\widetilde{e}_{ij}$ the affine parametrization of $e_{ij}$ starting at $v_i$ and ending at $v_j$, then
 the paths $\gamma_0=S\circ \widetilde{e}_{01}\colon [0,1]\to X$, $\gamma_1=S\circ \widetilde{e}_{10}\colon [0,1]\to X$ take values in $U_{j_0}$. 
 Therefore, we can define an element $g=(g_j)\in\bigoplus_{j\in\mathbb{N}} \Pi_{W_j} (U_j,V_j)$
as follows: $g_{j_0}=\{\gamma_{v_0},\gamma_{v_1}\}$, and $g_{j}=1$ for every $j\neq j_0$. 

Since $S(\Delta)\subseteq U_{j_0}$, for every $i_1,i_2,i_3\in\{0,\ldots,k\}$ the paths $S\circ (\widetilde{\gamma}_{i_1i_2}*\widetilde{\gamma}_{i_2i_3})$ and
$S\circ \widetilde{\gamma}_{i_1i_3}$ are strongly ad-homotopic. This implies that the element
$$
\Phi(g)\in \aut_{AD}(\mathcal{AD}_L(X)^1)
$$
acts on the $1$-skeleton of $\Delta$ in such a way that $\Phi(g)(v_0)=v_1$, $\Phi(g)(v_1)=v_0$, and $\Phi(g)(v_j)=v_j$ for every $j\in\{2,\ldots,k\}$. 
Let us denote by $\varphi_1\in \aut (\Delta\cup \mathcal{AD}_L(X)^1)$ the unique simplicial extension of $\Phi(g)$ to $\Delta\cup \mathcal{AD}_L(X)^1$ sending $\Delta$ to $\Delta$.
Using that the whole of $S(|\Delta|)$ is contained in $U_{j_i}$ for every $i=0,\ldots,k$, it is easy to check that $\varphi_1$ is admissible. 

Now Proposition~\ref{extension:prop} (applied to $K=K'=\Delta\cup \mathcal{AD}_L(X)^1$) ensures that $\varphi_1$ can be extended to the desired automorphism $\varphi\in\aut_{AD}(\mathcal{AD}_L(X))$.
 \end{proof}

Our final purpose is the study of locally finite chains in $X$, and to this aim the multicomplex
$\mathcal{AD}_{L}(X)$ seems to be too large (for example, Theorem~\ref{Teor-Az-Tilde-Gamma-Locally-Finite} would not hold with $\mathcal{AD}'_{L}(X)$ -- see the definition below -- replaced by $\mathcal{AD}_{L}(X)$). 
Following Gromov~\cite[Section~4.2, page~63]{Grom82}, we single out a suitable smaller submulticomplex of $\mathcal{AD}_{L}(X)$, namely the smallest
submulticomplex of $\mathcal{AD}_{L}(X)$ which contains $L$ and is left invariant by the action of $\aut_{AD}(\mathcal{AD}_L(X))$.

\begin{Definizione}
Henceforth we denote by $\mathcal{AD}_{L}'(X)$ the submulticomplex of $\mathcal{AD}_{L}(X)$ which is defined as follows:
$$\mathcal{AD}_{L}'(X) \coloneqq \bigcup_{\varphi \in \aut_{AD}(\mathcal{AD}_L(X))} \varphi(L)\ .$$
The fact that $\mathcal{AD}_{L}'(X)$ is indeed a multicomplex is straightforward, and by construction the group $\aut_{AD}(\mathcal{AD}_L(X))$
acts on $\mathcal{AD}_{L}'(X)$ via simplicial automorphisms.
\end{Definizione}

\section{Amenable subgroups of $\aut_{AD}(\mathcal{AD}_{L}'(X))$}\label{Subsec:proof-amen-cert-subgroups}

Our proofs of the Finiteness and the Vanishing Theorem  for locally finite homology make an essential use of the diffusion of chains that we described in Section~\ref{Locally-Finite-Diff}. 
For technical reasons, we need to distinguish the group $\aut_{AD}(\mathcal{AD}_L(X))$ from the group given by the restriction of elements of $\aut_{AD}(\mathcal{AD}_L(X))$ to $\mathcal{AD}_{L}'(X)$, which will be denoted henceforth by the symbol
$\G$ (by construction this group is a subgroup of $\aut_{AD}(\mathcal{AD}_{L}'(X))$).

Our next goal is the construction of a suitable family of subgroups of $\G$ with respect to which the action of $\G$ on the simplices
of $\mathcal{AD}_{L}'(X)$  is locally finite. 
In order to exploit the full power of diffusion of chains, we are also interested in studying when  these subgroups  are amenable.


Let us fix an open set $U_{\bar{j}}\in\calU$, and let us set
$$
J(\overline{j})=\{j\in J\, |\, U_j\cap U_{\overline{j}}\neq\emptyset\}\ .
$$
Since $U_{\overline j}$ is relatively compact and the family $\{U_{j}\}_{j\in \mathbb{N}}$ is locally finite, the set $J(\overline{j})$ is finite.

\begin{Definizione}\label{Gjdef}
For every $\overline{j}\in\mathbb{N}$ we denote by $A^{(\overline{j})}$ the subcomplex of $\mathcal{AD}'_L(X)$ given by all the simplices
of $\mathcal{AD}'_L(X)$ having no vertices in $\bigcup_{j\in J(\overline{j})} V_j$. 
We define $\Gamma^{(\bar{j})}$ to be the subgroup of $\G$ given by all the elements $\varphi\in \G$
such that $S|_{\mathcal{AD}'_L(X)}\circ \varphi$ is ad-homotopic to $S|_{\mathcal{AD}'_L(X)}$ relative to $A^{(\overline{j})}$.
\end{Definizione}

\begin{rem}\label{gamma2}
Gromov's definition of the groups $\Gamma^{(\bar{j})}$ is slightly different from ours: according to~\cite[page 63]{Grom82}, for an element $\varphi\in \G$
to lie in $\Gamma^{(\bar{j})}$ it is sufficient that $\varphi$ restricts to the identity of $A^{(\overline{j})}$. 
There is no apparent reason why this condition should imply that $S\circ \varphi$ is ad-homotopic to $S$ relative to $A^{(\overline{j})}$.
With this definition of $\Gamma^{(\bar{j})}$ we are not able to prove the fundamental Theorem~\ref{Gjamenable}, on which the proof of the Vanishing and the Finiteness Theorems
for locally finite homology are heavily based.
\end{rem}

Recall that, if $\Delta$ is a simplex of $\mathcal{AD}'_L(X)$, then $J(\Delta)$ is the set of colors of the vertices of $\Delta$.

\begin{lemma}\label{Jj}
 Let $\Delta$ be a simplex of $\mathcal{AD}'_L(X)$. If $j\in J(\Delta)$, then $J(\Delta)\subseteq J(j)$.
\end{lemma}
\begin{proof}
 If $\Delta$ lies in $L$, then the statement readily follows from the fact that the cover of $X$ given by the closed stars of the vertices of $L$ refines the cover $\{U_j\}_{j\in\mathbb{N}}$. In order to conclude it is now sufficient to observe that every simplex
 of $\mathcal{AD}'_L(X)$ lies in the $\G$-orbit of a simplex of $L$, and $\G$ acts on $\mathcal{AD}'_L(X)$ by color-preserving automorphisms.
\end{proof}

The following result shows that the groups $\G^{(\overline{j})}$ act transitively on certain $\G$-orbits of simplices in $\mathcal{AD}'_L(X)$.

\begin{Teorema}\label{transitive:action}
Let $\Delta$ be a $k$-dimensional simplex of $\mathcal{AD}'_L(X)$ such that at least one vertex of $\Delta$ lies in $V_{\overline{j}}$, and let 
$\varphi\in \G$.
Then there exists $\varphi'\in \G^{(\overline{j})}$ such that $\varphi'|_\Delta = \varphi|_\Delta$.
\end{Teorema}
\begin{proof}
Recall that $\mathcal{AD}_{L}'(X)$
is a union of $\aut_{AD}(\mathcal{AD}_L(X))$-orbits in $\mathcal{AD}_L(X)$. Therefore, 
 it suffices to construct an element $\varphi' \in \, \aut_{AD}(\mathcal{AD}_L(X))$ such that
\begin{itemize}
\item[(i)] $\varphi'|_\Delta = \varphi|_\Delta$;
\item[(ii)] $S|_{\mathcal{AD}_L(X)}\circ \varphi'$ is ad-homotopic to $S|_{\mathcal{AD}_L(X)}$ relative to $A^{(\overline{j})}$.
\end{itemize}

By definition of $\G$, there exists an admissible homotopy $H\colon |\mathcal{AD}'_L(X)|\times [0,1]\to X$ between $S|_{\mathcal{AD}'_L(X)}$ and $S|_{\mathcal{AD}'_L(X)}\circ \varphi$. Let us now consider the following
family of paths: if $v\in \bigcup_{j\in J(\overline{j})} V_j$, then $\gamma_v(t)=H(v,t)$ for every $t\in [0,1]$; if $v$ is any other vertex of $\mathcal{AD}_L(X)$, then $\gamma_v(t)=v$
for every $t\in [0,1]$. By Lemma~\ref{adm:cammini}, 
the (classes of the) paths $\{\gamma_v\}_{v\in V}$ define an element $\gamma$ of $\bigoplus_{j\in\mathbb{N}} \Pi_{W_j}(U_j,V_j)$.

Observe now that the action of $\Phi(\gamma)$ on $\Delta$ is determined by the paths $\gamma_v$, as $v$ varies among the vertices of $\Delta$. 
Therefore, since $J(\Delta)\subseteq J(\overline{j})$ (see Lemma~\ref{Jj}), Lemma~\ref{adm:cammini} implies that $\Phi(\gamma)$ and $\varphi$ coincide on the $1$-skeleton of $\Delta$. If $\Delta'=\varphi(\Delta)$,
$K=\Delta\cup \mathcal{AD}_L(X)^1$ and $K'=\Delta'\cup \mathcal{AD}_L(X)^1$
we may thus define $\varphi_1\colon K\to K'$ as the unique simplicial 
isomorphism which coincides with $\Phi(\gamma)$ on $ \mathcal{AD}_L(X)^1$ and with $\varphi$ on $\Delta$. It is clear from the proof of  Lemma~\ref{1dimensional:case}
that $S|_{|K'|}\circ\varphi_1$ is ad-homotopic to $S|_{|K|}$ via a homotopy $H_1$ which coincides with $H$ on $|\Delta|\times [0,1]$
and is constant on $A^{(\overline{j})}$. 
Hence, the simplicial isomorphism $\varphi_1$ is admissible. We can now apply
Proposition~\ref{extension:prop} to extend $\varphi_1$ to an automorphism $\varphi'\in\aut_{AD}(\mathcal{AD}_L(X))$ such that
$S\circ \varphi'$ is ad-homotopic to $\varphi'$ relative to $A^{(\overline{j})}$, and this concludes the proof.
\end{proof}

Recall that in order to exploit diffusion of chains we need to look for actions by amenable subgroups. To this aim we fix $\overline{j}\in\mathbb{N}$ and we make the following: 

\vspace{.3cm}\par
\noindent{\bf Assumption:}
$U_j$ is amenable in $W_j$ for every $j\in J(\overline j)$. 
\vspace{.3cm}\par

Under this assumption,
we would like to prove that the group
$\Gamma^{(\bar{j})}$ is amenable. 
Recall that $L$ is locally finite. Therefore, the maximal dimension of the simplices of $L$ having at least one vertex in 
$\bigcup_{j\in J(\overline{j})} V_j$ is equal to a natural number $N\in\mathbb{N}$ (of course, this number may depend on $\overline{j}$). 
Since each simplex of  $\mathcal{AD}_{L}'(X)$ is obtained from a simplex of $L$ via a translation by a color-preserving automorphism, also the dimension of any simplex of
 $\mathcal{AD}_{L}'(X)$ having at least one vertex in $\bigcup_{j\in J(\overline{j})} V_j$ is at most $N$.

Let $\Gamma^{(\bar{j})}_{i}$ be the normal subgroup of $\Gamma^{(\bar{j})}$ pointwise fixing the $i$-skeleton $\mathcal{AD}_{L}'(X)^i$ of $\mathcal{AD}_{L}'(X)$.
The following result reduces the amenability of $\Gamma^{(\bar{j})}$ to the amenability of the quotients
$\Gamma^{(\bar{j})}_{i} \slash \Gamma_{i+1}^{(\bar{j})}$, $i=1,\ldots, N-1$.

\begin{lemma}\label{reductionlem}
 Suppose that the quotient $$\Gamma^{(\bar{j})}_{i} \slash \Gamma_{i+1}^{(\bar{j})}$$ is amenable for every $i=1,\ldots, N-1$. Then 
  $\Gamma^{(\bar{j})}$ is amenable.
\end{lemma}
\begin{proof}
Since the dimension of any simplex of
 $\mathcal{AD}_{L}'(X)$ having at least one vertex in $\bigcup_{j\in J(\overline{j})} V_j$ is at most $N$ and $\Gamma^{(\bar{j})}$ acts
 as the identity on simplices with no vertices in $\bigcup_{j\in J(\overline{j})} V_j$, we have
 $\Gamma_{N}^{(\bar{j})}=\{1\}$. 
The extension of an amenable group by an amenable group is amenable, so an easy inductive argument based on the analysis of the short exact sequences 
$$
1 \rightarrow \Gamma^{(\bar{j})}_{i+1} \slash \Gamma_{N}^{(\bar{j})} \rightarrow \Gamma^{(\bar{j})}_{i} \slash \Gamma_{N}^{(\bar{j})} \rightarrow \Gamma^{(\bar{j})}_{i} \slash \Gamma_{i+1}^{(\bar{j})} \rightarrow 1\ ,
$$
$i=1,\ldots,N-1$, implies that $\Gamma^{(\bar{j})}_{1} \slash \Gamma_{N}^{(\bar{j})}=\Gamma^{(\bar{j})}_{1}$ is amenable. 

Let us now consider the restriction map
$$
R\colon \G^{(\bar{j})} \to \aut_{AD}(\mathcal{AD}_{L}'(X)^{1})\ .
$$
The kernel of $R$ is equal to the amenable group $\Gamma^{(\bar{j})}_{1}$, so in order to conclude it is sufficient to prove that also the image of $R$ is amenable.

By definition,  if $\varphi$ belongs to $\G^{(\overline{j})}$, then $S|_{\mathcal{AD}'_L(X)}\circ\varphi$ is ad-homotopic to $S|_{\mathcal{AD}'_L(X)}$ via a homotopy
which is constant on every vertex of $X$ not belonging to $\bigcup_{j\in J(\overline{j})} V_j$.
By Lemma~\ref{adm:cammini}, this implies that the restriction of $\varphi$ to $\mathcal{AD}_{L}'(X)^{1}$ coincides with the restriction to
$\mathcal{AD}_{L}'(X)^{1}$ of an element in 
$$
\Phi\left(\bigoplus_{j\in J(\overline{j})} \Pi_{W_{j}}(U_{j}, V_{j})\right)\subseteq  \aut_{AD} (\mathcal{AD}_L(X)^1)\ .
$$
Therefore, the group $R(\G^{(\overline{j})}) $ is isomorphic to a subgroup of a quotient of $\bigoplus_{j\in J(\overline{j})} \Pi_{W_{j}}(U_{j}, V_{j}).$
Now Lemma~\ref{prop-psi-pi-u-v-amenable} ensures that each $ \Pi_{W_{j}}(U_{j}, V_{j})$, $j\in J(\overline{j})$, is amenable, and the product
of a finite number of amenable groups is amenable. Since a subgroup of a quotient of an amenable group is amenable, this concludes the proof.

\end{proof}

We are now reduced to studying the amenability of the group $\Gamma^{(\bar{j})}_{i} \slash \Gamma_{i+1}^{(\bar{j})}$ for $i = 1, \cdots, N-1$, and this will require some work.
Let us fix $i\in\{1,\ldots,N-1\}$. 
For every simplex $\Delta$ of $\mathcal{AD}_{L}'(X)$ we denote by $V(\Delta)$ the set of vertices of $\Delta$. 
Recall that the group  $\Gamma^{(\bar{j})}$ acts as the identity on any simplex having no vertices in $\bigcup_{j\in J(\overline{j})} V_j$. We need to understand the action
of $\Gamma^{(\bar{j})}$ on the remaining simplices, and we are thus lead to the following definition:
\begin{align*}
 \Theta&=\left\{\Delta\, |\, \Delta\ \textrm{is\ an\ } (i+1)-\textrm{simplex\ of}\  \mathcal{AD}_{L}'(X)\ \textrm{s.t.}\ J(\Delta)\cap J(\overline{j})\neq\emptyset\right\}
 \\ &=\left\{\Delta\, |\, \Delta\ \textrm{is\ an\ } (i+1)-\textrm{simplex\ of}\  \mathcal{AD}_{L}'(X)\ \textrm{s.t.}\ V(\Delta)\cap \left(\bigcup_{j\in J(\overline{j})} V_j\right)\neq\emptyset\right\}\ .
 \end{align*}
 Recall that $\G$ leaves invariant each $V_j$, hence it acts on $\Theta$. Let us say that two simplices $\Delta_1,\Delta_2\in\Theta$ are equivalent if
 they share the same vertex set, i.e.~if $V(\Delta_1)=V(\Delta_2)$. 
  Since $L$ is locally finite and each $V_j$ is finite, the set $\Theta$ contains a finite number of 
 simplices of $L$. 
 Using that the action of $\G$ preserves the coloring of vertices, one can observe that $\Theta$ is the union of the $\Gamma$-orbits of this finite number of simplices, and this implies
 in turn (using again that the action of $\G$ preserves the coloring of vertices) that there exists a finite subset $F$ of the set of vertices of $X$ such that $V(\Delta)\subseteq F$ for every
 $\Delta\in\Theta$. As a consequence, the number of equivalence classes of simplices in $\Theta$ is finite. We denote by $\Theta_1,\ldots,\Theta_h$ these classes, and 
 for every $k=1,\ldots,h$ we
 choose
a representative $\Delta_k$ of $\Theta_k$.


For every $k=1,\ldots,h$, let us now fix a vertex $p_k\in V(\Delta_k)$ such that $p_k\in \bigcup_{j\in J(\overline{j})} V_j$. 
We will soon deal with special spheres obtained from simplices in $\Theta$. Therefore, for every $k=1,\ldots,h$ 
we fix an ordering on $V(\Delta_k)$ with minimal element $p_k$. 
In this way, the set of vertices of every $\Delta\in\Theta$ is endowed with an ordering such that equivalent simplices are endowed with the same ordering on vertices.
\begin{Definizione}
For every $k = 1, \cdots, h$, we denote by $\mathcal{A}^{*}_k$ the submulticomplex of $\mathcal{AD}_{L}'(X)$ which contains all the $(i+1)$-simplices of 
$\Theta_k$, 
together with their faces. 
\end{Definizione}

\begin{Definizione}
Let $\sigma$ be a simplex in $\Theta_k$. Then we define the map 
$$\Phi_\sigma\colon \Gamma^{(\bar{j})}_{i} \slash \Gamma_{i+1}^{(\bar{j})} \rightarrow \pi_{i+1}(|\calA_k^*|,p_k)$$
$$ [\varphi]\to [\dot{S}(\sigma, \varphi(\sigma))]\ ,$$
where $\dot{S}(\sigma, \varphi(\sigma))$ denotes the special sphere in $\mathcal{A}^{*}_{k}$ as  in Definition~\ref{special:def}. The definition is well posed:
if $\sigma\in \Theta_k$, the simplices $\sigma$ and $\varphi(\sigma)$ lie in $\calA^*_k$; moreover, since $\varphi$ fixes the $i$-skeleton, they may indeed be glued to define a special sphere.
Finally, if $\varphi\in  \Gamma_{i}^{(\bar{j})}$ and $\widetilde{\varphi}\in  \Gamma_{i+1}^{(\bar{j})}$, then $\varphi(\widetilde{\varphi}(\sigma))=\varphi(\sigma)$, 
so $\Phi_\sigma$ is well defined on the quotient $ \Gamma^{(\bar{j})}_{i} \slash \Gamma_{i+1}^{(\bar{j})}$.

For every $k=1,\ldots,h$ let 
$$
\widehat{W}_k=\bigcup_{j\in V(\Delta_k)} W_j\ ,
$$
and observe that by construction we have $S(|\calA^*_k|)\subseteq \widehat{W}_k$. Therefore, if we denote by
$S_k\colon |\calA^*_k|\to \widehat{W}_k$ the restriction of $S$, then for every $\sigma\in \Theta_k$ we can consider the composition
$$
\Psi_\sigma\colon \Gamma^{(\bar{j})}_{i} \slash \Gamma_{i+1}^{(\bar{j})} \rightarrow \pi_{i+1}(\widehat{W}_k,p_k)\, ,\quad \Psi_\sigma=(S_k)_*\circ \Phi_\sigma\ .
$$
We finally set
$$
\Psi \colon \Gamma^{(\bar{j})}_{i} \slash \Gamma_{i+1}^{(\bar{j})} \rightarrow 
\bigoplus_{k=1}^h \bigoplus_{\sigma \in \, \Theta_k} \pi_{i+1}(\widehat{W}_k, p_{k})\, ,\quad 
\Psi=\bigoplus_{k=1}^h \bigoplus_{\sigma \in \, \Theta_k} \Psi_\sigma\ .
$$
\end{Definizione}

\begin{Osservazione}\label{nohomorem}
The map $\Psi$ introduced in the definition above is \emph{not} a group homomorphism. 
A very similar map was shown to be a group homomorphism in Lemma~\ref{phi:hom}. However, in that context we were restricting our attention only to simplicial automorphisms
homotopic to the identity \emph{relatively to the $0$-skeleton} (see Definition~\ref{definizioneGamma}). Here elements in $\Gamma^{(\bar{j})}_{i}$ are probably homotopic to the identity.
We know that their composition with the projection $S$ is homotopic to $S$ itself, but this homotopy cannot be
relative to the $0$-skeleton, and this causes some technical issues that we need to take into account. The study of the behaviour of the map $\Psi$
with respect to the group structures of its domain and of its target is carried out in Proposition~\ref{algebraicpsi}.
\end{Osservazione}

\begin{Lemma}\label{Lemma-the-set-map-is-inj}
The map $\Psi$ is  injective.
\end{Lemma}
\begin{proof}
Let $\varphi_{1}, \varphi_{2} \in \, \Gamma^{(\bar{j})}_{i}$ be such that $\Psi([\varphi_{1}]) = \Psi([\varphi_{2}])$. 
We need to show that the restrictions of $\varphi_1$ and of $\varphi_2$ to $\mathcal{AD}_{L}'(X)^{i+1}$ coincide.
Let $\sigma$ be an $(i+1)$-simplex of $\mathcal{AD}_{L}'(X)$. If $\sigma\notin \Theta$, then since $\varphi_i\in \G^{(\bar{j})}$ for $i=1,2$ we have that
$\varphi_1(\sigma)=\varphi_2(\sigma)=\sigma$. We may thus suppose that $\sigma\in \Theta_k$ for some $k\in \{1,\ldots,h\}$. 
From $\Psi_\sigma(\varphi_1)=\Psi_\sigma(\varphi_2)$ we deduce that 
$$[S \circ \dot{S}(\sigma, \varphi_{1}(\sigma))] = [S \circ \dot{S}(\sigma, \varphi_{2}(\sigma))]\ \textrm{in}\ \pi_{i+1}(\widehat{W}_k,p_k)\ .$$
By Lemma~\ref{lemma-the-following-3-cond-equivalent} we thus have that 
$S \circ \varphi_{1}(\sigma)$ is homotopic to $S \circ \varphi_{2}(\sigma)$ in $\widehat{W}_k$ relative to $\partial \Delta^{i+1}$.
By definition, this means that the simplices $\varphi_1(\sigma)$ and $\varphi_2(\sigma)$ are strongly ad-homotopic. But $\varphi_1(\sigma)$ and $\varphi_2(\sigma)$
are simplices of the multicomplex $\mathcal{AD}_{L}(X)$, which is ad-minimal, so $\varphi_1(\sigma)=\varphi_2(\sigma)$, and this concludes the proof.
\end{proof}

Recall now that, if $(Y,y_0)$ is a pointed topological space, then for every $n\geq 1$ there is an action $\pi_1(Y,y_0)\actson \pi_n(Y,y_0)$ of the fundamental group of $Y$ on the $n$-th homotopy group
of $Y$
 such that, for every $\gamma\in \pi_1(Y,y_0)$, the 
action by $\gamma$ is a group automorphism of
$\pi_n(Y,y_0)$ (see e.g.~\cite[Chapter 7]{spanier}).

The following lemma is proved in~\cite[Theorem 14, page 386]{spanier}:

\begin{lemma}\label{homaction}
 Let $(Z,z_0)$ and $(Y,y_0)$ be pointed topological spaces, let $f,g\colon  Z\to Y$ be homotopic maps such that 
 $f(z_0)=g(z_0)=y_0$, and let $H\colon Z\times [0,1]\to Y$ be a homotopy such that $H(\cdot,0)=f$ and $H(\cdot,1)=g$.
 Let $n\geq 1$, and let $f_*,g_*\colon \pi_n(Z,z_0)\to \pi_n(Y,y_0)$ be the maps induced by $f,g$, respectively.
 Then we have
 $$
 f_*(\alpha)=\gamma\cdot g_*(\alpha)
 $$
 for every $\alpha\in \pi_n(Z,z_0)$, where  $\gamma\in \pi_1(Y,y_0)$ is the class of the loop $t\mapsto H(x_0,t)$.
\end{lemma}

Let now $H^1$ be a subgroup of $\pi_1(Y,y_0)$ and  $H^2$ be a subgroup of $\pi_n(Y,y_0)$. 
Then we denote by $\aut_{H^1}(H^2)$ the following group of automorphisms of $H^2$:
$$
\aut_{H^1}(H^2)=\{f\in \aut(H^2)\, |\, \exists \gamma\in H^1\ \textrm{s.t.}
\ f(h)=\gamma\cdot h\ \forall h\in H^2\}\ . 
$$
It is immediate to check that $\aut_{H^1}(H^2)$ is indeed a group.

For every $k=1,\ldots,h$ let $j_k\in\mathbb{N}$ be such that $p_k\in V_{j_k}$, and let
$H_k^1$ be the image of $\pi_1(U_{j_k},p_k)$ in $\pi_1(\widehat{W}_k,p_k)$ under the map induced by the inclusion $U_{j_k}\hookrightarrow \widehat{W}_k$, and 
$H_k^2<\pi_{i+1}(\widehat{W}_k,p_k)$ be the image  of  $\pi_{i+1}(|\calA^*_k|,p_k)$ under the map induced by the projection $S_k = S |_{|\calA^*_k|}$.

\begin{lemma}\label{homactionamenable}
 For every $k=1,\ldots,h$, 
 the group $\aut_{H^1_k}(H^2_k)$ is  amenable.
\end{lemma}
\begin{proof}
Recall that the image of $\pi_1(U_{j_k},p_k)$ in $\pi_1(W_{j_k},p_k)$ is amenable, so from the 
 inclusions $U_{j_k}\subseteq W_{j_k}\subseteq \widehat{W}_k$ we deduce that $H_k^1$, being the quotient of an amenable
 group, is amenable. But $\aut_{H^1_k}(H^2_k)$ is a quotient of a subgroup of $H^1_k$ (namely, of the subgroup given by those elements which restrict to automorphisms of $H^2_k$), whence the conclusion.
\end{proof}

From Lemma~\ref{homaction} we deduce the following:

\begin{lemma}\label{homaction2}
For every $k=1,\ldots,h$, there exists a homomorphism $$\omega_k\colon \Gamma^{(\overline{j})}_{i}/\Gamma^{(\overline{j})}_{i+1}\to \aut_{H^1_k}(H^2_k)$$ such that 
$$
(S_k)_*\circ \varphi_*=\omega_k([\varphi])\circ (S_k)_*\ ,
$$
for every $\varphi \in \Gamma^{(\overline{j})}_{i}$,
where $$(S_k)_*\colon \pi_{i+1}(|\calA^*_k|,p_k)\to \pi_{i+1}(\widehat{W}_k,p_k)\, ,\quad  \varphi_*\colon \pi_{i+1}(|\calA^*_k|,p_k)\to \pi_{i+1}(|\calA^*_k|,p_k)$$
are the maps induced by $S_k=S|_{|\mathcal{A}^*_k|}$ and by $\varphi$ on homotopy groups.
\end{lemma}
\begin{proof}
 Take $\varphi\in \Gamma^{(\overline{j})}_{i}$. By definition of $\Gamma^{(\overline{j})}_{i}$, there exists an admissible homotopy $H\colon |\mathcal{AD}_{L}'(X)|\times [0,1]\to X$
 between $S\circ\varphi$ and $S$. By definition of admissibility, this homotopy restricts to a homotopy $H'\colon |\calA^*_k|\times [0,1]\to \widehat{W}_k$ between $S_k\circ \varphi|_{|\calA^*_k|}$ and 
 $S_k$, so Lemma~\ref{homaction} implies that, if $\gamma$ is the class of the loop $\beta(t)=H(p_k,t)$, then for every $\alpha\in \pi_{i+1}(|\calA^*_k|,p_k)$ we have
 \begin{equation}\label{sk}
 (S_k)_* (\varphi_*(\alpha))=\gamma\cdot ((S_k)_*(\alpha))\ .
 \end{equation}
 Equation~\eqref{sk} shows that the action of $\gamma$ leaves $H^2_k$ invariant, thus restricting to a homomorphism
 of $H^2_k$ into itself. Being the restriction of an automorphism of $\pi_{i+1}(\widehat{W}_k,p_k)$, this homomorphism is injective, and by replacing
 $\alpha$ with $\varphi_*^{-1}(\alpha)$ in Equation~\eqref{sk} it is easily seen that this homomorphism is also surjective. In other words,
 the action of $\gamma$ on $H^2_k$ restricts to  an element $\psi(\varphi)\in \aut_{H^1_k}(H^2_k)$ such that 
 \begin{equation}\label{sk2}
 (S_k)_* \circ \varphi_*=\psi(\varphi)\circ (S_k)_*\ .
 \end{equation}
 Observe that Equation~\eqref{sk2} implies that $\psi(\varphi)$ is completely determined by $\varphi$ (recall that by definition the domain
 of $\psi(\varphi)$ is equal to the image of $(S_k)_*$). Using this it is immediate to check that the map $\psi\colon \Gamma^{(\overline{j})}_{i}\to \aut_{H^1_k}(H^2_k)$
 is a group homomorphism. In order to conclude, it is sufficient to observe that, if $\varphi\in \Gamma^{(\overline{j})}_{i+1}$, then $\varphi$ restricts to the identity of
 $\calA^*_k$ (whose dimension, by definition, is equal to $i+1$).
 This implies that $\psi(\varphi)$ is the identity,
 hence the homomorphism $\psi$ induces the desired homomorphism $\omega_k$ on $\Gamma^{(\overline{j})}_{i}/\Gamma^{(\overline{j})}_{i+1}$.
 \end{proof}

 The following proposition provides a precise description of the fact that the maps $\Psi_\sigma\colon \Gamma^{(\overline{j})}_{i}/\Gamma^{(\overline{j})}_{i+1} \to \pi_{i+1}(\widehat{W}_k,p_k)$ 
 introduced above are not group homomorphisms.
 
 \begin{prop}\label{algebraicpsi}
  Let $\sigma\in\Theta_k$. Then for every $\varphi_1,\varphi_2\in \Gamma^{(\overline{j})}_{i}/\Gamma^{(\overline{j})}_{i+1}$ we have
  $$
  \Psi_\sigma(\varphi_1\varphi_2)=\Psi_\sigma(\varphi_1)+\omega_k(\varphi_1)(\Psi_\sigma(\varphi_2))\ .
  $$
 \end{prop}
\begin{proof}
The proof is very similar to the argument for Lemma~\ref{phi:hom}. Indeed, by definition of sum
in $\pi_{i+1}(\widehat{W}_k,p_k)$ and by Lemma~\ref{homaction2} we have
\begin{align*}
\psi_\sigma(\varphi_1\varphi_2)&=\left[S_k\circ \dot{S} (\sigma,\varphi_1\varphi_2(\sigma))\right]=
\left[S_k\circ \dot{S} (\sigma,\varphi_1(\sigma))\right]+ \left[S_k\circ \dot{S} (\varphi_1(\sigma),\varphi_1\varphi_2(\sigma))\right]\\ &=
\Psi_\sigma(\varphi_1)+\left[S_k\circ \varphi_1\circ \dot{S} (\sigma,\varphi_2(\sigma))\right]\\ &=
\Psi_\sigma(\varphi_1)+(S_k)_*\left((\varphi_1)_*\left(\left[\dot{S} (\sigma,\varphi_2(\sigma))\right]\right)\right)\\
&=\Psi_\sigma(\varphi_1)+ \omega_k(\varphi_1)\left((S_k)_*\left(\left[\dot{S} (\sigma,\varphi_2(\sigma))\right]\right)\right)\\
&=\Psi_\sigma(\varphi_1)+ \omega_k(\varphi_1)(\Psi_\sigma(\varphi_2))\ .
\end{align*}
\end{proof}

 \begin{prop}\label{quotientamenable}
  Let $i\in \{1,\ldots,N-1\}$. Then the group $\Gamma^{(\overline{j})}_{i}/\Gamma^{(\overline{j})}_{i+1}$ is amenable.
 \end{prop}
\begin{proof}
 Let 
 $$
 \omega \colon \Gamma^{(\bar{j})}_{i} \slash \Gamma_{i+1}^{(\bar{j})} \rightarrow 
\bigoplus_{k=1}^h \aut_{H^1_k}(H^2_k)\, ,\quad 
\omega=\bigoplus_{k=1}^h \omega_k\ .
$$
The product of a finite number of amenable groups is amenable, so we know from Lemma~\ref{homactionamenable} that
the image of $\omega$ is amenable. Therefore, in order to conclude it is sufficient to show that also the kernel
$K=\ker \omega$ is amenable. However, by Proposition~\ref{algebraicpsi} the restriction 
$$
\Psi|_K\colon K\to \bigoplus_{k=1}^h \bigoplus_{\sigma \in \, \Theta_k} \pi_{i+1}(\widehat{W}_k, p_{\sigma})
$$
is now a group homomorphism. We know from Lemma~\ref{Lemma-the-set-map-is-inj} that $\Psi|_K$ is injective, so $K$
is in fact isomorphic to the image of $\Psi|_K$, which is abelian (hence amenable) since $i\geq 1$. 
\end{proof}

We can finally state the main result of this section:

\begin{thm}\label{Gjamenable}
Let $j\in\mathbb{N}$, and suppose that $U_j$ is amenable in $W_j$ for every $j\in J(\overline j)$. 
Then the group
$\Gamma^{(\bar{j})}$ is amenable.
\end{thm}
\begin{proof}
 The conclusion follows from Lemma~\ref{reductionlem} and Proposition~\ref{quotientamenable}.
\end{proof}

\chapter{Diffusion of locally finite chains \\ on the admissible multicomplex}\label{vanfin:proof:chap}



Let $K$ be a multicomplex. Then the chain complex of locally finite sums of simplices of $K$ gives rise to 
 the locally finite homology of $K$. 
 More precisely, recall that an (algebraic) $n$-simplex over $K$ is a pair $(\Delta,(v_0,\ldots,v_n))$, where $\Delta$ is a simplex of $K$ with vertices 
 $\{v_0,\ldots,v_n\}$ (in particular, the dimension of the geometric simplex $\Delta$ is smaller than or equal to $n$).
 As usual, when this does not cause any confusion,
 we  simply call \emph{simplices} both the geometric and the algebraic ones.
 We say that a family $\Omega$ of algebraic $n$-dimensional simplices of $K$
 is \emph{locally finite} if for every vertex $v$ of $K$ the number of elements of $\Omega$ having $v$ as a vertex is finite.
 (It is an easy exercise to check that $\Omega$ is locally finite if and only if the geometric realizations of the corresponding geometric simplices 
 define a locally finite family of subsets of  $|K|$.) 
 A locally finite (simplicial) $n$-chain on $K$ (with coefficients in $R$) is a possibly infinite  formal sum 
 $\sum_{\sigma\in \Omega} a_\sigma \sigma$, where $\Omega$ is a locally finite family of  $n$-simplices of $K$, and $a_\sigma$ is an element of $R$ for every $\sigma\in\Omega$. 
It is easy to check that the set $C_*^{\lf}(K;R)$ of locally finite chains is an $R$-module, 
and that
the obvious  extension of the usual boundary operator sends locally finite chains to locally finite chains, so it makes sense
to define the locally finite homology $H^{\lf}_*(K;R)$ as the homology of the complex $C_*^{\lf}(K;R)$.
As usual, we  simply denote by $H_*^{\lf}(K)$ (resp.~$C_*^{\lf}(K)$) the module $H_*^{\lf}(K;\R)$ (resp.~$C_*^{\lf}(K;\R)$).
 
 We have recalled in Theorem~\ref{simpl-hom-eq-sing-one} that 
 the simplicial homology of a multicomplex is canonically isomorphic to
the singular homology of its geometric realization. 
The same result holds true for locally finite homology.
More precisely, 
let 
\begin{displaymath}
\phi_* \colon C_{*}^{\lf}(K ; R) \rightarrow C_{*}^{\lf}(\lvert K \rvert ; R)
\end{displaymath}
be the $R$-linear chain map
sending every algebraic simplex $(\sigma,(v_0,\ldots,v_n))\in C_*^{\lf}(K;R)$ to the singular simplex 
$$
\Delta^n\to |K|\, ,\quad (t_0,\ldots,t_n)\mapsto (\sigma,t_0v_0+\ldots+\ldots t_nv_n)\ .
$$
The following result may be easily deduced from \cite[Theorem 7.4]{spanier2}. Indeed, the assumption that $K$ be locally finite is probably unnecessary (but
it does not conflict with our purposes).

\begin{thm}\label{simPL hom-eq-sing-one-lf}
For any locally finite multicomplex $K$, the homomorphism
\begin{displaymath}
H_{*}^\lf(K) \rightarrow H_{*}^\lf(\lvert K \rvert)
\end{displaymath}
induced by $\phi_*$ 
is an isomorphism in every degree.
\end{thm}

Just as for ordinary homology, one may also restrict to considering chains that are \emph{alternating}, in the following sense:
if $c=\sum_{\sigma\in\Omega} a_\sigma\cdot \sigma$ is a locally finite chain, and $\sigma=(\Delta,(v_0,\ldots,v_k))$, $\sigma'=(\Delta,(v_{\tau(0)},\ldots,v_{\tau(k)}))$ are algebraic simplices
which can be obtained one from the other via a permutation $\tau$ of the vertices, then $a_\sigma=\varepsilon(\tau)a_{\sigma'}$
(here $\varepsilon(\tau)=\pm 1$ denotes the sign of $\tau$).  Indeed, the linear operator
$\alt_*\colon C_*^\lf(K)\to C_*^\lf(K)$ such that
$$
\alt(\Delta,(v_0,\ldots,v_k))=\frac{1}{(k+1)!} \sum_{\tau\in\mathfrak{S}_{k+1}} \varepsilon(\tau)(\Delta,(v_{\tau(0)},\ldots,v_{\tau(k)})) 
$$
is well defined and homotopic to the identity.


\section{An important locally finite action}

Let us now come to the context we are interested in. 
We keep notation from the previous chapter, i.e.~we denote by $X$ the geometric realization of a locally finite simplicial complex $L$,
and by $\mathcal{AD}_L(X)$ and $\mathcal{AD}'_L(X)$ the multicomplexes described there. Also recall that the construction of $\mathcal{AD}_L(X)$ and  $\mathcal{AD}'_L(X)$
is based on 
fixed locally finite open covers $\{U_i\}_{i\in\mathbb{N}}$  and $\{W_i\}_{i\in\mathbb{N}}$ 
such that $U_i\subseteq W_i$ and $W_i$ is large for every $i\in\mathbb{N}$.
Finally, let $\G<\aut_{AD}(\mathcal{AD}'_L(X))$ be the group described at the beginning of Section~\ref{Subsec:proof-amen-cert-subgroups}.

\begin{lemma}\label{lfequiv}
Let $\Omega$ be a family of algebraic $n$-simplices in $\mathcal{AD}'_L(X)$. Then $\Omega$ is locally finite if and only if the 
family of singular simplices 
$$
\{S\circ (\phi_n(\sigma))\, |\, \sigma\in\Omega\}
$$
is locally finite in $X$.
\end{lemma}
\begin{proof}
Suppose first that $\Omega$ is locally finite, and let $Z$ be a compact subset of $X$. 
Since the $W_i$ are locally finite, the set
$F=\{j\in\mathbb{N}\, |\, Z\cap W_j\neq \emptyset \}$ is finite. Let us take an element  $\sigma\in\Omega$, and denote
by $\sigma'$ the geometric simplex of $\mathcal{AD}'_L(X)$ corresponding to the algebraic simplex $\sigma$.

Using that $\sigma'$ is admissible, we deduce that
$$
\emptyset \neq S(|\sigma'|)\cap Z\subseteq \left(\bigcup_{j\in J(\sigma')} W_j\right) \cap Z\ ,
$$
hence $J(\sigma')\cap F\neq \emptyset$.
We have thus shown that 
any simplex $\sigma\in\Omega$ such that the image of $S\circ(\phi_n(\sigma))$ intersects $Z$ 
has at least one vertex in the finite set $\bigcup_{j\in F} V_j$. By definition of locally finiteness for simplicial chains, this readily implies that the number of simplices
$\sigma\in\Omega$ such that  the image of $S\circ(\phi_n(\sigma))$ intersects $Z$  is finite, and this proves that the family $\{S\circ (\phi_n(\sigma))\, |\, \sigma\in\Omega\}$
is locally finite in $X$.

On the other hand, if the family $\{S\circ (\phi_n(\sigma))\, |\, \sigma\in\Omega\}$
is locally finite in $X$, then every vertex of $X$ appears in a finite number of simplices of the form $S\circ (\phi_n(\sigma))$,  $\sigma\in\Omega$,
which means that $\Omega$ is locally finite.
\end{proof}

Lemma~\ref{lfequiv} implies that the composition $S_*\circ\phi_*\colon C_*^{\lf}(\mathcal{AD}'_L(X))\to C_*^{\lf}(X)$ is well defined, and it induces a  map
$$
P_*\colon H_*^{\lf}(\mathcal{AD}'_L(X))\to H_*^{\lf}(X)\ .
$$

Our next aim is to prove 
 that the action of $\G$ over the set of algebraic simplices of $\mathcal{AD}_{L}'(X)$ is locally finite according to Definition \ref{Def-locally-finite-action}. 

Let us fix a natural number $k$ with $0\leq k\leq N$. We denote by $\Lambda(k)$ the set of algebraic $k$-simplices of $\mathcal{AD}'_L(X)$. For 
every $(j_0,\ldots,j_k)\in \mathbb{N}^{k+1}$, we denote by $\Lambda_{(j_0,\ldots,j_k)}\subseteq \Lambda(k)$ the set of algebraic simplices whose vertex set coloring is equal to 
$(j_0,\ldots,j_k)$, i.e.~we set
$$
\Lambda_{(j_0,\ldots,j_k)}=\{(\sigma,(v_0,\ldots,v_k))\in \Lambda(k)\, |\, v_i\in V_{j_i}\ \forall\ i=0,\ldots,k\}\ .
$$
We also denote by $\calC(k)$ the subset of $\mathbb{N}^{k+1}$ given by those $(k+1)$-tuples $(j_0,\ldots,j_k)$ for which
$\Lambda_{(j_0,\ldots,j_k)}\neq\emptyset$. 

\begin{lemma}\label{finitiCk}
Let $\overline{j_0}\in\mathbb{N}$ be fixed. Then, the set of $(k+1)$-tuples of $\calC(k)$ which contain $\overline{j_0}$ is finite.
\end{lemma}
\begin{proof}
It is sufficient to show that the set of $(k+1)$-tuples of $\calC(k)$ of the form
$(\overline{j_0},j_1,\ldots,j_k)$ is finite. So, suppose $(\overline{j_0},j_1,\ldots,j_k)\in\calC(k)$. Then, there exists a $k$-simplex $\sigma$ of $\mathcal{AD}_L'(X)$
such that $\{\overline{j_0},j_1,\ldots,j_k\}=J(\sigma)$. By Lemma~\ref{Jj}, from $\overline{j_0}\in J(\sigma)$ we deduce that 
$j_i\in J(\sigma)\subseteq J(\overline{j_0})$ for every $i=1,\ldots,k$. Since $J(\overline{j_0})$ is finite, this concludes the proof.
\end{proof}

\begin{lemma}\label{Lemma-numero-di-orbite-tilde-gamma-numerabili}
For every $(j_0,\ldots,j_k)\in \calC(k)$, the group
$\G$ acts on $\Lambda_{(j_0,\ldots,j_k)}$. Moreover, the set $\Lambda_{(j_0,\ldots,j_k)}$ is partitioned
into finitely many $\G$-orbits, whose number will be denoted by $M(j_0,\ldots,j_k)$.
\end{lemma}
\begin{proof}
Every element of $\G$ acts on the $0$-skeleton of $\mathcal{AD}'_L(X)$ by color-preserving automorphisms, and this implies
that $\G$ acts on $\Lambda_{(j_0,\ldots,j_k)}$. Since every $V_j$ is finite, there exists a finite set $F$ of vertices of $X$ such that
all the vertices of any simplex in  $\Lambda_{(j_0,\ldots,j_k)}$ are contained in $F$. Thus $\Lambda_{(j_0,\ldots,j_k)}$ contains a finite
number of simplices of $L$. Since every simplex of $\mathcal{AD}'_L(X)$ is $\G$-equivalent to a simplex
in $L$, this implies in turn that the number of the $\G$-orbits contained in $\Lambda_{(j_0,\ldots,j_k)}$ is finite.
\end{proof}

\begin{cor}
The action $\G \actson \Lambda(k)$ has countably many orbits.
\end{cor}

\begin{Definizione}
We denote by $$\{\Lambda_{(j_{0}, \cdots, j_{k})}^{i}\}_{i = 1,\ldots,M(j_0,\ldots,j_k)}$$
the orbits of the action $\G \actson \Lambda(k)$.
\end{Definizione}

We are now going  to prove that for every $k \in \mathbb{N}$ the action $\G \actson \Lambda(k)$ is locally finite.
To this end, for every $(j_0,\ldots,j_k)\in\calC(k)$ and every $i\in\{1,\ldots,M(j_0,\ldots,j_k)\}$ we need to choose
a subgroup  $\Gamma_{(j_{0}, \cdots, j_{k})}^{i}$ of $\aut_{AD}(\mathcal{AD}_{L}'(X))$. Indeed, it is sufficient to set
$$
\Gamma_{(j_{0}, \cdots, j_{k})}^{i}=\G^{(j_0)}\ ,
$$
where $\G^{(j_0)}$ denotes the group 
 described in Definition~\ref{Gjdef}. With this choice, we have the following:

\begin{Teorema}\label{Teor-Az-Tilde-Gamma-Locally-Finite}
The action $\G \actson \Lambda(k)$ is locally finite.
\end{Teorema}
\begin{proof}
We need to prove that the conditions described in Definition~\ref{Def-locally-finite-action} are satisfied. 
Theorem~\ref{transitive:action} ensures that $\G^i_{(j_0,\ldots,j_k)}$ acts transitively on $\Lambda^i_{(j_0,\ldots,j_k)}$, hence we are left 
to show that the actions of the $\Gamma_{(j_{0}, \cdots, j_{k})}^{i}$ are asymptotically disjoint. 
To this aim, let us consider an algebraic $k$-simplex $\sigma \subset \Lambda(k)$, and suppose that 
either $\sigma$ belongs to 
$\Lambda^i_{(j_{0}, \cdots, j_{k})}$, or $\sigma$ is not fixed by some element in $\G^i_{(j_{0}, \cdots, j_{k})}$ (or both). In any case we have
$J(\sigma)\cap J(j_0)\neq \emptyset$,
and thanks to Lemma~\ref{Jj} 
$$
J(\sigma)\subseteq \bigcup_{j\in J(j_0)} J(j)\ . 
$$
As a consequence, there exists a finite set $F$ of vertices of $X$, which only depends on $j_0$, such that all the vertices of $\sigma$ are contained in $F$.
This easily implies that 
 only  a finite number of $\G^{(j)}$ may act non-trivially on $\sigma$. Now the conclusion follows from the fact that, for any fixed
 $j\in\mathbb{N}$, there exists only a finite number of groups $\Gamma_{(j'_{0}, \cdots, j'_{k})}^{i'}$ which are equal to $\G^{(j)}$ (see Lemma~\ref{finitiCk}).
\end{proof}


\section{Locally finite functions versus locally finite chains}\label{lff:lfc}
Just as we did in Section~\ref{toy:sec}, 
we now need to translate chains into functions, and viceversa. The following formula
describes
a tautological correspondence between (possibly non-locally finite) $k$-dimensional chains on $\mathcal{AD}_L'(X)$ and (possibly non-locally finite) real valued
functions on $\Lambda(k)$:
$$
f\colon \Lambda(k)\to\mathbb{R} \quad \longleftrightarrow \quad \sum_{\sigma \in \Lambda(k)} f(\sigma) \cdot \sigma \ .
$$

Recall that $\Lambda(k)$ is partitioned into the $\G$-orbits $\Lambda^i_{(j_0,\ldots,j_k)}$, $(j_0,\ldots,j_k)\in\calC(k)$, $i=1,\ldots,M(j_0,\ldots,j_k)$. A function 
$f\colon \Lambda(k)\to\mathbb{R}$ will be said to be \emph{locally finite} if it is locally finite with respect to this partition (see Definition~\ref{locallyfinitefun}).

\begin{prop}\label{local:finit:translation}
Let $f\colon \Lambda(k)\to\mathbb{R}$. Then $$f\in \lf(\Lambda(k))\quad \Longleftrightarrow \quad \sum_{\sigma \in \Lambda(k)} f(\sigma) \cdot \sigma \ \in \ C_k^\lf(\mathcal{AD}_L'(X))\ .$$
\end{prop}
\begin{proof}
Recall that, for every $(j_0,\ldots,j_k)\in\calC(k)$, the set $\Lambda_{(j_0,\ldots,j_k)}$ is the union of a finite number of $\G$-orbits (see Lemma~\ref{Lemma-numero-di-orbite-tilde-gamma-numerabili}). Therefore, $f\in\lf(\Lambda(k))$ if and only if $\supp(f)\cap \Lambda_{(j_0,\ldots,j_k)}$ is finite for every $(j_0,\ldots,j_k)\in\calC(k)$. By Lemma~\ref{finitiCk},
this is in turn equivalent to the fact that, for any given $\overline{j_0}\in\mathbb{N}$, if $\calC(k,\overline{j_0})$ denotes the subset of $\calC(k)$
given by those $(k+1)$-tuples that contain $\overline{j_0}$, then
$$
\supp(f)\cap \left(\bigcup_{({j_0},j_1,\ldots,j_k)\in\calC(k,\overline{j}_0)} \Lambda_{({j_0},\ldots,j_k)}\right)
$$
is finite. Since the number of vertices of $\mathcal{AD}_L'(X)$ colored by $\overline{j_0}$ is finite, we have thus shown that $f\in\lf(\Lambda(k))$ if and only if, for every vertex $v$ of 
$\mathcal{AD}'_L(X)$, the set of simplices in $\supp(f)$ having $v$ as a vertex is finite. But this is exactly the definition of local finiteness for the chain 
$\sum_{\sigma \in \Lambda(k)} f(\sigma) \cdot \sigma$.
\end{proof}

Let us now discuss how diffusion of functions translates into the context of (locally finite) chains. To this aim, we fix an arbitrary ordering $\Lambda_s$, $s\in\mathbb{N}$,
on the orbits $\Lambda_{(j_0,\ldots,j_k)}^i$, as $(j_0,\ldots,j_k)$ varies in $\calC(k)$ and $i$ in $\{1,\ldots, M(j_0,\ldots,j_k)\}$. Of course,
we fix the corresponding ordering also on the subgroups $\G_{(j_0,\ldots,j_k)}^i$, which will be denoted by $\G_s$, $s\in\mathbb{N}$. 
For every $s\in\mathbb{N}$ we choose a probability measure $\mu_s$ on $\G$ with finite support contained in $\G_s$, and as usual we denote
by $\overline{\mu}$ the collection $\{\mu_s\}_{s\in\mathbb{N}}$. If
$$
c=\sum_{\sigma\in \Lambda(k)} f(\sigma)\sigma
$$
is a locally finite $k$-dimensional chain on $\mathcal{AD}'_L(X)$, then we set 
$$
c_s=\mu_s*(\mu_{s-1}*(\ldots *(\mu_1*c)\ldots))\ ,
$$
where we are identifying locally finite chains with locally finite functions on $\Lambda(k)$ thanks to Proposition~\ref{local:finit:translation}. For example,
we have
\begin{align*}
c_1&=\mu_1*c=\sum_{\sigma\in\Lambda(k)} \left(\sum_{\gamma\in\G} \mu_1(\gamma)f(\gamma^{-1}\sigma)\right)\sigma\\ &=
\sum_{\sigma\in\Lambda(k)} \sum_{\gamma\in\G} \mu_1(\gamma)f(\sigma)(\gamma\cdot\sigma)\\ &=\sum_{\gamma\in\G}\mu_1(\gamma) (\gamma\cdot c)
\end{align*}
(where we denote by $\gamma\cdot c$ the image of $c$ under the obvious action of $\gamma\in\G$ on $C^\lf_k(\mathcal{AD}'_L(X))$).

The following important result shows that diffusion with respect to the $\G_s$ does not alter the homology class of the push-forward via the natural projection $S$.
Recall  that we have a natural map $P_*\colon H_*^\lf(\mathcal{AD}_L'(X))\to H_*^\lf(X)$ induced
by the composition $S_*\circ\phi_*\colon C_*^\lf(\mathcal{AD}_L'(X))\to C_*^\lf(X)$. 


\begin{prop}\label{diffusion:homology}
Let $c\in C_k^\lf(\mathcal{AD}'_L(X))$ be a locally finite alternating cycle. Then 
$$
P_*([c])=P_*([\overline{\mu}*c])
$$
in $H_k^\lf(X)$.
\end{prop}
\begin{proof}
Let us set $c_0=c$, and for every $s\in\mathbb{N}$ set $c_s=\mu_s*(\mu_{s-1}*(\ldots *(\mu_1*c)\ldots))$. Observe that, since $c$ is alternating,
every $c_s$ is also alternating.

Let us now fix $s\geq 1$. 
Recall that $\G_s=\G^i_{(j_0,\ldots,j_k)}=\G^{(j_0)}$ for some $j_0=j_0(s)\in\mathbb{N}$. We then set
$$
\widehat{W}_s=\bigcup_{j\in J(j_0)} \bigcup_{i\in J(j)} W_i\ .
$$
We first observe that the $\widehat{W}_s$ are locally finite. In order to prove this fact, since the $W_i$ are locally finite,
it is sufficient to show that, for every $i\in\mathbb{N}$, we have $i\in \bigcup_{j\in J(j_0(s))}  J(j)$ only for a finite number of indices $s\in\mathbb{N}$. However, this easily follows
from Lemma~\ref{Lemma-numero-di-orbite-tilde-gamma-numerabili} (and the finiteness of $J(k)$ for every $k\in\mathbb{N}$).

We now claim that $S_*(\phi_*(c_s))-S_*(\phi_*(c_{s-1}))=\partial b_s$, where $b_s$ is a finite chain supported in $\widehat{W}_s$.
As observed above, we have
$$
c_{s}=\sum_{\gamma\in{\G}}\mu_{s}(\gamma) (\gamma\cdot c_{s-1})=\sum_{\gamma\in\G_{s}}\mu_{s}(\gamma) (\gamma\cdot c_{s-1})\ ,
$$
hence $c_s$ is a convex combination of elements of the form $\gamma\cdot c_{s-1}$, for suitable chosen elements $\gamma\in\G_s$. Therefore, in order to prove the claim it is sufficient to show that, for every $\gamma\in\G_s$, there exists a finite chain $b_s$ supported in $\widehat{W}_s$ and such that
 $S_*(\phi_*(\gamma\cdot c_{s-1}))-S_*(\phi_*(c_{s-1}))=\partial b_s$. 
 
 Let us take an element $\gamma\in\G_s$, 
 let 
 $H\colon |\mathcal{AD}'_L(X)|\times [0,1]\to X$ be an admissible homotopy between $S|_{\mathcal{AD}'_L(X)}\circ \gamma$ and $S|_{\mathcal{AD}'_L(X)}$ which is constant on every simplex having 
 no vertex colored by an element of $J(j_0)$, and let $T_k\colon C_k^{\lf}(\mathcal{AD}'_L(X))\to C_{k+1}^{\lf} (X)$ be the algebraic homotopy obtained by composing
 the map $\phi_k\colon C_k^{\lf}(\mathcal{AD}'_L(X))\to C_k^{\lf}(|\mathcal{AD}'_L(X)|)$ with the usual algebraic homotopy associated to $H$. 
 We denote by $\Omega'$ the subset of $\Lambda(k)$ given by the algebraic simplices with at least one vertex colored by an index in $J(j_0)$, and we set
 $\Omega''=\Lambda(k)\setminus \Omega'$. We also set $c_{s-1}=c_{\Omega'}+c_{\Omega''}$, where 
  $$
 c_{\Omega'}=\sum_{\sigma\in\Omega'} a_\sigma \sigma\ ,\qquad c_{\Omega''}=\sum_{\sigma\in\Omega''} a_\sigma \sigma\ .
 $$
 Let us now describe how the homotopy operator $T_k$ acts on simplices in $\Omega''$. 
If $\sigma=(\Delta,(v_0,\ldots,v_k))$ belongs to $\Omega''$, then  by definition of $\G_s$ we have $\gamma\cdot \sigma=\sigma$. Moreover,
 $H(x,t)=S(x)$ for every $x\in |\Delta|$, $t\in [0,1]$. This readily implies that 
 $$
 T_k(\sigma)=\sum_{i=0}^k S\circ (\phi_k(\Delta,(v_0,\ldots,v_i,v_i,\ldots,v_k)))
 $$
 is a degenerate singular simplex. But  $c_{s-1}$ is alternating, hence also $c_{\Omega''}$ is so. Putting together these facts we easily obtain 
 $T_k(c_{\Omega''})=0$. 
 
On the other hand, if $\sigma=(\Delta,(v_0,\ldots,v_k))$ belongs to $\Omega'$,
 then the colors of all the vertices of $\sigma$ lie in $\bigcup_{j\in J(j_0)} J(j)$ (see Lemma~\ref{Jj}). As a consequence, $H(|\Delta|\times [0,1])\subseteq \widehat{W}_s$. Moreover,
  since $c_{s-1}$ is locally finite, the chain $c_{\Omega'}$ is finite. As a consequence, the chain $b_s=T_k(c_{\Omega'})$ is finite and supported in $\widehat{W}_s$.
This concludes the proof of the claim, since
$$
S_*(\phi_*(\gamma\cdot c_{s-1}))-S_*(\phi_*(c_{s-1}))=\partial (T_k(c_{s-1}))=\partial (T_k(c_{\Omega'}))=\partial b_s\ .
$$

It is now straightforward to conclude the proof of the proposition.
Since the $\widehat{W}_s$ are locally finite,  the chain $b=\sum_{s=1}^\infty b_s$ is locally finite. Moreover, it readily follows from the definitions that
$S_*(\phi_*(\overline{\mu}*c))-S_*(\phi_*(c))=\partial b$, whence the conclusion.
\end{proof}

\section{Proof of the Vanishing Theorem for locally finite homology}
Let us briefly recall the setting of Theorem~\ref{Van-Theor-intro}.
We denote by $X$ be a connected non-compact triangulable topological space.
The space $X$ is endowed with  locally finite open covers $\calU=\{U_i\}_{i\in\mathbb{N}}$,   $\mathcal{W}=\{W_i\}_{i\in\mathbb{N}}$ such that each $U_i$ is relatively compact in $X$. 
Moreover:
\begin{enumerate}
\item each $U_i$, $i\in\mathbb{N}$, is amenable in $X$;
\item each $W_i$, $i\in\mathbb{N}$, is large;
\item $U_i\subseteq W_i$ for every $i\in\mathbb{N}$, and there exists $j\in\mathbb{N}$ such that $U_i$ is amenable in $W_i$ for every $i\geq j$;
\item the multiplicity of the cover $\calU$ is equal to $m$. 
\end{enumerate}
We are going to prove that, for every $k\geq m$ and every $h \in \, H^{\lf}_{k}(X)$, we have $$\lVert h \rVert _1= 0\ .$$

First observe that, after replacing $W_i$ with  $X$ for a finite number of indices, we may assume that $U_i$ is amenable in $W_i$ for every $i\in\mathbb{N}$.

Let $k\geq m$, and fix a class $\alpha\in H_k^\lf(X)$. 
Since $X$ is homeomorphic to the geometric realization of a locally finite simplicial complex $L$, 
Theorem~\ref{simPL hom-eq-sing-one-lf}
implies that there exists a locally finite simplicial cycle $$c \in \, C_{k}^{\lf}(L; \mathbb{R}) \subset C_{k}^{\lf}(\mathcal{AD}_{L}'(X); \mathbb{R})$$ such that $[S_*(\phi_*( c))] = \alpha$, where $S$ is the restriction of the natural projection.
Moreover, we may assume that  $c$ is alternating.

Let  ${\Gamma}$ be the group introduced in Section \ref{Subsec:proof-amen-cert-subgroups}, and denote by $\Lambda_s$ and $\G_s$ the ${\G}$-orbits in $\Lambda(k)$ 
and the subgroups of ${\G}$ described in Section~\ref{lff:lfc}, respectively. The following result exploits the fact that the multiplicity of $\calU$ is not bigger than $k$.

\begin{lemma}\label{sommazero}
Let $c=\sum_{\sigma\in\Lambda(k)} a_\sigma\cdot \sigma$ be as above. Then
$$
\sum_{\sigma\in\Lambda_s} a_\sigma=0
$$
for every $s\in\mathbb{N}$.
\end{lemma}
\begin{proof}
Let us take an algebraic simplex $\sigma=(\Delta,(v_0,\ldots,v_k))\in\Lambda_s$, and for every $i=0,\ldots,k$ let $j_i$ be the color of $v_i$. Since $k\geq m$, there exist 
distinct indices $i_1,i_2\in\{0,\ldots,k\}$ such that $j_{i_1}=j_{i_2}$. Let $\sigma'$ be obtained from $\sigma$ by switching $v_{i_1}$ and $v_{i_2}$, i.e.~let 
$\sigma'=(\Delta,(v'_0,\ldots,v'_k))$, where $v'_{i_1}=v_{i_2}$, $v'_{i_2}=v_{i_1}$, an $v'_j=v_j$ if $j\notin \{i_1,i_2\}$. Corollary~\ref{scambi} ensures that
$\sigma'\in\Lambda_s$. Moreover, since $c$ is alternating, we have $a_{\sigma'}=-a_\sigma$. This shows that the simplices of $\Lambda_s$ appearing in $c$
may be partitioned in pairs, in such a way that the sum of the coefficients of the simplices of each pair vanishes. This concludes the proof. 
\end{proof}

We are now ready to conclude the proof of the Theorem~\ref{Van-Theor-intro}. Let $\varepsilon>0$ be fixed, and for every $s\in\mathbb{N}$ let us set
$\varepsilon_s=2^{-s-1}\varepsilon$. Also denote by $f\colon \Lambda(k)\to\mathbb{R}$ the locally finite function associated to the locally finite cycle $c$ (see
Proposition~\ref{local:finit:translation}).
By Theorem~\ref{Gjamenable}, our assumptions on the covers $\{U_i\}_{i\in\mathbb{N}}$, $\{W_i\}_{i\in\mathbb{N}}$ ensure that every $\G_s$ is amenable. 
Together with Lemma~\ref{sommazero}, this implies that we can apply Proposition~\ref{Prop-locally-finite-diff-op-norm-min-eps} to construct  a
local diffusion operator $\overline{\mu}*\colon \lf(\Lambda(k))\to \lf(\Lambda(k))$ such that
$$
\|\overline{\mu}*f|_{\Lambda_s}\|_1\leq \varepsilon_s=2^{-s-1}\varepsilon\ ,
$$
whence
$$
\|\overline{\mu}*f\|_1=\sum_{s=0}^\infty \|\overline{\mu}*f|_{\Lambda_s}\|_1\leq \sum_{s=0}^\infty 2^{-s-1}\varepsilon=\varepsilon. 
$$

Since the $\ell^1$-norm of a locally finite function on $\Lambda(k)$ coincides with the $\ell^1$-norm of the associated locally finite chain, this implies that
$$
\|\overline{\mu}*c\|_1\leq \varepsilon\ .
$$
Moreover, by 
Proposition~\ref{diffusion:homology} we have
$$
\alpha=P_*([c])=P_*([\overline{\mu}*c])\ ,
$$
and the map $P_*\colon H_k^\lf(\mathcal{AD}'_L(X))\to H_k^\lf(X)$ does not increase the $\ell^1$-norm of chains. We can thus conclude that
$$
\|\alpha\|_1= \|P_*([\overline{\mu}*c])\|\leq \|\overline{\mu}*c\|_1\leq \varepsilon\ .
$$
Since $\varepsilon$ is arbitrary, this concludes the proof of  Theorem~\ref{Van-Theor-intro}.

\section{Proof of the Finiteness Theorem}
The proof of the Finiteness Theorem is completely analogous to the proof of the Vanishing Theorem for locally finite homology. Let $W$ and $\calU=\{U_i\}_{i\in\mathbb{N}}$ 
be as in the statement of Theorem~\ref{Fin-Theor-intro}. By definition, there exists a sequence of large sets $\mathcal{W}=\{W_i\}_{i\in\mathbb{N}}$
such that
$U_i\subseteq W_i$ for every $i\in\mathbb{N}$. 

Let $K=X\setminus W$. Since $W$ is large, $K$ is compact. Let $U_{-1}$ be a relatively compact open neighbourhood of $K$ in $X$, and let
$\widehat{\calU}=\calU\cup\{U_{-1}\}=\{U_i\}_{i\geq -1}$. By construction, $\widehat{\calU}$ is a locally finite open cover of $X$ by relatively compact sets. 
We set $W_{-1}=X$ and $\widehat{\calW}=\{W_i\}_{i\geq -1}$, so that each $W_i$, $i\geq -1$ is large, and $U_i\subseteq W_i$ for every $i\geq -1$. We consider
the admissible multicomplex $\mathcal{AD}_L'(X)$ associated to the covers $\widehat{\calU}$, $\widehat{\calW}$ just constructed. 

Let us denote by $\widehat{\calU}|_W$ the restriction of $\widehat{\calU}$ to $W$, i.e.~the open cover of $W$ defined by
$\widehat{\calU}|_W=\{U_i\cap W\}_{i\in I}$, where $I=\{i\geq -1\, |\, W\cap U_i\neq\emptyset\}$.

Let $m=\mult(\calU)$. We claim that, up to replacing $W$ with a smaller large set, we may assume that $\mult(\widehat{\calU}|_W)\leq m$. Indeed, since $U_{-1}$ is relatively compact,
there exists $j_0\in\mathbb{N}$ such that $U_{j}\cap U_{-1}=\emptyset$ for every $j\geq j_0$. We may then set $W=\bigcup_{i\geq j_0} U_i$, and it is now clear that
$\mult(\widehat{\calU}|_W)\leq \mult(\calU)=m$.


Let $k\geq m$, and let us fix a class $\alpha\in H_k^\lf(X)$. As above, we may take an alternating locally finite 
 simplicial cycle $$c \in \, C_{k}^{\lf}(L) \subset C_{k}^{\lf}(\mathcal{AD}_{L}'(X))$$ such that $[S_*(\phi_*( c))] = \alpha$, where $S$ is the restriction of the natural projection.
We have the following:

\begin{lemma}\label{sommazero2}
Let $c=\sum_{\sigma\in\Lambda(k)} a_\sigma\cdot \sigma$. Then there exists $s'\in\mathbb{N}$ such that
$$
\sum_{\sigma\in\Lambda_s} a_\sigma=0
$$
for every $s\geq s'$.
\end{lemma}
\begin{proof}
Since $W$ is large and the cover $\widehat{\calU}$ is locally finite, there exists ${j_0}\in\mathbb{N}$ such that $U_j\subseteq W$ for every $j\geq j_0$.
By Lemmas~\ref{finitiCk} and~\ref{Lemma-numero-di-orbite-tilde-gamma-numerabili},  
 there exists ${s}'\in\mathbb{N}$ such that, if $s\geq {s}'$, then the colors of all the  vertices of the algebraic simplices of $\Lambda_s$ 
are not smaller than $j_0$.

 Since $\mult(\widehat{\calU}|_W)\leq m$, this implies that, if $s\geq s'$, then 
every simplex of $\Lambda_s$ has two distinct vertices of the same color. The conclusion follows just as in the proof of Lemma~\ref{sommazero}.
\end{proof}

 Let us denote  by $f\colon \Lambda(k)\to\mathbb{R}$ the locally finite function associated to the locally finite cycle $c$.
By Theorem~\ref{Gjamenable}, our assumptions on the sequences $\{U_i\}_{i\in\mathbb{N}}$, $\{W_i\}_{i\in\mathbb{N}}$ ensure that there exists $s''\in\mathbb{N}$
such that $\G_s$ is amenable for every $s\geq s''$. Let $\overline{s}=\max\{s',s''\}$. 
We can now apply Proposition~\ref{Prop-locally-finite-diff-op-norm-min-eps} to construct  a
local diffusion operator $\overline{\mu}*\colon \lf(\Lambda(k))\to \lf(\Lambda(k))$ such that
$$
\|\overline{\mu}*f|_{\Lambda_s}\|_1\leq \varepsilon_s=2^{-s-1}\varepsilon\ ,
$$
for every $s\geq\overline{s}$. Thus, if $M=\sum_{s=0}^{\overline{s}-1} \|(\overline{\mu}*f)|_{\Lambda_s}\|_1$, then
$$
\|\overline{\mu}*c\|_1=\|\overline{\mu}*f\|_1=\sum_{s=0}^\infty \|\overline{\mu}*f|_{\Lambda_s}\|_1\leq M+\sum_{s=\overline{s}}^\infty 2^{-s-1} \leq M+1\ . 
$$
Thanks to Proposition~\ref{diffusion:homology} we can now conclude that 
$$
\|\alpha\|_1\leq \|P_*([\overline{\mu}*c])\|_1\leq \|\overline{\mu}*c\|_1 \leq M+1\ .
$$
This concludes the proof of the Finiteness Theorem~\ref{Fin-Theor-intro}.

\chapter{Some results on the simplicial volume\\ of open manifolds}
In this section we provide some applications of the Vanishing and the Finiteness Theorems for locally finite homology. 
In particular, we show that the simplicial volume of the cartesian product of three {tame} open PL manifolds always vanishes 
(recall from Definition~\ref{tame:defn} that an open topological manifold $X$ is \emph{tame} if it is homeomorphic to the internal part of a compact manifold with boundary, and see below for the definition of PL manifold).

\section{PL manifolds}\label{PL:section}

In this section we restrict our attention to PL manifolds, which define a special subclass of triangulable manifolds. For a much broader overview on PL topology we refer the reader to \cite{Wh61, Zeeman-PL, Stallings, RS-PL}.

\begin{Definizione}
A \emph{combinatorial $n$-manifold} is a simplicial complex of dimension $n$ with the property that the link (i.e.~the boundary of the closed star) of every vertex is combinatorially 
equivalent to an $(n-1)$-sphere (i.e.~it has a subdivision which is isomorphic to a subdivision of the boundary of the standard simplex).
\end{Definizione}

\begin{Definizione}\label{def:PL:mflds}
A topological manifold without boundary $X$ is a \emph{PL manifold} if it admits a piecewise linear structure. Thanks to~\cite[Exercise 2.21(i)]{RS-PL}, a topological manifold is PL if and only if it is homeomorphic to the geometric realization of a combinatorial $n$-manifold.
 \end{Definizione}

 A well-known result of Whitehead~\cite{Wh-Trian} ensures that any smooth (indeed, piecewise smooth) manifold is a PL manifold. 
 
\begin{Osservazione}
One may be tempted to assume that all topological triangulable manifolds without boundary are in fact PL manifolds. However, in dimension $n \geq 5$, there exist examples of triangulable manifolds which do not admit any PL structure. We refer the reader to the survey \cite{Rudyak} for such examples.
\end{Osservazione}

By working in the PL category we are allowed to exploit the useful notion of regular neighbourhood:

\begin{Definizione}
Let $K$ be a simplicial complex and let $L$ be a subcomplex of $K$. Let $L'', K''$ be the second barycentric subdivisions of $L,K$, respectively, and observe that $L''\subset K''$. 
The \emph{closed regular neighbourhood}\index{regular neighbourhood!closed} of $L$ in $K$ is the union of all the closed simplices of $K''$ which meet $L''$ (and of their faces). It will be denoted by the symbol $\mathcal{N}(L, K)$\index{$\mathcal{N}(L, K)$}.

Similarly, the \emph{open regular neighbourhood}\index{regular neighbourhood!open} of $L$ in $K$, which will be denoted by $\mathcal{O}(L, K)$\index{$\mathcal{O}(L, K)$}, is the union of all the open simplices in $K''$ whose closures meet $L''$.

Equivalently, the open (resp.~closed) regular neighbourhood of $L$ in $K$ is the union of all the open (resp.~closed) stars in $K''$ of the vertices of $L''$.
\end{Definizione}

The following fundamental result of Whitehead states that any open PL manifold $X$ is homeomorphic to the open regular neighbourhood of a codimension-1 subcomplex of $X$.

\begin{Teorema}[{\cite[Lemma 2.1]{Wh61}}]\label{Teor:Whitehead:Open:Manifolds}
Let $X$ be an open PL $n$-manifold homeomorphic to the geometric realization of a combinatorial $n$-manifold $K$. 
Then, there exists an $(n-1)$-dimensional subcomplex $L \subset K$ such that $X$ is homeomorphic to the open regular neighbourhood $\mathcal{O}(L, K)$ of $L$ in $K$.
If $X$ is tame, then the subcomplex $L$ may be chosen to be finite.
\end{Teorema}

\section{Locally coamenable subcomplexes}

Let us now fix an $n$-dimensional PL manifold $X$, suppose that $X$ is homeomorphic to the geometric realization of a simplicial complex $K$, and let
$P$ be a subcomplex of $K$. For every vertex $v$ of $P''$, we denote by $\overline{B}^n_v$ (resp.~$B^n_v$) the closed star (resp.~open star) of $v$ in $X''$, so that
$$\mathcal{N}(P, X) \cong \bigcup_{v \in \, P''} \overline{B}^n_{v}\, ,\qquad \mathcal{O}(P, X) \cong \bigcup_{v \in \, P''} {B}^n_{v}$$ 
(here and henceforth, for ease of notation we will denote subcomplexes of $K$ (and their subdivisions) and their geometric realizations (which are subsets of $X$) with the same symbol).
We will also denote by $S^{n-1}_v$ the link of $v$ in $K''$ (which coincides with the topological boundary of $\overline{B}^n_v$ in $X$, and is therefore homeomorphic to an $(n-1)$-dimensional sphere).
The following definition will play a fundamental role in the sequel.

\begin{Definizione}
Let $P$ be a subcomplex of $K$. Then $P$ is \emph{locally coamenable} (in $K$ or in $X$) if the following conditions hold: $\dim P\leq \dim K-2=n-2$, and 
for every vertex $v$ of $P''$ the  group $$\pi_1(S^{n-1}_{v} \setminus (P \cap S^{n-1}_{v}))$$ is amenable.
\end{Definizione}

\begin{Esempio}\label{Esempi-locally-coamenable}
 If the codimension of $P\subseteq K$ is at least $3$, then $P$ is locally coamenable.
Indeed, for every vertex $v$ of $P''$, the subset $P\cap S^{n-1}=P''\cap S^{n-1}$ is a subcomplex of $S^{n-1}$ of codimension at least 3, and this readily implies that
$\pi_1(S^{n-1}_{v} \setminus (P \cap S^{n-1}_{v}))$ is trivial, hence amenable.



\end{Esempio}

The following theorems provide criteria for the vanishing or the finiteness of the simplicial volume of the open regular neighbourhood of a locally coamenable subcomplex
 of $K$.
 
\begin{Teorema}\label{teor-co-am}
Let $X=|K|$ be a tame open $n$-dimensional PL manifold,
let $P\subseteq K$ be a finite and locally coamenable subcomplex of $K$, and let $i\geq 2+\dim P $. Then
$$
\| h\|_1=0
$$
for every  $h \in \, H^{\lf}_{i}(\mathcal{O}(P, K))$.

\end{Teorema}

\begin{proof}

Our proof is based on a construction inspired by~\cite[Theorem 5.3]{Loh-Sauer}. Let us set $m=\dim P+1$, and
let $\calU$ be the open cover of $\mathcal{O}(P, K)$ given by the family $\{B^n_v\}$, as $v$ varies in the set $(P'')^0$ of vertices of $P''$. 
We denote by $\mathcal{U} |_{P}$ the restriction of the cover $\calU$ to $P$, i.e.~the open cover of $P$ given by
the family $\{B^n_v\cap P\ |\ B^n_v\in\calU\}$. Since $P$ is finite, the open covers $\calU$ and $\calU|_P$ are finite. Moreover,
it is readily seen that, for every vertex $v$ of $P''$, the set $B^n_v\cap P$ is the open star of $v$ in $P''$. This readily implies that 
$\mult(\calU |_P)=m$. 

We claim that also $\mult(\calU) =m$. 
Of course $\mult(\calU)\geq \mult(\calU|_P)=m$. Suppose now that $B_{v_{1}}^n \cap B_{v_2}^n \cap \cdots \cap B_{v_{m'}}^n \neq \emptyset$, where the $v_i$ are vertices of $P''$. Then there exists a 
simplex $\sigma$ of $K''$ whose set of vertices is equal to
$\{v_1, \cdots, v_{m'}\}$. Since $P''$ is a full subcomplex of $K''$, the simplex $\sigma$ also belongs to $P''$. Thus the internal part of $\sigma$ is contained
in $(B_{v_{1}}^n\cap P) \cap (B_{v_2}^n\cap P) \cap \cdots \cap (B_{v_{m'}}^n\cap P)$. This implies that $m'\leq \mult(\calU|_P)=m$, hence $\mult(\calU)\leq m$ and $\mult(\calU)= m$.

Our next goal is the construction of an open cover of $\mathcal{O}(P, K)$ of multiplicity $m+1=2 + \dim P$ satisfying the hypothesis of Theorem \ref{Van-Theor-intro}. 
To this aim we would like to refine the open cover $\calU$ into an open cover of $\calO(P,K)$ by relatively compact sets, while keeping some control on the multiplicity.

Let us consider locally finite open covers $\mathcal{R}_v$, $v\in P''$, of the real line $\mathbb{R}$ satisfying the following properties: each set in any $\calR_v$ is a bounded interval;  
$\mult(\mathcal{R}_v)= 2$ for every $v \in \, P''$;  $\mult(\mathcal{R}_v \sqcup \mathcal{R}_w) = 3$ for every $v \neq w \in \, P''$. The fact that such open covers indeed exist is observed  in~\cite[Theorem~5.3]{Loh-Sauer}.

Let now $f \colon \mathcal{O}(P, K) \rightarrow \mathbb{R}$ be a proper map, and define a new cover of $\mathcal{O}(P, K)$ as follows:
$$
\mathcal{U}' =\left\{B^n_v \cap f^{-1}(W) \ |\ v \in \, (P'')^0,\  W\in\calR_v\right\}\ .
$$
We check that the open cover $\mathcal{U}'$  satisfies the hypothesis of Theorem \ref{Van-Theor-intro} (where we replace
$X$ by $\calO(P,K)$). 

Using that each $\calR_v$ is a locally finite cover of $\R$ one easily shows that 
$\calU'$ is locally finite. In addition, since $f$ is a proper map and the elements of the $\mathcal{R}_v$ are bounded,  every open set of $\mathcal{U}'$ is relatively compact in $\calO(P,K)$.  

Let us prove that $\mult(\mathcal{U}') \leq m+1$. By contradiction, suppose  that there exists a collection $\{U_1,\ldots,U_{m+2}\} \subset \mathcal{U}'$ of $m+2$ distinct sets 
such that $\bigcap_{i=1}^{m+2} U_i\neq\emptyset$. Let $v_i\in (P'')^0$ be such that $U_i=B^n_{v_i}\cap f^{-1}(W)$ for some $W\in\calR_{v_i}$.
Since $\mult(\calU)=m$, at least one of the following possibilities must hold: there exist distinct indices $i,j,k\in\{1,\ldots,m+2\}$ such that $v_i=v_j=v_k$; or
there exist distinct indices $i,j,k,h$ such that $v_i=v_j$ and $v_k=v_h$. The first case contradicts $\mult(\calR_{v_i})=2$, while the second one
is not compatible with $\mult(\calR_{v_i}\sqcup \calR_{v_k})=3$. Thus $\mult(\mathcal{U}') \leq m+1$.

The cover $\mathcal{U}'$  of $\mathcal{O}(P, K)$ is amenable, because 
for every $v \in \, (P'')^0$, $W \in \, \calR_v$ 
the inclusion of $B_v^n \cap f^{-1}(W)$ into $\calO(P, K)$ factors through $B_v^n$, which is contractible.

Let us now prove that $\mathcal{U}'$ is amenable at infinity. For every $0<r<1$ and every vertex $v$ of $P''$, we denote by $\overline{B}_v^n(r)$ the closed ball of $\calO(P,K)$ of radius $r$ centered at $v$,
i.e.~the set of points in $B^n_v$ whose barycentric coordinate with respect to $v$ is not smaller than $1-r$. For every $i\in\mathbb{N}$, $i\geq 1$, we  set 
$$
K_i = P \cup \left(\bigcup_{v\in (P'')^0} \overline{B}_v^n\left(1 - \frac{1}{i}\right)\right)
$$
(the subcomplex $P$ is automatically contained in the union of balls of radius $1-1/i$ provided that $i\geq n$). 
It is readily seen that the family $\{K_i\}_{i\geq 1}$ provides an exhaustion of $\calO(P,K)$ by compact sets. Therefore, if we set $Z_i=\calO(P,K)\setminus K_i$,
then each $Z_i$ is large in $\calO(P,K)$, and the family $\{Z_i\}_{i\geq 1}$ is locally finite. Let us now arbitrarily order the elements of $\calU'$ (which are countable) by setting
$\calU'=\{F_s\}_{s\in\mathbb{N}}$. Since $\calU'$ is locally finite, only a finite number of elements of $\calU'$ intersect $K_1$. Therefore, there exists $s_0\in\mathbb{N}$ such that,
for every $s\geq s_0$, the set $F_s$ is contained in $Z_{i}$ for some $i\in\mathbb{N}$. Let $i(s)=\max \{i\, |\, F_s\subseteq Z_i\}$ (for every $s\geq s_0$, this maximum exists since
$\bigcap_{i\geq 1} Z_i=\emptyset$), and for every $s\geq s_0$ set $W_s=Z_{i(s)}$. In order to prove that the cover $\mathcal{U}'$ is amenable at infinity it is sufficient
to show that the family $\{W_s\}_{s\geq s_0}$ is locally finite, and that for every $s\geq s_0$ the set $F_s$ is amenable in $W_s$. 

Since $\{Z_i\}_{i\geq 1}$ is locally finite, the local finiteness of $\{W_s\}_{s\geq s_0}$ readily follows from the fact that the map $s\mapsto i(s)$ is finite-to-one. In fact, if this were
not the case, then for some $i\geq 1$ there would be an infinite number of elements of $\calU'$ intersecting $Z_{i}\setminus Z_{i+1}$, which is relatively compact in $\calO(P,K)$. This would contradict
the local finiteness of $\calU'$. We have thus proved that the family $\{W_s\}_{s\geq s_0}$ is locally finite. Let us now fix $s\geq s_0$. By construction, $F_s\subseteq W_s$
(in particular, $F_s\cap P=\emptyset$), and $F_s=B_v^n \cap f^{-1}(W)$ for some $v \in \, (P'')^0$, $W\in\calR_v$. This implies that $F_s$ is contained in the set $G = B_v^n \setminus (B_v^n \cap K_s)$. 

We claim that $G$ has an amenable fundamental group. Since
the inclusion of $F_s$ into $W_s$ factors through $G$, this proves that $F_s$ is amenable in $W_s$, thus concluding
the proof that $\calU'$ is amenable at infinity.
In order to prove the claim,  observe that $G$ is 
a deformation retract of $B^n_v\setminus (B_v^n \cap P)$, which in turn is homotopy equivalent to $\overline{B}_v^n \setminus (\overline{B}_v^n \cap P)$, hence to $S^{n-1}_v \setminus (S^{n-1}_v \cap P)$.
The  claim now follows from the local  amenability of $P$ in $K$. We have thus proved  that $\calU'$ is amenable at infinity.

Now the conclusion follows by applying  Theorem~\ref{Van-Theor-intro} to $\calO(P, K)$ with respect to the open cover $\calU'$.
\end{proof}

\begin{Teorema}\label{teor-co-am2}
Let $X=|K|$ be a compact $n$-dimensional PL manifold, and let $P\subseteq K$ be a locally coamenable sucomplex of $K$. Assume that $i\geq 2+\dim P $. Then
$$
\| h\|_1<+\infty
$$
for every  $h \in \, H^{\lf}_{i}(X\setminus P)$.
\end{Teorema}
\begin{proof}
Let us concentrate our attention on the open manifold $X \setminus P$. Since $X$ is compact, the open set $W =\mathcal{O}(P, K) \setminus P$ is large
in $X\setminus P$.
We would like to obtain an open cover of $W$ satisfying the hypothesis of the Finiteness Theorem \ref{Fin-Theor-intro}. To this end, we slightly modify the construction provided in the proof of the previous theorem.

Let $\calU = \{B_v^n\}_{v \in \, (P'')^0}$ be the open cover of $\calO(P, K)$ given by the open stars of the vertices of  $P''$. 
We have seen in the proof of Theorem~\ref{teor-co-am} that
$\mult(\calU) = 1 + \dim P$. 
By setting $\mathcal{U'} = \{B_v^n \setminus P\ |\ B_v^n \in \, \calU\}$,
we restrict $\calU$ to an open covering $\mathcal{U'}$ of $W$ such that $\mult(\calU')=\mult(\calU)=1+\dim P$.

Let us now refine $\mathcal{U}'$ as follows.
We consider the family of open covers $\{\mathcal{R}_v\}_{v \in \, (P'')^0}$ of ${\R}$
described in the proof of Theorem~\ref{teor-co-am}.
Let $f \colon X \setminus P \rightarrow \mathbb{R}$ be a proper map, and consider the refinement of $\calU'$ given by
$$\mathcal{U}'' = \{(B_n^v \setminus P) \cap f^{-1}(A)\ |\  A \in \, \calR_v\}_{v \in \, (P'')^0}\ .$$
Using that $f$ is proper, one easily checks that every element of $\calU''$ is relatively compact in $X\setminus P$. Moreover, the family $\calU''$ is locally finite
in $X\setminus P$, and 
$\mult(\mathcal{U}'')\leq 2 + \dim P$ (for all these facts, see again the proof of Theorem~\ref{teor-co-am}).

Let us prove that $\calU''$ is amenable at infinity. 
We inductively define $\calO^1 = \calO(P,K)$, $\calO^2 = \calO(P, \calO^1)$ and $\calO^j = \calO(P, \calO^{j-1})$ for every $j \in \, \mathbb{N}$. This construction provides a decreasing sequence of open regular neighbourhoods of $P$ becoming thinner and thinner.   For every $j \in \, \mathbb{N}$ we set
$$Z_j \coloneqq \calO^j  \setminus P.$$ 
By construction, $Z_1 = \calO^1 \setminus P=W$, and
$Z_j$ is large for every $j \in \, \mathbb{N}$. Moreover, 
the family $\{Z_j\}_{j \in \, \mathbb{N}}$ is locally finite in $X\setminus P$. 
The open cover $\calU''$ is countable, hence we can arbitrarily order it by setting $\calU'' = \{F_s\}_{s \in \, \mathbb{N}}$. 
Just as in the proof of Theorem~\ref{teor-co-am}, for every $j\in\mathbb{N}$ we can define 
the maximum integer $s(j)$ such that $F_s \subset Z_j$, and the family $\{W_s\}_{s \in \, \mathbb{N}}$ turns out to be locally finite in $X\setminus P$. 

Finally, by arguing as in the proof of Theorem~\ref{teor-co-am} one proves that  $F_s$ is an amenable subset of $W_s$ for every $s \in \, \mathbb{N}$. 
We can then apply the Finiteness Theorem~\ref{Fin-Theor-intro} to conclude the proof.
\end{proof}

\section{The simplicial volume of the product of three open manifolds}\label{tre-sec}

One of the most striking applications of the Theorem~\ref{teor-co-am} is the following result, which ensures that the simplicial volume of the product of three tame open manifolds necessarily vanishes.
As mentioned in the introduction, there exist no examples of open manifolds $X_1,X_2$ for which the simplicial volume $\|X_1\times X_2\|$ is known to be positive and finite.

\begin{thm:repeated:tre-intro}
Let $X_1, X_2,X_3$ be tame open oriented PL manifolds of positive dimension.
 Then
$$\lVert X_1 \times X_2 \times X_3 \rVert = 0\ .$$
\end{thm:repeated:tre-intro}
\begin{proof}
Let $n_i=\dim X_i$, $i=1,2,3$.
By Theorem \ref{Teor:Whitehead:Open:Manifolds} we know that each $X_i$
is homeomorphic to the open regular neighbourhood $\calO(P_i,X_i)$ of a finite subpolyhedron $P_i \subset X_i$ 
such that $\dim P_i=n_i-1$. The product $P_1\times P_2\times P_3$ naturally embeds as a subpolyhedron in the product
$X_1\times X_2\times X_3$, and it is readily seen that the regular neighbourhood $\calO(P_1\times P_2\times P_3,X_1\times X_2\times X_3)$ is homeomorphic
to the product $\calO(P_1,X_1)\times \calO(P_2,X_2)\times \calO(P_3,X_3)$. 
Let now $n=\dim X_1\times X_2\times X_3=n_1+n_2+n_3$.
The codimension of $P_1\times P_2\times P_3$ in $X_1\times X_2\times X_3$ is equal to 3, so by Example~\ref{Esempi-locally-coamenable} 
$P_1\times P_2\times P_3$ is locally coamenable in $X_1\times X_2\times X_3$. Therefore, the Theorem~\ref{teor-co-am}
implies that, if $i\geq 2+\dim (P_1\times P_2\times P_3)=n-1$ and 
$h\in H^\lf_i(\calO(P_1\times P_2\times P_3,X_1\times X_2\times X_3))$, then $\|h\|_1=0$. 
Now the conclusion follows from the fact that, as observed above, $X_1\times X_2\times X_3$ is homeomorphic to
$\calO(P_1\times P_2\times P_3,X_1\times X_2\times X_3)$.
\end{proof}

\begin{Osservazione}
The proof of the previous corollary shows in fact a bit more: if $X=X_1\times X_2\times X_3$ is the product of three tame open manifolds of positive dimension
and $n=\dim X$, then the $\ell^1$-seminorm vanishes
both on $H^\lf_n(X)$ and on $H^\lf_{n-1}(X)$.
\end{Osservazione}

\bibliographystyle{amsalpha}
\bibliography{biblionote}

\newcommand{\etalchar}[1]{$^{#1}$}
\providecommand{\bysame}{\leavevmode\hbox to3em{\hrulefill}\thinspace}
\providecommand{\MR}{\relax\ifhmode\unskip\space\fi MR }
\providecommand{\MRhref}[2]{%
  \href{http://www.ams.org/mathscinet-getitem?mr=#1}{#2}
}
\providecommand{\href}[2]{#2}
\begin{thebibliography}{BBM{\etalchar{+}}10}

\bibitem[AK16]{Alpert1}
H.~Alpert and G.~Katz, \emph{Using simplicial volume to count multi-tangent
  trajectories of traversing vector fields}, Geom. Dedicata \textbf{180}
  (2016), no.~1, 323--338.

\bibitem[Alp16]{Alpert2}
H.~Alpert, \emph{Using simplicial volume to count maximally broken {M}orse
  trajectories}, Geom. Topol. \textbf{20} (2016), 2997--3018.

\bibitem[BBB{\etalchar{+}}10]{BBBMP}
L.~Bessi{\`e}res, G.~Besson, M.~Boileau, S.~Maillot, and J.~Porti,
  \emph{Collapsing irreducible 3-manifolds with nontrivial fundamental group},
  Invent. Math. \textbf{179} (2010), 435--360.

\bibitem[BBF{\etalchar{+}}14]{BBFIPP}
M.~Bucher, M.~Burger, R.~Frigerio, A.~Iozzi, C.~Pagliantini, and M.~B.
  Pozzetti, \emph{Isometric properties of relative bounded cohomology}, J.
  Topol. Anal. \textbf{6} (2014), no.~1, 1--25.

\bibitem[BBM{\etalchar{+}}10]{geom:book}
L.~Bessi{\`e}res, G.~Besson, S.~Maillot, M.~Boileau, and J.~Porti,
  \emph{Geometrisation of 3-manifolds}, EMS Tracts in Mathematics, vol.~13,
  European Mathematical Society (EMS), Z\"urich, 2010.

\bibitem[BCL18]{BCL}
M.~Bucher, C.~Connell, and J.~F. Lafont, \emph{Vanishing simplicial volume for
  certain affine manifolds}, Proc. Amer. Math. Soc. \textbf{146} (2018),
  1287--1294.

\bibitem[BFP15]{BFP}
M.~Bucher, R.~Frigerio, and C.~Pagliantini, \emph{The simplicial volume of
  3-manifolds with boundary}, J. Topol. \textbf{8} (2015), 457--475.

\bibitem[BFP17]{BFP2}
\bysame, \emph{A quantitative version of a theorem by {J}ungreis}, Geom.
  Dedicata \textbf{187} (2017), 199--218.

\bibitem[BG11]{BuGe}
M.~Bucher and T.~Gelander, \emph{The generalized {C}hern conjecture for
  manifolds that are locally a product of surfaces}, Adv. Math. \textbf{228}
  (2011), 1503--1542.

\bibitem[BK07]{michelleprimo}
M.~Bucher-Karlsson, \emph{Simplicial volume of locally symmetric spaces covered
  by {$SL_3\mathbb{R}/SO(3)$}}, Geom. Dedicata \textbf{125} (2007), 203--224.

\bibitem[BK08a]{Bucher}
\bysame, \emph{The proportionality constant for the simplicial volume of
  locally symmetric spaces}, Colloq. Math. \textbf{111} (2008), 183--198.

\bibitem[BK08b]{Bucher3}
\bysame, \emph{The simplicial volume of closed manifolds covered by $
  \mathbb{H}^ 2\times\mathbb{H}^2$}, J. Topol. \textbf{1} (2008), 584--602.

\bibitem[BK14]{MichelleKim}
M.~Bucher and I.~Kim, \emph{Proportionality principle for the simplicial volume
  of families of $\mathbb{Q}$-rank 1 locally symmetric spaces}, Math. Z.
  \textbf{276} (2014), 153--172.

\bibitem[BK19]{Bal}
F.~Balacheff and S.~Karam, \emph{Schoen conjecture for manifolds with non-zero
  simplicial volume}, Trans. Amer. Math. Soc. \textbf{372} (2019), no.~1-3,
  7071--7086.

\bibitem[Bla16]{Blank}
M.~Blank, \emph{Relative bounded cohomology for groupoids}, Geom. Dedicata
  \textbf{184} (2016), 27--66.

\bibitem[BM99]{BM1}
M.~Burger and N.~Monod, \emph{Bounded cohomology of lattices in higher rank
  {L}ie groups}, J. Eur. Math. Soc. \textbf{1} (1999), 199--235.

\bibitem[BM02]{BM2}
\bysame, \emph{Continuous bounded cohomology and applications to rigidity
  theory}, Geom. Funct. Anal. \textbf{12} (2002), 219--280.

\bibitem[Bro81]{Brooks}
R.~Brooks, \emph{Some remarks on bounded cohomology}, Riemann surfaces and
  related topics: Proceedings of the 1978 Stony Brook Conference (State Univ.
  New York, Stony Brook, N.Y., 1978), Ann. of Math. Stud., vol.~97, Princeton
  Univ. Press, Princeton, N.J., 1981, pp.~53--63.

\bibitem[BRW14]{BRW}
M.~Boileau, J.~H. Rubinstein, and S.~Wang, \emph{Finiteness of 3-manifolds
  associated with non-zero degree mappings}, Comment. Math. Helv. \textbf{89}
  (2014), 33--68.

\bibitem[BS]{BaSa}
I.~Babenko and S.~Sabourau, \emph{Volume entropy semi-norm}, arXiv:1909.10803.

\bibitem[Buc09]{Michelle:fiber}
M.~Bucher, \emph{Simplicial volume of products and fiber bundles}, Discrete
  groups and geometric structures, Contemp. Math., vol. 501, Amer. Math. Soc.,
  Providence, RI, 2009, pp.~79--86.

\bibitem[B{\"u}h11]{Bualg}
T.~B{\"u}hler, \emph{On the algebraic foundations of bounded cohomology}, Mem.
  Amer. Math. Soc. \textbf{214} (2011), xxii+97.

\bibitem[CS19]{Campa}
C.~Campagnolo and R.~Sauer, \emph{Counting maximally broken {M}orse
  trajectories on aspherical manifolds}, Geom Dedicata \textbf{202} (2019),
  no.~1, 387--399.

\bibitem[CW]{CW2}
C.~Connell and S.~Wang, \emph{Some remarks on the simplicial volume of
  nonpositively curved manifolds}, arXiv:1801.08597.

\bibitem[CW19]{CW1}
\bysame, \emph{Positivity of simplicial volume for nonpositively curved
  manifolds with a {R}icci-type curvature condition}, To appear in Groups Geom.
  Dyn., DOI: 10.4171/GGD/512 (2019), 1007--1034.

\bibitem[Der10]{Derbez}
P.~Derbez, \emph{Topological rigidity and {G}romov simplicial volume}, Comment.
  Math. Helv. \textbf{179} (2010), 1--37.

\bibitem[Eil44]{eil-sing}
S.~Eilenberg, \emph{Singular homology theory}, Ann. of Math. \textbf{45}
  (1944), no.~3, 407--444.

\bibitem[ES52]{E-Steenrod}
S.~Eilenberg and N.~Steenrod, \emph{Foundations of algebraic topology},
  Princeton University Press, 1952.

\bibitem[EZ50]{eil-zil}
S.~Eilenberg and J.~A. Zilber, \emph{Semi-simplicial complexes and singular
  homology}, Ann. of Math. \textbf{51} (1950), 499--513.

\bibitem[FFL19]{FLF}
D.~Fauser, S.~Friedl, and C.~L{\"o}h, \emph{Integral approximation of
  simplicial volume of graph manifolds}, Bull. Lond. Math. Soc. \textbf{51}
  (2019), 715--731.

\bibitem[FFM12]{FFM}
S.~Francaviglia, R.~Frigerio, and B.~Martelli, \emph{Stable complexity and
  simplicial volume of manifolds}, J. Topology \textbf{5} (2012), 977--1010.

\bibitem[FLMQ]{FLMQ}
D.~Fauser, C.~L{\"o}h, M.~Moraschini, and J.~P. Quintanilha, \emph{Stable
  integral simplicial volume of $3$-manifolds}, arXiv:1910.06120.

\bibitem[FLPS16]{FLPS}
R.~Frigerio, C.~L{\"o}h, C.~Pagliantini, and R.~Sauer, \emph{Integral foliated
  simplicial volume of aspherical manifolds}, Israel J. Math. \textbf{216}
  (2016), 707--751.

\bibitem[FM]{Maffei}
R.~Frigerio and A.~Maffei, \emph{A remark on the {M}ayer-{V}ietoris double
  complex for singular homology}, arXiv:1912.07736.

\bibitem[FM19]{FMo}
R.~Frigerio and M.~Moraschini, \emph{Ideal simplicial volume of manifolds with
  boundary}, To appear in Int. Math. Res. Not. IMRN, DOI:10.1093/imrn/rny302
  (2019).

\bibitem[FP90]{Piccinini}
R.~Fritsch and A.~Piccinini, \emph{Cellular structures in topology}, Cambridge
  studies in advanced mathematics, no.~19, Cambridge University Press, New
  York, 1990.

\bibitem[FP10]{FP}
R.~Frigerio and C.~Pagliantini, \emph{The simplicial volume of hyperbolic
  manifolds with geodesic boundary}, Algebr. Geom. Topol. \textbf{10} (2010),
  979--1001.

\bibitem[FP12]{FP2}
\bysame, \emph{Relative measure homology and continuous bounded cohomology of
  topological pairs}, Pacific J. Math. \textbf{257} (2012), 95--130.

\bibitem[Fra16]{Franceschini1}
F.~Franceschini, \emph{Proportionality principle for the {L}ipschitz simplicial
  volume}, Geom. Dedicata \textbf{182} (2016), 287--306.

\bibitem[Fra18]{Franceschini2}
\bysame, \emph{A characterization of relatively hyperbolic groups via bounded
  cohomology}, Groups Geom. Dyn. \textbf{12} (2018), 919--960.

\bibitem[Fri]{Fri:amenable}
R.~Frigerio, \emph{Amenable covers and $\ell^1$-invisibility},
  arXiv:1907.10547.

\bibitem[Fri11]{Frigerio}
\bysame, \emph{({B}ounded) continuous cohomology and {G}romov's proportionality
  principle}, Manuscripta Math. \textbf{134} (2011), 435--474.

\bibitem[Fri17]{frigerio-libro}
\bysame, \emph{Bounded cohomology of discrete groups}, Mathematical Surveys and
  Monographs, vol. 227, American Mathematical Society, 2017, 193 pp.

\bibitem[GG12]{GuthGromov}
M.~Gromov and L.~Guth, \emph{Generalizations of the {K}olmogorov-{B}arzdin
  embedding estimates}, Duke Math. J. \textbf{161} (2012), 2549--2603.

\bibitem[GJ99]{GJ}
P.~G. Goerss and J.~F. Jardine, \emph{Simplicial homotopy theory}, Progress in
  Mathematics, Birkh{\"a}user Verlag, Basel, 1999.

\bibitem[Gra01a]{grandis1}
M.~Grandis, \emph{Finite sets and symmetric simplicial sets}, Theory Appl.
  Categ. \textbf{8} (2001), no.~8, 244--252.

\bibitem[Gra01b]{grandis2}
\bysame, \emph{Higher functors for simplicial sets}, Cah. Topol. G\'{e}om.
  Diff\'{e}r. Cat\'{e}g. \textbf{42} (2001), 101--136.

\bibitem[Gro82]{Grom82}
M.~Gromov, \emph{Volume and bounded cohomology}, Inst. Hautes \'Etudes Sci.
  Publ. Math. \textbf{56} (1982), 5--99.

\bibitem[Gro93]{gromovasymptotic}
\bysame, \emph{Asymptotic invariants of infinite groups, {G}eometric group
  theory, vol.~2}, London Math. Soc. Lectures Notes Ser., vol. 182, Cambridge
  Univ. Press, Cambridge, 1993.

\bibitem[Gro99]{Gromovbook}
\bysame, \emph{Metric structures for {R}iemannian and non-{R}iemannian spaces},
  Progress in Mathematics, vol. 152, Birkh{\"a}user, 1999, with appendices by
  M. Katz, P. Pansu, and S. Semmes, translated by S.M. Bates.

\bibitem[Gro09]{gromov-cycles}
\bysame, \emph{Singularities, expanders and topology of maps. {P}art {I}:
  {H}omology versus volume in the spaces of cycles}, Geom. Funct. Anal.
  \textbf{19} (2009), 743--841.

\bibitem[Gut11]{Guth}
L.~Guth, \emph{Volumes of balls in large {R}iemannian manifolds}, Ann. of Math.
  (2) \textbf{173} (2011), 51--76.

\bibitem[Hat02]{hatcher}
A.~Hatcher, \emph{Algebraic topology}, Cambridge University Press, Cambridge,
  2002.

\bibitem[HK01]{HK}
M.~Hoster and D.~Kotschick, \emph{On the simplicial volumes of fiber bundles},
  Proc. Amer. Math. Soc. \textbf{129} (2001), 1229--1232.

\bibitem[IM16]{homotopy}
S.~O. Ivanov and R.~Mikhailov, \emph{On nontriviality of certain homotopy
  groups of spheres}, Homol. Homot. Appl. \textbf{18} (2016), 337--344.

\bibitem[Iva]{ivanov3}
N.~V. Ivanov, \emph{Notes on the bounded cohomology theory}, arXiv:1708.05150.

\bibitem[Iva87]{Ivanov}
\bysame, \emph{Foundation of the theory of bounded cohomology},
  J.~Soviet.~Math. \textbf{37} (1987), 1090--1114.

\bibitem[IY82]{Inoue}
H.~Inoue and K.~Yano, \emph{The {G}romov invariant of negatively curved
  manifolds}, Topology \textbf{21} (1982), 83--89.

\bibitem[Jar04]{jardine}
J.~F. Jardine, \emph{Simplicial approximation}, Theory Appl. Categ. \textbf{12}
  (2004), no.~2, 34--72.

\bibitem[Kan58]{kan-comb}
D.~M. Kan, \emph{A combinatorial definition of homotopy groups}, Ann. of Math.
  \textbf{67} (1958), no.~2, 282--312.

\bibitem[Kat16]{Katz}
G.~Katz, \emph{Complexity of shadows and traversing flows in terms of the
  simplicial volume}, J. Topol. Anal. \textbf{8} (2016), 501--543.

\bibitem[KK15]{KimKue}
S.~Kim and T.~Kuessner, \emph{Simplicial volume of compact manifolds with
  amenable boundary}, J. Topol. Anal. (2015), no.~7, 23--46.

\bibitem[Kue15]{Kuessner}
T.~Kuessner, \emph{Multicomplexes, bounded cohomology and additivity of
  simplicial volume}, Bull. Korean Math. Soc. (2015), no.~52, 1855--1899.

\bibitem[LN19]{Neo3}
J.~F. Lafont and C.~Neofytidis, \emph{Steenrod problem and the domination
  relation}, Topology Appl. \textbf{255} (2019), 32--40.

\bibitem[L{\"o}h06]{Loh}
C.~L{\"o}h, \emph{Measure homology and singular homology are isometrically
  isomorphic}, Math. Z. \textbf{253} (2006), 197--218.

\bibitem[L{\"o}h07]{Lothesis}
\bysame, \emph{Homology and simplicial volume}, Ph.D. thesis, WWU M{\"u}nster,
  2007, available online at
  http://nbn-resolving.de/urn:nbn:de:hbz:6-37549578216.

\bibitem[L{\"o}h08]{Loeh}
\bysame, \emph{Isomorphisms in $l^1$-homology}, M{\"u}nster J. Math. \textbf{1}
  (2008), 237--266.

\bibitem[LP16]{LP}
C.~L{\"o}h and C.~Pagliantini, \emph{Integral foliated simplicial volume of
  hyperbolic 3-manifolds}, Groups Geom. Dyn. \textbf{10} (2016), 825--865.

\bibitem[LS06]{Lafont-Schmidt}
J.~F. Lafont and B.~Schmidt, \emph{Simplicial volume of closed locally
  symmetric spaces of non-compact type}, Acta Math. \textbf{197} (2006),
  129--143.

\bibitem[LS09a]{Loh-Sauer}
C.~L{\"o}h and R.~Sauer, \emph{Degree theorems and {L}ipschitz simplicial
  volume for non-positively curved manifolds of finite volume}, J. Topol.
  \textbf{2} (2009), 193--225.

\bibitem[LS09b]{Loh-Sauer2}
\bysame, \emph{Simplicial volume of {H}ilbert modular varieties}, Comment.
  Math. Helv. \textbf{84} (2009), 457--470.

\bibitem[May92]{may}
J.~P. May, \emph{Simplicial objects in algebraic topology}, Chicago Lectures in
  Mathematics, University of Chicago Press, Chicago, IL, 1992.

\bibitem[McC66]{mccord}
M.~C. McCord, \emph{Singular homology groups and homotopy groups of finite
  topological spaces}, Duke Math. J. \textbf{33} (1966), 465--474.

\bibitem[Mil57]{milnor-geom}
J.~Milnor, \emph{The geometric realization of a semi-simplicial complex}, Ann.
  of Math. \textbf{65} (1957), no.~2, 357--362.

\bibitem[Mon01]{Monod}
N.~Monod, \emph{Continuous bounded cohomology of locally compact groups},
  Lecture notes in Mathematics, no. 1758, Springer-Verlag, Berlin, 2001.

\bibitem[Mon06]{Monod:inv}
\bysame, \emph{An invitation to bounded cohomology}, International Congress of
  Mathematicians. Vol. II, Eur. Math. Soc., Z{\"u}rich, 2006, pp.~1183--1211.

\bibitem[Moo58]{moore}
J.~Moore, \emph{Semi simplicial complexes and {P}ostnikov systems}, Symposium
  de Topolog\'{i}a Algebraica Mexico (1958), 232--248.

\bibitem[Mor18]{Marco:tesi}
M.~Moraschini, \emph{On {G}romov's theory of multicomplexes: the original
  approach to bounded cohomology and simplicial volume}, Ph.D. thesis,
  Universit\`a di Pisa, December 2018, available at
  https://etd.adm.unipi.it/t/etd-11242018-182437/.

\bibitem[Mun84]{munkres}
J.~R. Munkres, \emph{Elements of algebraic topology}, Avalon Publishing, 1984.

\bibitem[MY07]{Yaman}
I.~Mineyev and A.~Yaman, \emph{Relative hyperbolicity and bounded cohomology},
  available online at available online at
  \href{http://www.math.uiuc.edu/~mineyev/math/art/rel-hyp.pdf}{http://www.math.uiuc.edu/{$\sim$}mineyev/math/art/rel-hyp.pdf},
  2007.

\bibitem[Neo17]{Neo1}
C.~Neofytidis, \emph{Degrees of self-maps of products}, Int. Math. Res. Not.
  IMRN \textbf{2017} (2017), 6977--6989.

\bibitem[Neo18]{Neo2}
\bysame, \emph{Ordering {T}hurston's geometries by maps of nonzero degree}, J.
  Topol. Anal. \textbf{10} (2018), 853--872.

\bibitem[Par03]{Park}
H.~S. Park, \emph{Relative bounded cohomology}, Topology Appl. \textbf{131}
  (2003), 203--234.

\bibitem[Pat88]{Paterson}
A.~L.~T. Paterson, \emph{Amenability}, Mathematical Surveys and Monographs,
  no.~29, American Mathematical Society, Providence, RI, 1988.

\bibitem[RS82]{RS-PL}
C.~Rourke and B.~Sanderson, \emph{Introduction to piecewise-linear topology},
  Springer Study Edition, Springer-Verlag, Berlin Heidelberg, 1982, VIII + 126
  pp.

\bibitem[RT03]{rosick-symm}
J.~Rosick\'{y} and W.~Tholen, \emph{Left determined model categories and
  universal homotopy theories}, Trans. Amer. Math. Soc. \textbf{355} (2003),
  no.~2, 3611--3623.

\bibitem[Rud16]{Rudyak}
Y.~Rudyak, \emph{Piecewise linear structures on topological manifolds}, World
  Scientific, 2016, 128 pp.

\bibitem[Sam99]{Samb1}
A.~Sambusetti, \emph{Minimal entropy and simplicial volume}, Manuscripta Math.
  \textbf{99} (1999), 541--560.

\bibitem[Sau09]{Sauer}
R.~Sauer, \emph{Amenable covers, volume and {$L^2$}-{B}etti numbers of
  aspherical manifolds}, J. Reine Angew. Math. \textbf{636} (2009), 47--92.

\bibitem[Sch05]{Sthesis}
M.~Schmidt, \emph{{$L^2$}-{B}etti numbers of $\mathcal{R}$-spaces and the
  integral foliated simplicial volume}, Ph.D. thesis, Westf{\"a}lische
  Wilhelms-Universit\"at M{\"u}nster, 2005.

\bibitem[Spa66]{spanier}
E.~H. Spanier, \emph{Algebraic topology}, McGraw-Hill Book Co., New
  York-Toronto, Ont.-London, 1966, xiv+528 pp.

\bibitem[Spa93]{spanier2}
\bysame, \emph{Singular homology and cohomology with local coefficients and
  duality for manifolds}, Pacific J. Math. \textbf{160} (1993), 165--200.

\bibitem[Sta68]{Stallings}
J.~R. Stallings, \emph{Lectures in polyhedral topology}, Tata institute of
  fundamental research, 1968, 260 pp.

\bibitem[Str]{strz-unp}
K.~Strza{\l}kowski, \emph{Lipschitz simplicial volume of connected sum},
  arXiv:1704.04636.

\bibitem[Str11]{Strom}
J.~Strom, \emph{Modern classical homotopy theory}, Graduate Studies in
  Mathematics, no. 127, American Mathematical Society, Providence, 2011.

\bibitem[Str17]{Str-Lip}
K.~Strza{\l}kowski, \emph{Piecewise straightening and lipschitz simplicial
  volume}, J. Topol. Anal. \textbf{1} (2017), 167--193.

\bibitem[Thu79]{Thurston}
W.~P. Thurston, \emph{The geometry and topology of $3$-manifolds}, Princeton,
  1979, mimeo\-graphed notes.

\bibitem[Whi40]{Wh-Trian}
J.~H.~C. Whitehead, \emph{On ${C}^1$-complexes}, Ann. of Math. \textbf{41}
  (1940), no.~4, 809--824.

\bibitem[Whi61]{Wh61}
\bysame, \emph{The immersion of an open $3$-manifold in {E}uclidean $3$-space},
  Proc. Lond. Math. Soc. \textbf{3} (1961), no.~11, 81--90.

\bibitem[Whi78]{whitehead}
G.~W. Whitehead, \emph{Elements of homotopy theory}, Graduate Texts in
  Mathematics, no.~61, Springer-Verlag, Chicago, IL, 1978.

\bibitem[Zee63]{Zeeman-PL}
E.~C. Zeeman, \emph{Seminar on combinatorial topology}, Institut des {H}autes
  \'{E}tudes {S}cientifiques, 1963.

\bibitem[Zee64]{zeeman}
\bysame, \emph{Relative simplicial approximation}, Math. Proc. Cambridge
  Philos. Soc. \textbf{60} (1964), 39--43.

\end{thebibliography}

\end{document}